\documentclass[12pt]{article}
\title{Dp-finite fields III: inflators and directories}
\author{Will Johnson}

\usepackage{amsmath, amssymb, amsthm}    	
\usepackage{fullpage} 	
\usepackage{amscd}
\usepackage{hyperref}
\usepackage[all]{xy}
\usepackage{centernot}

\DeclareMathOperator*{\forkindep}{\raise0.2ex\hbox{\ooalign{\hidewidth$\vert$\hidewidth\cr\raise-0.9ex\hbox{$\smile$}}}}

\newcommand{\Ab}{\operatorname{Ab}}

\newcommand{\End}{\operatorname{End}}
\newcommand{\bdn}{\operatorname{bdn}}

\newcommand{\rk}{\operatorname{rk}}

\newcommand{\Frac}{\operatorname{Frac}}

\newcommand{\Hom}{\operatorname{Hom}}

\newcommand{\res}{\operatorname{res}}
\newcommand{\Aut}{\operatorname{Aut}}

\newcommand{\Set}{\operatorname{Set}}
\newcommand{\Fun}{\operatorname{Fun}}

\newcommand{\Mod}{\operatorname{Mod}}
\newcommand{\Vect}{\operatorname{Vect}}

\newcommand{\coker}{\operatorname{coker}}

\newcommand{\val}{\operatorname{val}}

\newcommand{\Stab}{\operatorname{Stab}}

\newcommand{\img}{\operatorname{im}}
\newcommand{\coim}{\operatorname{coim}}
\newcommand{\dpr}{\operatorname{dp-rk}}
\newcommand{\redrk}{\operatorname{rk}_0}
\newcommand{\botrk}{\operatorname{rk}_\bot}

\newcommand{\Sub}{\operatorname{Sub}}
\newcommand{\soc}{\operatorname{soc}}
\newcommand{\qsoc}{\operatorname{qsoc}}
\newcommand{\Dir}{\operatorname{Dir}}
\newcommand{\Pro}{\operatorname{Pro}}
\newcommand{\Ind}{\operatorname{Ind}}

\newtheorem{theorem}{Theorem}[section] 
\newtheorem{lemma}[theorem]{Lemma}
\newtheorem{lemmadefinition}[theorem]{Lemma-Definition}
\newtheorem{corollary}[theorem]{Corollary}
\newtheorem{fact}[theorem]{Fact}

\newtheorem{assumption}[theorem]{Assumption}
\newtheorem{conjecture}[theorem]{Conjecture}
\newtheorem{dream}{Dream}[section]

\newtheorem{proposition}[theorem]{Proposition}
\newtheorem{proposition-eh}[theorem]{Proposition(?)}
\newtheorem*{theorem-star}{Theorem}
\newtheorem*{conjecture-star}{Conjecture}
\newtheorem*{lemma-star}{Lemma}

\theoremstyle{definition}
\newtheorem{definition}[theorem]{Definition}
\newtheorem{example}[theorem]{Example}

\theoremstyle{remark}
\newtheorem{remark}[theorem]{Remark}
\newtheorem{claim}[theorem]{Claim}

\newtheorem*{acknowledgment}{Acknowledgments}

\newtheorem*{warning}{Warning}

\newcommand{\Qq}{\mathbb{Q}}

\newcommand{\Rr}{\mathbb{R}}

\newcommand{\Zz}{\mathbb{Z}}
\newcommand{\Kk}{\mathbb{K}}
\newcommand{\Nn}{\mathbb{N}}
\newcommand{\Cc}{\mathcal{C}}
\newcommand{\Dd}{\mathcal{D}}

\newcommand{\Hh}{\mathcal{H}}

\newcommand{\Ff}{\mathbb{F}}

\newcommand{\Oo}{\mathcal{O}}
\newcommand{\mm}{\mathfrak{m}}

\newenvironment{claimproof}[1][\proofname]
               {
                 \proof[#1]
                 
               }
               {
                 \endproof
               }

\begin{document}
\maketitle

\begin{abstract}
  We develop some tools for analyzing dp-finite fields, including a
  notion of an ``inflator'' which generalizes the notion of a
  valuation/specialization on a field.  For any field $K$, let
  $\Sub_K(K^n)$ denote the lattice of $K$-linear subspaces of $K^n$.
  An ordinary valuation on $K$ with residue field $k$ induces
  order-preserving dimension-preserving specialization maps from
  $\Sub_K(K^n)$ to $\Sub_k(k^n)$, satisfying certain compatibility
  across $n$.  An $r$-inflator is a similar family of maps
  $\{\Sub_K(K^n) \to \Sub_k(k^{rn})\}_{n \in \Nn}$ scaling dimensions
  by $r$.  We show that 1-inflators are equivalent to valuations, and
  that $r$-inflators naturally arise in fields of dp-rank $r$.  This
  machinery was ``behind the scenes'' in \S 10 of \cite{prdf}.
  We rework \S 10 of \cite{prdf} using the machinery of
  $r$-inflators.
\end{abstract}

\section{Introduction}

This paper continues \cite{prdf,prdf2}, and is concerned with the
problem of classifying fields of finite dp-rank.  See \cite{dp-add}
for background on dp-rank and \cite{halevi-hasson-jahnke} for
background on the classification problem for NIP fields.
``Dp-minimal'' means dp-rank 1 and ``dp-finite'' means dp-rank $n$ for
some $n < \omega$.

In the present paper, we develop a set of algebraic tools for
analyzing dp-finite fields.  These tools---inflators and
directories---were implicit in \cite{prdf} \S 9-10; we will re-work \S
10 of \cite{prdf} in the language of inflators.  In a future paper
\cite{prdf4}, we will use inflators to carry out a detailed analysis
of fields of dp-rank 2, yielding some new results.

\subsection{Model-theoretic motivation}

Dp-minimal fields were classified in \cite{arxiv-myself}.  Given an
unstable dp-minimal field $K$, one embeds $K$ into a monster model
$\Kk \succeq K$, and uses this strategy:
\begin{enumerate}
\item \label{step-1} Define a group of ``$K$-infinitesimals'' $I_K \le (\Kk,+)$.
\item \label{step-2} Show that $I_K$ is an ideal in a valuation ring
  $\Oo_K$ on $\Kk$.
\item \label{step-3} Show that $\Oo_K$ is henselian.
\item \label{step-4} Use canonical henselian valuations (specifically \cite{JK}) to find a
  henselian valuation $\Oo'$ on $\Kk$ with a controlled residue field.
\item \label{step-5} Use the Ax-Kochen-Ershov principle to determine
  the complete theory of the original fields $K$ and $\Kk$.
\end{enumerate}
All these steps generalize to dp-finite fields, except
Step~\ref{step-2}.  Step~\ref{step-1} was done in \cite{prdf},
Step~\ref{step-3} was done in \cite{prdf2}, Step~\ref{step-4} was done
by Halevi, Hasson, and Jahnke \cite{halevi-hasson-jahnke}, and
Step~\ref{step-5} was done by Sinclair \cite{sinclair}.

The main gap is thus
\begin{conjecture}[Valuation conjecture]\label{val-conj}
  If $K$ is an unstable dp-finite field embedded in a monster model
  $\Kk \succeq K$, then the group $I_K$ of $K$-infinitesimals is an
  ideal in a valuation ring on $\Kk$.
\end{conjecture}
Here, $I_K$ is the group of $K$-infinitesimals constructed in
\cite{prdf}.  By later work (Lemma 5.8, Theorem 5.9 in \cite{prdf2},
and Corollary 6.19 in \cite{prdf}), we can characterize $I_K$ as the
smallest additive subgroup of $\Kk$ that is type-definable over $K$
and has full rank $\dpr(I_K) = \dpr(\Kk)$.

If $J$ is an additive subgroup of $\Kk$, the ``stabilizer''
\begin{equation*}
  \Stab(J) := \{a \in \Kk : a \cdot J \subseteq J\}
\end{equation*}
is always a subring of $\Kk$.  The valuation conjecture says that
$\Stab(I_K)$ is a valuation ring.  The focus of the present paper is
on understanding the algebraic structure of rings $\Stab(J)$ when $J$
is a type-definable group in a dp-finite field.

It is worth noting two related conjectures:
\begin{conjecture}[Shelah conjecture for dp-finite fields]\label{shelah-conj}
  Let $\Kk$ be a saturated, infinite, dp-finite field.  Then $\Kk$
  admits a non-trivial henselian valuation ring.
\end{conjecture}
\begin{conjecture}[Henselianity conjecture for dp-finite fields]\label{hens-conj}
  If $(K,\Oo)$ is a dp-finite valued field, then $\Oo$ is henselian.
\end{conjecture}
These two conjectures are much more likely than the valuation
conjecture.  Assuming the Shelah conjecture, the henselianity
conjecture follows \cite{hhj-v-top}, and the full classification of
dp-finite fields is known \cite{halevi-hasson-jahnke}.  The valuation
conjecture implies the Shelah and henselianity conjectures
\cite{prdf2}.

In the dp-minimal case, the valuation conjecture is true for a very
simple reason.  Let $\Lambda^{00}$ be the poset of type-definable
subgroups $G \le (\Kk,+)$ such that $G = G^{00}$, ordered by
inclusion.  The $K$-infinitesimals $I_K$ are an element of
$\Lambda^{00}$.

If $\dpr(\Kk) = 1$, then $\Lambda^{00}$ is totally ordered.
Therefore, for any $a \in \Kk^\times$,
\begin{equation*}
  a \cdot I_K \subseteq I_K \textrm{ or } I_K \subseteq a \cdot I_K.
\end{equation*}
Equivalently,
\begin{equation*}
  a \cdot I_K \subseteq I_K \textrm{ or } a^{-1} \cdot I_K \subseteq I_K.
\end{equation*}
In other words, $\Stab(I_K)$ is a valuation ring
(Conjecture~\ref{val-conj}).

If $\dpr(K)= r > 1$, the structure of $\Lambda^{00}$ is much less
constrained, though we can still say the following:
\begin{itemize}
\item $\Lambda^{00}$ is a bounded modular lattice.
\item For $n > r$, there are no strict $n$-cubes in $\Lambda^{00}$.
\end{itemize}
Here, a ``strict $n$-cube'' in a modular lattice $M$ means an
unbounded sublattice\footnote{``Unbounded sublattice'' means a
  substructure of the unbounded lattice $(M,\vee,\wedge)$, rather than
  a substructure of the bounded lattice $(M,\vee,\wedge,\top,\bot)$.}
isomorphic to the boolean algebra of size $2^n$.  We say that a
modular lattice $M$ is \emph{cube-bounded} if there is a uniform
finite bound on the size of strict cubes in $M$.

Note that any two incomparable elements yield a strict 2-cube.  Thus
for $r = 1$, the second condition says that $\Lambda^{00}$ is totally
ordered.  We can think of ``cube-bounded'' as the natural
generalization of ``totally ordered'' to higher ranks.

Now, one would like to somehow deduce the valuation conjecture from
cube-boundedness of $\Lambda^{00}$.

\subsection{Multi-valuation rings, magic fields, and pedestals}\label{sec:has-pc}
For any small model $K_0 \preceq \Kk$, let $\Lambda_{K_0}$ denote the
lattice of type-definable $K_0$-linear subspaces of $\Kk$.  If $K_0$
is a \emph{magic subfield} (Definition~\ref{def:magic}), then
$\Lambda_{K_0} \subseteq \Lambda^{00}$, and so $\Lambda_{K_0}$ is
cube-bounded.  We preferentially work in $\Lambda_{K_0}$ rather than
$\Lambda^{00}$ because the lattice operations are simpler; no
$(-)^{00}$'s are involved.

Let $r$ be maximal such that a strict $r$-cube exists in
$\Lambda_{K_0}$.  Say that $A$ is a \emph{$K_0$-pedestal} if $A$ is
the base of a strict $r$-cube in $\Lambda_{K_0}$.

Say that a subring $R$ of a field $K$ is a \emph{multi-valuation ring on
  $K$} if $R$ is a finite intersection of valuation rings on $K$.  Say
that a subset $M \subseteq K$ is a \emph{multi-valuation ideal} on $K$
if $M$ is an $R$-submodule of $K$ for some multi-valuation ring $R$ on
$K$.

It turns out that pedestals can be used to verify the valuation
conjecture:
\begin{lemma-star}[Pedestal criterion]
  Let $K$ be an unstable dp-finite field.  Let $\Kk \succeq K$ be a
  monster model.  Let $K_0 \succeq \Kk$ be a magic subfield.  Let $A$
  be a $K_0$-pedestal.
  \begin{enumerate}
  \item If $A$ contains a non-zero multi-valuation ideal, then $I_K$
    is a valuation ideal.
  \item If $\Stab(A)$ contains a non-zero multi-valuation ideal, then
    $I_K$ is a valuation ideal.
  \end{enumerate}
\end{lemma-star}
This is essentially Theorem~8.11 in \cite{prdf2}.\footnote{There is
  some subtlety in the second point when $A = 0$.  See
  Lemma~\ref{true-pc} for the details.}

\subsection{Flattening and inflators}
Let $(M,\vee,\wedge,\bot)$ be a lower-bounded modular lattice.  In \S
9.4 of \cite{prdf}, we defined a set of \emph{quasi-atoms} in $M$, and
a modular pregeometry on the quasi-atoms.  Let $M^\flat$ be the
lattice of closed sets in this pregeometry.  Then $M^\flat$ is an
atomic modular lattice.  By Corollary~9.39.2 in \cite{prdf}, every
element in $M$ determines a closed set in the pregeometry, yielding an
order-preserving map
\begin{equation*}
  f : M \to M^\flat.
\end{equation*}
We call $M^\flat$ the \emph{flattening of $M$}, and $f : M \to
M^\flat$ the \emph{flattening map}.

Now let $\Kk$ be a dp-finite field.  Fix a magic subfield $K_0$, and
suppress it from the notation.  Let $\Lambda_n$ be the lattice of
type-definable $K_0$-linear subspaces of $\Kk^n$.  For any $A \in
\Lambda_1$, we can build a family of maps
\begin{align*}
  \Sub_K(K^n) \to \Lambda_n &\to [A^n,\Kk^n] \to [A^n,\Kk^n]^\flat \\
  V \mapsto V & \mapsto V + A^n \mapsto f(V + A^n).
\end{align*}
Here the notation $\Sub_R(M)$ denotes the lattice of $R$-submodules of
$M$, for any ring $R$ and $R$-module $M$.

As predicted in Speculative Remark~10.10 of \cite{prdf}, there is a
$K_0$-algebra $R$ and a semisimple $R$-module $M$ such that
\begin{equation*}
  [A^n,\Kk^n]^\flat \cong \Sub_R(M^n)
\end{equation*}
for each $n$.  We get a family of maps
\begin{equation*}
  \varsigma_n : \Sub_\Kk(\Kk^n) \to \Sub_R(M^n).
\end{equation*}
These maps satisfy the following properties, predicted in Speculative
Remark~10.10 of \cite{prdf}:
\begin{enumerate}
\item \label{uc1} Each $\varsigma_n$ is order-preserving.
\item Each $\varsigma_n$ is $GL_n(K_0)$-equivariant.
\item \label{uc3} There is compatibility with $\oplus$:
  \begin{equation*}
    \varsigma_{n+m}(V \oplus W) = \varsigma_n(V) \oplus \varsigma_m(W)
  \end{equation*}
\item The map $\varsigma_n$ scales lengths by a fixed factor:
  \begin{equation*}
    \ell_R(\varsigma_n(V)) = \ell_R(M) \cdot \dim_\Kk(V).
  \end{equation*}
\end{enumerate}
We call such a configuration an \emph{inflator} on $\Kk$.  If $d =
\ell_R(M)$, we call this a \emph{$d$-inflator}; a $d$-inflator
inflates lengths by a factor of $d$:
\begin{equation*}
  \ell_R(\varsigma_n(V)) = d \cdot \dim_\Kk(V).
\end{equation*}

It is helpful to bundle the lattices $\Sub_R(M^\bullet)$ into a
multi-sorted structure
\begin{equation*}
  \Dir_R(M) := (\Sub_R(M^1),\Sub_R(M^2),\Sub_R(M^3),\ldots)
\end{equation*}
with the poset structure and $GL_n(K_0)$-action on $\Sub_R(M^n)$,
and the connecting maps
\begin{equation*}
  \oplus : \Sub_R(M^n) \times \Sub_R(M^m) \to \Sub_R(M^{n+m}).
\end{equation*}
We call such structures \emph{directories}.  Then conditions
(\ref{uc1})-(\ref{uc3}) say that $\varsigma_\bullet$ is a \emph{morphism
  of directories}.

Inflators can be seen as generalized valuation data.  In fact, if $K$
is a valued field with valuation ring $\Oo$, maximal ideal $\mm$, and
residue field $k$, then there is a 1-inflator
\begin{align*}
  \Dir_K(K) & \to \Dir_k(k) \\
  \Sub_K(K^n) & \to \Sub_k(k^n) \\
  V & \mapsto (V \cap \Oo^n + \mm^n)/\mm^n.
\end{align*}
Essentially all 1-inflators arise this way.

More generally, if $R = \Oo_1 \cap \cdots \cap \Oo_d$ is an
intersection of $d$ incomparable valuation rings on a field $K$, there
is a natural $d$-inflator on $K$ essentially given by
\begin{align*}
  \Dir_K(K) & \to \Dir_R(R/J) \\
  \Sub_K(K^n) & \to \Sub_R(R^n/J^n) \\
  V & \mapsto (V \cap R^n + J^n)/J^n,
\end{align*}
where $J$ is the Jacobson radical $\mm_1 \cap \cdots \cap \mm_d$ of
$R$.  These are the motivating examples of $d$-inflators.  Thus the
intuition is that any inflator should have something to do with a
multi-valuation ring.

\subsection{The fundamental ring}
Fix a $d$-inflator $\varsigma : \Dir_K(K) \to \Dir_S(M)$, where $M$ is a
semisimple $S$-module of length $d$.  For any endomorphism $\varphi
\in \End_S(M)$, let $\Theta_\varphi$ be the graph of $\varphi$:
\begin{equation*}
  \Theta_\varphi = \{(x,\varphi(x)) : x \in M\} \subseteq M^2.
\end{equation*}
Similarly, for $b \in K$ let $\Theta_b$ be the line with slope $b$:
\begin{equation*}
  \Theta_b = \{(x,bx) : x \in K\} \subseteq K^2.
\end{equation*}
It turns out that the set
\begin{equation*}
  \{(b,\varphi) : \varsigma_2(\Theta_b) = \Theta_\varphi\}
\end{equation*}
is the graph of a ring homomorphism
\begin{equation*}
  \widehat{\res} : R \to \End_S(M)
\end{equation*}
for some subring $R \subseteq K$.  We call $R$ the \emph{fundamental
  ring} and $\widehat{\res}$ the \emph{generalized residue map}.

For 1-inflators, $R$ is the corresponding valuation ring and
$\widehat{\res}$ is the usual residue map.  Similarly, if $\varsigma$ is
the $n$-inflator induced by a multi-valuation ring $R$, then the
fundamental ring is $R$.

If $\Kk$ is a saturated unstable dp-finite field, if $A$ is a
$K_0$-pedestal, and if $\varsigma$ is the induced inflator, then the
fundamental ring $R_\varsigma$ of $\varsigma$ turns out to be the stabilizer
ring of $A$:
\begin{equation*}
  R_\varsigma = \Stab(A) := \{x \in \Kk : x \cdot A \subseteq x\}.
\end{equation*}
In particular, if $R_\varsigma$ were a multi-valuation ring, the Valuation
Conjecture would hold in $\Kk$, by the Pedestal Criterion of
\S\ref{sec:has-pc}.  We say that $\varsigma$ has \emph{multi-valuation
  type} if its fundamental ring is a multi-valuation ring.

In fact, we only need $\Stab(A)$ to contain a non-zero multi-valuation
ideal on $K$.  We say that $\varsigma$ is \emph{weakly multi-valuation
  type} if $R_\varsigma$ contains a non-zero multi-valuation ideal on $K$.

\subsection{Breaking and repairing}
It would therefore be nice if every inflator was weakly
multi-valuation type.  Unfortunately, one can produce unwanted
examples of inflators using the identity
\begin{equation*}
  \dim_K(V) + \dim_K(W) = \dim_K(V + W) + \dim_K(V \cap W).
\end{equation*}
For example, if $(K,\sigma)$ is a difference field, there is a
2-inflator
\begin{align*}
  \Dir_K(K) &\to \Dir_K(K) \times \Dir_K(K) \\
  V & \mapsto (V + \sigma(V), V \cap \sigma(V)).
\end{align*}
The fundamental ring turns out to be the fixed field of $\sigma$,
which usually contains no multi-valuation ideal on the big field $K$.

Fortunately, there is a way to ``twist'' or ``mutate'' inflators that
seems to undo the corrupting influence of the $(V,W) \mapsto (V + W, V
\cap W)$ map.  If
\begin{equation*}
  \varsigma : \Dir_K(K) \to \Dir_S(M)
\end{equation*}
is a $d$-inflator, and $L$ is a 1-dimensional subspace of $K^m$, one
can define a new inflator $\varsigma'$ by the formula
\begin{equation*}
  \varsigma'_n(V) = \varsigma_{nm}(V \otimes L).
\end{equation*}
We call $\varsigma'$ a \emph{mutation} of $\varsigma$.  It turns out that
$R_{\varsigma'} \supseteq R_{\varsigma}$.  In fact,
\begin{equation*}
  R^\infty_\varsigma = \{R_{\varsigma'} : \varsigma' \textrm{ a mutation of } \varsigma\}
\end{equation*}
turns out to be a directed union, and $R^\infty_\varsigma$ is a
multi-valuation ring on the field $K$.

In the case of inflators on dp-finite fields, this gives the
construction of a weakly definable non-trivial multi-valuation ring.
This was already proven in \cite{prdf}, Theorem~10.25, but the proof
we give here is the original proof using inflators.

\subsection{Directions for future work}

One could be more optimistic, and conjecture the following:
\begin{dream}\label{dream-0}
  If $\varsigma$ is an inflator, then some mutation $\varsigma'$ of $\varsigma$ is
  weakly multi-valuation type.
\end{dream}
It turns out that Dream~\ref{dream-0} implies the Valuation
Conjecture~\ref{val-conj}.  If $\varsigma$ comes from a pedestal $A$, and
$\varsigma'$ is obtained from $\varsigma$ by mutation, then for some reason
$\varsigma'$ also comes from a pedestal $A'$.  So the Pedestal Criterion
of \S\ref{sec:has-pc} applies.

Unfortunately, Dream~\ref{dream-0} is too optimistic, and there are
examples where no mutation of $\varsigma$ is weakly multi-valuation type.

The inflators on dp-finite fields have an additional technical
property called \emph{malleability} (see Definition~\ref{def:mal}).
Most of the bad examples of inflators fail to be malleable, so we
could conjecture
\begin{dream}\label{dream-1}
  If $\varsigma$ is a malleable inflator, then some mutation $\varsigma'$ of
  $\varsigma$ is weakly multi-valuation type.
\end{dream}
This is probably still false, but much closer to the truth.  In a
future paper \cite{prdf4}, we will investigate malleable 2-inflators
and prove enough of a characterization to get the following
description of infinitesimals:
\begin{theorem}[to appear in \cite{prdf4}]
  Let $\Kk$ be an unstable dp-finite field of rank 2 and
  characteristic 0 in which the valuation conjecture fails---$I_K$ is
  not a valuation ideal.  Then there is a field $(L,\partial,\val)$
  with a derivation and a valuation such that
  \begin{equation*}
    (\Kk,+,\cdot,I_K) \equiv (L,+,\cdot,J)
  \end{equation*}
  where $J = \{x \in L : \val(x) > 0 < \val(\partial x)\}$.
  Furthermore, the derivation and valuation on $L$ satisfy some
  independence conditions---for example every set of the form
  \begin{equation*}
    \{x \in L : \val(a - x) > \gamma \textrm{ and } \val(b - \partial x) \ge 0\}
  \end{equation*}
  is non-empty.
\end{theorem}
Using this, we show
\begin{theorem}[to appear in \cite{prdf4}] \label{to-generalize}
  If $K$ is an unstable dp-finite field of dp-rank 2 and
  characteristic 0, then $K$ admits a unique definable V-topology.
  The canonical topology is also definable, and refines this
  V-topology.
\end{theorem}
The hope is to generalize these results to higher ranks.  If
Theorem~\ref{to-generalize} were generalized to all ranks, it would
imply the Shelah and henselianity conjectures, completing the
classification.

\subsection{Notation and conventions}
Following \cite{prdf2} but not \cite{prdf}, the monster model will be
denoted $\Kk$.  We will always assume that $\Kk$ is a field, and never
assume that $\Kk$ is a pure field.  We will sometimes abuse
terminology and use ``saturated'' to mean ``sufficiently saturated and
sufficiently strongly homogeneous,'' rather than the official meaning
of ``saturated in the size of the model.''

If $A$ is an additive subgroup of a field $K$, we will call the ring
\begin{equation*}
  \Stab(A) = \{x \in K : x \cdot A \subseteq A\}
\end{equation*}
the ``stabilizer'' of $A$, for want of a better name.  Sometimes the
scare quotes will be omitted.

Rings are always unital but not always commutative.  An $R$-module is
always a left $R$-module; the category of $R$-modules is denoted
$R\Mod$, or $R\Vect$ if $R$ is a field.  If $K$ is a field, then
$K\Vect^f$ will denote the full subcategory of finite-dimensional
$K$-vector spaces.

Lattices are not assumed to have top and bottom elements.  A ``bounded
lattice'' is a lattice with top and bottom elements.  We use $\vee,
\wedge, \bot, \top$ to denote the (bounded) lattice operations.  A
homomorphism of lattices need not preserve $\top, \bot$ when they
exist, but a bounded lattice homomorphism must preserve $\top, \bot$.
A sublattice need not contain $\top, \bot$, but a bounded sublattice
must.

A modular lattice is ``atomic'' if it is lower-bounded ($\bot$ exists)
and every element is a finite join of atoms.
A modular lattice has ``finite length'' if there is a maximal chain of
finite length.  By Jordan-H\"older, this implies all maximal chains
have finite length.  A modular lattice is ``semisimple'' if it is an
atomic modular lattice of finite length.

For objects in abelian categories, ``semisimple'' will always mean
``semisimple of finite length,'' i.e., a finite sum of simple objects,
even in cases where the more general notion of semisimple would make
sense.  So an object $A$ is semisimple if and only if the modular
lattice of subobjects is semisimple.

A subquotient of an object in an abelian category is a quotient of a
subobject, or equivalently, a subobject of a quotient.  If $A$ is an
object in an abelian category, we let
\begin{itemize}
\item $\End_\Cc(A)$ denote the endomorphism ring of $A$
\item $\Sub_\Cc(A)$ denote the modular lattice of subobjects of $A$.
\item $\Dir_\Cc(A)$ denote the \emph{directory} of $A$; see
  Definition~\ref{def:the-directory}.
\end{itemize}
If $\Cc$ is $R\Mod$ or $K\Vect$, we will shorten the subscript to $R$
or $K$.

We adopt the following definition from \cite{prdf}, Definition 8.3:
\begin{definition}\label{def:magic}
  Let $\Kk$ be a saturated dp-finite field.  A small submodel $K_0
  \preceq \Kk$ is \emph{magic} if for every type-definable subgroup $G
  \le (\Kk^n,+)$, we have
  \begin{equation*}
    K_0 \cdot G \subseteq G \implies G = G^{00}.
  \end{equation*}
  In other words, type-definable $K_0$-linear subspaces of $\Kk^n$ are
  00-connected.
\end{definition}
By \cite{prdf} (Theorem~8.4 and the proof of Corollary~8.7), all
sufficiently large submodels of $\Kk$ are magic.  In particular, magic
subfields exist.

We adopt the following changes in notation from \cite{prdf},
introduced in \S8 of \cite{prdf2}:
\begin{itemize}
\item The lattice of type-definable $K_0$-subspaces of $\Kk^n$ is
  denoted $\Lambda_n$, not $\mathcal{P}_n$.
\item We refer to the bases of maximal strict cubes in $\Lambda_1$ as
  \emph{$K_0$-pedestals}, rather than ``special groups.''
\end{itemize}
Later we will give a more general notion of ``pedestal'' in an
abstract setting; see Definition~\ref{def:pedestal1.5}.

Following \cite{prdf2}, a \emph{multi-valuation ring} on a field $K$
is a finite intersection of valuation rings on $K$.  Multi-valuation
rings are the same thing as Bezout domains with finitely many maximal
ideals; see \S6 in \cite{prdf2}.

For most of the paper there will be a small but infinite field $K_0$
lurking in the background.  All the fields will extend $K_0$, all the
rings will be $K_0$-algebras, all the abelian categories will be
$K_0$-linear abelian categories, and all the additive subgroups of
$\Kk$ will be $K_0$-linear subspaces.  In the dp-finite setting, $K_0$
will generally be a magic subfield.  In other cases one can usually
take $K_0$ to be $\Qq$ of $\Ff_p^{alg}$.  The field $K_0$ serves only
one purpose, which is to provide a simple criterion for being a
multi-valuation ring:
\begin{lemma-star}
  Fix $q_1, \ldots, q_n$ distinct elements of $K_0$.  Let $K$ be a
  field extending $K_0$ and $R$ be a $K_0$-subalgebra of $K$.  Then
  the following are equivalent:
  \begin{itemize}
  \item $R$ is an intersection of $n$ valuation rings on $K$.
  \item For any $x \in K$, at least one of the following is in $R$:
    \begin{equation*}
      x, \frac{1}{x - q_1}, \ldots, \frac{1}{x - q_n}.
    \end{equation*}
  \end{itemize}
\end{lemma-star}
See Lemma~\ref{algebra-case}.  Note that for $n = 1$ and $q_1 = 0$,
this generalizes the usual test for a valuation ring.

\subsection{Outline}
The paper is divided into three parts.  Each part begins
with a brief synopsis.  Part~\ref{part:1} defines directories and
inflators, and works through their basic algebraic theory.  In
particular, we see how inflators generalize valuations.
Part~\ref{part:2} constructs inflators on dp-finite fields.  A major
theme is lifting the analysis of modular lattices in \S 9 in
\cite{prdf} to the more natural setting of abelian categories.
Part~\ref{part:3} continues the algebraic investigation of inflators,
showing how they naturally give rise to multi-valuation rings.  As an
application, this gives the construction of non-trivial
multi-valuation rings on unstable dp-finite fields.

There are several appendices.  The first two,
Appendices~\ref{app:ab}-\ref{app:ind} review the category theory of
abelian categories and pro-objects.  Appendix~\ref{app:dirs}
double-checks some ``obvious'' statements from Part~\ref{part:1}.
Appendix~\ref{app:dirs2} contains some further speculations on
directories.  This speculation is important for motivating
directories, but not important for the main line of proofs.  Finally,
Appendix~\ref{app:ranks} is a remark on subadditive rank functions on
abelian categories, which helps motivate the notion of ``reduced
rank.''

\begin{acknowledgment}
The author would like to thank
\begin{itemize}
\item Martin Hils and Franziska Jahnke for the invitation to speak at
  the conference ``Model Theory of Valued Fields and Applications.''
  This paper began as supplementary notes for the talk.
\item Jan Dobrowolski, for some helpful discussions about groups of
  finite burden.
\item Meng Chen, for hosting the author at Fudan University, where
  this research was carried out.
\end{itemize}
{\tiny This material is based upon work supported by the National Science
Foundation under Award No. DMS-1803120.  Any opinions, findings, and
conclusions or recommendations expressed in this material are those of
the author and do not necessarily reflect the views of the National
Science Foundation.}  
\end{acknowledgment}

\tableofcontents

\part{Directories and inflators} \label{part:1}
Sections \ref{sec:directories}-\ref{sec:basic-theory} work through the
basic theory of directories and inflators.

Section \ref{sec:directories} defines the category of directories.
This category helps simplify the definition and construction of
inflators.  For example, we could define a $d$-inflator on a field $K$
as a family of maps
\begin{equation*}
  \varsigma_n : \Sub_K(K^n) \to \Sub_R(M^n)
\end{equation*}
where $R$ is a ring and $M$ is a semisimple $R$-module of length $d$,
satisfying the following conditions:
\begin{enumerate}
\item \label{cu1} Each $\varsigma_n$ is order-preserving
\item $\varsigma_n$ is $GL_n(K_0)$-equivariant
\item \label{cu3} The $\varsigma_n$ are compatible with $\oplus$
\item \label{cu4} Each $\varsigma_n$ scales lengths by a factor of $d$.
\end{enumerate}
The first three conditions are meaningful in greater generality, and
it is helpful to group them into the notion of a \emph{morphism of
  directories}.  The inflators we construct in Part~\ref{part:2} will
be constructed as a composition of simpler directory morphisms, each
satisfying conditions \ref{cu1}-\ref{cu3}, but only the final composition
will satisfy \ref{cu4}.  Moreover, most of the directory morphisms we
will use are generated out of a few simple examples in
\S\ref{sec:directory-morphisms}.  Conditions \ref{cu1}-\ref{cu3} only
need to be checked on these generating examples.  Thus, the category
of directories helps to suppress many of the boring details in proofs.

Also, the ring $R$ and module $M$ are not important; what we really
care about is the collection of lattices
\begin{equation*}
  (\Sub_R(M), \Sub_R(M^2), \Sub_R(M^n), \ldots).
\end{equation*}
For example, in Part~\ref{part:2} we will construct an inflator
\begin{equation*}
  \varsigma_n : \Sub_\Kk(\Kk^n) \to D_n
\end{equation*}
on dp-finite field $\Kk$.  The codomain lattices $D_n$ will be
constructed abstractly.  So the lattices are canonical, not the ring
$R$ and module $M$.  The notion of \emph{directory} focuses in on the
important information.

In \S\ref{sec:lambda} we show that if $\Kk$ is a saturated field,
possibly with extra structure, then there are three natural
directories:
\begin{itemize}
\item $\Delta_\bullet$, where $\Delta_n$ is the lattice of definable
  $K_0$-linear subspaces of $\Kk^n$.
\item $\Lambda_\bullet$, where $\Lambda_n$ is the lattice of
  type-definable $K_0$-linear subspaces of $\Kk^n$.
\item $\Lambda^{00}_\bullet$, where $\Lambda^{00}_n$ is the quotient
  of $\Lambda_n$ by 00-commensurability.
\end{itemize}
These structures certainly feel like directories, but there are a few
things to check.  Section~\ref{sec:lambda} is the only appearance of
model theory in Part~\ref{part:1}.

In \S\ref{sec:inflators} we precisely define inflators, and give a few
examples.  Sections \ref{sec:aaron}-\ref{sec:dorothy} cover the
natural examples of inflators arising from valuation rings, field
extensions, and multi-valuation rings.  Section \ref{garbage-ex} goes
through the unnatural examples of inflators that derail the analysis
of dp-finite fields.

In \S\ref{sec:basic-theory} we begin to analyze inflators, a project
which is continued in Part~\ref{part:3}.  In
\S\ref{sec:r-i}-\ref{sec:1-fold-classify} we show how to use the
inflator axioms to construct the generalized residue map and
fundamental ring, and we classify 1-inflators.  In
\S\ref{sec:tame-locus} we give a criterion for whether the fundamental
ring is a multi-valuation ring.  Lastly, \S\ref{sec:malleable}
introduces the notion of ``malleable'' inflators.  Malleability rules
out most of the unwanted inflators of \S\ref{garbage-ex}, but
continues to hold for the inflators on dp-finite fields constructed in
Part~\ref{part:2}.  The hope, then, is to classify malleable
inflators.  Malleability will play a key role in the sequel
\cite{prdf4}.

\section{Directories} \label{sec:directories}
Fix a small infinite field $K_0$.  See Appendix~\ref{app:ab} for a
review of abelian categories.
\begin{definition}\label{def:the-directory}
  Let $A$ be an object in a $K_0$-linear abelian category
  $\mathcal{C}$.  The \emph{directory of $A$} is the multi-sorted
  structure
  \begin{equation*}
    \Dir(A) := (\Sub(A), \Sub(A^2), \Sub(A^3), \ldots)
  \end{equation*}
  with the following functions and relations:
  \begin{itemize}
  \item The lattice structure on each $\Sub_\Cc(A^n)$.
  \item For any $n, m$, the map
    \begin{equation*}
      \oplus : \Sub(A^n) \times \Sub(A^m) \to \Sub(A^{n+m})
    \end{equation*}
  \item For each $n$, the action of $GL_n(K_0)$ on $\Sub(A^n)$.
  \end{itemize}
  We write $\Dir_\Cc(A)$ when we need to specify the category $\Cc$.
  
  A \emph{directory} is a structure isomorphic to $\Dir_\Cc(A)$ for
  some $K_0$-linear abelian category $\Cc$ and some $A \in \Cc$.
\end{definition}
Perhaps this should be called a ``$K_0$-linear directory.''  But we
will have no use for plain/$\Zz$-linear directories.


If $D_\bullet$ is a directory, each $D_n$ is a bounded modular
lattice.  There is a natural action of the $n$th symmetric group
$\mathcal{S}_n$ on $D_n$ via permutation matrices in $GL_n(K_0)$.
\begin{definition}
  Let $d$ be a nonnegative integer.  A directory $D_\bullet \cong
  \Dir_\Cc(A)$ \emph{has length $d$} if one of the following
  equivalent conditions holds:
  \begin{itemize}
  \item $A$ has length $d$.
  \item $D_1$ is a modular lattice of length $d$.
  \item For every $n$, $D_n$ is a modular lattice of length $dn$.
  \end{itemize}
  We say that $D$ \emph{has finite length} if $D_\bullet$ has length
  $d$ for some $d \in \Zz_{\ge 0}$.  We let $\ell(D_\bullet)$ denote
  the length of $D_\bullet$, if it is finite.
\end{definition}
\begin{definition}
  A directory $D_\bullet \cong \Dir_\Cc(A)$ is \emph{semisimple} if
  one of the following equivalent conditions holds:
  \begin{itemize}
  \item $A$ is semisimple (of finite length)
  \item $D_1$ is an atomic modular lattice of finite length.
  \item For every $n$, $D_n$ is an atomic modular lattice of finite
    length.
  \end{itemize}
\end{definition}
There are two intuitions for $\Dir_\Cc(A)$.  On one hand, the
directory can be viewed as a generalization of the endomorphism
$K_0$-algebra $\End_\Cc(A)$.  One can interpret $\End_\Cc(A)$ in
$\Dir_\Cc(A)$ by Proposition~\ref{prop:interpret-end} below.  In the
case of semisimple directories, the directory and endomorphism algebra
appear to even be bi-interpretable---see
Proposition~\ref{prop:semisimple-bi-interp}.

On the other hand, $\Dir_\Cc(A)$ is also a generalization of the
subobject lattice $\Sub_\Cc(A)$.  The subobject lattice $\Sub_\Cc(A)$
is trivially interpretable in $\Dir_\Cc(A)$, and most of the
configurations that yield maps between subobject lattices yield maps
between directories (see \S \ref{sec:directory-morphisms}).

\begin{remark}
  Very loosely, we can regard $\Dir_\Cc(A)$ as a version of
  $\Sub_\Cc(A)$ decorated with extra information which makes things
  like synthetic projective geometry work better.  For example, for
  each $n$ there is a corrspondence
  \begin{equation*}
    K \mapsto \Sub_K(K^n)
  \end{equation*}
  from skew fields to $(n-1)$-dimensional projective spaces.  For $n
  \ge 4$, this is a perfect correspondence.  But for $n = 3$ the map
  is not onto, because of non-Desarguesian projective planes.  For $n
  \le 2$, the map is not injective, because projective lines are
  structureless.  In contrast,
  \begin{equation*}
    K \mapsto \Dir_K(K)
  \end{equation*}
  is a perfect correspondence between skew fields and length-1
  directories.  (See \S\ref{sec:all-that}.)  The ``extra structure''
  of $\Sub_K(K^n)$ for $n > 1$ overcomes the problems of
  non-Desarguesian projective planes and structureless projective
  lines.  The general form of synthetic projective geometry for
  directories is the statement that any directory $(D_1,D_2,\ldots)$
  is isomorphic to $\Dir_\Cc(A)$ for some object $A$ in an abelian
  category $\Cc$.

  Of course we ``cheated'' and made this true by definition of
  ``directory.''  It would be better to define directories
  axiomatically as a collection of bounded modular lattices satisfying
  some axioms.  See Appendix~\ref{app:dirs2} for some speculation on
  what the axioms might be.  Having a list of axioms would simplify
  life in the cases when we need to prove that a certain structure is
  a directory, as in \S\ref{sec:lambda} and \S\ref{sec:dirflat} below.
\end{remark}

\subsection{Neighborhoods}\label{sec:neighborhoods}
If $A$ is an object in an abelian category $\mathcal{C}$, we define
the \emph{neighborhood} of $A$ to be the full subcategory of objects
isomorphic to subquotients of finite powers of $A$.  The neighborhood
of $A$ is in a sense the smallest abelian full subcategory
$\mathcal{C}' \subseteq \mathcal{C}$ containing $A$ and all its
subobjects.

If $\Cc'$ is the neighborhood of $A \in \Cc$, then $\Dir_{\Cc'}(A)
\cong \Dir_\Cc(A)$.  In particular, the neighborhood of $A$ determines
the directory of $A$.  The converse appears to be true: the directory
of $A$ seems to determine the neighborhood of $A$ up to equivalence of
categories.  See \S\ref{sec:directory-speculation}.

\subsection{Semisimple directories}\label{sec:all-that}
For any (noncommutative) ring $R$, let $M_n(R)$ denote the ring of $n
\times n$ matrices.  Recall the Artin-Wedderburn theorem:
\begin{theorem}[Artin-Wedderburn] \label{thm:aw}
  Let $R$ be a ring.  The following are equivalent:
  \begin{enumerate}
  \item \label{aw1} $R$ is a semisimple ring, i.e., $R$ is semisimple
    as an $R$-module.
  \item \label{aw2} $R$ is a finite product
    \begin{equation*}
      M_{n_1}(D_1) \times \cdots \times M_{n_k}(D_k)
    \end{equation*}
    of matrix rings over division rings $D_i$.
  \end{enumerate}
\end{theorem}
The same circle of ideas which prove the Artin-Wedderburn theorem
also prove the following proposition (see \S\ref{app:ssd} for the proof):
\begin{proposition}\label{prop:to-prove}
  Let $A$ be a semisimple object in an abelian category $\Cc$.
  \begin{enumerate}
  \item The directory $\Dir_\Cc(A)$ is isomorphic to $\Dir_R(R)$, where
    $R$ is a semisimple ring, namely $\End_\Cc(A)^{op}$.
  \item The directory $\Dir_\Cc(A)$ is isomorphic to $\Dir_S(M)$,
    where $S$ is a finite product $D_1 \times \cdots \times D_n$ of
    division rings, and $M$ is a finitely-generated $S$-module.
  \end{enumerate}
\end{proposition}
Conversely, if $R$ is a semisimple ring and $M$ is a finitely
generated $R$-module, then $M$ is a semisimple $R$-module.  So we get
the following characterization of semisimple directories:
\begin{theorem}\label{thm:semisimple-directories}
  Let $D_\bullet$ be a directory.  The following are equivalent:
  \begin{enumerate}
  \item $D_\bullet$ is a semisimple directory.
  \item $D_\bullet \cong \Dir_R(M)$ for some semisimple ring $R$ and
    finitely generated $R$-module $M$.
  \item $D_\bullet \cong \Dir_R(R)$ for some semisimple ring $R$.
  \item $D_\bullet \cong \Dir_R(M)$ where $R$ is a product of division
    algebras and $M$ is a finitely generated $R$-module.
  \end{enumerate}
\end{theorem}

\subsection{Morphisms of directories} \label{sec:directory-morphisms}
\begin{definition}\label{def:morphisms}.
  If $D_\bullet$ and $D'_\bullet$ are directories, a \emph{morphism}
  from $D_\bullet$ to $D'_\bullet$ is a system of maps $f_n : D_n \to
  D'_n$ satisfying the following constraints:
  \begin{enumerate}
  \item For each $n$, the map $f_n$ is order-preserving:
    \begin{equation*}
      x \ge y \implies f_n(x) \ge f_n(y).
    \end{equation*}
  \item For each $n$, the map $f_n$ is $GL_n(K_0)$-equivariant.
  \item The maps $f_\bullet$ are compatible with $\oplus$:
    \begin{equation*}
      f_{n+m}(x \oplus y) = f_n(x) \oplus f_n(y).
    \end{equation*}
  \end{enumerate}
\end{definition}
Note that we do not require $f_n$ to preserve the lattice structure.
\begin{example} \label{from-morphism}
  For any morphism $f : A \to B$ in $\Cc$, there are pushforward and
  pullback morphisms
  \begin{align*}
    f^* : \Dir(B) \to \Dir(A) \\
    f_* : \Dir(A) \to \Dir(B).
  \end{align*}
  defined by inverse and direct image along the componentwise maps
  $f^{\oplus n} : A^n \to B^n$.  These are functorial in the obvious
  way:
  \begin{align*}
    (f \circ g)_* &= f_* \circ g_* \\
    (f \circ g)^* &= g^* \circ f^*.
  \end{align*}
  Moreover,
  \begin{enumerate}
  \item If $f$ is an isomorphism, then $f_*$ and $f^*$ are
    isomorphisms, and $(f^{-1})_* = f^*$.
  \item If $f$ is a monomorphism, each map $f_{*,n} : \Sub(A^n) \to
    \Sub(B^n)$ is injective, and $f^* \circ f_* = id$.
  \item If $f$ is an epimorphism, each map $f^*_n : \Sub(B^n) \to
    \Sub(A^n)$ is injective, and $f_* \circ f^* = id$.
  \end{enumerate}
\end{example}
(We verify that $f^*$ and $f_*$ are directory morphisms in
Proposition~\ref{prop:ffs1} below.)
\begin{example}\label{from-functor}
  If $F : \Cc \to \Cc'$ is a left-exact $K_0$-linear functor, and $A
  \in \Cc$, then there is a morphism of directories
  \begin{equation*}
    F_* : \Dir_\Cc(A) \to \Dir_{\Cc'}(F(A)).
  \end{equation*}
  The map $\Sub_\Cc(A^n) \to \Sub_{\Cc'}(F(A)^n)$ is defined by
  sending a monomorphism
  \begin{equation*}
    X \stackrel{i}{\hookrightarrow} A^n
  \end{equation*}
  to
  \begin{equation*}
    F(X) \stackrel{F(i)}{\hookrightarrow} F(A^n) \cong F(A)^n.
  \end{equation*}
  This gives a well-defined map on subobjects.
\end{example}
(We verify that $F_*$ is a directory morphism in
Proposition~\ref{prop:ffs2} below.)
\begin{example}
  If $F : \Cc \to \Cc'$ is a ($K_0$-linear) equivalence of categories,
  then
  \begin{equation*}
    F_* : \Dir_\Cc(A) \to \Dir_{\Cc'}(F(A))
  \end{equation*}
  is an isomorphism of directories.
\end{example}

\subsection{Interval subdirectories}\label{sec:intervals}
\begin{proposition}\label{prop:sq}
  Let $D_\bullet$ be a directory and $a \le b$ be elements of $D_1$.
  Let $D^{[a,b]}_n$ be the interval $[a^{\oplus n},b^{\oplus n}]$
  inside $D_n$.
  \begin{enumerate}
  \item $D^{[a,b]}_\bullet$ forms a substructure of
    $D_\bullet$, i.e., $D^{[a,b]}_\bullet$ is closed under $\oplus,
    \vee, \wedge,$ and the $GL_\bullet(K_0)$-action.
  \item  The substructure $D^{[a,b]}_\bullet$ is itself a
    directory.
  \item  The inclusion maps $i_n : D^{[a,b]}_n
    \stackrel{\subseteq}{\to} D_n$ form a directory morphism $i :
    D^{[a,b]}_\bullet \to D_\bullet$.
  \item There is a directory morphism $r : D_\bullet \to
    D^{[a,b]}_\bullet$ given by
    \begin{equation*}
      r_n(x) = (x \vee a^{\oplus n}) \wedge b^{\oplus n} = (x \wedge
      b^{\oplus n}) \vee a^{\oplus n},
    \end{equation*}
    and $r$ is a retract of $i$: $r \circ i$ is the identity map on
    $D^{[a,b]}_\bullet$.
  \item \label{sq5-o} If $f : D' \to D$ is some directory morphism, then
    $f$ factors through $D^{[a,b]} \to D$ if and only if $f_n(x) \in
    [a^{\oplus n},b^{\oplus n}]$ for all $n$ and all $x \in D'_n$.
  \item \label{sq6-o} If $D_\bullet = \Dir(C)$ for some object $C$ in an
    abelian category, and if $a, b$ correspond to subobjects $A
    \subseteq B \subseteq C$, then $D^{[a,b]}_\bullet$ is isomorphic
    to $\Dir(B/A)$ via the maps from the isomorphism theorems.
  \end{enumerate}
\end{proposition}
See \S\ref{app:interval-proof} for the proof.
\begin{definition}\label{def:intervals}
  Let $D_\bullet$ be a directory.  An \emph{interval subdirectory} is
  a directory of the form $D^{[a,b]}_\bullet$.  A morphism of
  directories is an \emph{interval inclusion} or an \emph{interval
    retract} if it is one of the maps
  \begin{align*}
    i : D^{[a,b]}_\bullet \hookrightarrow D_\bullet \\
    r : D_\bullet \twoheadrightarrow D^{[a,b]},
  \end{align*}
  of Proposition~\ref{prop:sq}, respectively.
\end{definition}
\begin{lemma}
  Let $D_\bullet$ be a directory and $a, b$ be two elements of $D_1$.
  Then there is an isomorphism of directories $D^{[a \wedge b, a]} \to
  D^{[b,a \vee b]}$ given at each level by the usual isomorphism
  \begin{align*}
    [a^n \wedge b^n, a^n] & \to [b^n, a^n \vee b^n] \\
    x & \mapsto x \vee b^n.
  \end{align*}
\end{lemma}
\begin{proof}
  We may assume $D_\bullet = \Dir(C)$, and $a, b$ correspond to
  subobjects $A, B \subseteq C$.  By
  Proposition~\ref{prop:sq}.\ref{sq6-o}, the map in question is the
  isomorphism
  \begin{equation*}
    \Dir(C)^{[A \cap B, A]} \cong \Dir(A/(A \cap B)) \cong
    \Dir((A+B)/B) \cong \Dir(C)^{[B,A+B]},
  \end{equation*}
  where the middle isomorphism $\Dir(A/(A \cap B)) \cong
  \Dir((A+B)/B)$ is induced by the standard isomorphism $A/(A \cap B)
  \cong (A+B)/B$.
\end{proof}

\subsection{Products of directories}\label{sec:products}
If $M$ and $M'$ are two lattices, the product $M \times M'$ is
naturally a lattice (as lattices are an algebraic theory), and the
order on $M \times M'$ is determined as follows:
\begin{equation*}
  (a,a') \le (b,b') \iff (a \le b \textrm{ and } a' \le b').
\end{equation*}
If $\Cc, \Cc'$ are two abelian categories, the product category $\Cc
\times \Cc'$ is itself an abelian category (\cite{cat-sheaves},
Remark~8.3.6(i)).  For any object $(A,A') \in \Cc \times \Cc'$, one
has
\begin{equation*}
  \Sub_{\Cc \times \Cc'}(A,A') = \Sub_\Cc(A) \times \Sub_\Cc(A'),
\end{equation*}
because a subobject of $(A,A')$ is a pair $(B,B')$ where $B$ is a
subobject of $A$ and $B'$ is a subobject of $A'$.
\begin{lemmadefinition}
  If $D_\bullet$ and $D'_\bullet$ are two directories, there is a
  product directory $(D \times D')_\bullet$ given by
  \begin{equation*}
    (D_1 \times D'_1, D_2 \times D'_2, \cdots)
  \end{equation*}
  in which the structure maps are given componentwise; for example if
  $(V,V') \in D_n \times D'_n$ and $(W,W') \in D_m \times D'_m$, then
  \begin{equation*}
    (V,V') \oplus (W,W') := (V \oplus W, V' \oplus W') \in D_{n+m} \times D'_{n+m}.
  \end{equation*}
  Furthermore, the two projections from $D \times D'$ to $D$ and to
  $D'$ are both morphisms of directories.
\end{lemmadefinition}\label{ld:products}
\begin{proof}
  Write $D_\bullet$ as $\Dir_{\Cc}(A)$ and $D'_\bullet$ as
  $\Dir_{\Cc'}(A')$.  Let $A''$ be the object $(A,A')$ in the product
  category $\Cc \times \Cc'$.  Then for any $n$,
  \begin{equation*}
    \Sub_{\Cc \times \Cc'}(A'') = \Sub_\Cc(A) \times \Sub_{\Cc'}(A').
  \end{equation*}
  Then $\Dir_{\Cc \times \Cc'}(A'')$ is the desired product directory.
  The projections of $\Dir_{\Cc \times \Cc'}(A'')$ onto $\Dir_\Cc(A)$
  and $\Dir_\Cc'(A')$ are induced by the (left-)exact projection
  functors
  \begin{align*}
    \Cc \times \Cc' & \to \Cc \\
    \Cc \times \Cc' & \to \Cc'.
  \end{align*}
  as in Example~\ref{from-functor}.  Therefore these projections are
  morphisms of directories.
\end{proof}
\begin{remark}\label{rem:whatever}
  If $A, A'$ are two objects in an abelian category $\Cc$, the (left-)exact functor
\begin{equation*}
  \oplus : \Cc \times \Cc \to \Cc
\end{equation*}
yields by Example~\ref{from-functor} a morphism of directories
\begin{align*}
  \Dir_\Cc(A) \times \Dir_\Cc(A') &\to \Dir_\Cc(A \oplus A') \\
  (V,W) & \mapsto V \oplus W
\end{align*}
by Example~\ref{from-functor}.\footnote{Actually, the morphism is
  $(V,W) \mapsto (V \oplus W)^T$, where $(-)^T$ is the transpose map
  $A^n \oplus (A')^n \to (A \oplus A')^n$.}  This morphism is rarely
an isomorphism, unless every subobject of $A^n \oplus (A')^n$ happens
to be of the form $V \oplus W$ for some $V \in \Sub_\Cc(A^n)$ and $W
\in \Sub_\Cc((A')^n)$.\footnote{This condition looks a lot like
  model-theoretic orthogonality.  When $A, A'$ are finite length, this
  condition should be equivalent to the condition that the simple
  factors of $A$ are pairwise non-isomorphic to the simple factors of
  $A'$.}
\end{remark}
\begin{remark}\label{ring-splitting}
  If $R_1, R_2$ are two rings, the category $(R_1 \times R_2)\Mod$ is
  equivalent to the product category $R_1\Mod \times R_2\Mod$.  If
  $M_i$ is an $R_i$-module for $i = 1, 2$, then
  \begin{equation*}
    \Dir_{R_1 \times R_2}(M_1 \times M_2) \cong \Dir_{R_1}(M_1) \times \Dir_{R_2}(M_2).
  \end{equation*}
  In particular,
  \begin{equation*}
    \Dir_{R_1 \times R_2}(R_1 \times R_2) \cong \Dir_{R_1}(R_1) \times
    \Dir_{R_2}(R_2).
  \end{equation*}
  Combined with Theorem~\ref{thm:semisimple-directories}, this implies
  that the semisimple directories are exactly those of the form
  \begin{equation*}
    \Dir_{k_1}(k_1^{d_1}) \times \cdots \times \Dir_{k_n}(k_n^{d_n})
  \end{equation*}
  where $k_1, \ldots, k_n$ are division algebras over $K_0$, and $d_1,
  \ldots, d_n$ are positive integers.
\end{remark}

Usually, we aren't very interested in category-theoretic constructions
in the category of directories.  Nevertheless, the following fact will
come in handy:
\begin{proposition}\label{prop:dir-prod}
  If $D^1, D^2$ are two directories, the product directory $D^1 \times
  D^2$ of Lemma-Definition~\ref{ld:products} is the category-theoretic
  product in the category of directories.
\end{proposition}
\begin{proof}
  Let $T$ be any directory.  We must show bijectivity of the map
  \begin{equation*}
    \Hom(T,D^1 \times D^2) \to \Hom(T,D^1) \times \Hom(T,D^2)
  \end{equation*}
  induced by the projections $D^1 \times D^2 \to D^1, D^2$.
  Equivalently, if $f^1, f^2$ are two systems of maps
  \begin{align*}
    f^1_n : T_n &\to D^1_n \\
    f^2_n : T_n &\to D^2_n,
  \end{align*}
  and $g$ is the induced system of maps
  \begin{equation*}
    g_n : T_n \to D^1_n \times D^2_n,
  \end{equation*}
  then we must show that
  \begin{equation}
    g_\bullet \in \Hom(T,D^1 \times D^2) \iff (f^1_\bullet \in
    \Hom(T,D^1) \textrm{ and } f^2_\bullet \in
    \Hom(T,D^2)). \label{jk-trivial}
  \end{equation}
  But (\ref{jk-trivial}) is somehow automatic, because\ldots
  \begin{itemize}
  \item \ldots $D^1_n \times D^2_n$ is the category-theoretic product
    of $D^1_n$ and $D^2_n$ in the category of posets.
  \item \ldots $D^1_n \times D^2_n$ is the category-theoretic product
    of $D^1_n$ and $D^2_n$ in the category of sets with $GL_n(K_0)$-actions.
  \item \ldots $D^1_\bullet \times D^2_\bullet$ is the
    category-theoretic product of $D^1_\bullet$ and $D^2_\bullet$ in
    the category of multi-sorted structures $(M_1,M_2,\ldots)$ with
    operators $\oplus_{m,n} : M_m \times M_n \to M_{m+n}$.
  \end{itemize}
  So, somehow everything works because the structure on
  $D^1_\bullet \times D^2_\bullet$ was defined componentwise.
\end{proof}

\section{Directories from model theory} \label{sec:lambda}
Let $\Kk$ be a field, possibly with extra structure, assumed to be
very saturated and very homogeneous.  Let $K_0$ be a small, infinite
subfield.  For each $n$, let
\begin{itemize}
\item $\Delta_n$ denote the lattice of definable $K_0$-linear
  subspaces of $\Kk^n$.
\item $\Lambda_n$ denote the lattice of type-definable\footnote{We
  mean ``type-definable over a small set.''  Unless we have an
  inaccessible cardinal on hand, there is some ambiguity in ``small.''
  In what follows, it should suffice to fix a small cardinal $\kappa_0
  \gg \aleph_0$ such that $\Kk$ is $\kappa$-saturated for some $\kappa
  \gg \kappa_0$.  Then we can take $\Lambda$ to be the lattice of
  $K_0$-linear subspaces that are type-definable over a small model of
  size less than $\kappa_0$.  The small cardinal $\kappa_0$ needs to
  be bigger than the size of the theory, and larger than the size of a
  magic subfield.  Further properties, like being a regular cardinal
  or a limit cardinal, seem to be unnecessary.}  $K_0$-linear
  subspaces of $\Kk^n$.
\item $\Lambda^{00}_n$ denote the quotient of $\Lambda_n$ by the
  00-commensurability relation:
  \begin{equation*}
  G \approx H \iff \left(G/(G \cap H) \textrm{ and } H/(G \cap H) \textrm{
    are bounded} \right).
  \end{equation*}
\end{itemize}
Let $\Lambda_\bullet$ be the multi-sorted structure
\begin{equation*}
  (\Lambda_1, \Lambda_2, \Lambda_3, \ldots)
\end{equation*}
with the following functions and relations:
\begin{itemize}
\item The lattice structure on each $\Lambda_n$
\item For any $n, m$, the map
  \begin{equation*}
    \oplus : \Lambda_n \times \Lambda_m \to \Lambda_{n+m}
  \end{equation*}
\item For each $n$, the action of $GL_n(K_0)$ on $\Lambda_n$.
\end{itemize}
Define the structures $\Lambda_\bullet^{00}, \Delta_\bullet$
similarly.
\begin{theorem}\label{thm:dtd00-dir}
  The structures $\Delta_\bullet, \Lambda_\bullet,
  \Lambda^{00}_\bullet$ are ($K_0$-linear) directories.
\end{theorem}
We give the proof over the next three sections
\S\ref{sec:def-setting}-\ref{sec:tdef00-setting}.  The cases of
$\Delta_\bullet$ and $\Lambda_\bullet$ are intuitively unsurprising,
though there are some details to check.  The case of
$\Lambda^{00}_\bullet$ is slightly more subtle.
\begin{remark}\label{rem:del00}
  One could define $\Delta^{00}_\bullet$ analogously, but it is always
  identical to $\Delta_\bullet$.  Indeed, if $G, H \in \Delta_n$ and
  $G \approx H$, then the groups $G/(G \cap H)$ and $H/(G \cap H)$ are
  bounded.  They are also interpretable, so they are \emph{finite}.
  They are also $K_0$-vector spaces, and $K_0$ is infinite, so $G/(G
  \cap H)$ and $H/(G \cap H)$ are in fact \emph{trivial}, implying $G
  = H$.  Thus $\approx$ is a trivial equivalence relation on
  $\Delta_n$, and $\Delta^{00}_n = \Delta_n$.
\end{remark}
\begin{remark}
  If $\Kk$ is NIP, we can alternatively view $\Lambda^{00}_n$ as the
  set of $G \in \Lambda_n$ such that $G = G^{00}$.  Note that
  $\Lambda^{00}_n$ is not a sublattice of $\Lambda_n$---the
  lattice operations are
  \begin{align*}
    G \vee H &= G + H \\ G \wedge H &= (G \cap H)^{00}.
  \end{align*}
\end{remark}
\begin{remark}
  If $\Kk$ is dp-finite and $K_0$ is a magic subfield, then the
  relation $\approx$ on $\Lambda_n$ is trivial, so
  $\Lambda^{00}_\bullet \cong \Lambda_\bullet$.
\end{remark}

\subsection{The definable case}\label{sec:def-setting}
Let $\Dd$ be the abelian category of interpretable $K_0$-vector
spaces.\footnote{There's an ambiguity here.  The most general sort of
  interpretable $K_0$-vector space would be a vector space
  $(V,+,\cdot)$ such that (1) the underlying set $V$ is interpretable
  in $\Kk$, (2) the addition operation $V \times V \to V$ is
  definable, and (3) for each $a \in K_0$, the map $V \to V$ sending
  $x$ to $ax$ is definable.  There is also a more restricted version
  where the map $x \mapsto ax$ is definable uniformly across $a$.  The
  choice of $\Dd$ ultimately doesn't matter, since the ``neighborhood'' of
  $\Kk$ in $\Dd$ is the same in each case (see \S\ref{sec:neighborhoods}).
  If you're unhappy with this state of affairs, you can also take the
  minimal necessary $\Dd$, which is defined similar to $\Hh$ in the
  next section, but replacing ``type-definable'' with ``definable.''}
We can view $\Kk$ as an object in $\Dd$, and then there is an obvious
isomorphism
\begin{equation*}
  \Delta_n \cong \Sub_\Dd(\Kk^n)
\end{equation*}
for each $n$.  Then $\Delta_\bullet$ is simply $\Dir_\Dd(\Kk)$.

\begin{remark}\label{rem:delta-variant}
  If $R$ is a definable subring of $\Kk$ containing $K_0$, then there
  is a smaller directory $\Delta^R_\bullet$, where $\Delta^R_n$ is the
  lattice of definable $R$-submodules of $\Kk^n$.  This structure is a
  directory because it is the directory of $\Kk$ in the category of
  definable $R$-modules.
\end{remark}
\subsection{The type-definable case}\label{sec:tdef-setting}
For $\Lambda_\bullet$, we should be able to proceed analogously to
$\Delta_\bullet$, replacing the category of interpretable $K_0$-vector
spaces with the category of hyper-interpretable $K_0$-vector spaces.
It's not clear to me that this category is well-defined or
well-behaved, so we instead build a minimal sufficient category $\Hh$.
Morally, $\Hh$ is the neighborhood of $\Kk$ in the full category of
hyper-interpretable $K_0$-vector spaces.

We prove the following proposition and theorem in \S\ref{sec:annoy}.
\begin{proposition}\label{prop:here-is-hh-copy}
  There is a $K_0$-linear pre-additive category $\Hh$ in which
  \begin{itemize}
  \item an object is a quotient $A/B$ where $B, A \in \Lambda_n$
    for some $n$, and $A \ge B$.
  \item a morphism from $A/B$ to $A'/B'$ is a $K_0$-linear function $f
    : A/B \to A'/B'$ such that the set
    \begin{equation*}
      \{(x,y) \in A \times A' : y + B' = f(x + B)\}
    \end{equation*}
    is type-definable.
  \item the composition of $f : A/B \to A'/B'$ and $g : A'/B' \to
    A''/B''$ is the usual composition $g \circ f$.
  \item the $K_0$-vector space structure on $\Hom_\Hh(A/B,A'/B')$ is
    induced by the usual operations, i.e., induced as a subspace of
    $\Hom_{K_0\Vect}(A/B,A'/B')$.
  \end{itemize}
\end{proposition}
\begin{theorem}\label{thm:h-lambda}
  Let $\Hh$ be the category of Proposition~\ref{prop:here-is-hh-copy}.
  \begin{itemize}
  \item $\Hh$ is a $K_0$-linear abelian category.
  \item The forgetful functor $\Hh \to K_0\Vect$ is a $K_0$-linear
    exact functor.
  \item $\Lambda_\bullet$ is isomorphic to $\Dir_\Hh(\Kk)$, and is
    therefore a directory.
  \end{itemize}
\end{theorem}
\begin{remark}
  There is a $K_0$-linear functor $\Kk\Vect^f \to \Hh$ sending
  $\Kk^n$ to the object $\Kk^n/0$ in $\Hh$.  This functor is exact and
  faithful (see Remark~\ref{out-of-finvec}).  The composition
  \begin{equation*}
    \Kk\Vect^f \to \Hh \to K_0\Vect
  \end{equation*}
  is equivalent to the forgetful functor $\Kk\Vect^f \to
  K_0\Vect$.  This configuration will be important in
  \S\ref{sec:pedestals}-\ref{sec:model-theory-inflator}.
\end{remark}

\subsection{Serre quotients and 00-commensurability}\label{sec:tdef00-setting}
In the category $\mathcal{H}$, the small/bounded objects form a Serre
subcategory.  We define $\mathcal{H}^{00}$ to be the Serre quotient.
(See \S\ref{serre-quot} for a review of Serre subcategories and
Serre quotients.)

Concretely, a morphism from $X$ to $Y$ in $\mathcal{H}^{00}$ is an equivalence class of type-definable
$K_0$-linear subspaces $\Gamma \subseteq X \times Y$ such that
\begin{itemize}
\item $\Gamma + (0 \times Y) \approx X \times Y$.
\item $\Gamma \cap (0 \times Y) \approx 0 \times 0$.
\end{itemize}
where $\approx$ denotes 00-commensurability.  Two subspaces $\Gamma$
and $\Gamma'$ are regarded as equivalent if $\Gamma \approx \Gamma'$.
These equivalence classes of relations are essentially the
``endogenies'' of (\cite{stab-groups}, \S1.5) or the
``quasi-homomorphisms'' of (\cite{frank-jan}, \S3).  By analogy with
abelian varieties, we call $\Hh^{00}$ the \emph{isogeny category}.

\begin{proposition} \label{prop:h-lambda-00}
  $\mathcal{H}^{00}$ is a well-defined $K_0$-linear abelian category,
  and
  \begin{equation}
    \Dir_{\mathcal{H}^{00}}(\Kk) \cong \Lambda_\bullet^{00}, \label{oh-that}
  \end{equation}
  so $\Lambda_\bullet^{00}$ is a directory.
\end{proposition}
This follows by basic facts about Serre quotients, discussed in
\S\ref{serre-quot}.  The identity (\ref{oh-that}) follows by
Proposition~\ref{oh-that-2}.

\begin{remark}\label{rem:that-requires-serre-to-be-exact}
  The localization functor $\mathcal{H} \to \mathcal{H}^{00}$ is
  exact, yielding a directory morphism
  \begin{equation*}
    \Lambda_\bullet \to \Lambda^{00}_\bullet
  \end{equation*}
  The $n$th component of this morphism,
  \begin{equation*}
    \Lambda_n \to \Lambda^{00}_n,
  \end{equation*}
  is simply the quotient map.
\end{remark}

Now assume $\Kk$ is dp-finite.  Recall Definition~\ref{def:magic}: a
subfield $K_0$ is \emph{magic} if it is large enough to ensure that
type-definable $K_0$-linear subspaces of $\Kk^n$ are 00-connected.
These exist by \cite{prdf}, Corollary~8.7.

Assuming $K_0$ is magic, it follows that the localization map
\begin{equation*}
  \mathcal{H} \to \mathcal{H}^{00}
\end{equation*}
is an equivalence of categories, because all non-trivial $A/B \in
\mathcal{H}$ are unbounded\footnote{Because $K_0$ is magic, $A =
  A^{00}$ and $B = B^{00}$.  If $A/B$ is bounded, then $A^{00} =
  B^{00}$, implying $A = B$, implying $A/B$ is trivial.}.  In this
case, $\Lambda_\bullet \cong \Lambda^{00}_\bullet$.

\section{Inflators: definition and examples}\label{sec:inflators}
Fix a small field $K_0$.  (Usually it will be $\Qq$ or a magic
subfield of a field of finite dp-rank.)
\begin{definition}
  Let $K/K_0$ be a field.  A \emph{$d$-inflator on $K$}
  is a morphism of directories
  \begin{equation*}
    \varsigma : \Dir_K(K) \to D_\bullet
  \end{equation*}
  to some semisimple directory $D_\bullet$ of length $d$, satisfying
  the following additional axiom:
  \begin{equation*}
    \ell(\varsigma_n(V)) = d \cdot \dim_K(V)
  \end{equation*}
  for any $V \subseteq K^n$.  In other words, the maps $\varsigma_n$
  scale lengths by a factor of $d$.
\end{definition}
Explicitly, then, a $d$-inflator on $K$ consists of a
family of maps
\begin{equation*}
  \varsigma_n : \Sub_K(K^n) \to \Sub_\Cc(B^n)
\end{equation*}
where $B$ is a semisimple object of length $d$ in a $K_0$-linear
abelian category, and the maps $\varsigma_n$ need to satisfy the
identities\footnote{Compare with Speculative Remark~10.10 in
  \cite{prdf}.  Inflators were called ``$r$-fold specializations'' in
  introductory \S1.2 of \cite{prdf}.}
\begin{align*}
  V \subseteq W & \implies \varsigma_n(V) \subseteq \varsigma_n(W) \\
  \varsigma_{n + m}(V \oplus W) &= \varsigma_n(V) \oplus \varsigma_m(W) \\
  \ell(\varsigma_n(V)) &= d \cdot \dim_K(V) \\
  \varsigma_n(\mu \cdot V) &= \mu \cdot \varsigma_n(V) \qquad \textrm{for } \mu \in GL_n(K_0).
\end{align*}
\begin{definition}
  Two $d$-inflators
  \begin{align*}
    \varsigma : \Dir_K(K) & \to D_\bullet \\
    \varsigma' : \Dir_K(K) & \to D'_\bullet
  \end{align*}
  are \emph{equivalent} if there is an isomorphism of directories $f :
  D_\bullet \to D'_\bullet$ such that $\varsigma' = f \circ \varsigma$.
\end{definition}
By Theorem~\ref{thm:semisimple-directories}, every $d$-inflator
is equivalent to one of the form
\begin{align*}
  \Dir_K(K) \to \Dir_R(A),
\end{align*}
where $R$ is a semisimple $K_0$-algebra and $A$ is a
finitely-generated $R$-module.  Moreover, we may assume either of the
following (but not both):
\begin{itemize}
\item $A = R$.
\item $R$ is a finite product of division algebras over $K_0$.
\end{itemize}

\subsection{1-inflators from field embeddings} \label{sec:aaron}
Let $K_0 \subseteq K_1 \subseteq K_2$ be a chain of field extensions.
The inclusion of $K_1$ into $K_2$ induces a 1-inflator from
$\Dir_{K_1}(K_1)$ to $\Dir_{K_2}(K_2)$
\begin{align*}
  \Sub_{K_1}(K_1^n) & \to \Sub_{K_2}(K_2^n) \\
  V & \mapsto K_2 \otimes_{K_1} V.
\end{align*}
This is a directory morphism because it arises as in
Example~\ref{from-functor} from the left-exact $K_0$-additive functor
$K_2 \otimes_{K_1} -$ from $K_1\Vect$ to $K_2\Vect$.

More generally, if $K_1, K_2$ are two fields extending $K_0$, any
$K_0$-linear embedding of $K_1$ into $K_2$ induces a 1-inflator from
$\Dir_{K_1}(K_1)$ to $\Dir_{K_2}(K_2)$.  Even more generally, we
can allow $K_2$ to be a $K_0$-division algebra.

\subsection{Inflators from restriction of scalars} \label{sec:beth}
Let $L/K$ be a finite field extension.  The forgetful functor from
$L\Vect$ to $K\Vect$ is exact, hence induces a morphism of directories
\begin{equation*}
  \Dir_L(L) \to \Dir_K(L).
\end{equation*}
This map clearly scales dimensions by a factor of $[L : K]$, yielding
an $[L : K]$-inflator on $L$.

\subsection{1-inflators from valuations} \label{sec:calvin}
Let $(K,\Oo,\mm)$ be a valued field.  Suppose $K$ extends $K_0$, and
the valuation on $K_0$ is trivial, so that $\Oo$ is a $K_0$-algebra.
Let $k$ be the residue field.

\begin{theorem}\label{thm:calvin}
  The family of maps
  \begin{align*}
    \varsigma_n : \Sub_K(K^n) & \to \Sub_k(k^n) \\
    V & \mapsto (V \cap \Oo^n + \mm^n)/\mm^n
  \end{align*}
  is a 1-inflator from $\Dir_K(K)$ to $\Dir_k(k)$.
\end{theorem}
\begin{proof}
  First we verify that $\varsigma$ is a morphism of directories.
  Indeed, it is the composition
  \begin{equation*}
    \Dir_K(K) \stackrel{i}{\hookrightarrow} \Dir_\Oo(K)
    \stackrel{f}{\to} \Dir_\Oo(\Oo) \stackrel{g}{\to} Dir_\Oo(k),
  \end{equation*}
  where $i$ is induced by the forgetful functor $K\Vect \to \Oo\Mod$,
  where $f$ is pullback along the monomorphism $\Oo \hookrightarrow
  K$, and where $g$ is pushforward along the epimorphism $\Oo \to
  \Oo/\mm \cong k$.  See Examples~\ref{from-morphism} and
  \ref{from-functor}.

  Next, we must show that 
  \begin{equation}
    \label{gr-to-gr}
    \dim_k(\varsigma_n(V)) = \dim_K(V)
  \end{equation}
  for any
  subspace $V \subseteq K^n$.  By the following lemma, we may change
  coordinates and assume $V = K^\ell \oplus 0^{n-\ell}$, in which case
  (\ref{gr-to-gr}) is clear.
\end{proof}
\begin{lemma}\label{v-o-lemma}
  If $(K,\Oo,\mm,k)$ is a valued field and $V$ is a subspace of $K^n$,
  then there is $\mu \in GL_n(\Oo)$ such that $\mu \cdot V = K^\ell
  \oplus 0^{n-\ell}$ for some $\ell$.
\end{lemma}
Although this lemma is well-known\footnote{It is the reason why
  Grassmannian varieties are proper.}, we give a proof for
completeness.
\begin{proof}
  Let $r : \Oo^n \twoheadrightarrow k^n$ be the coordinatewise residue
  map.  Let $\vec{a}_1, \ldots, \vec{a}_\ell$ be elements of $V \cap
  \Oo^n$ such that $r(\vec{a}_1), \ldots, r(\vec{a}_\ell)$ form a basis
  for the $k$-linear subspace $r(V \cap \Oo^n)$, the image of $V \cap
  \Oo^n$ under $r$.  We may find $\vec{a}_{\ell+1},\ldots,\vec{a}_n \in
  \Oo^n$ such that $r(\vec{a}_1),\ldots,r(\vec{a}_n)$ are a basis of
  $k^n$.  Let $\nu$ be the matrix built out of
  $\vec{a}_1,\ldots,\vec{a}_n$.  Then $\nu$ is in $GL_n(\Oo)$ because
  its determinant has nonzero residue, hence is invertible.  Let $\mu
  = \nu^{-1}$.  Replacing $V$ with $\mu \cdot V$, we may assume that
  $\vec{a}_i$ is the $i$th standard basis vector $\vec{e}_i$.  For $1
  \le i \le \ell$,
  \begin{equation*}
    \vec{e}_i = \vec{a}_i \in V \cap \Oo^n \subseteq V,
  \end{equation*}
  so $K^\ell \oplus 0^{n-\ell} \subseteq V$.  If $V$ is strictly larger than
  $K^\ell \oplus 0^{n-\ell}$, we can find non-zero
  \begin{equation*}
    \vec{x} \in V \cap (0^\ell \oplus K^{n-\ell}).
  \end{equation*}
  Rescaling $\vec{x}$ by its coordinate of least valuation, we may
  assume $\vec{x} \in \Oo^n \setminus \mm^n$.  Then $\vec{x} \in V
  \cap \Oo^n$.  Also $r(\vec{x})$ is a non-zero element of $0^\ell \oplus
  k^{n-\ell}$, and therefore isn't in the span of the first $\ell$ standard
  basis vectors $r(\vec{a}_1),\ldots,r(\vec{a}_\ell)$.  But we chose the
  vectors $\vec{a}_1,\ldots,\vec{a}_\ell$ to ensure that
  $r(\vec{a}_1),\ldots,r(\vec{a}_\ell)$ span all of $r(V \cap \Oo^n)$.
\end{proof}

Composing the 1-inflator of Theorem~\ref{thm:calvin} with the
1-inflators from \S\ref{sec:aaron}, we get a more general class of
1-inflators:
\begin{theorem}\label{thm:1-fold-example-2}
  Let $\Oo$ be a valuation $K_0$-algebra, let $L$ be a $K_0$-division
  algebra, and $f$ be a $K_0$-linear embedding of the residue field of
  $\Oo$ into $L$.  Let $K = \Frac \Oo$.  Then there is a 1-inflator
  from $\Dir_K(K)$ to $\Dir_L(L)$ sending $V \subseteq
  K^n$ to
  \begin{equation*}
    L \otimes_{\Oo/\mm} ((V \cap \Oo^n + \mm^n)/\mm^n)
  \end{equation*}
\end{theorem}
In \S\ref{sec:1-fold-classify} we shall see that all 1-inflators have
this form.  Thus, a 1-inflator from $K$ to $L$ is equivalent to a
valuation ring on $K$ and an embedding of the residue field into $L$.

\subsection{Inflators from multivaluations}\label{sec:dorothy}
Let $K$ be a field extending $K_0$.  For $i = 1,\ldots,n$, let
\begin{equation*}
  \varsigma^i : \Dir_K(K) \to D^i
\end{equation*}
be a $d_i$-inflator.  Then the product map
\begin{align*}
  \varsigma : \Dir_K(K) &\to (D^1 \times \cdots \times D^n) \\ 
  V & \mapsto (\varsigma^1(V),\ldots,\varsigma^n(V))
\end{align*}
of Proposition~\ref{prop:dir-prod} is a $(d_1 + \cdots +
d_n)$-inflator.  The length equation is a consequence of the following
trivial fact:
\begin{remark}\label{obvious-additive}
  For $i = 1, \ldots, n$, let $M_i$ be a modular lattice of finite
  length and let $x_i$ be an element of $M_i$.  In the product lattice
  $\prod_{i = 1}^n M_i$, the length of $(x_1,x_2,\ldots,x_n)$ over the
  bottom element is exactly
  \begin{equation*}
    \sum_{i = 1}^n \ell_{M_i}(x_i/\bot).
  \end{equation*}
\end{remark}
\begin{example}\label{dorothy}
  If $K$ is a field extending $K_0$, and if $\Oo_1,
  \ldots, \Oo_n$ are valuation $K_0$-algebras on $K$ with residue fields
  $k_1, \ldots, k_n$, then there is an $n$-inflator
  \begin{equation*}
    \Dir_K(K) \to \prod_{i = 1}^n \Dir_{k_i}(k_i)
  \end{equation*}
  sending $V$ to $(\varsigma^1(V),\ldots,\varsigma^n(V))$, where
  $\varsigma^i$ is the 1-inflator coming from the $i$th
  valuation.
\end{example}

The original hope was that essentially \emph{all} $d$-inflators would
arise in this manner.  In the next section, we will see that this is
far from being the case.

\subsection{Unwanted examples}\label{garbage-ex}

\begin{lemma}\label{lem:corruptor}
  If $K$ is a field (extending $K_0$ as always), there is a
  length-preserving directory morphism
  \begin{align*}
    \tau : \Dir_K(K) \times \Dir_K(K) & \to
    \Dir_K(K) \times \Dir_K(K) \\ (V,W) & \mapsto (V + W,
    V \cap W).
  \end{align*}
\end{lemma}
\begin{proof}
  First note that
  \begin{align*}
    \Dir_K(K) \times \Dir_K(K) &\to \Dir_K(K) \\
    (V,W) & \mapsto V + W
  \end{align*}
  is a directory morphism, because it is the composition of the morphism
  \begin{align*}
    \Dir_K(K) \times \Dir_K(K) &\to \Dir_K(K \oplus K) \\
    (V,W) & \mapsto V \oplus W
  \end{align*}
  from Remark~\ref{rem:whatever} with the morphism $\Dir_{K \oplus K}(K
  \oplus K) \to \Dir_K(K)$ obtained by pushforward along
  \begin{align*}
    K \oplus K & \to K \\
    (x,y) & \mapsto x+y.
  \end{align*}
  Similarly,
    \begin{align*}
    \Dir_K(K) \times \Dir_K(K) &\to \Dir_K(K) \\
    (V,W) & \mapsto V \cap W
  \end{align*}
  is a directory morphism, because it is the composition of the morphism
  \begin{align*}
    \Dir_K(K) \times \Dir_K(K) &\to \Dir_K(K \oplus K) \\
    (V,W) & \mapsto V \oplus W
  \end{align*}
  from Remark~\ref{rem:whatever} with the morphism $\Dir_{K \oplus K}(K \oplus K) \to
  \Dir_K(K)$ obtained by pullback along the diagonal inclusion
  \begin{align*}
    K & \hookrightarrow K \oplus K \\
    x & \mapsto (x,x).
  \end{align*}
  By Proposition~\ref{prop:dir-prod}, the product map $\tau$ is a
  morphism of directories.  The map $\tau$ is obviously
  length-preserving (note Remark~\ref{obvious-additive}).
\end{proof}
Using this, we can produce a number of perverse examples of 2-inflators.  Some (but not all) of these examples will be ruled
out by the ``malleability'' condition of \S\ref{sec:malleable}.

\begin{example}\label{eric}
  There is a 2-inflator
  \begin{align*}
    \Dir_{\mathbb{C}}(\mathbb{C}) &\to \Dir_{\mathbb{C}}(\mathbb{C}) \times \Dir_{\mathbb{C}}(\mathbb{C}) \\
    V &\mapsto (V + \overline{V}, V \cap \overline{V}),
  \end{align*}
  where $\overline{V}$ denotes the complex conjugate of $V$.
\end{example}
\begin{example}\label{fiona}
  In fact, both $V + \overline{V}$ and $V \cap \overline{V}$ descend
  to $\Rr$, yielding a 2-inflator
  \begin{align*}
    \Dir_{\mathbb{C}}(\mathbb{C}) &\to \Dir_{\mathbb{R}}(\mathbb{R}) \times \Dir_{\mathbb{R}}(\mathbb{R}) \\
    V &\mapsto (V + \overline{V}, V \cap \overline{V}),
  \end{align*}
\end{example}
\begin{example}\label{gerald}
  Let $K$ be a field with two independent valuation rings $\Oo_1,
  \Oo_2$.  Suppose both residue fields are isomorphic to some field
  $k$.  Then there is a 2-inflator
  \begin{align*}
    \Dir_K(K) &\to \Dir_k(k) \times \Dir_k(k) \\
    V & \mapsto (\varsigma^1(V) + \varsigma^2(V), \varsigma^1(V) \cap \varsigma^2(V))
  \end{align*}
  where $\varsigma^i : \Dir_K(K) \to \Dir_k(k)$ is the 1-inflator induced
  by $\Oo_i$.
\end{example}

\begin{example}\label{harriet}
Let
\begin{equation*}
  \ldots,p_{-2},p_{-1},p_0,p_1,p_2,\ldots
\end{equation*}
be some enumeration of $\{2,3,5,7,11,\ldots\}$.  Let
\begin{equation*}
  K = \Qq(\sqrt{2},\sqrt{3},\sqrt{5},\ldots).
\end{equation*}
For $i \in \Zz$, let $\sigma_i : K \to K$ be the automorphism of $K$
over $\Qq$ characterized by
\begin{equation*}
  \sigma_i(\sqrt{p_j}) = 
  \begin{cases}
    -\sqrt{p_j} & j = i \\
    \sqrt{p_j} & j \ne i
  \end{cases}
\end{equation*}
Extend $\sigma_i$ to a map $K^n \to K^n$ by coordinatewise
application.  Let $\tau_i$ be the map
\begin{align*}
  \Dir_K(K) \times \Dir_K(K) &\to \Dir_K(K) \times \Dir_K(K) \\
  (V,W) &\mapsto (V + \sigma_i(W), V \cap \sigma_i(W))
\end{align*}
Note that
\begin{equation*}
  (V',W') = \tau_i(V,W) \implies \dim(V') + \dim(W') = \dim(V) + \dim(W).
\end{equation*}
For any integers $i \le j$, there is a 2-inflator
$\varsigma^{i,j} : \Dir_K(K) \to \Dir_{K \times K}(K,K)$ given by
\begin{equation*}
  \varsigma^{i,j}(V) = (\tau_i \circ \tau_{i+1} \circ \cdots \circ \tau_j)(V,V)
\end{equation*}
Then we can define $\varsigma : \Dir_K(K) \to \Dir_{K \times K}(K,K)$ by
\begin{equation*}
  \varsigma(V) = \lim_{j \to \infty} \varsigma^{-j,j}(V)
\end{equation*}
The limit is well-defined because if $V$ is defined over
$\Qq(\sqrt{p_i},\sqrt{p_{i+1}},\ldots,\sqrt{p_j})$, then $\tau_k$ has
no effect for $k < i$ or $k > j$.
\end{example}
One can generalize the examples of this section to $d$-inflators for
$d > 2$, by using variants of $\tau$, such as the map
\begin{align*}
  \Dir_K(K^d) &\to \Dir_K(W) \times \Dir_K(K^d/W) \\
  V & \mapsto (V \cap W^n, (V + W^n)/W^n)
\end{align*}
for any $K$-linear subspace $W \subseteq K^d$.  (In the case of
Lemma~\ref{lem:corruptor}, $W$ is the diagonal in $K^2$.)

\section{Algebraic properties of inflators} \label{sec:basic-theory}
\subsection{Basic techniques}\label{sec:basic-tech}
Fix a $d$-inflator on $K$:
\begin{equation*}
  \varsigma : \Dir_K(K) \to \Dir_\Cc(B),
\end{equation*}
where $B$ is semisimple of length $d$.

We shall repeatedly use the following basic facts:
\begin{lemma}\label{duh}
  $\varsigma_n(0) = 0$ and $\varsigma_n(K^n) = B^n$.
\end{lemma}
\begin{proof}
  The map $\varsigma_n$ must scale lengths by a factor of $d$, so
  \begin{align*}
    \ell(\varsigma_n(0)) &= 0 \\
    \ell(\varsigma_n(K^n)) &= nd.
  \end{align*}
  But $\ell(B^n) = nd$.  The only subobject of $B^n$ having length $0$
  is $0$, and the only subobject having length $nd$ is $B^n$.
\end{proof}
\begin{lemma} \label{basic-tool}
  Let $V, W$ be subspaces of $K^n$.
  \begin{enumerate}
  \item \label{bt1} If
    \begin{equation*}
      \ell(\varsigma_n(V) \cap \varsigma_n(W)) = d \cdot \dim(V \cap W)
    \end{equation*}
    then
    \begin{equation*}
      \varsigma_n(V \cap W) = \varsigma_n(V) \cap \varsigma_n(W).
    \end{equation*}
  \item \label{bt2} Dually, if
    \begin{equation*}
      \ell(\varsigma_n(V) + \varsigma_n(W)) = d \cdot \dim(V + W)
    \end{equation*}
    then
    \begin{equation*}
      \varsigma_n(V + W) = \varsigma_n(V) + \varsigma_n(W).
    \end{equation*}
  \end{enumerate}
\end{lemma}
\begin{proof}
  By order-preservation, the following holds in general:
  \begin{equation} \label{cap-one-way}
    \varsigma_n(V \cap W) \subseteq \varsigma_n(V) \cap \varsigma_n(W).
  \end{equation}
  Suppose that
  \begin{equation*}
    \ell(\varsigma_n(V) \cap \varsigma_n(W)) = d \cdot \dim(V \cap W)
  \end{equation*}
  holds.  By the length-scaling law, it follows that
  \begin{equation*}
    \ell(\varsigma_n(V) \cap \varsigma_n(W)) = \ell(\varsigma_n(V \cap W)).
  \end{equation*}
  Therefore, equality must hold in (\ref{cap-one-way})---the two sides
  have the same length.  This proves (\ref{bt1}), and (\ref{bt2}) is
  similar.
\end{proof}
Lemma~\ref{basic-tool} can be re-stated as follows: Assume
\begin{align*}
  \varsigma_n(V) &= V' \\
  \varsigma_n(W) &= W'.
\end{align*}
Then
\begin{align*}
  \varsigma_n(V \cap W) & = V' \cap W' \\
  \varsigma_n(V + W) &= V' + W',
\end{align*}
\emph{provided that} the lengths and dimensions satisfy the expected
requirements:
\begin{align*}
  \ell(V' \cap W') &\stackrel{?}{=} d \cdot \dim(V \cap W) \\
  \ell(V' + W') &\stackrel{?}{=} d \cdot \dim(V + W).
\end{align*}
  
\subsection{The fundamental ring and ideal}\label{sec:r-i}
Fix some inflator $\varsigma : K \to \Dir_\Cc(B)$.  For $\alpha \in K$ and
$\varphi \in \End_\Cc(B)$, let
\begin{align*}
  \Theta_\alpha &= \{(x,\alpha x) : x \in K\} = K \cdot (1,\alpha) \\
  \Theta_\varphi &= \{(y,\varphi(y)) : y \in B\}.
\end{align*}
Say that $\alpha \in K$ \emph{specializes to} $\varphi \in \End(B)$ if
\begin{equation*}
  \varsigma_2(\Theta_\alpha) = \Theta_\varphi.
\end{equation*}
\begin{lemma}\label{basic-composition}
  If $\alpha$ specializes to $\varphi$ and $\alpha'$ specializes to
  $\varphi'$, then $\alpha \cdot \alpha'$ specializes to $\varphi
  \circ \varphi'$.
\end{lemma}
\begin{proof}
  By compatibility with $\oplus$ we have
  \begin{align*}
    \varsigma_3(\{(x,\alpha x, y) ~|~ x, y \in K\}) &= \{(x,\varphi
    x, y) ~|~ x, y \in B\} \\
    \varsigma_3(\{(w,x, \alpha' x) ~|~ w, x \in K\}) &= \{(w,x,
    \varphi' x) ~|~ w, x \in B\}
  \end{align*}
  By Lemma~\ref{basic-tool}.\ref{bt1},
  \begin{equation*}
    \varsigma_3(\{(x, \alpha x, \alpha' \alpha x) ~|~ x \in K\}) =
    \{(x, \varphi x, \varphi' \varphi x) ~|~ x \in B\}.
  \end{equation*}
  Meanwhile, by Lemma~\ref{duh} and $\oplus$-compatibility,
  \begin{equation*}
    \varsigma_3(\{(0,y,0) ~|~ y \in K\}) = \varsigma_3(0 \oplus K
    \oplus 0) = \{(0,y,0)~|~ y \in B\}.
  \end{equation*}
  By Lemma~\ref{basic-tool}.\ref{bt2}, it follows that
  \begin{equation*}
    \varsigma_3(\{(x, y, \alpha' \alpha x) ~|~ x, y \in K\}) =
    \{(x,y, \varphi' \varphi x) ~|~ x, y \in B\}.
  \end{equation*}
  By equivariance under $GL_3(K_0)$, we can permute the coordinates:
  \begin{equation*}
    \varsigma_3(\{(x, \alpha' \alpha x, y) ~|~ x, y \in K\}) =
    \{(x, \varphi' \varphi x, y) ~|~ x, y \in B\}.
  \end{equation*}
  By $\oplus$-compatibility, this implies
  \begin{equation*}
    \varsigma_2(\{(x,\alpha' \alpha x) ~|~ x \in K\}) \oplus B =
    \{(x, \varphi' \varphi x) ~|~ x \in B\} \oplus B.
  \end{equation*}
  Therefore, the desired identity must hold:
  \begin{equation*}
    \varsigma_2(\{(x, \alpha' \alpha x) ~|~ x \in K\}) =
    \{(x, \varphi' \varphi x) ~|~ x \in B\}. \qedhere
  \end{equation*}
\end{proof}
\begin{lemma}
  If $\alpha$ specializes to $\varphi$ and $\alpha'$ specializes to
  $\varphi'$, then $\alpha + \alpha'$ specializes to $\varphi +
  \varphi'$.
\end{lemma}
\begin{proof}
  Using a method similar to the proof of
  Lemma~\ref{basic-composition}, we see that
  \begin{equation*}
    \varsigma_3(\{(x,\alpha x, \alpha' x) ~|~ x \in K\}) =
    \{(x,\varphi x, \varphi' x) ~|~ x \in B\}.
  \end{equation*}
  By $\oplus$-compatibility,
  \begin{equation*}
    \varsigma_4(\{(x,\alpha x, \alpha' x, 0) ~|~ x \in K\}) =
    \{(x,\varphi x, \varphi' x, 0) ~|~ x \in B\}.
  \end{equation*}
  By $GL_4(K_0)$-equivariance,
  \begin{equation*}
    \varsigma_4(\{(x,\alpha x, \alpha' x, \alpha x + \alpha' x) ~|~ x \in K\}) =
    \{(x,\varphi x, \varphi' x, \varphi x + \varphi' x) ~|~ x \in B\}.
  \end{equation*}
  Adding in $0 \oplus K \oplus K \oplus 0$, we see that
  \begin{equation*}
    \varsigma_4(\{(x,y, z, \alpha x + \alpha' x) ~|~ x,y,z \in K\}) =
    \{(x,y, z, \varphi x + \varphi' x) ~|~ x,y,z \in B\}.
  \end{equation*}
  Permuting coordinates and splitting things off via
  $\oplus$-compatibility, we get the desired identity
  \begin{equation*}
    \varsigma_2(\{(x,\alpha x + \alpha' x) ~|~ x \in K\}) =
    \{(x, \varphi x + \varphi' x) ~|~ x \in B\}. \qedhere
  \end{equation*}
\end{proof}
\begin{lemma}
  If $\alpha \in K_0$, then $\alpha$ specializes to $\alpha$ (i.e., to
  $\alpha$ times $id_B$).
\end{lemma}
\begin{proof}
  By $\oplus$-compatibility,
  \begin{equation*}
    \varsigma_2(\{(x,0) ~|~ x \in K\}) = \{(x,0) ~|~ x \in B\}.
  \end{equation*}
  Applying the matrix $\begin{pmatrix} 1 & 0 \\ \alpha &
    1\end{pmatrix}$, the $GL_2(K_0)$-equivariance implies
  \begin{equation*}
    \varsigma_2(\{(x,\alpha x) ~|~ x \in K\}) = \{(x, \alpha x) ~|~ x \in B\}. \qedhere
  \end{equation*}  
\end{proof}
\begin{lemma}\label{o-i-jac}
  If $\alpha \in K$ specializes to an automorphism $\varphi \in
  \Aut(B)$, then $\alpha^{-1}$ specializes to $\varphi^{-1}$.
\end{lemma}
\begin{proof}
  This follows by $GL_2(K_0)$-equivariance using the matrix
  \begin{equation*}
    \begin{pmatrix}
      0 & 1 \\
      1 & 0
    \end{pmatrix} \qedhere
  \end{equation*}
\end{proof}
Putting together all the lemmas yields the
\begin{proposition} \label{o-i}
  Let $R$ be the set of $x \in K$ specializing to any endomorphism of
  $B$.  Let $I$ be the set of $x \in K$ specializing to 0.  Let
  \begin{equation*}
    \widehat{\res} : R \to \End(B)
  \end{equation*}
  be the map sending $x$ to $\varphi$ if $x$ specializes to $\varphi$.
  \begin{enumerate}
  \item $R$ is a $K_0$-subalgebra of $K$.
  \item $\widehat{\res} : R \to \End(B)$ is a $K_0$-algebra homomorphism.
  \item $I$ is the kernel of $\widehat{\res}$, so $I$ is an ideal in
    $R$.
  \item \label{o-i-jac2} $I$ is contained in the Jacobson radical of
    $R$.
  \end{enumerate}
\end{proposition}
\begin{proof}
  The only subtle point is the last one.  By basic commutative
  algebra, it suffices to show that $1 + I \subseteq R^\times$,
  which follows from Lemma~\ref{o-i-jac}.
\end{proof}
Note that the properties of $R$ and $I$ are analogous to the
properties of the ring $R_J$ and ideal $I_J$ appearing in
(\cite{prdf}, Proposition~10.15).  In fact, the $R_J$ and $I_J$ of
\cite{prdf} are the fundamental ring and ideal of a specific inflator
we will construct later in Theorem~\ref{yeah-main}.

\begin{definition}
  If $\varsigma$ is a $d$-inflator, the \emph{fundamental ring} and
  \emph{fundamental ideal} of $\varsigma$ are the ring $R$ and ideal
  $I$ appearing in Proposition~\ref{o-i}.  The natural map
  \begin{equation*}
    \widehat{\res} : R \to \End(B)
  \end{equation*}
  is the \emph{generalized residue map}; its kernel is $I$.
\end{definition}
\begin{lemma}\label{o-alt-criterion}
  Fix an inflator $K \to \Dir_\Cc(B)$.  Let $\alpha$ be an element of
  $K$, and let $\Theta_\alpha = K \cdot (1, \alpha)$.  The following
  are equivalent:
  \begin{itemize}
  \item $\alpha \in R$, i.e., $\varsigma_2(\Theta_\alpha)$ is the graph of
    an endomorphism on $B$.
  \item $\varsigma_2(\Theta_\alpha) + (0 \oplus B) = B \oplus B$, i.e., the projection
    of $\varsigma_2(\Theta_\alpha)$ onto the first coordinate is onto.
  \item $\varsigma_2(\Theta_\alpha) \cap (0 \oplus B) = 0 \oplus 0$, i.e., $\varsigma_2(\Theta_\alpha)$ is
    the graph of a partial endomorphism.
  \end{itemize}
\end{lemma}
\begin{proof}
  The first condition is equivalent to the conjunction of the second
  and third conditions.  By the length-scaling law, $\varsigma_2(\Theta_\alpha)$ has
  half the length of $B \oplus B$.  This ensures that the second and
  third conditions are, in fact, equivalent.
\end{proof}
\begin{lemma}\label{i-alt-criterion}
  Fix an inflator $K \to \Dir_\Cc(B)$.  Let $\alpha$ be an element of
  $K$, and let $\Theta_\alpha = K \cdot (1, \alpha)$.  The following
  are equivalent:
  \begin{itemize}
  \item $\alpha \in I$, i.e., $\varsigma_2(\Theta_\alpha) = B \oplus 0$
  \item $\varsigma_2(\Theta_\alpha) \supseteq B \oplus 0$.
  \item $\varsigma_2(\Theta_\alpha) \subseteq B \oplus 0$.
  \end{itemize}
\end{lemma}
\begin{proof}
  Similar to Lemma~\ref{o-alt-criterion}, but easier.
\end{proof}

\subsection{Examples}\label{sec:r-i-ex}
\begin{example}\label{aaron-ring}
  Let $K \hookrightarrow L$ be a ($K_0$-linear) embedding of $K$ into
  a division ring $L$.  As in \S\ref{sec:aaron}, this gives a 1-inflator
  \begin{align*}
    \varsigma: \Sub_K(K^n) & \to \Sub_L(L^n) \\
    V & \mapsto L \otimes_K V.
  \end{align*}
  If $a \in K$ and $\Theta_a = K \cdot (1,a)$, then
  $\varsigma(\Theta_a)$ is just $L \cdot (1,a)$.  In particular, the
  fundamental ring of $\varsigma$ is all of $K$, and the generalized residue
  map is the embedding of $K$ into $L$.
\end{example}
\begin{example}\label{beth-ring}
  Let $[L : K]$ be a finite field extension of degree $d$ and let
  $\varsigma : \Dir_L(L) \to \Dir_K(L)$ be the restriction of
  scalars $d$-inflator, as in \S\ref{sec:beth}.  For any $a
  \in L$, let $\Theta_a = L \cdot (1,a)$.  Then $\varsigma(\Theta_a)$
  is the graph of the $K$-linear map $L \to L$ induced by
  multiplication by $a$.  Thus the fundamental ring of $\varsigma$ is
  all of $L$, and the generalized residue map is the obvious forgetful map
  $L \to \End_K(L)$.
\end{example}
\begin{example}\label{calvin-ring}
  Let $\Oo$ be a valuation ring on a field $K$, and let
  \begin{align*}
    \varsigma : \Dir_K(K) &\to \Dir_k(k) \\
    \varsigma_n(V) &= (V \cap \Oo^n + \mm^n)/\mm^n
  \end{align*}
  be the induced 1-inflator as in Theorem~\ref{thm:calvin}.

  For $a \in K$, let $\Theta_a = K \cdot (1,a)$.  Note that
  \begin{equation*}
    \Theta_a \cap \Oo^2 = 
    \begin{cases}
      \Oo \cdot (1,a) & a \in \Oo \\
      \Oo \cdot (a^{-1},1) & a \notin \Oo.
    \end{cases}
  \end{equation*}
  Then $\varsigma_2(\Theta_a)$ is the image of $\Theta_a \cap \Oo^2$
  under the residue map $\Oo^2 \twoheadrightarrow k^2$, which is
  \begin{equation*}
    \varsigma_2(\Theta_a) = 
    \begin{cases}
      k \cdot (1, \res(a)) & a \in \Oo \\
      k \cdot (0,1) & a \notin \Oo.
    \end{cases}
  \end{equation*}
  It follows that $\Oo$ is the fundamental ring of $\varsigma$, $\mm$
  is the fundamental ideal, and the canonical map
  \begin{equation*}
    \Oo/\mm \to \End_k(k) = k
  \end{equation*}
  is the residue map.
\end{example}
\begin{example}\label{oops}
  More generally, let $\varsigma : \Dir_K(K) \to \Dir_L(L)$ be a
  1-inflator on $K$ constructed from a valuation ring $\Oo$ on $K$ and
  an embedding of the residue field into a division ring $L$.  Then
  the fundamental ring of $\varsigma$ is $\Oo$, the generalized residue
  map is the composition
  \begin{equation*}
    \Oo \twoheadrightarrow k \hookrightarrow L,
  \end{equation*}
  and the fundamental ideal is $\mm$.  We leave the details as an
  exercise to the reader.
\end{example}
\begin{proposition}
  For $i = 1, \ldots, n$, let $\varsigma^i : \Dir_K(K) \to
  \Dir_{\mathcal{C}_i}(B_i)$ be a $d_i$-inflator.  Let
  $R_i$ and $I_i$ be the fundamental ring and ideal of $\varsigma^i$.
  Let $\varsigma$ be the product inflator of \S\ref{sec:dorothy}:
  \begin{align*}
    \varsigma : \Dir_K(K) &\to \prod_{i = 1}^n \Dir_{\mathcal{C}_i}(B_i) \\
    V & \mapsto (\varsigma^1(V),\ldots,\varsigma^n(V)).
  \end{align*}
  Then the fundamental ring and ideal of $\varsigma$ are exactly
  \begin{align*}
    R &= \bigcap_{i = 1}^n R_i \\
    I &= \bigcap_{i =1}^n I_i,
  \end{align*}
  and the canonical map
  \begin{equation*}
    R/I \to \prod_{i = 1}^n \End_{\mathcal{C}_i}(B_i)
  \end{equation*}
  is the product of the canonical maps for the $\varsigma^i$.
\end{proposition}
\begin{proof}
  Note that for $a \in K$ and $\Theta_a = K \cdot (1,a)$, we have
  \begin{equation*}
    \varsigma(\Theta_a) = (\varsigma^1(\Theta_a),\ldots,\varsigma^n(\Theta_a)).
  \end{equation*}
  By Lemma~\ref{o-alt-criterion}, $a \in R$ if and only if
  \begin{equation*}
    \varsigma(\Theta_a) \cap (0 \oplus (B_1,\ldots,B_n)) = 0 \oplus 0,
  \end{equation*}
  i.e., if and only if
  \begin{equation*}
    (\varsigma^1(\Theta_a) \cap (0 \oplus B_1),\ldots,
    \varsigma^n(\Theta_a) \cap (0 \oplus B_n)) = (0 \oplus 0, \ldots,
    0 \oplus 0).
  \end{equation*}
  By Lemma~\ref{o-alt-criterion} this holds if and only if $a \in R_i$
  for all $i$.  A similar argument using Lemma~\ref{i-alt-criterion}
  shows $I = \bigcap_{i = 1}^n I_i$.  Finally, if $a \in R$, and $a$
  specializes to $\varphi_i \in \End_{\Cc_i}(B_i)$ for each $i$, then
  \begin{equation*}
    \varsigma^i(\Theta_a) = \Theta_{\varphi_i},
  \end{equation*}
  so
  \begin{equation*}
    \varsigma(\Theta_a) = (\Theta_{\varphi_1},\ldots,\Theta_{\varphi_n}).
  \end{equation*}
  But the right hand side is the graph of the endomorphism
  $(\varphi_1,\ldots,\varphi_n)$ in
  \begin{equation*}
    \End_{\Cc_1 \times \cdots \times \Cc_n}((B_1,\ldots,B_n)) = \End_{\Cc_1}(B_1) \times \cdots \times \End_{\Cc_n}(B_n). \qedhere
  \end{equation*}
\end{proof}
\begin{example}\label{dorothy-ring}
  Suppose $\Oo_1, \ldots, \Oo_n$ are valuation $K_0$-algebras on a
  field $K$, and $k_i$ is the residue field of $\Oo_i$.  Let $\varsigma_i : \Dir_K(K) \to \Dir_{k_i}(k_i)$ be the 1-inflator from $\Oo_i$, and let
  \begin{equation*}
    \varsigma(V) = (\varsigma_1(V),\ldots,\varsigma_n(V))
  \end{equation*}
  be the resulting $n$-inflator from $\Dir_K(K)$ to
  $\prod_{i = 1}^n \Dir_{k_i}(k_i)$, as in Example~\ref{dorothy}.
  Then the fundamental ring of $\varsigma$ is $\bigcap_{i = 1}^n
  \Oo_i$, the fundamental ideal is $\bigcap_{i = 1}^n \mm_i$, and the
  fundamental map
  \begin{equation*}
    R/I \to \prod_{i = 1}^n \End_{k_i}(k_i)
  \end{equation*}
  is the obvious map sending $x$ to $(\res_1(x),\ldots,\res_n(x))$.
\end{example}

\begin{example}\label{fiona-ring}
  Consider the 2-inflator
\begin{align*}
  \varsigma : \Dir_{\mathbb{C}}(\mathbb{C}) &\to \Dir_{\mathbb{R}}(\mathbb{R}) \times \Dir_{\mathbb{R}}(\mathbb{R}) \\
  V & \mapsto (V + \overline{V}, V \cap \overline{V}).
\end{align*}
of Example~\ref{fiona}.

If $a \in \mathbb{C}$ and $\Theta_a = \mathbb{C} \cdot (1,a)$, then
$\overline{\Theta}_a = \mathbb{C} \cdot (1,\overline{a})$, so
\begin{equation*}
  \varsigma(\Theta_a) = 
  \begin{cases}
    (\Rr \cdot (1,\overline{a}),\Rr \cdot (1,\overline{a})) & a \in \Rr \\
    (\Rr^2,0) & a \notin \Rr.
  \end{cases}
\end{equation*}
Let $R$ and $I$ be the fundamental ring and ideal of $\varsigma$.  By
Lemmas~\ref{o-alt-criterion} and \ref{i-alt-criterion},
\begin{align*}
  a \in R &\iff \varsigma(\Theta_a) \cap (0 \oplus \Rr, 0 \oplus \Rr) = (0 \oplus 0,0 \oplus 0) \\
  & \iff a \in \Rr \\
  a \in I & \iff \varsigma(\Theta_a) \subseteq (\Rr \oplus 0, \Rr \oplus 0) \\
  & \iff a = 0
\end{align*}
Thus $R = \Rr$ and $I = 0$.
\end{example}

\begin{example}\label{gerald-ring}
  Let $K, k$ be two fields extending $K_0$.  Suppose $\Oo_1, \Oo_2$ are two valuation $K_0$-algebras on $K$ and the two residue fields are both isomorphic to $k$.  Let
  \begin{align*}
    \varsigma : \Dir_K(K) &\to \Dir_k(k) \times \Dir_k(k) \\
    \varsigma(V) &= (\varsigma^1(V) + \varsigma^2(V), \varsigma^1(V) \cap \varsigma^2(V))
  \end{align*}
  be the 2-inflator of Example~\ref{gerald}, where
  $\varsigma^i$ is the inflator from the valuation ring $\Oo_i$.

  Let us calculate the fundamental ring and ideal.  If $a \in K$ and
  $\Theta_a = K \cdot (1,a)$, then
  \begin{align*}
    \varsigma^1(\Theta_a) &= 
    \begin{cases}
      k \cdot (1,\res_1(a)) & a \in \Oo_1 \\
      k \cdot (0,1) & a \notin \Oo_1
    \end{cases} \\
    \varsigma^2(\Theta_a) &= 
    \begin{cases}
      k \cdot (1,\res_2(a)) & a \in \Oo_2 \\
      k \cdot (0,1) & a \notin \Oo_2
    \end{cases} \\
    \varsigma(\Theta_a) &= (\varsigma^1(\Theta_a) + \varsigma^2(\Theta_a), \varsigma^1(\Theta_a) \cap \varsigma^2(\Theta_a))
  \end{align*}
  Now $\varsigma^1(\Theta_a) + \varsigma^2(\Theta_a)$ will be all of
  $k^2$ unless $\res_1(a) = \res_2(a)$ (including the possibility that
  $\res_1(a) = \infty = \res_2(a)$).  Using this, one sees that the
  fundamental ring and ideal are
  \begin{align*}
    R &= \{x \in \Oo_1 \cap \Oo_2 : \res_1(a) = \res_2(a)\} \\
    I &= \{x \in \Oo_1 \cap \Oo_2 : \res_1(a) = \res_2(a) = 0\}.
  \end{align*}
\end{example}
In both Example~\ref{fiona-ring} and \ref{gerald-ring},
$\varsigma(\Theta_a)$ tends to be some fixed value not depending on
$a$.  In other words the maps
\begin{equation*}
  \varsigma_n : \Sub_K(K^n) \to \Sub_{\mathcal{C}}(B^n)
\end{equation*}
are a bit trivial for $n = 2$.  This suggests that the inflators
in question are degenerate without much interesting structure, at
least compared to their unclobbered counterparts
\begin{align*}
  \Dir_{\mathbb{C}}(\mathbb{C}) &\to \Dir_{\mathbb{C}}(\mathbb{C}) \times \Dir_{\mathbb{C}}(\mathbb{C}) \\
  V & \mapsto (V,\overline{V}) \\
  \\
  \Dir_K(K) & \to \Dir_k(k) \times \Dir_k(k) \\
  V & \mapsto (\varsigma_1(V),\varsigma_2(V))
\end{align*}
However, we will see in \S\ref{sec:mut4} that in both cases, the original
unclobbered counterpart can be recovered from the clobbered version,
by considering the maps $\varsigma_n$ for $n > 2$.  The information is
still there; it is merely hidden.

\subsection{1-inflators}\label{sec:1-fold-classify}
\begin{proposition}\label{prop:1-fold}
  Let $\varsigma$ be a 1-inflator from $K$ to $\Dir(B)$.  Then the
  fundamental ring $\Oo$ is a valuation ring on $K$.  The fundamental
  ideal $I$ is the maximal ideal of $\Oo$.  The induced map
  \begin{equation*}
    \Oo/I \to \End(B)
  \end{equation*}
  is the inclusion of the residue field into a division ring.
\end{proposition}
\begin{proof}
  The fact that $\varsigma$ is a 1-inflator implies that $B$ is simple
  ($\ell(B) = 1$), so $\End(B)$ is indeed a division ring.  We may
  assume that the ambient abelian category is the category of
  $D$-modules, and $B = D$.  For any $\alpha \in K$,
  \begin{equation*}
    \varsigma_2(\{(x, \alpha x) ~|~ x \in K\})
  \end{equation*}
  must be a 1-dimensional submodule of $D^2$.  This \emph{must} have
  one of the following forms:
  \begin{align*}
    \{(x, x \cdot \beta) ~&|~ x \in D\} \\
    \{(x \cdot \beta, x) ~&|~ x \in D\}
  \end{align*}
  for some $\beta \in D$.  So, replacing $\alpha$ with $\alpha^{-1}$ if
  necessary, we get $\alpha \in \Oo$.  Thus $\Oo$ is a valuation ring.  Let
  $\mm$ be the maximal ideal.

  By Proposition~\ref{o-i}.\ref{o-i-jac2}, we know $I \subseteq \mm$.
  Conversely, suppose $\alpha \in \mm$.  Then $\alpha^{-1} \notin
  \Oo$.  Therefore,
  \begin{equation*}
    \varsigma_2(\{(x, \alpha^{-1} x) ~|~ x \in K\})
  \end{equation*}
  is a 1-dimensional subspace of $D^2$, \emph{not} of the form
  \begin{equation*}
    \{(x, x \cdot \beta) ~|~ x \in D\}.
  \end{equation*}
  There is only one such subspace, namely $0 \oplus D$.  Therefore
  \begin{equation*}
    \varsigma_2(\{(x, \alpha^{-1} x) ~|~ x \in K\}) = 0 \oplus D.
  \end{equation*}
  Swapping the coordinates (using $GL_2(K_0)$-equivariance), we see
  \begin{equation*}
    \varsigma_2(\{(x,\alpha x) ~|~ x \in K\}) =
    \varsigma_2(\{(\alpha^{-1} x, x) ~|~ x \in K\}) = D \oplus 0.
  \end{equation*}
  Thus $\alpha$ specializes to 0, meaning that $\alpha \in I$.  Thus
  $\mm = I$.

  Now $\Oo/I$ is a field and $\Oo/I \hookrightarrow D$ is an injective
  ring homomorphism, because $I := \ker(\Oo \to D)$.  So $D$ is a
  division ring extending the residue field $\Oo/I$.
\end{proof}

We can now classify 1-inflators:
\begin{theorem}\label{thm:1-fold}
  Every 1-inflator $\varsigma : \Dir_K(K) \to \Dir_L(L)$ is
  of the form constructed in Theorem~\ref{thm:1-fold-example-2}.
\end{theorem}
\begin{proof}
  By Proposition~\ref{prop:1-fold}, there is a valuation ring $\Oo$ on
  $K$ and an embedding of the residue field $k = \Oo/\mm$ into $L$
  such that
  \begin{itemize}
  \item $\Oo$ is the fundamental ring of $\varsigma$
  \item $\mm$ is the fundamental ideal
  \item The composition
    \begin{equation*}
      \Oo \twoheadrightarrow \Oo/\mm = k = k^{op} \hookrightarrow L^{op}
    \end{equation*}
    is the generalized residue map $\Oo \to \End_L(L) = L^{op}$.
  \end{itemize}
  By Theorem~\ref{thm:1-fold-example-2}, this induces a 1-inflator
  $\varsigma' : \Dir_K(K) \to \Dir_L(L)$, arising as a
  composition
  \begin{equation*}
    \Dir_K(K) \stackrel{\varsigma''}{\to} \Dir_k(k) \stackrel{L
      \otimes_k -}\hookrightarrow \Dir_L(L)
  \end{equation*}
  where $\varsigma''$ is the inflator
  \begin{equation*}
    V \mapsto (V \cap \Oo^n + \mm^n)/\mm^n
  \end{equation*}
  of Theorem~\ref{thm:calvin}.
  
  We claim that $\varsigma = \varsigma'$.  Let $V$ be a
  $m$-dimensional $K$-linear subspace of $K^n$.  We will show
  $\varsigma_n(V) = \varsigma'_n(V)$.  By Lemma~\ref{v-o-lemma}, there
  are vectors $\vec{a}_1, \ldots, \vec{a}_m \in V$ such that
  \begin{itemize}
  \item $\vec{a}_1,\ldots,\vec{a}_m$ freely generate $V$ as a
    $K$-vector space.
  \item $\vec{a}_1,\ldots,\vec{a}_m$ freely generate $V \cap \Oo^n$ as
    an $\Oo$-module.
  \item The residues of $\vec{a}_1,\ldots,\vec{a}_m$ freely generate
    the $m$-dimensional $k$-subspace
    \begin{equation*}
      \varsigma''(V) = (V \cap \Oo^n + \mm^n)/\mm^n \subseteq k^n.
    \end{equation*}
  \end{itemize}
  Let $\rho$ be the induced composition $\Oo \to \Oo/\mm = k
  \hookrightarrow L$.  Then $\rho(\vec{a}_1), \ldots, \rho(\vec{a_m})$
  freely generate $\varsigma'(V)$.

  Let $\vec{a}_i = (a_{i,1},\ldots,a_{i,n})$.  Let $b_{i,j} =
  \rho(a_{i,j})$.  Because $\rho$ is the generalized residue map from
  $\Oo$ to $L^{op} = \End_L(L)$,
  \begin{equation*}
    \varsigma(\{(x,a_{i,j}x) : x \in K\}) = \{(x,xb_{i,j}) : x \in L\}
  \end{equation*}
  Inside $K^m \times K^{nm} \times K^n$, consider the following
  subspaces:
  \begin{align*}
    W_{i,j} &= \{(\vec{x},\vec{y},\vec{z}) : y_{i,j} = a_{i,j}x_i\} \\
    U_j &= \{(\vec{x},\vec{y},\vec{z}) : z_j = y_{1,j} + \cdots + y_{m,j}\} \\
    M &= \{(\vec{x},\vec{y},\vec{z}) : \vec{z} = 0\}
  \end{align*}
  By the techniques of \S\ref{sec:basic-tech}-\ref{sec:r-i},
  \begin{align*}
    \varsigma(W_{i,j}) &= \{(\vec{x},\vec{y},\vec{z}) : y_{i,j} = x_ib_{i,j}\} \\
    \varsigma(U_j) &= \{(\vec{x},\vec{y},\vec{z}) : z_j = y_{1,j} + \cdots + y_{m,j}\} \\
    \varsigma(M) &= \{(\vec{x},\vec{y},\vec{z}) : \vec{z} = 0\}
  \end{align*}
  Note that the intersection of all the $W_{i,j}$ and $U_j$ is the set
  of tuples $(\vec{x},\vec{y},\vec{z})$ where
  \begin{align*}
    y_{i,j} &= a_{i,j}x_i \\
    z_j &= y_{1,j} + \cdots + y_{m,j} = a_{1,j}x_1 + \cdots + a_{m,j}x_m.
  \end{align*}
  By Lemma~\ref{basic-tool}.\ref{bt1},
  \begin{equation*}
    \varsigma \left(\bigcap_{i,j} W_{i,j} \cap \bigcap_j U_j \right) =
    \bigcap_{i,j} \varsigma(W_{i,j}) \cap \bigcap_j \varsigma(U_j),
  \end{equation*}
  which is the set of tuples $(\vec{x},\vec{y},\vec{z}) \in L^m \times
  L^{nm} \times L^n$ such that
  \begin{align*}
    y_{i,j} &= x_ib_{i,j} \\
    z_j &= x_1b_{1,j} + \cdots + x_mb_{m,j}.
  \end{align*}
  Then
  \begin{equation*}
    M + \left(\bigcap_{i,j} W_{i,j} \cap \bigcap_j U_j\right)
  \end{equation*}
  is the set of $(\vec{x},\vec{y},\vec{z})$ such that $\vec{z}$ is in
  the span of $\vec{a}_1,\ldots,\vec{a}_m$, i.e., $K^m \times K^{nm}
  \times V$.  Similarly,
  \begin{equation*}
    \varsigma(M) + \left( \bigcap_{i,j} \varsigma(W_{i,j}) \cap
    \bigcap_j \varsigma(U_j) \right)
  \end{equation*}
  is the set of $(\vec{x},\vec{y},\vec{z})$ such that $\vec{z}$ is in
  the $L$-module generated by $\vec{b}_1,\ldots,\vec{b}_m$.  We chose
  $\vec{a}_1,\ldots,\vec{a}_m$ to ensure that this module is
  $\varsigma'(V)$.  The identity
  \begin{equation*}
    \varsigma(M) + \left( \bigcap_{i,j} \varsigma(W_{i,j}) \cap
    \bigcap_j \varsigma(U_j) \right) \subseteq \varsigma\left( M +
    \left(\bigcap_{i,j} W_{i,j} \cap \bigcap_j U_j\right) \right)
  \end{equation*}
  therefore says
  \begin{equation*}
    L^n \times L^{nm} \times \varsigma'(V) \subseteq \varsigma(K^n
    \times K^{nm} \times V).
  \end{equation*}
  As in the proof of Lemma~\ref{basic-tool}.\ref{bt2}, equality holds,
  because both sides have the same dimension.  Therefore $\varsigma(V)
  = \varsigma'(V)$ as desired.
\end{proof}

\begin{corollary}
  A 1-inflator from $\Dir_K(K)$ to $\Dir_L(L)$ is
  equivalent to the following data:
  \begin{itemize}
  \item A valuation ring $\Oo$ on $K$, with trivial valuation on
    $K_0$.
  \item An embedding (over $K_0$) of the residue field of $\Oo$ into
    $L$.
  \end{itemize}
\end{corollary}
\begin{proof}
  Every such datum determines a 1-inflator
  (Theorem~\ref{thm:1-fold-example-2}), every 1-inflator arises in
  this way (Theorem~\ref{thm:1-fold}), and the data can be recovered
  from the 1-inflator (Example~\ref{oops}).
\end{proof}

\subsection{Tame locus}\label{sec:tame-locus}
Fix a $d$-inflator $\varsigma : \Dir_K(K) \to \Dir(B)$.  Let
$R$ and $I$ be the fundamental ring and ideal.
\begin{lemma}\label{tame-wild}
  For any $\alpha \in K$, consider the set
  \begin{equation*}
    S_\alpha = \{\alpha\} \cup \left\{ \frac{1}{\alpha - q} ~|~ q \in K_0\right\}.
  \end{equation*}
  Then one of two things happens:
  \begin{itemize}
  \item No element of $S_\alpha$ is in $R$.
  \item Every element of $S_\alpha$ is in $R$, with at most $d$
    exceptions.
  \end{itemize}
\end{lemma}
\begin{proof}
  First suppose that $\alpha \in R$, so $\alpha$ specializes to some
  $\varphi \in \End(B)$.  Then for any $q \in K_0$, the following
  identies hold
  \begin{align*}
    \varsigma_2(\{(x, \alpha x) ~|~ x \in K) &= \{(x, \varphi x) ~|~
    x \in B\} \\
    \varsigma_2(\{(x, \alpha x - qx) ~|~ x \in K) &= \{(x, \varphi x
    - qx) ~|~ x \in B\} \\
    \varsigma_2(\{((\alpha - q)x, x) ~|~ x \in K) &= \{((\varphi -
    q)x, x) ~|~ x \in B\}
  \end{align*}
  where each line follows from the preceding one by
  $GL_2(K_0)$-equivariance.  Let
  \begin{equation*}
    B_q := \{((\varphi - q)x, x) ~|~ x \in B\}
  \end{equation*}
  Then $(\alpha - q)^{-1} \in R$ if and only if $B_q$ is the graph
  of an endomorphism.  By Lemma~\ref{o-alt-criterion},
  \begin{equation*}
    (\alpha - q)^{-1} \in R \iff B_q \cap (0 \oplus B) = 0 \oplus 0.
  \end{equation*}
  Thus $(\alpha - q)^{-1} \notin R$ if and only if there is a
  non-zero $x \in B$ annihilated by $\varphi - q$.  For any $q$, let
  $V_q \subseteq B$ be the eigenspace $\ker(\varphi - q)$.  We have
  seen
  \begin{equation*}
    (\alpha - q)^{-1} \notin R \iff V_q \ne 0.
  \end{equation*}
  But as usual, the $V_q$ are lattice-theoretically independent: if
  $q_1, q_2, \ldots$ is a sequence of distinct values in $K_0$, then
  \begin{equation*}
    V_{q_n} \cap \sum_{i < n} V_{q_i} = 0.
  \end{equation*}
  Otherwise, take $n$ minimal for which this fails.  Then there is an
  equation
  \begin{equation*}
    t_1 + \cdots + t_n = 0
  \end{equation*}
  where $t_i \in V_{q_i}$ and $t_n \ne 0$.  By definition of $V_q$, we
  have
  \begin{equation*}
    0 = \varphi(t_1 + \cdots + t_n) = q_1t_1 + \cdots + q_nt_n.
  \end{equation*}
  Then
  \begin{equation*}
    (q_1 - q_n)t_1 + (q_2 - q_n)t_2 + \cdots + (q_{n-1} - q_n)t_{n-1} = 0,
  \end{equation*}
  contradicting the choice of $n$.

  By independence of the $V_q$, it follows that at most $d$ of the
  $V_q$ can be non-trivial, so $(\alpha - q)^{-1} \in R$ for all but
  at most $d$ values of $q \in K_0$.

  We now turn to the general case.  Let $PGL_2(K_0)$ be the group of
  fractional linear transformations over $K_0$ and let $Aff(K_0)$ be
  the subgroup of affine transformations.  Both groups act on $K \cup
  \{\infty\}$.  Note that $Aff(K_0)$ fixes $R$ setwise, because $R$ is
  a $K_0$-algebra.  Set
  \begin{align*}
    \gamma_q(x) &= \frac{1}{x - q} \\
    \gamma_\infty(x) &= x
  \end{align*}
  for $q \in K_0$.  Then $\{\gamma_q : q \in K_0 \cup \{\infty\}\}$ is
  a set of coset representatives for $Aff(K_0)$: every $\gamma \in
  PGL_2(K_0)$ can be uniquely written as
  \begin{equation*}
    \tau \cdot \gamma_q
  \end{equation*}
  for some $\tau \in Aff(K_0)$ and $q \in K_0 \cup \{\infty\}$.
  Suppose that $\gamma_{q_0}(\alpha) \in R$ for at least one $q_0$.
  By the first case we considered,
  \begin{equation*}
    \gamma_q(\gamma_{q_0})(\alpha) \in R
  \end{equation*}
  for all $q \in K_0 \cup \{\infty\}$, with at most $d$ exceptions.
  Now $\{ \gamma_q \cdot \gamma_{q_0}\}$ is also a set of coset
  representatives.  Therefore, there is a bijection and a map
  \begin{align*}
    \pi : K_0 \cup \{\infty\} & \stackrel{\sim}{\to} K_0 \cup \{\infty\} \\
    \tau : K_0 \cup \{\infty\} & \to Aff(K_0)
  \end{align*}
  such that
  \begin{equation*}
    \gamma_q \cdot \gamma_{q_0} = \tau(q) \cdot \gamma_{\pi(q)}
  \end{equation*}
  for all $q \in K_0 \cup \{\infty\}$.  As $Aff(K_0)$ preserves $R$,
  we see that
  \begin{equation*}
    \gamma_q(\gamma_{q_0}(\alpha)) \in R \implies \gamma_{\pi(q)}(\alpha) \in R.
  \end{equation*}
  Thus
  \begin{equation*}
    \gamma_{\pi(q)}(\alpha) \in R
  \end{equation*}
  for all $q$, with at most $d$ exceptions.  As $\pi$ is a
  permutation, it follows that $\gamma_q(\alpha) \in R$ for all $q$,
  with at most $d$ exceptions.  This is the statement of the lemma.
\end{proof}
\begin{definition}\label{def:tame}
  We say that $\alpha \in K$ is \emph{tame} if $\frac{1}{\alpha - q}
  \in R$ for almost all $q \in K_0$, and \emph{wild} otherwise.
\end{definition}
By the lemma, all elements of $R$ are tame.

\begin{lemma}\label{algebra-case}
  Let $K$ be a field, $K_0$ be a subfield, and $R$ be a
  $K_0$-subalgebra.  Let $q_1, \ldots, q_n$ be $n$ distinct elements
  of $K_0$.  The following are equivalent:
  \begin{itemize}
  \item $R$ is an intersection of $n$ or fewer valuation rings on $K$.
  \item For any $x \in K$, at least one of
    \begin{equation*}
      x, \frac{1}{x - q_1}, \ldots, \frac{1}{x - q_n}
    \end{equation*}
    is in $R$.
  \end{itemize}
\end{lemma}
\begin{proof}
  First suppose that $R = \Oo_1 \cap \cdots \cap \Oo_m$ for some $m
  \le n$.  Given $x \in K$ and $1 \le i \le m$, note that almost all
  of the values
  \begin{equation*}
    x, \frac{1}{x - q_1}, \ldots, \frac{1}{x - q_n}
  \end{equation*}
  are in $\Oo_i$, with at most one exception.  Indeed,
  \begin{itemize}
  \item $x \notin \Oo_i \iff \res_i(x) = \infty$.
  \item $\frac{1}{x - q_j} \notin \Oo_i \iff \res_i(x) = q_j$.
  \end{itemize}
  As there are $n+1$ total numbers, and at most $n$ valuation rings,
  at least one number remains in all the $\Oo_i$, hence in $R$.

  Conversely, suppose that for any $x \in K$, at least one of the
  numbers
  \begin{equation*}
    x, \frac{1}{x - q_1}, \ldots, \frac{1}{x - q_n}
  \end{equation*}
  is in $R$.  (We repeat the arguments following \cite{prdf},
  Claim~10.26.)  Then first of all, any finitely-generated
  $R$-submodule of $K$ is singly-generated.  It suffices to consider
  $aR + bR$ for some $a, b \in K$.  By the hypothesis (applied to $x =
  a/b$), at least one of the numbers
  \begin{equation*}
    a/b, \frac{b}{a - q_1b}, \ldots, \frac{b}{a - q_nb}
  \end{equation*}
  lies in $R$.  If $a/b \in R$, then $aR + bR = bR$.  If $b/(a - q_ib)
  \in R$, then
  \begin{equation*}
    aR + bR = (a - q_ib)R.
  \end{equation*}
  As in the proof of \cite{prdf2}, Proposition~6.2, it follows that
  $\Frac(R) = K$ and $R$ is a Bezout domain.  We claim that $R$ has
  finitely many maximal ideals, in fact, no more than $n$ maximal
  ideals.  Otherwise, let $\mm_0, \ldots, \mm_n$ be $n+1$ distinct
  maximal ideals.  By the Chinese remainder theorem, find elements $a,
  b \in R$ such that
  \begin{align*}
    a & \equiv 1 \pmod{\mm_0} \\
    b & \equiv 0 \pmod{\mm_0} \\
    a & \equiv q_i \pmod{\mm_i} \\
    b & \equiv 1 \pmod{\mm_i}.
  \end{align*}
  Then $a/b$ cannot be in $R$, or else
  \begin{equation*}
    b \equiv 0 \pmod{\mm_0} \implies a \equiv 0 \pmod{\mm_0}.
  \end{equation*}
  Similarly, $b/(a - q_ib)$ cannot be in $R$ or else
  \begin{equation*}
    a - q_ib \equiv 0 \pmod{\mm_i} \implies b \equiv 0 \pmod{\mm_i}.
  \end{equation*}
  So this contradicts the hypothesis.

  Therefore, $R$ has at most $n$ maximal ideals.  By Proposition~6.9
  in \cite{prdf2}, $R$ is a multi-valuation ring.  By Corollary~6.7 in
  \cite{prdf2}, $R$ is an intersection of at most $n$ valuation rings
  on $K$.
\end{proof}

\begin{proposition}
  Let $\varsigma$ be a $d$-inflator on a field $K$, with
  fundamental ring $R$.  The following are equivalent:
  \begin{itemize}
  \item The ring $R$ is a multi-valuation ring.
  \item The ring $R$ is an intersection of $d$ or fewer valuation
    rings.
  \item Every $\alpha \in K$ is tame.
  \end{itemize}
\end{proposition}
\begin{proof}
  Lemmas~\ref{tame-wild} and \ref{algebra-case}.
\end{proof}
\begin{definition}\label{def:mvt}
  A $d$-inflator is of \emph{multi-valuation type} if the
  fundamental ring $R$ is a multi-valuation ring.  Equivalently,
  every $\alpha \in K$ is tame.
\end{definition}
\begin{definition}\label{def-almost}
  A $d$-inflator is \emph{weakly of multi-valuation type}
  if the fundamental ring $R$ contains a nonzero ideal of some
  multi-valuation ring.
\end{definition}
\begin{example}
  The multi-valuation inflators of Example~\ref{dorothy} are of
  multi-valuation type, by Example~\ref{dorothy-ring}.
\end{example}

\subsection{Malleable $d$-inflators}\label{sec:malleable}
\begin{definition}
  Let $f : \mathcal{L} \to \mathcal{L}'$ be an order-preserving
  $\{\bot,\top\}$-preserving map between two finite-length modular
  lattices.  Say that $f$ has \emph{atom-lifting} if for every atom
  $x' \in \mathcal{L}'$, there is an atom $x \in \mathcal{L}$ such
  that $f(x) \ge x'$.
\end{definition}
\begin{definition}
  Let $f : \mathcal{L} \to \mathcal{L}'$ be an order-preserving map
  between two finite-length modular lattices.  Say that $f$ is
  \emph{malleable} if for every interval $[x,y] \subseteq
  \mathcal{L}$, the map
    \begin{equation*}
      [x,y] \to [f(x),f(y)]
    \end{equation*}
    has atom-lifting.
\end{definition}
Concretely, malleability of $f$ means that for any $x, z \in \mathcal{L}$
and $y \in \mathcal{L}'$ such that
\begin{equation*}
  x \le z
\end{equation*}
\begin{equation*}
  f(x) \le y \le f(z)
\end{equation*}
\begin{equation*}
  \ell(y) = 1 + \ell(f(x))
\end{equation*}
there exists $y' \in \mathcal{L}$ such that
\[ x \le y' \le z \]
\[ f(x) \le y \le f(y') \le f(z)\]
\[ \ell(y') = \ell(x) + 1.\]
\begin{definition}\label{def:mal}
  We say that a morphism of directories $f : D_\bullet \to D'_\bullet$
  is \emph{malleable} if every $n$ the map $f_n : D_n \to D'_n$ is
  malleable.  We say that a $d$-inflator on $K$ is
  \emph{malleable} if the associated directory morphism
  $\Dir_K(K) \to D_\bullet$ is malleable.
\end{definition}

\begin{proposition}
  A 1-inflator is malleable if and only if it is
  (equivalent to) a 1-inflator
  \begin{equation*}
    \Dir_K(K) \to \Dir_k(k)
  \end{equation*}
  induced as in Theorem~\ref{thm:calvin} by a valuation $K_0$-algebra
  $\Oo \subseteq K$ with residue field $k$.
\end{proposition}
\begin{proof}
  First suppose that $\varsigma : \Dir_K(K) \to
  \Dir_k(k)$ comes from a valuation ring $\Oo$.  Then
  \begin{equation*}
    \varsigma_n(V) = (V \cap \Oo^n + \mm^n)/\mm^n.
  \end{equation*}
  for any $n, V \in \Sub_K(K^n)$.  Let $r : \Oo^n \twoheadrightarrow
  \Oo^n/\mm^n$ be the coordinatewise residue map.  Then
  $\varsigma_n(V)$ is the direct image $r(V \cap \Oo^n)$.

  We check malleability.  Suppose $V, W$ are subspaces of $K^n$ with
  $V \subseteq W$.  Suppose $X$ is a subspace of $k^n$ with
  \begin{equation*}
    r(V \cap \Oo^n) \subseteq X \subseteq r(W \cap \Oo^n),
  \end{equation*}
  and $\dim_k(X) = 1 + \dim_k(r(V \cap \Oo^n))$.  Take $\vec{a} \in W
  \cap \Oo^n$ such that $r(\vec{a}) \in X \setminus r(V \cap \Oo^n)$.  Then
  \begin{equation*}
    X = k \cdot r(\vec{a}) + r(V \cap \Oo^n).
  \end{equation*}
  Let $V' = V + K \cdot \vec{a}$.  Then $V \subseteq V' \subseteq W$
  and $\dim(V') = \dim(V) + 1$.  Also, $\vec{a} \in V' \cap \Oo^n$,
  and so
  \begin{equation*}
    r(\vec{a}) \in r(V' \cap \Oo^n)
  \end{equation*}
  It follows that
  \begin{equation*}
    r(V \cap \Oo^n) \subseteq X \subseteq r(V' \cap \Oo^n),
  \end{equation*}
  proving malleability.

  Now suppose that $\varsigma$ is any 1-inflator on $K$.
  By Theorem~\ref{thm:1-fold}, we may assume $\varsigma$ is a composition
  \begin{equation*}
    \Dir_K(K) \to \Dir_k(k) \to \Dir_L(L)
  \end{equation*}
  where the first map arises from a valuation ring $\Oo$ on $K$ with
  residue field $k$, as in Theorem~\ref{thm:calvin}, and the second
  map arises from a field extension $L/k$.  Specifically, the second
  map sends $V \in \Sub_k(k^n)$ to $L \otimes_k V \in \Sub_L(L^n)$.

  Assume $\varsigma$ is malleable; we must show that $k = L$.  If not,
  take $a \in L \setminus k$ and let
  \begin{equation*}
    \Theta_a = L \cdot (1,a) = \{(x,xa) : x \in L\}.
  \end{equation*}
  Then
  \begin{equation*}
    \varsigma(0 \oplus 0) = 0 \oplus 0 \subseteq \Theta_a \subseteq L
    \oplus L = \varsigma(K \oplus K)
  \end{equation*}
  By malleability, there is a 1-dimensional subspace $V \subseteq K^2$
  such that
  \begin{equation*}
    \Theta_a \subseteq \varsigma(V).
  \end{equation*}
  The right-hand side must have dimension equal to $1 \cdot \dim_K(V)
  = 1$, so equality must hold: $\Theta_a = \varsigma(V)$.  However,
  $\varsigma(V)$ is of the form $L \otimes_k W$ for some subspace $W
  \subseteq k^2$.  The subspace $\Theta_a$ does not have this form, by
  choice of $a$.  This contradiction shows that $k = L$.  Thus
  $\varsigma$ is a 1-inflator coming from a valuation as in
  Theorem~\ref{thm:calvin}.
\end{proof}

\begin{example}\label{eric-not-malleable}
  Consider the 2-inflator
  \begin{align*}
    \varsigma : \Dir_{\mathbb{C}}(\mathbb{C}) &\to \Dir_{\mathbb{C}}(\mathbb{C}) \times \Dir_{\mathbb{C}}(\mathbb{C}) \\
    V & \mapsto (V + \overline{V}, V \cap \overline{V})
  \end{align*}
  from Example~\ref{eric}.  Then $\varsigma$ is \emph{not} malleable.
  Otherwise, from the chain
  \begin{equation*}
    \varsigma_2(0^2) = (0^2,0^2) \subseteq (0^2,\mathbb{C} \cdot (1,i)) \subseteq (\mathbb{C}^2, \mathbb{C}^2) = \varsigma_2(\mathbb{C}^2),
  \end{equation*}
  malleability implies the existence of a line $L \subseteq \mathbb{C}^2$ such that
  \begin{equation*}
    (0^2,\mathbb{C} \cdot (1,i)) \subseteq (L + \overline{L}, L \cap \overline{L}).
  \end{equation*}
  Then $(1,i) \in L \cap \overline{L}$, so $\overline{L}$ contains
  both $(1,i)$ and $(1,-i)$, failing to be a line.
\end{example}
\begin{example}\label{fiona-malleable}
  Consider the 2-inflator
  \begin{align*}
    \varsigma : \Dir_{\mathbb{C}}(\mathbb{C}) &\to \Dir_{\mathbb{R}}(\mathbb{R}) \times \Dir_{\mathbb{R}}(\mathbb{R}) \\
    V & \mapsto (V + \overline{V}, V \cap \overline{V})
  \end{align*}
  from Example~\ref{fiona}.  Then $\varsigma$ \emph{is} malleable.  First
  note that we can rewrite this example more accurately as
  \begin{align*}
    \varsigma : \Dir_{\mathbb{C}}(\mathbb{C}) &\to \Dir_{\mathbb{R}}(\mathbb{R}) \times \Dir_{\mathbb{R}}(\mathbb{R}) \\
    V & \mapsto (\pi(V),V \cap \Rr^n)
  \end{align*}
  where $\pi : \mathbb{C}^n \to \Rr^n$ is the projection onto the real
  axis (the coordinatewise real part map).  Now suppose that $V \subseteq W \subseteq
  \mathbb{C}^n$, and
  \begin{equation*}
    (\pi(V),V \cap \Rr^n) \subseteq (H_1,H_2) \subseteq (\pi(W),W \cap \Rr^n),
  \end{equation*}
  where the length of $(H_1,H_2)$ over $(\pi(V),V \cap \Rr^n)$ is 1.
  This length is exactly
  \begin{equation*}
    \dim_\Rr(H_1/\pi(V)) + \dim_\Rr(H_2/(V \cap \Rr^n)),
  \end{equation*}
  so we are in one of two cases:
  \begin{enumerate}
  \item $\dim_\Rr(H_1/\pi(V)) = 1$ and $H_2 = V \cap \Rr^n$.  Take
    $\vec{a} \in H_1 \setminus \pi(V)$.  The dimension of $H_1/\pi(V)$
    ensures that $H_1 = \pi(V) + \Rr \cdot \vec{a}$.  The fact that
    $H_1 \subseteq \pi(W)$ implies that $\vec{a} = \pi(\vec{b})$ for
    some $\vec{b} \in W$.  Let $V' = V + \mathbb{C} \cdot \vec{b}$.
    Then
    \begin{itemize}
    \item $V \subseteq V' \subseteq W$, and $\dim_{\mathbb{C}}(V') = 1 + \dim_{\mathbb{C}}(V)$.
    \item $\pi(V') \supseteq \pi(V)$ and $\vec{a} \in \pi(V')$, so
      $\pi(V') \supseteq H_1$.
    \end{itemize}
    Then
    \begin{equation*}
      (\pi(V),V \cap \Rr^n) \subseteq (H_1,H_2) = (H_1,V \cap \Rr^n)
      \subseteq (\pi(V'),V' \cap \Rr^n)
    \end{equation*}
    verifying malleability.
  \item $H_1 = \pi(V)$ and $\dim_\Rr(H_2/(V \cap \Rr^n)) = 1$.  Take
    $\vec{a} \in H_2 \setminus (V \cap \Rr^n)$.  Then $\vec{a} \in H_2
    \subseteq W \cap \Rr^n$.  Also $H_2 = V \cap \Rr^n + \Rr \cdot
    \vec{a}$.  Let $V' = V + \mathbb{C} \cdot \vec{a}$.  Then
    $\dim_{\mathbb{C}}(V') = 1 + \dim_{\mathbb{C}}(V)$.  We have $V'
    \subseteq W$, because $\vec{a} \in W$.  Note that $V' \cap \Rr^n$
    contains both $V \cap \Rr^n$ (trivially) and $\vec{a}$ (because
    $\vec{a} \in \Rr^n$).  Thus $V' \cap \Rr^n$ must contain $V \cap
    \Rr^n + \Rr \cdot \vec{a} = H_2$, ensuring that
    \begin{equation*}
      (\pi(V),V \cap \Rr^n) \subseteq (H_1,H_2) = (\pi(V),H_2)
      \subseteq (\pi(V'),V' \cap \Rr^n).
    \end{equation*}
  \end{enumerate}
\end{example}

\part{Inflators on dp-finite fields}  \label{part:2}
In Part~\ref{part:2}, we carry out the main construction of inflators
on dp-finite fields.  Here is a brief synopsis.

Let $\Kk$ be a saturated, unstable, dp-finite field.  Let $\Lambda_1$
be the lattice of type-definable additive subgroups---or rather,
type-definable $K_0$-subspaces.  If $K_0$ is a magic field, there is a
uniform bound on the size of cubes in the lattice $\Lambda_1$; we say
that $\Lambda_1$ is a \emph{cube-bounded} modular lattice.

The notion of cube-boundedness naturally extends to directories and
abelian categories: a directory $D_\bullet$ is cube-bounded if every
$D_n$ is cube-bounded; an abelian category is cube-bounded if every
subobject lattice is cube-bounded.  In a dp-finite setting, the
category $\Hh$ of \S\ref{sec:tdef-setting} and the directory $\Lambda$
of \S\ref{sec:lambda} are both cube-bounded.  For abelian categories,
cube-boundedness is equivalent to the existence of a subadditive
finite rank function on objects.  Ultimately, the category $\Hh$ is
cube-bounded because of dp-rank.

In \S9.4 of \cite{prdf}, we constructed a modular pregeometry on the
``quasi-atoms'' in any lower-bounded modular lattice
$(M,\vee,\wedge,\bot)$.  If $M$ is a cube-bounded lattice, this
pregeometry has finite rank.  The lattice of closed sets is therefore
a semisimple modular lattice $M^\flat$.  We call $M^\flat$ the
\emph{flattening} of $M$.  There is an analogue of flattening for
abelian categories and directories as well.  In each case, flattening
turns something cube-bounded into something semisimple.  If $D_\bullet
= (D_1,D_2,D_3,\ldots)$ is a cube-bounded directory, we show that the
levelwise flattenings $(D_1)^\flat, (D_2)^\flat, (D_3)^\flat,\ldots$
naturally assemble into a semisimple directory
\begin{equation*}
  D^\flat_\bullet = (D_1^\flat, D_2^\flat, D_3^\flat, \ldots).
\end{equation*}
Moreover, there is a ``flattening map,'' a directory morphism
$D_\bullet \to D^\flat_\bullet$.

Applying this to the dp-finite field $\Kk$, for any type-definable
subgroup $A \subseteq \Kk$ we get a composition
\begin{equation} \label{the-big-one}
  \Dir_\Kk(\Kk) \hookrightarrow \Dir_\Hh(\Kk) \twoheadrightarrow
  \Dir_\Hh(\Kk/A) \to \left(\Dir_\Hh(\Kk/A)\right)^\flat
\end{equation}
in which the maps are respectively induced by
\begin{enumerate}
\item The faithful functor $\Kk\Vect^f \hookrightarrow \Hh$.
\item Pushforward along $\Kk \twoheadrightarrow \Kk/A$.
\item Flattening.
\end{enumerate}
For ``suitable'' $A$, the composition of (\ref{the-big-one}) turns out
to be an inflator, as predicted by Speculative Remark~10.10 in
\cite{prdf}.

Let $r$ be maximal such that a cube of size $2^r$ exists in
$\Lambda_1$.  Then $r$ is at most the dp-rank of $\Kk$.  The
``suitability'' requirement on $A$ is that $A$ is a \emph{pedestal}, i.e.,
the base of an $r$-cube in $\Lambda_1$.  The resulting inflator is an
$r$-inflator.  Moreover, this $r$-inflator knows something about
$A$---the fundamental ring of the inflator is exactly
\begin{equation*}
  \Stab(A) = \{x \in \Kk : x \cdot A \subseteq A\}.
\end{equation*}
The main construction actually works in greater generality than dp-finite
fields.  Abstractly, the only thing we need is a $K_0$-linear functor
from $K\Vect^f$ to a cube-bounded $K_0$-linear abelian category
$\Cc$.  For the dp-finite case, $K$ is $\Kk$ and $\Cc$ is $\Hh$.  But
the machinery also applies when $K$ is a field of finite burden,
and $\Cc$ is the isogeny category $\Hh^{00}$ of
\S\ref{sec:tdef00-setting}.

Flattening turns out to be closely related to pro-categories.  If $M$
is a cube-bounded modular lattice, then $\Pro M$ is also a modular
lattice, and the pregeometry on quasi-atoms in $M$ is equivalent to
the geometry on atoms in $\Pro M$.  The flattening operation on a
cube-bounded abelian category $\Cc$ is essentially the socle functor
on $\Pro \Cc$.

Using pro-categories, the composition of (\ref{the-big-one}) can be
expressed less opaquely as
\begin{equation*}
  \Dir_\Kk(\Kk) \hookrightarrow \Dir_{\Pro \Hh}(\Kk) \twoheadrightarrow \Dir_{\Pro \Hh}(A^+/A)
\end{equation*}
where
\begin{itemize}
\item $A^+/A$ is the socle of $\Kk/A$ in the category $\Pro \Hh$.
\item The first map is induced by the faithful exact functor
  \begin{equation*}
    \Kk\Vect^f \to \Hh \to \Pro \Hh.
  \end{equation*}
\item The second map is the interval retract onto the subquotient
  $A^+/A$ of $\Kk$.
\end{itemize}

We discuss cube-boundedness in \S\ref{sec:cube-bounded}, flattening in
\S\ref{sec:flattening}, the abstract form of the main construction in
\S\ref{sec:pedestals}, and the dp-finite case in
\S\ref{sec:model-theory-inflator}.

\section{Cube-boundedness} \label{sec:cube-bounded}

\subsection{Cube bounded lattices}
Recall the notion of strict $r$-cubes and reduced rank (Definitions
9.13, 9.17 in \cite{prdf}).  A strict $r$-cube in a modular lattice
$M$ is an injective homomorphism of (unbounded) lattices from the
powerset of $r$ to $M$.  The \emph{base} of the cube is the image of
$\emptyset$ under this homomorphism.  Equivalently, a strict $r$-cube
is an (unbounded) sublattice isomorphic to the boolean algebra of size
$2^r$, and the base of the cube is the minimum element of the
sublattice.

The reduced rank $\redrk(M)$ is the maximum $r$ such that a strict
$r$-cube exists in $M$, or $\infty$ if there is no maximum.  If $a \ge
b$ are elements of $M$, then $\redrk(a/b)$ is the reduced rank of the
sublattice $[b,a] \subseteq M$.

\begin{definition}
  \label{def:cube-bounded}
  A bounded modular lattice $(M,\vee,\wedge,\bot,\top)$ is
  \emph{cube-bounded} if $\redrk(M) < \infty$, i.e., there is a
  uniform bound on the size of strict cubes in $M$.
\end{definition}
\begin{remark}\label{rem:mod-fin-cb}
  A modular lattice of finite length is always cube-bounded, because a
  strict $n$-cube contains a chain of length $n$.
\end{remark}
Cube-boundedness is equivalent to something like a uniform
Baldwin-Saxl property:
\begin{proposition}\label{prop:bs}
  Let $M$ be a bounded modular lattice and $n > 1$ be an integer.
  The following are equivalent:
  \begin{enumerate}
  \item \label{bs1} There is a strict $n$-cube in $M$.
  \item \label{bs2} There are $a_1, \ldots, a_n \in M$ such that for
    every $1 \le i \le n$,
    \begin{equation*}
      a_1 \vee \cdots \vee a_n \ne a_1 \vee \cdots \vee \widehat{a_i} \vee \cdots \vee a_n
    \end{equation*}
  \item \label{bs3} There are $a_1, \ldots, a_n \in M$ such that for
    every $1 \le i \le n$,
    \begin{equation*}
      a_1 \wedge \cdots \wedge a_n \ne a_1 \wedge \cdots \wedge
      \widehat{a_i} \wedge \cdots \wedge a_n
    \end{equation*}    
  \end{enumerate}
\end{proposition}
\begin{proof}
  We give the implications
  (\ref{bs3})$\implies$(\ref{bs1})$\implies$(\ref{bs2}); the reverse
  implications follow by duality.  First suppose (\ref{bs3}) holds.
  For each $i$, let
  \begin{equation*}
    b_i = a_1 \wedge \cdots \wedge \widehat{a_i} \wedge \cdots \wedge a_n.
  \end{equation*}
  and let $c = a_1 \wedge \cdots \wedge a_n$.  Note that $b_i \ge c$
  for each $i$.  By assumption, $b_i > c$.  We claim that the sequence
  $b_1, \ldots, b_n$ is independent over $c$ (see \cite{prdf},
  Definition~9.11.1).  Indeed, for any $k$,
  \begin{equation*}
    c \le (b_1 \vee \cdots \vee b_{k-1}) \wedge b_k  \le a_k \wedge b_k = c,
  \end{equation*}
  because $b_i = b_i \wedge a_k \le a_k$ for $i \ne k$.  Thus $\{b_1,
  \ldots, b_n\}$ is an independent sequence over $c$.  By
  Proposition~9.15 in \cite{prdf}, the $b_i$ generate a strict
  $n$-cube in $M$.

  Next suppose (\ref{bs1}) holds.  Then there is a strict $n$-cube
  $\{b_S\}_{S \subseteq \{1,\ldots,n\}}$ in $M$.  Let $a_i =
  b_{\{i\}}$.  The map $S \mapsto b_S$ preserves $\vee$, so for any $i$,
  \begin{equation*}
    a_1 \vee \cdots \vee \widehat{a_i} \vee \cdots \vee a_n =
    b_{\{1,\ldots,\widehat{i},\ldots,n\}} < b_{\{1,\ldots,n\}} = a_1 \vee
    \cdots \vee a_n.
  \end{equation*}
  Therefore (\ref{bs2}) holds.
\end{proof}

\begin{proposition}\label{prop-cs}
  Let $G$ be a definable abelian group with finite burden, and $M$ be
  the lattice of type-definable subgroups of $G$, modulo
  00-commensurability.  Then $M$ is cube-bounded; in fact $\redrk(M)$
  is at most the burden of $G$.
\end{proposition}
\begin{proof}
  Let $n = \bdn(G)$.  By the proof of Proposition~4.5.2 in \cite{CKS},
  one knows that if $A_1, \ldots, A_{n+1}$ are type-definable
  subgroups of $G$, then there is an $i$ such that
  \begin{equation*}
    A_1 \cap \cdots \cap \widehat{A_i} \cap \cdots \cap A_{n+1}
    \approx A_1 \cap \cdots \cap A_{n+1}.
  \end{equation*}
  (Alternatively, the dual statement in terms of sums rather than
  intersections is essentially Lemma~2.6 in \cite{frank-jan}.)  Then
  the desired bound $\redrk(M) < n + 1$ follows by
  Proposition~\ref{prop:bs}.
\end{proof}

\begin{lemma}\label{meta-lol}
  Let $R = \Oo_1 \cap \cdots \cap \Oo_n$ be an intersection of $n$
  pairwise incomparable valuation rings on a field $K$.  Then
  $\Sub_R(K)$ has reduced rank exactly $n$.
\end{lemma}
\begin{proof}
    First note that
  \begin{equation*}
    \Oo_1 \cap \cdots \cap \Oo_n \ne \Oo_1 \cap \cdots \cap \widehat{\Oo_i}
    \cap \cdots \cap \Oo_n
  \end{equation*}
  for any $1 \le i \le n$, by \cite{prdf2}, Corollary~6.7.  By the
  implication (\ref{bs3}$\implies$\ref{bs1}) of
  Proposition~\ref{prop:bs}, it follows that $\Sub_R(K)$ has reduced
  rank at least $n$.  Now suppose for the sake of contradiction that
  $\Sub_R(K)$ has reduced rank greater than $n$.  By the implication
  (\ref{bs1}$\implies$\ref{bs3}) of Proposition~\ref{prop:bs}, there
  are $R$-submodules $A_0, \ldots, A_n \le K$ such that
  \begin{equation} \label{as-if}
    A_0 \cap \cdots \cap A_n \ne A_0 \cap \cdots \cap \widehat{A_j} \cap \cdots \cap A_n
  \end{equation}
  for each $0 \le j \le n$.  Let $\val_i : \Kk \to \Gamma_i$ be the
  valuation associated to $\Oo_i$.  By Proposition~6.2.4 in
  \cite{prdf2}, there exist cuts $\Xi_{i,j}$ in $\Gamma_i$ for $0 \le
  j \le n$ such that
  \begin{equation*}
    A_j = \bigcap_{i = 1}^n \{x \in \Kk : \val_i(x) > \Gamma_{i,j}\}.
  \end{equation*}
  Then
  \begin{equation*}
    \bigcap_{j = 0}^n A_j = \bigcap_{i = 1}^n \bigcap_{j = 0}^n \{x
    \in \Kk : \val_i(x) > \Gamma_{i,j}\} = \bigcap_{i = 1}^n \{x \in
    \Kk : \val_i(x) > \Gamma'_i\},
  \end{equation*}
  where
  \begin{equation*}
    \Gamma'_i = \max_{0 \le j \le n} \Gamma_{i,j}.
  \end{equation*}
  Take $f : \{1,\ldots,n\} \to \{0,\ldots,n\}$ such that $\Gamma'_i =
  \Gamma_{i,f(i)}$.  Then
  \begin{equation*}
    A_0 \cap \cdots \cap A_n = A_{f(1)} \cap \cdots \cap A_{f(n)},
  \end{equation*}
  contradicting (\ref{as-if}), since $f$ is not a surjection.
\end{proof}

\subsection{Cube-bounded objects} \label{sec:cbo}
Recall that in an abelian category, the subobject poset $\Sub(A)$ is
always a bounded modular lattice.
\begin{definition}\label{def:redrk}
  Let $\Cc$ be an abelian category.
  \begin{enumerate}
  \item If $A \in \Cc$, then the \emph{reduced rank} of $A$ is defined
    to be
    \begin{equation*}
      \redrk(A) := \redrk(\Sub_\Cc(A)),
    \end{equation*}
    which is possibly infinite.
  \item If $A \in \Cc$, then $A$ is \emph{cube-bounded} if $\redrk(A)
    < \infty$, i.e., $\Sub_\Cc(A)$ is a cube-bounded lattice.
  \item $\Cc$ itself is \emph{cube-bounded} if every object $A \in
    \Cc$ is cube-bounded.
  \end{enumerate}
\end{definition}
\begin{remark}\label{rem:cb-descr}
  The reduced rank of $A$ has a very concrete meaning: $\redrk(A) \ge
  r$ if and only if there are $C \le B \le A$ such that the
  subquotient $B/C$ is a direct sum of $r$ non-trivial objects.
\end{remark}
\begin{remark}
  By Remark~\ref{rem:mod-fin-cb}, any object of finite length is
  cube-bounded.
\end{remark}
Reduced rank behaves a bit like dp-rank:
\begin{proposition}\label{redrk}
  Let $\Cc$ be an abelian category.
  \begin{enumerate}
  \item \label{rr1} If $f : A \to B$ is an epimorphism, then
    $\redrk(A) \ge \redrk(B)$.
  \item \label{rr2} If $f : A \to B$ is a monomorphism, then
    $\redrk(A) \le \redrk(B)$.
  \item \label{rr3} If
    \begin{equation*}
      0 \to A \to B \to C \to 0
    \end{equation*}
    is a short exact sequence, then
    \begin{equation*}
      \redrk(B) \le \redrk(A) + \redrk(C).
    \end{equation*}
    Equality holds if the sequence splits.
  \item \label{rr4} $\redrk(A) = 0 \iff A \cong 0$.
  \end{enumerate}
\end{proposition}
\begin{proof}
  First note that if $A \cong A'$, then $\Sub(A) \cong \Sub(A')$ so
  $\redrk(A) = \redrk(A')$.  In other words, reduced rank is an
  isomorphism invariant.

  Let $0 \to A \to B \to C \to 0$ be a short exact sequence; view $A$
  as a subobject of $B$.  We have isomorphisms
  \begin{align*}
    \Sub(A) \cong [0,A] &\subseteq \Sub(B) \\
    \Sub(C) \cong [A,B] &\subseteq \Sub(B)
  \end{align*}
  by the third isomorphism theorem.  Now
  \begin{equation*}
    \max(\redrk(A/0),\redrk(B/A)) \le \redrk(B/0) \le \redrk(A/0) + \redrk(B/A)
  \end{equation*}
  by \cite{prdf}, Proposition~9.28.1.  This implies (\ref{rr1}),
  (\ref{rr2}), and the first half of (\ref{rr3}).  If the sequence
  splits, there is a subobject $A' \le B$ complementary to $A \le B$.
  Then
  \begin{equation*}
    \redrk(B/0) = \redrk(A/0) + \redrk(A'/0)
  \end{equation*}
  by \cite{prdf}, Proposition~9.28.1.  As $A' \cong C$, it follows that
  \begin{equation*}
    \redrk(B) = \redrk(A) + \redrk(A') = \redrk(A) + \redrk(C).
  \end{equation*}

  Finally, part (\ref{rr4}) is trivial: there is a 1-cube in a modular
  lattice if and only if the modular lattice contains more than one
  element, and $A \not \cong 0$ if and only if $A$ has more than
  one subobject.
\end{proof}
\begin{corollary}\label{cor:cb-serre}
  Given a short exact sequence
  \begin{equation*}
    0 \to A \to B \to C \to 0,
  \end{equation*}
  in an abelian category $\Cc$, the following are equivalent:
  \begin{itemize}
  \item $B$ is cube-bounded.
  \item $A$ and $C$ are cube-bounded.
  \end{itemize}  
\end{corollary}
\begin{proof}
  This follows from the identity
  \begin{equation*}
    \max(\redrk(A),\redrk(C)) \le \redrk(B) \le \redrk(A) + \redrk(C). \qedhere
  \end{equation*}
\end{proof}
\begin{corollary}
  If $A$ is a cube-bounded object in an abelian category $\Cc$, then
  the neighborhood of $A$ is a cube-bounded abelian category
  (\S\ref{sec:neighborhoods}).
\end{corollary}
\begin{remark}
  An abelian category $\Cc$ is cube-bounded if and only if it has a
  ``rank'' function from objects to $\Nn$ satisfying certain axioms.
  See Appendix~\ref{app:ranks}.
\end{remark}
The following variant of reduced rank will be important in
\S\ref{sec:flattening}.
\begin{definition}[= \cite{prdf}, Definition~9.47]
  Let $M$ be a bounded modular lattice.  Then $\botrk(M)$ is the
  maximum $n$ such that a strict $n$-cube exists in $M$ with base
  $\bot$.  If no maximum exists, $\botrk(M) = \infty$.
\end{definition}
\begin{definition}\label{def:botrk}
  If $A$ is an object in an abelian category, then $\botrk(A) :=
  \botrk(\Sub(A))$.
\end{definition}
\begin{remark}\label{botrk-rem}
  Let $A$ be an object in a cube-bounded abelian category.
  \begin{enumerate}
  \item $\botrk(A)$ is the supremum over all $n$ such that some
    subobject of $A$ is a direct sum of $n$ non-trivial objects.
    (Compare with Remark~\ref{rem:cb-descr}).
  \item\label{br2} $\botrk(A) \le \redrk(A) < \infty$.
  \item If $A$ is non-trivial, then $1 \le \botrk(A)$.
  \end{enumerate}
\end{remark}

\subsection{Cube-bounded directories}
\begin{definition}
  A directory $D_\bullet \cong \Dir_\Cc(A)$ is \emph{cube-bounded} if
  it satisfies one of the following equivalent conditions:
  \begin{itemize}
  \item $D_1$ is a cube-bounded lattice.
  \item $D_n$ is a cube-bounded lattice for all $n$.
  \item $A$ is a cube-bounded object.
  \item The neighborhood of $A$ in $\Cc$ is a cube-bounded abelian
    category.
  \end{itemize}
\end{definition}
These are equivalent by Corollary~\ref{cor:cb-serre}.  Note that for a
directory $D_\bullet$,
\begin{equation*}
  \textrm{semisimple} \implies \textrm{finite-length} \implies
  \textrm{cube-bounded}
\end{equation*}

\subsection{Cube-boundedness in the model-theoretic case}
\begin{proposition}\label{prop:genug}
  Let $\Kk$ be a saturated field of finite burden extending a small
  infinite field $K_0$.  Let the directories $\Lambda_\bullet,
  \Lambda^{00}_\bullet, \Delta_\bullet$ be as in \S\ref{sec:lambda}.
  Let the categories $\mathcal{D}, \Hh, \Hh^{00}$ be as in
  \S\ref{sec:def-setting}-\ref{sec:tdef00-setting}.
  \begin{itemize}
  \item The categories $\mathcal{D}$ and $\Hh^{00}$ are cube-bounded.
  \item The directories $\Delta_\bullet$ and $\Lambda^{00}_\bullet$
    are cube-bounded; in fact
    \begin{align*}
      \redrk(\Delta_n) &\le n \cdot \bdn(\Kk) \\
      \redrk(\Lambda^{00}_n) & \le n \cdot \bdn(\Kk).
    \end{align*}
  \end{itemize}
  Suppose moreover that $\Kk$ is NIP and $K_0$ is a magic subfield.
  \begin{itemize}
  \item The category $\Hh$ is cube-bounded.
  \item The directory $\Lambda_\bullet$ is cube-bounded; in fact
    \begin{equation*}
      \redrk(\Lambda_n) \le n \cdot \dpr(\Kk).
    \end{equation*}
  \end{itemize}
\end{proposition}
\begin{proof}
  We first consider the category $\Dd$.  The objects of $\Dd$ are
  interpretable groups, so each has a well-defined finite burden.
  (Finiteness follows by sub-multiplicativity of burden, Corollary 2.6
  in \cite{Ch}.)  Burden satisfies the following well-known
  properties:
  \begin{enumerate}
  \item If $A \twoheadrightarrow B$ is an epimorphism, then $\bdn(A)
    \ge \bdn(B)$.
  \item If $A \hookrightarrow B$ is a monomorphism, then $\bdn(A) \le
    \bdn(B)$.
  \item $\bdn(A \times B) \ge \bdn(A) + \bdn(B)$.
  \item $\bdn(A) > 0$ iff $A$ is infinite.  Since $A$ is a
    $K_0$-vector space and $K_0$ is infinite,
    \begin{equation*}
      \bdn(A) > 0 \iff A \not\cong 0.
    \end{equation*}
  \end{enumerate}
  \begin{claim}\label{genug-claim}
    For any $A \in \mathcal{D}$, $\bdn(A) \ge \redrk(A)$.
  \end{claim}
  \begin{claimproof}
    If $\redrk(A) \ge n$, then there are subobjects $C \subseteq B
    \subseteq A$ and a direct sum decomposition
    \begin{equation*}
      B/C \cong D_1 \oplus \cdots \oplus D_n
    \end{equation*}
    with the $D_i \not\cong 0$.  Then
    \begin{equation*}
      \bdn(A) \ge \bdn(B) \ge \bdn(B/C) = \bdn(D_1 \oplus \cdots
      \oplus D_n) \ge \sum_{i = 1}^n \bdn(D_i) \ge \sum_{i = 1}^n 1 =
      n.
    \end{equation*}
    Thus $\bdn(A) \ge n$ for any $n \le \redrk(A)$.
  \end{claimproof}
  It follows that $\mathcal{D}$ is a cube-bounded abelian category.
  Moreover, $\redrk(\Kk) \le \bdn(\Kk)$.  By
  Proposition~\ref{redrk}.\ref{rr3},
  \begin{equation*}
    \redrk(\Delta_n) = \redrk(\Sub_{\mathcal{D}}(\Kk^n)) =
    \redrk(\Kk^n) \le n \cdot \redrk(\Kk) \le n \cdot \bdn(\Kk).
  \end{equation*}
  Next consider the lattice $\Lambda_1^{00}$.  The lattice
  $\Lambda_1^{00}$ embeds into the lattice of type-definable subgroups
  of $\Kk$ modulo 00-commensurability, so
  \begin{equation*}
    \redrk(\Lambda_1^{00}) \le \bdn(\Kk)
  \end{equation*}
  by Proposition~\ref{prop-cs}.  Using
  Proposition~\ref{prop:h-lambda-00} to relate $\Lambda^{00}_\bullet$
  to $\Hh^{00}$, it follows that in the category $\Hh^{00}$,
  \begin{equation*}
    \redrk(\Kk) = \redrk(\Sub_{\Hh^{00}}(\Kk)) = \redrk(\Lambda_1^{00})
    \le \bdn(\Kk).
  \end{equation*}
  By Proposition~\ref{redrk}.\ref{rr3},
  \begin{equation*}
    \redrk(\Lambda_n^{00}) = \redrk(\Sub_{\Hh^{00}}(\Kk^n)) =
    \redrk(\Kk^n) \le n \cdot \redrk(\Kk) \le n \cdot \bdn(\Kk).
  \end{equation*}
  By construction of $\Hh$ and $\Hh^{00}$, every object of $\Hh^{00}$
  is a subquotient of $\Kk^n$, so $\Hh^{00}$ is cube-bounded by
  Proposition~\ref{redrk}.\ref{rr1}-\ref{rr2}.

  Lastly, if $\Kk$ is NIP and $K_0$ is a magic subfield, then
  $\Lambda_\bullet$ is isomorphic to $\Lambda^{00}_\bullet$ and $\Hh$
  is equivalent to $\Hh^{00}$, by the discussion at the end of
  \S\ref{sec:tdef00-setting}.  In an NIP context, burden agrees with
  dp-rank.
\end{proof}
\begin{remark}
  There is something funny about the situation with $\mathcal{D}$ in
  the finite burden case.  By Proposition~\ref{redrk}, there is a rank
  function $\redrk : \mathcal{D} \to \Zz_{\ge 0}$ which satisfies the
  sub-additivity properties of dp-rank.  For example, if
  \begin{equation*}
    0 \to A \to B \to C \to 0
  \end{equation*}
  is an exact sequence in $\Dd$, then
  \begin{equation*}
    \redrk(B) \le \redrk(A) + \redrk(C).
  \end{equation*}
  The analogous property for burden is unknown.  So somehow we found a
  way to upgrade weakly sub-additive ranks into strongly sub-additive
  ranks.\footnote{On the other hand, the new rank is only defined on
    interpretable \emph{abelian groups} or vector spaces, rather than
    on interpretable \emph{sets}.  Nothing miraculous is going on
    here.}  For a more general statement, see
  Appendix~\ref{app:ranks}.
\end{remark}

\section{Flattening}\label{sec:flattening}
In this section, we carry out three parallel ``flattening'' constructions.

If $M$ is a cube-bounded modular lattice, we will define a semisimple
modular lattice $M^\flat$ called the \emph{flattening} of $M$, as well
as a \emph{flattening map} $M \to M^\flat$ that is surjective and
preserves $\wedge$.

If $\Cc$ is a cube-bounded abelian category, we will define a
semisimple abelian category $\Cc^\flat$, called the \emph{flattening}
of $\Cc$, as well as a \emph{quasi-socle functor} $\qsoc : \Cc \to
\Cc^\flat$, that is essentially surjective and left-exact.  The
induced maps on subobject-lattices
\begin{equation*}
  \Sub_\Cc(A) \to \Sub_{\Cc^\flat}(\qsoc(A))
\end{equation*}
are flattening maps on modular lattices.

If $D_\bullet$ is a cube-bounded directory, we will define a
semisimple directory $D^\flat_\bullet$, called the \emph{flattening}
of $D_\bullet$, as well as a directory morphism $D_\bullet \to
D^\flat_\bullet$, called the \emph{flattening map}.  At each level, the map
\begin{equation*}
  D_n \to D^\flat_n
\end{equation*}
is a flattening map on modular lattices.  One constructs
$D_\bullet^\flat$ by choosing an isomorphism $D_\bullet \cong
\Dir_\Cc(A)$, and setting $D^\flat_\bullet =
\Dir_{\Cc^\flat}(\qsoc(A))$.

For the case of lattices, flattening is essentially the ``modular
pregeometry on quasi-atoms'' constructed in \S9.4 of \cite{prdf}.  But
we will see that the pro-construction in category theory gives a
better way to understand flattening.

\subsection{Flattening a lattice}
Let $(M,\wedge,\vee,\bot)$ be a modular lattice with minimum element
$\bot$.  Recall that a ``quasi-atom'' in $M$ (Definition~9.32 in \cite{prdf})
is an element $q > \bot$ such that the interval $(\bot,q]$ is closed
under intersection (i.e., a sublattice).  In \S9.4 of \cite{prdf}, we
constructed a (finitary) modular pregeometry on the set of
quasi-atoms, characterized by the fact that a finite set
$\{q_1,\ldots,q_n\}$ is independent in the pregeometry if and only if
it is ``lattice-theoretically independent:''
\begin{equation*}
  q_1 \wedge q_2 = \bot, ~ (q_1 \vee q_2) \wedge q_3 = \bot, ~ (q_1
  \vee q_2 \vee q_3) \wedge q_4 = \bot, ~ \ldots
\end{equation*}
See Corollary~9.39 and Proposition~9.41 in \cite{prdf} for details.
For any $x \in M$, the set
\begin{equation*}
  V(x) = \{q \in M : q \textrm{ is a quasi-atom and } q \wedge x >
  \bot\}
\end{equation*}
is a closed set in this pregeometry (\cite{prdf}, Corollary~9.39).  If
the pregeometry has finite rank, then every closed set is of this form
(\cite{prdf}, Proposition~9.45.2).

\begin{proposition}\label{soft-cube-bound}
  Let $(M,\wedge,\vee,\bot)$ be a modular lattice with minimum element
  $\bot$.  The following are equivalent:
  \begin{enumerate}
  \item \label{kj1} There is no infinite sequence $a_1, a_2, a_3,
    \ldots > \bot$ that is independent, in the sense that
    \begin{equation*}
      a_1 \wedge a_2 = \bot, ~ (a_1 \vee a_2) \wedge a_3 = \bot, ~
      (a_1 \vee a_2 \vee a_3) \wedge a_4 = \bot, ~ \ldots
    \end{equation*}
  \item \label{kj2} For every $a > \bot$, there is a quasi-atom $q \le
    a$, and the pregeometry on quasi-atoms has finite rank.
  \item \label{kj3} $\botrk(M) < \infty$, i.e., there is a finite
    bound on the length of independent sequence as in (\ref{kj1}).
  \end{enumerate}
\end{proposition}
\begin{proof}
  We will prove
  (\ref{kj1})$\implies$(\ref{kj2})$\implies$(\ref{kj3})$\implies$(\ref{kj1}).
  
  Assume (\ref{kj1}) holds.  The pregeometry on quasi-atoms certainly
  has finite rank; otherwise we could find an infinite independent
  set, which would yield an infinite independent sequence.  Say that
  an element $a \in M$ is ``bad'' if $a > \bot$ but there is no
  quasi-atom $q \le a$.  Let $B \subseteq M$ be the set of bad
  elements.  We claim $B$ is empty.  Note that if $a \in B$ then $a$
  is not a quasi-atom, so there exist $\bot < b,c \le a$ such that $b
  \wedge c = \bot$.  A fortiori, both $b$ and $c$ are bad.  If $B$ is
  non-empty, recursively build two sequences of bad elements
  \begin{align*}
    &a_0, a_1, a_2, \ldots \\
    &b_0, b_1, b_2, \ldots
  \end{align*}
  where
  \begin{itemize}
  \item $a_0$ is some bad element.
  \item For each $i \ge 0$, we have $\bot < b_i, a_{i+1} \le a_i$ and
    $b_i \wedge a_{i+1} = \bot$.
  \end{itemize}
  Then
  \begin{equation*}
    a_0 \ge a_1 \ge a_2 \ge \cdots
  \end{equation*}
  and $b_i \le a_i$ for all $i$.  Then for any $i < j$,
  \begin{equation*}
    b_i \wedge (b_{i+1} \vee \cdots \vee b_j) \le b_i \wedge (a_{i+1} \vee \cdots \vee a_j) = b_i \wedge a_{i+1} = \bot.
  \end{equation*}
  It follows that for any $j$, the sequence
  \begin{equation*}
    b_j, b_{j-1}, \ldots, b_1, b_0
  \end{equation*}
  is independent.  By symmetry of independence (\cite{prdf}, Proposition 9.3), the sequence
  \begin{equation*}
    b_0, b_1, \ldots, b_j
  \end{equation*}
  is independent.  This holds for all $j$, so the sequence of $b_i$'s
  is an infinite independent sequence, contradicting (\ref{kj1}).

  Next suppose (\ref{kj2}) holds.  Let $n$ be the rank of the
  pregeometry on quasi-atoms.  We claim that there is no independent
  sequence $b_1, b_2, \ldots, b_{n+1} > \bot$.  Otherwise, take $q_i$
  a quasi-atom below $b_i$.  Then the sequence $q_1, q_2, \ldots,
  q_{n+1}$ is an independent sequence in the pregeometry,
  contradicting the choice of $n$.  Thus
  (\ref{kj2})$\implies$(\ref{kj3}).  Finally, the implication
  (\ref{kj3})$\implies$(\ref{kj1}) is trivial.
\end{proof}
Cube-bounded lattice satisfy the equivalent conditions of
Proposition~\ref{soft-cube-bound}.  So do Noetherian modular lattices:
if $b_1, b_2, \ldots$ were an infinite independent sequence of
elements $b_i > \bot$, then for each $n$ the initial subsequence $b_1,
\ldots, b_n$ generates a strict $n$-cube (by \cite{prdf},
Proposition~9.15.3), and so
\begin{equation*}
  b_1 < b_1 \vee b_2 < b_1 \vee b_2 \vee b_3 < \cdots < (b_1 \vee \cdots \vee b_n).
\end{equation*}
Thus the sequence $b_1, b_1 \vee b_2, \ldots$ is an infinite ascending
sequence.

\begin{definition}
  Let $M$ be a bounded modular lattice with $\botrk(M) < \infty$.  Let
  $M^\flat$ be the lattice of closed sets in the pregeometry on
  quasi-atoms.  The \emph{standard flattening map} is the map
  \begin{equation*}
    V : M \to M^\flat
  \end{equation*}
  sending $x$ to the set $V(x)$ of quasi-atoms $q$ with $q \wedge x >
  \bot$.

  More generally, if $(P,\le)$ is a poset and $f : M \to P$ is a
  function, we say that $f$ is a \emph{flattening map} if it is
  isomorphic to $V$, i.e., there is a poset isomorphism $g : M^\flat
  \to P$ such that $f = g \circ V$.
\end{definition}
The punchline of the next few sections is that if
\begin{equation*}
  D_\bullet = (D_1,D_2,\ldots)
\end{equation*}
is a cube-bounded directory, then so is
\begin{equation*}
  D^\flat_\bullet = (D_1^\flat, D_2^\flat, \ldots),
\end{equation*}
and there is a morphism of directories $D_\bullet \to D_\bullet^\flat$
whose $n$th component is the flattening map
\begin{equation*}
  D_n \to D_n^\flat.
\end{equation*}
\begin{proposition}\label{fattening-facts}
  Let $M$ be a bounded modular lattice with $\botrk(M) < \infty$.  Let
  $f : M \to M'$ be a flattening map.
  \begin{enumerate}
  \item\label{ff1} $M'$ is a semisimple modular lattice of length equal to
    $\botrk(M)$.
  \item \label{ff2} $f$ is surjective.
  \item If $x \ge y$, then $f(x) \ge f(y)$.
  \item\label{ff4} $f(x \wedge y) = f(x) \wedge f(y)$.
  \item $f(x \vee y) \ge f(x) \vee f(y)$.
  \item \label{ff6} $x > \bot \iff f(x) > \bot$.
  \end{enumerate}
\end{proposition}
\begin{proof}
  By definition, we may assume $f = V$ and $M' = M^\flat$.  Then
  $M^\flat$ is a modular lattice because the pregeometry on
  quasi-atoms is modular (\cite{prdf}, Proposition~9.41), and the
  length of $M^\flat$ is finite and equal to $\botrk(M)$ by
  \cite{prdf}, Remark~9.48.3.  The map $V$ is surjective because every
  closed set of $M$ is of the form $V(x)$ (\cite{prdf},
  Proposition~9.45.2).  If $x \le y$, then
  \begin{equation*}
    V(x) = \{q : q \wedge x > \bot\} \subseteq V(y) = \{q : q \wedge y > \bot\}.
  \end{equation*}
  This implies the identities
  \begin{align*}
    V(x \wedge y) &\subseteq V(x) \cap V(y) \\
    V(x \vee y) &\subseteq V(x) \vee V(y)
  \end{align*}
  where $V(x) \vee V(y)$ is the closed set generated by $V(x) \cup
  V(y)$.  The reverse inclusion
  \begin{equation*}
    V(x) \cap V(y) \subseteq V(x \wedge y)
  \end{equation*}
  is Lemma~9.37 in \cite{prdf}, or can be seen as follows: if $q
  \wedge x > \bot$ and $q \wedge y > \bot$, then the two elements $q
  \wedge x$ and $q \wedge y$ are non-trivial elements below $q$.  As
  $q$ is a quasi-atom, their meet $(q \wedge x) \wedge (q \wedge y)$
  is also non-trivial, implying that $(x \wedge y) \wedge q > \bot$
  and $q \in V(x \wedge y)$.

  If $x = \bot$, then $x \wedge q = \bot$ for all $q$, so $V(x) =
  \emptyset$.  Conversely, if $x > \bot$, then there is some
  quasi-atom $q \le x$ by Proposition~\ref{soft-cube-bound}.  Then $q
  \wedge x = q > \bot$, implying $q \in V(x)$ and $V(x) > \emptyset$.
\end{proof}

\subsection{Quasi-atoms as pro-objects}
If $\Cc$ is a category, then $\Pro \Cc$ denotes the category of
pro-objects.  See Appendix~\ref{app:ind} for a review of pro-objects
and ind-objects.  We will use the facts listed in \S\ref{sec:pro}.

Let $M$ be a bounded lattice, viewed as a poset, viewed as a category.
The category $\Pro M$ is dual to the poset of filters, ordered by
inclusion.  Here, a \emph{filter} is a subset $F \subseteq M$ such
that
\begin{align*}
  \top \in F \\
  x, y \in F \implies &x \wedge y \in F \\
  (x \in F \textrm{ and } y \ge x) \implies & y \in F
\end{align*}
The embedding of $M$ into $\Pro M$ sends an element $a \in M$ to the
principal filter
\begin{equation*}
  \{x \in M : x \ge a\}.
\end{equation*}
Note that $\Pro M$ is itself a (complete!) bounded lattice.

\begin{lemma}\label{pro-hom}
  The embedding $M \hookrightarrow \Pro M$ is a homomorphism of
  bounded lattices.
\end{lemma}
\begin{proof}
  If $a = \top$, then the principal filter generated by $a$ is
  $\{\top\}$, which is clearly the minimum filter.

  If $a = \bot$, then the principal filter generated by $a$ is $M$,
  clearly the maximum filter.

  If $a, b$ are two elements, then
  \begin{equation*}
    \{x \in M : x \ge a\} \cap \{x \in M : x \ge b\} = \{x \in M : x \ge (a \vee b)\}
  \end{equation*}
  So the embedding $M \hookrightarrow \Pro M$ preserves $\vee$.

  It remains to show that the filter generated by
  \begin{equation*}
    \{x \in M : x \ge a\} \cup \{x \in M : x \ge b\}
  \end{equation*}
  is the filter generated by $a \wedge b$.  In other words, we must
  show that the filter generated by $\{a,b\}$ is the filter generated
  by $\{a \wedge b\}$.  From the definition of filter, it is clear
  that if $F$ is a filter, then
  \begin{equation*}
    F \supseteq \{a, b\} \iff F \supseteq \{a \wedge b\},
  \end{equation*}
  implying the desired statement.  Thus the embedding $M
  \hookrightarrow \Pro M$ preserves $\wedge$.
\end{proof}
\begin{proposition}
  If $M$ is a modular lattice, then $\Pro M$ is a modular lattice.
\end{proposition}
\begin{proof}
  The dual of a modular lattice is a modular lattice, so it suffices
  to show that the lattice of filters is modular.  If $A, B$ are two
  filters on $M$, let $A + B$ denote the upwards closure of the set
  \begin{equation*}
    S := \{a \wedge b : a \in A, ~ b \in B\}
  \end{equation*}
  Then $A + B$ is a filter:
  \begin{itemize}
  \item As $A, B$ are filters, $\top \in A, \top \in B$, and so $\top
    = \top \wedge \top \in S \subseteq A + B$.
  \item Suppose $x_1, x_2 \in A + B$.  Then $x_i \ge a_i \wedge b_i$
    for some $a_1, a_2 \in A$ and $b_1, b_2 \in B$.  Then
    \begin{equation*}
      x_1 \wedge x_2 \ge (a_1 \wedge b_1) \wedge (a_2 \wedge b_2) = (a_1 \wedge a_2) \wedge (b_1 \wedge b_2) \in S.
    \end{equation*}
  \item $A + B$ is upwards-closed by fiat.
  \end{itemize}
  Also $A \cup B \subseteq A + B$, because of terms like $a \wedge
  \top$ and $\top \wedge b$.  From all this, it follows that $A + B$
  is exactly the filter generated by $A \cup B$.

  To show modularity, suppose $A, B, C$ are filters on $M$ and $A
  \subseteq B$.  We must show
  \begin{equation*}
    (C + A) \cap B \stackrel{?}{\subseteq} (C \cap B) + A.
  \end{equation*}
  Suppose $x \in (C + A) \cap B$.  Then $x \ge c \wedge a$ for some $c
  \in C$ and $a \in A$.  Take $b = x \wedge a$.  Then $x \ge b \le a$, and so
  \begin{equation*}
    (x \ge c \wedge a \textrm{ and } x \ge b) \implies x \ge (c \wedge
    a) \vee b \stackrel{!}{=} (c \vee b) \wedge a,
  \end{equation*}
  where $\stackrel{!}{=}$ is by modularity of $M$.  Also $x \in B$ and
  $a \in A \subseteq B$, so $b = x \wedge a \in B$ because $B$ is a
  filter.  Then $c \vee b \in C \cap B$ and $a \in A$, so
  \begin{equation*}
    x \ge (c \vee b) \wedge a \implies x \in (C \cap B) + A.
  \end{equation*}
  This proves that the lattice of filters on $M$ is modular, which in
  turn implies $\Pro M$ is modular.
\end{proof}
\begin{lemma}\label{enough-atoms}
  If $M$ is a modular lattice, the lattice $\Pro M$ has enough atoms:
  if $x \in \Pro M$ and $x > \bot$, there is an atom $a \in \Pro M$
  with $a \le x$.
\end{lemma}
\begin{proof}
  A filter $F \subseteq M$ is proper if and only if $\bot \notin F$.
  By Zorn's lemma, every proper filter is contained in a maximal
  proper filter.
\end{proof}

\begin{lemma}\label{lem:qa-filters-one}
  Let $M$ be a bounded modular lattice with $\botrk(M) < \infty$, and
  $Q$ be the set of quasi-atoms in $M$.
  \begin{enumerate}
  \item If $q \in Q$, the set
    \begin{equation*}
      F_q := \{x \in M : x \wedge q > \bot\} = \{x \in M : q \in V(x)\}
    \end{equation*}
    is a proper filter on $M$, containing $q$.
  \item If $q \in Q$ and $F'$ is a proper filter containing $q$, then
    $F' \subseteq F_q$.  Therefore, $F_q$ is a maximal proper filter.
  \item Every maximal proper filter is of the form $F_q$ for some $q
    \in Q$.
  \end{enumerate}
\end{lemma}
\begin{proof}
  \begin{enumerate}
  \item $F_q$ is clearly upwards-closed.  We check that it is closed
    under intersection.  Suppose $x \wedge q > \bot$ and $y \wedge q >
    \bot$.  Then $\{x \wedge q, y \wedge q\}$ is a subset of
    $(\bot,q]$.  By definition of quasi-atom, $(\bot,q]$ is closed
        under $\wedge$, and so
    \begin{equation*}
      (x \wedge y) \wedge q = (x \wedge q) \wedge (y \wedge q) \in (\bot,q].
    \end{equation*}
    Therefore $x \wedge y \in F_q$.  It is clear that $q \in F_q$ and
    $\bot \notin F_q$.
  \item Suppose $q \in Q \cap F'$ but $F' \not \subseteq F_q$.  Take
    $a \in F' \setminus F_q$.  Then $a \wedge q = \bot$ by definition
    of $F_q$.  On the other hand, $F'$ contains $a$ and $q$, so it
    must contain $\bot$, therefore failing to be a proper filter.
  \item Let $F$ be a maximal proper filter.  The flattening map $V : M \to
    M^\flat$ is order-preserving, so the image $V(F)$ of $F$ under
    this map is downwards directed.  As $M^\flat$ has finite length,
    it follows that $V(F)$ contains a minimum element.  Thus, there is
    $a \in F$ such that $V(a) \subseteq V(x)$ for any $x \in F$.
    Properness of $F$ implies $a > \bot$, which implies $V(a) \ne
    \emptyset$ by Proposition~\ref{fattening-facts}.\ref{ff6}.
    Take $q$ one of the quasi-atoms in the set $V(a)$.  Then $q \in
    V(a) \subseteq V(x)$ for all $x \in F$, implying that $F \subseteq
    F_q$. \qedhere
  \end{enumerate}
\end{proof}
\begin{lemma}\label{lem:qa-filters-two}
  Let $M$ be a modular lattice with $\botrk(M) < \infty$.  Let $V : M
  \to M^\flat$ be the flattening map.
  \begin{enumerate}
  \item For every $A \in M^\flat$, the set
    \begin{equation*}
      F_A = \{x \in M : V(x) \ge A\} 
    \end{equation*}
    is a filter on $M$.
  \item The resulting map
    \begin{equation*}
      A \mapsto F_A
    \end{equation*}
    is an order-reversing isomorphism (of posets) from $M^\flat$ to its image, and
    satisfies the identity
    \begin{equation*}
      F_{A \vee B} = F_A \cap F_B.
    \end{equation*}
  \item A filter $F \subseteq M$ is of the form $F_A$ for some $A \in
    M^\flat$ if and only if $F$ is a finite intersection of zero or
    more maximal proper filters.
  \end{enumerate}
\end{lemma}
\begin{proof}
  For $A, A' \in M^\flat$, note that
  \begin{align*}
    F_{A \vee A'} = &\{x \in M : V(x) \ge A \vee A'\} \\ = & \{x \in M :
    V(x) \ge A\} \cap \{x \in M : V(x) \ge A'\} = F_A \cap F_{A'}.
  \end{align*}
  If $A$ is an atom in $M^\flat$, then $A$ is the closure of $\{q\}$ for some
  quasi-atom $q \in M$, and
  \begin{align*}
    F_A &= \{x \in M : V(x) \supseteq A\} \\ &= \{x \in M : V(x)
    \supseteq \{q\}\} \\ &= \{x \in M : V(x) \ni q\} \\ &= \{x \in M :
    x \wedge q > \bot\},
  \end{align*}
  which is a filter by Lemma~\ref{lem:qa-filters-one}.  In general, we
  can write $B = A_1 \vee \cdots \vee A_n$ where the $A_i$ are atoms,
  and
  \begin{equation*}
    F_B = \bigcap F_{A_i}.
  \end{equation*}
  Each $F_{A_i}$ is a maximal proper filter, and every maximal proper
  filter is of the form $F_A$ for some atom $A$, by
  Lemma~\ref{lem:qa-filters-one}.  So as $B$ ranges over $M^\flat$,
  $F_B$ ranges over finite intersections of maximal proper filters,
  proving the first and third points.

  It remains to show that $A \mapsto F_A$ is strictly order-reversing.
  For $A, A' \in M^\flat$, we have equivalences
  \begin{align*}
    F_{A'} \subseteq F_A & \iff \forall x \in M : \left(V(x) \ge A' \implies V(x) \ge A\right) \\
    & \iff \forall B \in M^\flat : \left(B \ge A' \implies B \ge A\right) \\
    & \iff A' \ge A,
  \end{align*}
  because $V : M \to M^\flat$ is surjective.
\end{proof}
\begin{remark}
  Let $M$ be a module over some ring.  Recall that ``semisimple''
  means ``semisimple of finite length.''  Let $A, B$ be submodules of
  $M$.  One has the following well-known facts:
  \begin{itemize}
  \item If $A, B$ are both semisimple, then $A + B$ is semisimple.
  \item If $A \subseteq B$ and $B$ is semisimple, then $A$ is
    semisimple.
  \end{itemize}
  The proofs generalize to modular lattices.  Let $M$ be a bounded
  modular lattice.  Let $x, y$ be elements of $M$.  The interval
  $[\bot,x]$ is a sublattice of $M$, which is semisimple if and only
  if $x$ is a finite join of atoms.
  \begin{itemize}
  \item If $[\bot, x]$ and $[\bot, y]$ are semisimple, then so is
    $[\bot, x \vee y]$.
  \item If $x \le y$ and $[\bot, y]$ is semisimple, then so is $[\bot,
    x]$.
  \end{itemize}
\end{remark}
\begin{theorem}\label{thm:nice}
  Let $M$ be a bounded modular lattice.  Suppose $\botrk(M) < \infty$.
  \begin{enumerate}
  \item \label{tn-soc} There is a unique maximum $s \in \Pro M$ such that $[\bot,s]
    \subseteq \Pro M$ is a semisimple sublattice of $\Pro M$.
  \item \label{tn-2} The composition
    \begin{equation*}
      M \hookrightarrow \Pro M \stackrel{x \mapsto x \wedge s}{\longrightarrow} [\bot,s]
    \end{equation*}
    is a flattening map.
  \item In particular, the lattice $[\bot,s]$ is isomorphic to the
    lattice $M^\flat$ of closed sets in the pregeometry on quasi-atoms in $M$.
  \item In particular, equivalence classes of quasi-atoms in $M$
    correspond exactly to atoms in $\Pro M$.
  \end{enumerate}
\end{theorem}
\begin{proof}
  The lattice $\Pro M$ is dual to the lattice of filters on $M$.  By
  Lemma~\ref{lem:qa-filters-two}, there is therefore a map $g :
  M^\flat \to \Pro M$ with the following properties:
  \begin{itemize}
  \item $g$ is an isomorphism onto its image: for any $A, B \in M^\flat$,
    \begin{equation*}
      A \subseteq B \iff g(A) \subseteq g(B)
    \end{equation*}
  \item $g(A \vee B) = g(A) \vee g(B)$.
  \item An element $x \in \Pro M$ is in the image of $g$ if and only
    if $x$ is a finite join of atoms, or equivalently, if and only if
    $[\bot,x]$ is a semisimple modular lattice.
  \end{itemize}
  As $M^\flat$ has a maximal element, there is a maximal $s \in \Pro
  M$ such that $[\bot,s]$ is semisimple.  Then for any $x \in \Pro M$,
  the following are equivalent:
  \begin{itemize}
  \item $x$ is in the image of $g$
  \item $[\bot,x]$ is semisimple
  \item $x$ is in $[\bot,s]$.
  \end{itemize}
  Therefore the image of $g$ is $[\bot,s]$.  Then $g$ is an
  isomorphism of posets from $M^\flat$ to $[\bot,s]$, hence an
  isomorphism of lattices, and $g \circ V : M \to [\bot,s]$ is a
  flattening map.

  It remains to prove the formula
  \begin{equation*}
    g(V(x)) \stackrel{?}{=} x \wedge s.
  \end{equation*}
  For any $A \in M^\flat$, $g(A)$ is dual to the filter
  \begin{equation*}
    F_A = \{z \in M : V(z) \supseteq A\}
  \end{equation*}
  by definition of $g$.  For any $x \in M$, the element $x \in \Pro M$
  is dual the principal filter
  \begin{equation*}
    \{z \in M : z \ge x\}
  \end{equation*}
  generated by $x$.  Therefore, for $x, y \in M$ we have an equivalence
  \begin{align*}
    x \ge g(V(y)) &\iff \{z \in M : z \ge x\} \subseteq \{z \in M :
    V(z) \supseteq V(y)\} \\ &\iff x \in \{z \in M : V(z) \supseteq
    V(y)\} \\ & \iff V(x) \supseteq V(y).
  \end{align*}
  As $g$ is strictly order-preserving,
  \begin{equation*}
    x \ge g(V(y)) \iff g(V(x)) \ge g(V(y)).
  \end{equation*}
  Now $g \circ V : M \to [\bot,s]$ is surjective
  (Proposition~\ref{fattening-facts}.\ref{ff2}), so for any $x \in M$
  and $w \in [\bot,s]$,
  \begin{equation*}
    x \ge w \iff g(V(x)) \ge w.
  \end{equation*}
  Then for any $x \in M$ and $w \in [\bot,s]$,
  \begin{equation*}
    x \wedge s \ge w \iff x \ge w \iff g(V(x)) \ge w,
  \end{equation*}
  implying that $x \wedge s = g(V(x))$.
\end{proof}

\subsection{Flattening an abelian category}\label{s7.3}
\begin{fact}\label{nonsense}
  Let $\Cc, \mathcal{D}$ be categories with finite limits, and $F :
  \Cc \to \mathcal{D}$ be a functor.  Then $F$ preserves finite limits
  if and only if for every $A \in \mathcal{D}$, the functor
  \begin{align*}
    \Cc &\to \Set \\
    X &\mapsto \Hom_{\mathcal{D}}(A,F(X))
  \end{align*}
  preserves finite limits.
\end{fact}
The following assumption will be in force for all of \S\ref{s7.3}
\begin{assumption}\label{asm:0}
  $\Cc$ is a small $K_0$-linear abelian category that is cube-bounded,
  or satisfies the following weaker condition:
  \begin{equation*}
    \forall A \in \Cc : \botrk(A) < \infty.
  \end{equation*}
  For example, $\Cc$ could be the category of modules over a
  Noetherian $K_0$-algebra.
\end{assumption}
As discussed in Appendix~\ref{sec:pro}, the category $\Pro \Cc$ is
naturally a $K_0$-linear abelian category, and the embedding $\Cc \to
\Pro \Cc$ is fully faithful and exact.  Moreover, for any object $A
\in \Cc$, there is an isomorphism
\begin{equation*}
  \Sub_{\Pro \Cc}(A) \cong \Pro \Sub_\Cc(A).
\end{equation*}
By Theorem~\ref{thm:nice}.\ref{tn-soc} in the previous section, if $A
\in \Cc$, then the pro-object $A \in \Pro \Cc$ has a socle---a maximum
semisimple\footnote{As always, ``semisimple'' means ``semisimple of
  finite length,'' even though $\Pro \Cc$ might have a more general
  notion of semisimplicity.} subobject.  We define the
\emph{quasi-socle} $\qsoc(A)$ to be the socle of $A$-as-a-pro-object.
\begin{remark}\label{doy}
  If $B \subseteq A$ in $\Cc$, then
  \begin{equation*}
    \qsoc(B) = B \cap \qsoc(A)
  \end{equation*}
  where the intersection is taken inside $\Pro \Cc$.  If $A, B$ are
  arbitrary objects in $\Cc$, then
  \begin{equation*}
    \qsoc(A \oplus B) \cong \qsoc(A) \oplus \qsoc(B).
  \end{equation*}
  Both statements follow from general facts about socles.
\end{remark}
\begin{lemma}\label{lem:heyhey} (Under \ref{asm:0}.) For every $A \in \Cc$, the induced
  map
  \begin{align*}
    \Sub_\Cc(A) & \to \Sub_{\Pro\Cc}(\qsoc(A)) \\
    B & \mapsto \qsoc(B) = B \cap \qsoc(A)
  \end{align*}
  is a flattening map.
\end{lemma}
\begin{proof}
  Theorem~\ref{thm:nice}.\ref{tn-2} and Remark~\ref{doy}.
\end{proof}
\begin{theorem}\label{thm:hey}
  (Under \ref{asm:0}.)  Let $\Cc^\flat$ be
  the full subcategory of $\Pro \Cc$ consisting of quasi-socles
  $\qsoc(A)$ for $A \in \Cc$.  Then
  \begin{enumerate}
  \item $\Cc^\flat$ is a small semisimple $K_0$-linear abelian
    category.
  \item \label{th2} The functor $\qsoc : \Cc \to \Cc^\flat$ is left exact and
    essentially surjective.
  \item \label{th3} For every $A \in \Cc$, the induced map
    \begin{align*}
      \Sub_\Cc(A) & \to \Sub_{\Cc^\flat}(\qsoc(A)) \\
      B & \mapsto \qsoc(B)
    \end{align*}
    is a flattening map.
  \end{enumerate}
\end{theorem}
\begin{proof}
  First note that for any $A \in \Cc$, the map
  \begin{align*}
      \Sub_\Cc(A) & \to \Sub_{\Pro \Cc}(\qsoc(A)) \\
      B & \mapsto \qsoc(B) = B \cap \qsoc(A)
  \end{align*}
  is a flattening map, by Lemma~\ref{lem:heyhey}.  As
  flattening maps are surjective
  (Proposition~\ref{fattening-facts}.\ref{ff2}), it then follows that
  $\Cc^\flat$ is closed under taking subobjects in $\Pro \Cc$.  It is
  also closed under direct sums, by Remark~\ref{doy}.  Therefore
  $\Cc^\flat$ is a semisimple abelian category, and
  \begin{equation*}
    \Sub_{\Cc^\flat}(\qsoc(A)) \cong \Sub_{\Pro \Cc}(\qsoc(A)).
  \end{equation*}
  So the natural map $\Sub_\Cc(A) \to \Sub_{\Cc^\flat}(\qsoc(A))$ is a
  flattening map.  The functor $\qsoc(A)$ is essentially surjective by
  definition of $\Cc^\flat$.  Note that for $A \in \Cc$ and $B \in
  \Cc^\flat$, there is a natural isomorphism
  \begin{equation*}
    \Hom_{\Cc^\flat}(B,\qsoc(A)) = \Hom_{\Pro \Cc}(B,\qsoc(A)) \cong
    \Hom_{\Pro \Cc}(B,A).
  \end{equation*}
  By Fact~\ref{nonsense}, the (left-)exactness of the embedding $\Cc
  \hookrightarrow \Pro \Cc$ implies the left-exactness of $\qsoc : \Cc
  \to \Cc^\flat$.
\end{proof}
We call $\Cc^\flat$ the \emph{flattening of $\Cc$}.
\begin{remark}
  There is an alternative approach to construct $\Cc^\flat$,
  proceeding as follows: say that a monomorphism $f : A \to B$ in
  $\Cc$ is ``dense'' if the image $\img(f)$ non-trivially intersects every
  non-trivial subobject of $B$.  The class of dense monomorphisms
  turns out to admit a calculus of right fractions, and $\Cc \to
  \Cc^\flat$ is the localization of $\Cc$ obtained by inverting the
  dense monomorphisms.  We prefer the approach using $\Pro \Cc$
  because it makes the calculations easier.
\end{remark}

\begin{proposition}\label{prop:qsoc-botrk} (Under \ref{asm:0}.)
  If $A \in \Cc$, then the length of $\qsoc(A)$ in $\Cc^\flat$ is
  exactly $\botrk(A)$ in $\Cc$.
\end{proposition}
\begin{proof}
  There is a flattening map
  \begin{equation*}
    \Sub_\Cc(A) \to \Sub_{\Cc^\flat}(\qsoc(A)),
  \end{equation*}
  so by Proposition~\ref{fattening-facts}.\ref{ff1},
  \begin{equation*}
    \botrk(A) := \botrk(\Sub_\Cc(A)) =
    \ell(\Sub_{\Cc^\flat}(\qsoc(A))) =: \ell(\qsoc(A)). \qedhere
  \end{equation*}
\end{proof}

\begin{lemma}[Intersection lemma]\label{lem:qsoc-prop} (Under \ref{asm:0}.)
  If $A \in \Cc$ and $B_1, B_2$ are two subobjects, then
  \begin{equation*}
    \qsoc(B_1) \cap \qsoc(B_2) = \qsoc(B_1 \cap B_2).
  \end{equation*}
  Consequently,
  \begin{itemize}
  \item The length of $\qsoc(B_1) \cap \qsoc(B_2)$ is $\botrk(B_1 \cap B_2)$.
  \item $\qsoc(B_1) \cap \qsoc(B_2) = 0 \iff B_1 \cap B_2 = 0$
  \end{itemize}
\end{lemma}
\begin{proof}
  This follows by properties of flattening maps, namely
  Proposition~\ref{fattening-facts}.\ref{ff4},\ref{ff6}.
\end{proof}

\begin{proposition}\label{finite-rank-case}
  Suppose every object of $\Cc$ has finite length.
  \begin{enumerate}
  \item For any $A \in \Cc$, the induced map
    \begin{equation*}
      \Sub_\Cc(A) \hookrightarrow \Sub_{\Pro \Cc}(A)
    \end{equation*}
    is an isomorphism of lattices.
  \item The quasi-socle $\qsoc(A)$ is the ordinary socle $\soc(A)$.
    In particular, the pro-object $\qsoc(A) \in \Pro \Cc$ is an
    ordinary object in $\Cc$.
  \item $\Cc^\flat$ is equivalent to the category of semisimple
    objects in $\Cc$, and $\qsoc : \Cc \to \Cc^\flat$ is equivalent to
    the ordinary socle functor.
  \item In particular, if $\Cc$ is a semisimple abelian category, then
    $\Cc \to \Cc^\flat$ is an equivalence of categories.
  \end{enumerate}
\end{proposition}
\begin{proof}
  If $M$ is a modular lattice of finite length, every filter on $M$
  has a minimum element, hence is principal.  Therefore the induced
  embedding $M \hookrightarrow \Pro M$ is an isomorphism.  As
  \begin{equation*}
    \Sub_{\Pro \Cc}(A) \cong \Pro \Sub_\Cc(A),
  \end{equation*}
  this proves the first point.

  Then for any $A \in \Cc$, the sub-proobjects of $A$ are the same
  thing as ordinary subobjects from $\Cc$.  The quasi-socle is
  therefore the largest semsimple subobject of $\Cc$, which is the
  ordinary socle, proving the second point.

  The essential image of the socle functor is exactly the category of
  semisimple objects in $\Cc$, because every semisimple object is its
  own socle.  This proves the third point.  The fourth point is then
  clear.
\end{proof}

\subsection{Flattening directories}\label{sec:dirflat}
\begin{corollary} \label{cor:dir-flat}
  Let $D_\bullet$ be a cube-bounded directory.  Then there is a
  directory $D^\flat_\bullet$ and a morphism of directories $D_\bullet
  \to D^\flat_\bullet$ such that each map
  \begin{equation*}
    D_n \to D^\flat_n
  \end{equation*}
  is a flattening map.  The resulting structure is unique up to
  isomorphism.
\end{corollary}
\begin{proof}
  We may assume $D_\bullet = \Dir_\Cc(A)$ for some object $A$ in a
  $K_0$-linear abelian category $\Cc$.  Then $A$ is a cube-bounded
  object in $\Cc$.  Replacing $\Cc$ with the neighborhood of $A$, we may
  assume $\Cc$ is cube-bounded.  Then we can construct $\qsoc : \Cc
  \to \Cc^\flat$ as in Theorem~\ref{thm:hey}.  By
  Theorem~\ref{thm:hey}.\ref{th2}, the functor $\qsoc$ is left-exact,
  hence induces a morphism
  \begin{align*}
    \Dir_\Cc(A) & \to \Dir_{\Cc^\flat}(\qsoc(A)) \\
    \Sub_\Cc(A^n) & \to \Sub_{\Cc^\flat}(\qsoc(A)^n) \\
    B & \mapsto \qsoc(B)
  \end{align*}
  as in Example~\ref{from-functor}.  Let $D^\flat_\bullet$ denote
  $\Dir_{\Cc^\flat}(\qsoc(A))$.  By Theorem~\ref{thm:hey}.\ref{th3},
  the morphism $D_\bullet \to D^\flat_\bullet$ is a levelwise
  flattening map, proving existence.

  Alternatively, we can form the composition
  \begin{equation*}
    \Dir_\Cc(A) \to \Dir_{\Pro \Cc}(A) \to \Dir_\Cc(\qsoc(A)),
  \end{equation*}
  where the first map is induced by the exact functor $\Cc
  \hookrightarrow \Pro \Cc$, and the second map is pullback along
  $\qsoc(A) \hookrightarrow A$.  This composition is a levelwise
  flattening map by Lemma~\ref{lem:heyhey}.

  For uniqueness, suppose that $D_\bullet \to D'_\bullet$ is some
  other directory morphism such that each map $D_n \to D'_n$ is a
  flattening map.  Because flattening maps are unique up to
  isomorphism, we may move $D'_\bullet$ by an isomorphism and arrange
  the following:
  \begin{itemize}
  \item $D'_n$ and $D^\flat_n$ have the same underlying poset
  \item The functions $D_n \to D'_n$ and $D_n \to D^\flat_n$ are the
    same underlying function.
  \end{itemize}
  After arranging this, we claim that $D'_\bullet = D^\flat_\bullet$.
  The underlying sets are the same, and we only need to check that the
  directory structures agree, specifically the $\oplus$ operators and the
  $GL_n(K_0)$-actions.  The two maps
  \begin{align*}
    D_n &\twoheadrightarrow D'_n \\
    D_n &\twoheadrightarrow D^\flat_n
  \end{align*}
  are surjective and $GL_n(K_0)$-equivariant, so $D'_n$ and
  $D^\flat_n$ must have the same $GL_n(K_0)$-action, namely, the
  action induced by the action on $D_n$.  A similar argument shows
  that $D^\flat_\bullet$ and $D'_\bullet$ have the same $\oplus$
  operators.  Thus $D^\flat_\bullet = D'_\bullet$.
\end{proof}
Apparently, Corollary~\ref{cor:dir-flat} also works if $D_\bullet$ is
a Noetherian directory, in the sense that the ascending chaing
condition holds in $D_1$.  (This implies the ascending chain condition
in $D_n$ for all $n$.)

We call the morphism $D_\bullet \to D^\flat_\bullet$ of
Corollary~\ref{cor:dir-flat} the \emph{flattening morphism}.
\begin{remark}\label{alt-forms}
  From the proof of Corollary~\ref{cor:dir-flat}, we get two explicit
  descriptions of flattening.  If $D_\bullet = \Dir_\Cc(A)$, the
  flattening $D^\flat_\bullet$ can be described as the morphism
  \begin{equation*}
    \Dir_\Cc(A) \to \Dir_{\Cc^\flat}(\qsoc(A))
  \end{equation*}
  induced by the left-exact functor $\qsoc(A)$.  It can also be
  described as the morphism
  \begin{equation*}
    \Dir_\Cc(A) \to \Dir_{\Pro \Cc}(A) \to \Dir_{\Pro \Cc}(\qsoc(A))
  \end{equation*}
  where the first map is induced by the exact functor $\Cc
  \hookrightarrow \Pro \Cc$, and the second map is intersection with
  $\qsoc(A)$ (i.e., pullback along the monomorphism $\qsoc(A)
  \hookrightarrow A$).  These two descriptions are helpful when doing
  calculations.
\end{remark}

\subsection{Balanced objects}
Let $\Cc$ be a cube-bounded abelian category.
\begin{remark} \label{bot-red-bound}
  If $A \in \Cc$, then
  \begin{equation*}
    \ell(\qsoc(A)) = \botrk(A) \le \redrk(A)
  \end{equation*}
  by Proposition~\ref{prop:qsoc-botrk} and
  Remark~\ref{botrk-rem}.\ref{br2}.
\end{remark}
\begin{definition}\label{def:balanced}
  An object $A \in \Cc$ is \emph{balanced} if
  \begin{equation*}
    \ell(\qsoc(A)) = \redrk(A).
  \end{equation*}
\end{definition}
\begin{lemma}\label{sub-balanced}
  Subobjects of balanced objects are balanced.
\end{lemma}
\begin{proof}
  Let $B \subseteq A$ with $A$ balanced.  As $\Cc^\flat$ is a
  semisimple category, there is some $C' \subseteq \qsoc(A)$
  complementary to $\qsoc(B) \subseteq \qsoc(A)$.  The induced map
  \begin{equation*}
    \Sub_\Cc(A) \to \Sub_{\Cc^\flat}(\qsoc(A))
  \end{equation*}
  is a flattening map, hence surjective, so we may take $C' =
  \qsoc(C)$ for some $C \subseteq A$.  Then
  \begin{equation*}
    \qsoc(C) \cap \qsoc(B) = C' \cap \qsoc(B) = 0.
  \end{equation*}
  By Intersection Lemma~\ref{lem:qsoc-prop}, $C \cap B = 0$.  It
  follows that $C \oplus B \cong C + B \subseteq A$.  Then
  \begin{equation*}
    \redrk(C) + \redrk(B) = \redrk(C + B) \le \redrk(A)
  \end{equation*}
  by Proposition~\ref{redrk}.\ref{rr3}.  Note that
  \begin{equation*}
    \ell(\qsoc(C)) + \ell(\qsoc(B)) = \ell(\qsoc(A)).
  \end{equation*}
  because $\qsoc(C)$ and $\qsoc(B)$ are complementary inside
  $\qsoc(A)$.  Meanwhile,
  \begin{align*}
    \ell(\qsoc(C)) &\le \redrk(C) \\
    \ell(\qsoc(B)) &\le \redrk(B)    
  \end{align*}
  by Remark~\ref{bot-red-bound}.  By assumption, $\redrk(A) =
  \ell(\qsoc(A))$.  Putting everything together,
  \begin{align*}
    \redrk(C) + \redrk(B) &\ge \ell(\qsoc(C)) + \ell(\qsoc(B)) \\ &=
    \ell(\qsoc(A)) = \redrk(A) \\ &\ge \redrk(C) + \redrk(B).
  \end{align*}
  Therefore equality holds:
  \begin{align*}
    \ell(\qsoc(C)) &= \redrk(C) \\
    \ell(\qsoc(B)) &= \redrk(B). \qedhere
  \end{align*}
\end{proof}
\begin{remark}\label{add-balanced}
  If $A, B$ are balanced, then $A \oplus B$ is balanced.  Indeed,
  \begin{align*}
    \ell(\qsoc(A \oplus B)) &= \ell(\qsoc(A) \oplus \qsoc(B)) =
    \ell(\qsoc(A)) + \ell(\qsoc(B)) \\ &= \redrk(A) + \redrk(B) = \redrk(A
    \oplus B).
  \end{align*}
\end{remark}

\section{The pedestal machine} \label{sec:pedestals}
We give an abstract machine for generating inflators from cube-bounded
lattices on fields.
\subsection{The cube-bounded configuration}\label{sec:cbc}
If $K$ is a field, let $K\Vect^f$ denote the category of
finite-dimensional $K$-vector spaces.  For
\S\ref{sec:cbc}-\ref{yo-sec} we assume the following:
\begin{assumption}[Cube-bounded configuration] \label{asm:1}
  We have the following configuration:
  \begin{enumerate}
  \item A field $K$ extending the small field $K_0$.
  \item A $K_0$-linear abelian category $\Cc$
  \item A $K_0$-linear functor $F : K\Vect^f \to \Cc$ such that
    $F(K)$ is non-trivial but cube-bounded.
  \end{enumerate}
\end{assumption}
By Remark~\ref{out-of-finvec}, $F$ is faithful, exact, and
conservative.  Abusing notation, we view $K\Vect^f$ as a
subcategory of $\Cc$, and suppress $F$.

The motivating example is where $K$ is an infinite, saturated field of
finite burden, $\Cc$ is the category $\Hh^{00}$ of
\S\ref{sec:tdef00-setting}, and $F$ is the composition
\begin{equation*}
  K\Vect^f \to \Hh \to \Hh^{00}.
\end{equation*}
If $K$ is NIP and $K_0$ is a magic subfield, then the second map is an
equivalence of categories, and we can take $\Cc = \Hh$ instead.  We
give another example in \S\ref{sec:another-example}.

Let $d = \redrk(K)$ (in the category $\Cc$).  By assumption $d <
\infty$.
\begin{definition}\label{def:pedestal1.5}
  A $\Cc$-subobject $A \subseteq K$ is a \emph{pedestal} if
  $\botrk(K/A) = d$, i.e., $A$ is the base of a strict $d$-cube in the
  lattice $\Sub_\Cc(K)$.
\end{definition}

\begin{remark}\label{special-balanced1.5}
  If $A$ is a pedestal, then
  \begin{equation*}
    \ell(\qsoc(K/A)) = \botrk(K/A) = d = \redrk(K) \ge \redrk(K/A) \ge \botrk(K/A),
  \end{equation*}
  so $K/A$ is \emph{balanced} (Definition~\ref{def:balanced}).
\end{remark}

\subsection{The inflator}\label{sec:the-machine}
Continue Assumption~\ref{asm:1}.
\begin{lemma}[$\approx$ Lemma~10.9 in \cite{prdf}]\label{lem:miracle}
  Let $d = \redrk(K)$ and $A \subseteq K$ be a pedestal.  Then for
  any $V \in \Sub_K(K^n)$, we have
  \begin{equation*}
    \ell(\qsoc((V + A^n)/A^n)) = \botrk((V + A^n)/A^n) = \redrk((V +
    A^n)/A^n) = d \cdot \dim_K(V).
  \end{equation*}
\end{lemma}
\begin{proof}
  The first equality is Proposition~\ref{prop:qsoc-botrk}.  The second
  holds by Lemma~\ref{sub-balanced} because $(V + A^n)/A^n$ is
  isomorphic to a subobject of $K^n/A^n$, which is balanced by
  Remark~\ref{add-balanced}.  It remains to show the third equality.
  Let $W \in \Sub_K(K^n)$ be a complementary subspace, so
  \begin{align*}
    V + W & = K^n \\
    V \cap W & = \{0\} \\
    \dim(V) + \dim(W) &= n.
  \end{align*}
  The isomorphisms
  \begin{align*}
    V &\cong K^{\dim V} \\
    W &\cong K^{\dim W}
  \end{align*}
  imply that in $\mathcal{C}$,
  \begin{align*}
    \redrk(V) &= d \cdot \dim(V) \\
    \redrk(W) &= d \cdot \dim(W).
  \end{align*}
  The surjection $V \twoheadrightarrow (V + A^n)/A^n$ implies
  \begin{align*}
    \redrk((V + A^n)/A^n) & \le d \cdot \dim(V) \\
    \redrk((W + A^n)/A^n) & \le d \cdot \dim(W).
  \end{align*}
  Now $K^n/A^n$ is the join of the subobjects $(V + A^n)/A^n$ and $(W
  + A^n)/A^n$, so
  \begin{equation*}
    \redrk(K^n/A^n) \le \redrk((V + A^n)/A^n) + \redrk((W +
    A^n)/A^n).
  \end{equation*}
  Assembling all the inequalities,
  \begin{align*}
    dn &= n \cdot \redrk(K/A) = \redrk(K^n/A^n) \\ & \le \redrk((V +
    A^n)/A^n) + \redrk((W + A^n)/A^n) \\ &\le d \cdot \dim(V) + d \cdot
    \dim(W) = dn.
  \end{align*}
  Therefore all the inequalities are equalities and
  \begin{align*}
    \redrk((V + A^n)/A^n) & = d \cdot \dim(V) \\
    \redrk((W + A^n)/A^n) & = d \cdot \dim(W). \qedhere
  \end{align*}
\end{proof}

\begin{theorem}[Pedestal machine] \label{thm:actually-important}
  Under Assumption~\ref{asm:1}, if $d = \redrk(K)$ and $A \in
  \Sub_\Cc(K)$ is a pedestal, then the following composition is a
  $d$-inflator on $K$:
  \begin{equation*}
    \Dir_K(K) \stackrel{=}{\to} \Dir_{K\Vect^f}(K) \hookrightarrow
    \Dir_\Cc(K) \to \Dir_\Cc(K/A) \to \Dir_\Cc(K/A)^\flat
  \end{equation*}
  where the second map comes from the exact functor $K\Vect^f \to
  \Cc$, the third map comes from the pushforward along $K
  \twoheadrightarrow K/A$, and the final map is the flattening map of
  Corollary~\ref{cor:dir-flat}.
\end{theorem}
\begin{proof}
  By Remark~\ref{alt-forms}, we can identify $\Dir_\Cc(K/A)^\flat$ with
  $\Dir_{\Cc^\flat}(\qsoc(K/A))$, and the directory morphism in question is
  simply the map
  \begin{equation*}
    V \mapsto (V + A^n)/A^n \mapsto \qsoc((V + A^n)/A^n).
  \end{equation*}
  Lemma~\ref{lem:miracle} ensures that this is a $d$-inflator.
\end{proof}
\begin{proposition}\label{bleal-form}
  Under Assumption~\ref{asm:1}, the $d$-inflator of
  Theorem~\ref{thm:actually-important} can also be described as
  \begin{equation*}
    \Dir_K(K) \hookrightarrow \Dir_{\Pro \Cc}(K) \twoheadrightarrow
    \Dir_{\Pro \Cc}(A^+/A),
  \end{equation*}
  where
  \begin{itemize}
  \item $A^+$ is the pro-subobject of $K$ such that $A^+/A$ is the
    socle of $K/A$.
  \item the first map is induced by the exact functor $K\Vect^f
    \hookrightarrow \Pro \Cc$.
  \item the second map is an interval retract onto the interval
    $[A,A^+]$, as in Definition~\ref{def:intervals}.
  \end{itemize}
\end{proposition}
\begin{proof}
  By Remark~\ref{alt-forms}, we can identify $\Dir_\Cc(K/A)^\flat$ with
  $\Dir_{\Pro \Cc}(\qsoc(K/A)) = \Dir_{\Pro \Cc}(A^+/A)$, and the map
  in question is
  \begin{equation*}
    V \mapsto (V + A^n)/A^n \cap (A^+)^n/A^n.
  \end{equation*}
  The notation is unambiguous because the embedding $\Cc
  \hookrightarrow \Pro \Cc$ is an exact functor which preserves
  everything.  In other words, there is a commutative diagram
  \begin{equation*}
    \xymatrix{ \Dir_\Cc(K) \ar[r] \ar[d] & \Dir_\Cc(K/A) \ar[d]
      \\ \Dir_{\Pro \Cc}(K) \ar[r]  & \Dir_{\Pro
        \Cc}(K/A). }
  \end{equation*}
  So the map in question is the composition
  \begin{equation*}
    \Dir_K(K) \to \Dir_{K\Vect^f}(K) \to \Dir_\Cc(K) \to
    \Dir_{\Pro \Cc}(K) \to \Dir_{\Pro \Cc}(K/A) \to \Dir_{\Pro
      \Cc}(A^+/A).
  \end{equation*}
  The composition of the final two maps is the interval retract.
\end{proof}

\subsection{A special case}
Continue Assumption~\ref{asm:1}.
\begin{lemma}
  If $M$ is a modular lattice of finite length, then $M \cong \Pro M$.
\end{lemma}
\begin{proof}
  Every filter has a minimum element, and is therefore principal.
\end{proof}
\begin{corollary}\label{cor:c-like-pro-c}
  If $A$ is an object of finite length in an abelian category $\Cc$,
  then
  \begin{equation*}
    \Sub_\Cc(A) \to \Sub_{\Pro \Cc}(A)
  \end{equation*}
  is an isomorphism.  In particular, $A$ has the same length in $\Cc$
  and $\Pro \Cc$, and $A$ is semisimple in $\Cc$ if and only if $A$ is
  semisimple in $\Pro \Cc$.
\end{corollary}
\begin{proposition}\label{prop:special-case}
  Under Assumption~\ref{asm:1}, suppose there are subobjects $A \le
  A^+ \le K$ such that $A^+/A$ is semisimple of length $d =
  \redrk(K)$.  Then $A$ is a pedestal and the associated $d$-inflator
  is isomorphic to
  \begin{equation*}
    \Dir_K(K) \to \Dir_\Cc(K) \to \Dir_\Cc(A^+/A),
  \end{equation*}
  where the first map is induced by the inclusion $K\Vect^f
  \hookrightarrow \Cc$ and the second map is the interval retract
  \begin{align*}
    \Sub_\Cc(K^n) & \to \Sub_\Cc((A^+/A)^n) \\
    V & \mapsto (V \cap (A^+)^n + A^n)/A^n.
  \end{align*}
\end{proposition}
\begin{proof}
  In the category $\Pro \Cc$, we see that $A^+/A$ is semisimple.  Therefore,
  \begin{equation*}
    A^+/A \subseteq \qsoc(K/A),
  \end{equation*}
  and then $\botrk(K/A) = \ell(\qsoc(K/A)) \ge \ell(A^+/A) = d$,
  showing that $A$ is a pedestal.  Equality holds, so $A^+/A =
  \qsoc(K/A)$.  Then the induced $d$-inflator is
  \begin{align*}
    \Sub_K(K^n) & \to \Sub_{\Pro \Cc}((A^+/A)^n) \\
    V & \mapsto (V \cap (A^+)^n + A^n)/A^n.
  \end{align*}
  Since $A^+$ comes from $\Cc$,
  this map factors through the natural map
  \begin{equation*}
    \Sub_\Cc((A^+/A)^n) \hookrightarrow \Sub_{\Pro \Cc}((A^+/A)^n),
  \end{equation*}
  which is an isomorphism by Corollary~\ref{cor:c-like-pro-c}.
\end{proof}

\subsection{The ring and the ideal}
Continue Assumption~\ref{asm:1}.

Note that there is an action of $\End_\Cc(K)$ on $\Sub_\Cc(K)$, by
direct images.  Via the embedding $K = \End_K(K)
\hookrightarrow \End_\Cc(K)$, we get an action of the monoid
$(K,\cdot)$ on the poset $\Sub_\Cc(K)$.

In the concrete cases of interest, $a \in K$ sends a type-definable
subgroup $A \subseteq K$ to its rescaling $a \cdot A$.
\begin{proposition}\label{prop:r-i}
  Under Assumption~\ref{asm:1}, let $A$ be a pedestal and let $\varsigma$
  be the $d$-inflator of Theorem~\ref{thm:actually-important}.  Then
  the fundamental ring is the ``stabilizer''
  \begin{equation*}
    R = \{x \in K : xA \subseteq A\}
  \end{equation*}
  and the fundamental ideal is
  \begin{equation*}
    I = \{x \in R : \redrk(A/xA) = d\}.
  \end{equation*}
\end{proposition}
\begin{proof}
  Let $b$ be an element of $K$ and let $\Theta_b = K \cdot (1,b)$.
  Then $\varsigma_2(\Theta_b)$ is exactly
  \begin{equation*}
    \qsoc((\Theta_b + A^2)/A^2).
  \end{equation*}
  By Lemma~\ref{o-alt-criterion}, Intersection
  Lemma~\ref{lem:qsoc-prop}, and modularity of subobject lattices,
  \begin{align*}
    b \in R & \iff \qsoc((\Theta_b + A^2)/A^2) \cap (0 \oplus \qsoc(K/A)) = 0 \\
    & \iff \qsoc((\Theta_b + A^2)/A^2) \cap \qsoc((A \oplus K)/A^2) = 0 \\
    & \iff (\Theta_b + A^2)/A^2 \cap (A \oplus K)/A^2 = 0 \\
    & \iff (\Theta_b + A^2) \cap (A \oplus K) = A^2 \\
    & \iff (\Theta_b \cap (A \oplus K)) + A^2 = A^2 \\
    & \iff \Theta_b \cap (A \oplus K) \subseteq A^2.
  \end{align*}
  The final condition is equivalent to $bA \subseteq A$.  Next suppose
  $b \in R$, so $bA \subseteq A$.  By Lemma~\ref{i-alt-criterion}, the
  Intersection Lemma~\ref{lem:qsoc-prop}, and modularity of subobject
  lattices,
  \begin{align*}
    b \in I & \iff \qsoc((\Theta_b + A^2)/A^2) \supseteq (\qsoc(K/A) \oplus 0) \\
    & \iff \ell(\qsoc((\Theta_b + A^2)/A^2) \cap (\qsoc(K/A) \oplus 0)) \ge \ell(\qsoc(K/A)) \\
    & \iff \ell(\qsoc((\Theta_b + A^2)/A^2) \cap (\qsoc(K/A) \oplus 0)) \ge d \\
    & \iff \ell(\qsoc((\Theta_b + A^2)/A^2) \cap \qsoc((K \oplus A)/A^2)) \ge d \\
    & \iff \botrk((\Theta_b + A^2)/A^2 \cap (K \oplus A)/A^2) \ge d \\
    & \iff \botrk(((\Theta_b + A^2) \cap (K \oplus A))/A^2) \ge d \\
    & \iff \botrk((\Theta_b \cap (K \oplus A) + A^2)/A^2) \ge d
  \end{align*}
  By a diagram chase,
  \begin{equation*}
    (\Theta_b \cap (K \oplus A) + A^2)/A^2 \cong (\Theta_b \cap (K
    \oplus A))/(\Theta_b \cap A^2) \cong b^{-1}A/A.
  \end{equation*}
  But $b^{-1}A/A$ is a subobject of the balanced object $K/A$, so by
  Lemma~\ref{sub-balanced},
  \begin{equation*}
    \botrk(b^{-1}A/A) = \redrk(b^{-1}A/A).
  \end{equation*}
  Thus
  \begin{equation*}
    b \in I \iff \botrk(b^{-1}A/A) \ge d \iff \redrk(b^{-1}A/A) \ge d
    \iff \redrk(A/bA) \ge d. \qedhere
  \end{equation*}
\end{proof}
\begin{warning}
  In the $\mathcal{H}^{00}$ case, the notation must be understood
  modulo commensurability.  For example, the fundamental ring is the
  set of $x \in K$ such that $xA$ is below $A$ in the lattice
  $\Lambda^{00}$, i.e., $xA \cap A$ has bounded index in $xA$.
\end{warning}
In the good case of dp-finite fields, the ring and ideal appearing in
Proposition~\ref{prop:r-i} are the same ring and ideal appearing in
\cite{prdf}, Proposition~10.15.

\subsection{Malleability}\label{yo-sec}
To obtain malleability from the pedestal machine, we need an
additional assumption:
\begin{assumption}[Special cube-bounded configuration] \label{asm:2}
  Assumption~\ref{asm:1} holds, and
  there is a faithful exact $K_0$-linear functor $G : \Cc \to K_0\Vect$
  such that the composition
  \begin{equation*}
    K\Vect^f \stackrel{F}{\to} \Cc \stackrel{G}{\to} K_0\Vect
  \end{equation*}
  is isomorphic to the forgetful functor $K\Vect^f \to K_0\Vect$.
\end{assumption}
For example, Assumption~\ref{asm:2} holds for
\begin{equation*}
  F : \Kk\Vect^f \to \Hh^{00}
\end{equation*}
when $\Kk$ is NIP and $K_0$ is a magic subfield.  In this case,
$\Hh^{00} = \Hh$ and we can take $G$ to be the natural forgetful
functor to $K_0\Vect$.

\begin{proposition}\label{prop:ped-mal}
  Under Assumption~\ref{asm:2}, the $d$-inflator of
  Theorem~\ref{thm:actually-important} is malleable.
\end{proposition}
\begin{proof}
  The faithful exact functors $F$ and $G$ allow us to view objects of
  $\Kk\Vect^f$ and $\Cc$ as $K_0$-vector spaces with extra
  structure.  (On the other hand, $\Cc^\flat$ is still opaque.)  We
  can then suppress $F$ and $G$ from the notation.  The fact that $G
  \circ F$ is the usual forgetful functor $\Kk\Vect^f \to
  K_0\Vect$ ensures that this is notationally safe.

  Let $B = \qsoc(K/A)$, so that the $d$-inflator is
  \begin{align*}
    \Dir_K(K) & \to \Dir_{\Cc^\flat}(B) \\
    V & \mapsto \qsoc((V + A^n)/A^n).
  \end{align*}
  Suppose $V \subseteq W \subseteq K^n$ and $L$ is a subobject of
  $B^n$ such that
  \begin{equation*}
    \qsoc((V + A^n)/A^n) \subseteq L \subseteq \qsoc((W + A^n)/A^n)
  \end{equation*}
  with the length of $L/\qsoc((V + A^n)/A^n)$ equal to 1.  We must
  find a $K$-linear subspace $V' \subseteq W$ such that
  \begin{align*}
    V &\subseteq V' \\
    L &\subseteq \qsoc((V' + A^n)/A^n) \\
    \dim_K(V') &\le \dim_K(V) + 1.
  \end{align*}
  Because $\Cc^{\flat}$ is a semisimple category, we can find a
  subobject $S$ of $L$ complementary to $\qsoc((V + A^n)/A^n)$.  Then
  $S$ has length 1, and $L$ is generated by $\qsoc((V + A^n)/A^n)$ and
  $S$.  Now the map
  \begin{equation*}
    \Sub_\Cc((W + A^n)/A^n) \to \Sub_{\Cc^\flat}(\qsoc((W + A^n)/A^n))
  \end{equation*}
  is a flattening map (Theorem~\ref{thm:hey}.\ref{th3}) and flattening
  maps are surjective (Proposition~\ref{fattening-facts},\ref{ff2}),
  so there is a $\Cc$-subobject $Q$ of $W + A^n$ such that
  \begin{equation*}
    A^n \subseteq Q \subseteq W + A^n
  \end{equation*}
  and $\qsoc(Q/A^n) = S$.  Then $\ell(\qsoc(Q/A^n)) = 1$, so $Q/A^n
  \ne 0$.  Take $x_0 \in Q \setminus A^n$.  We can write $x_0$ as $x_1
  + x_2$, where $x_1 \in W$ and $x_2 \in A^n$.  Then $x_1 \in W$
  and $x_1 \in Q \setminus A^n$.  Let $V'$ be the $K$-linear subspace
  of $W$ generated by $V$ and $x_1$.  Then
  \begin{align*}
    x_1 & \in V' + A^n \\
    x_1 & \in Q \\
    x_1 & \notin A^n.
  \end{align*}
  It follows that $(V' + A^n)/A^n$ and $Q/A^n$ have non-trivial
  intersection.  By the Intersection Lemma~\ref{lem:qsoc-prop},
  \begin{equation*}
    \qsoc((V' + A^n)/A^n) \cap \qsoc(Q/A^n) \ne 0.
  \end{equation*}
  But $\qsoc(Q/A^n)$ has length 1, so this in fact implies
  \begin{equation*}
    \qsoc((V' + A^n)/A^n) \supseteq \qsoc(Q/A^n) = S.
  \end{equation*}
  Also $V' \supseteq V$, so
  \begin{equation*}
    \qsoc((V' + A^n)/A^n) \supseteq \qsoc((V + A^n)/A^n) + S = L.
  \end{equation*}
  Thus $V'$ has all the desired properties.
\end{proof}

\subsection{An example}\label{sec:another-example}
Let $K$ be a field, $K_0$ be a subfield, and let $\Oo_1, \ldots,
\Oo_n$ be pairwise incomparable valuation rings on $K$, with $K_0
\subseteq \Oo_i$.  Let $R$ be the multivaluation ring $\Oo_1 \cap
\cdots \cap \Oo_n$; this is a $K_0$-subalgebra of $K$.  Let $\Cc$ be
the category $R\Mod$, and let
\begin{equation*}
  K\Vect^f \stackrel{F}{\to} R\Mod \stackrel{G}{\to} K_0\Mod
\end{equation*}
be the forgetful functors.  Then we are in the Special Cube-bounded
Configuration of Assumption~\ref{asm:2}.  Indeed, the reduced rank $d$
of $K \in R\Mod$ is exactly $n$, by Lemma~\ref{meta-lol}.

For each $i$, let $\mm_i$ be the maximal ideal of $\Oo_i$, and let $J$ be the intersection
\begin{equation*}
  J = \mm_1 \cap \cdots \cap \mm_n.
\end{equation*}
By \cite{prdf2}, Proposition~6.2, $R$ has exactly $n$ distinct maximal
ideals $M_i = R \cap \mm_i$, the quotients $R/M_i$ are isomorphic to
$k_i := \Oo/\mm_i$, and $J = M_1 \cap \cdots \cap M_n$ is the Jacobson
radical of $R$.  The quotient $R/J$ is isomorphic to $k_1 \times
\cdots \times k_n$.  By Propositions~\ref{prop:special-case},
\ref{prop:ped-mal}, we see that $J$ is a pedestal, and there is a
malleable $n$-inflator
\begin{align*}
  \Dir_K(K) &\to \Dir_R(R/J) \\
  V & \mapsto (V \cap R^n + J^n)/J^n.
\end{align*}
If $n = 1$ and $R$ is a valuation ring, this is the 1-inflator of
Theorem~\ref{thm:calvin}.  If $n > 1$, note that
\begin{equation*}
  \Dir_R(R/J) = \Dir_{R/J}(R/J) = \prod_{i = 1}^n \Dir_{k_i}(k_i) = \prod_{i = 1}^n \Dir_R(k_i)
\end{equation*}
by Remark~\ref{ring-splitting} and the isomorphism $R/J \cong k_1
\times \cdots \times k_n$.  The $i$th projection map $\Dir_R(R/J) \to
\Dir_R(k_i)$ can be described in several ways; one of them is that it
is the pushforward along the quotient map $R/J \to R/M_i$.  Thus the
composition
\begin{equation*}
  \Dir_K(K) \to \Dir_R(R/J) \to \Dir_R(k_i)
\end{equation*}
is given by
\begin{equation*}
  V \mapsto (V \cap R^n + M_i^n)/M_i^n.
\end{equation*}
Let $f$ be the $R$-module morphism from $R/M_i$ to $\Oo_i/\mm_i$
induced by the inclusion.  By Proposition~6.2.6 in \cite{prdf2}, $f$
is an isomorphism.  Let $f^{\oplus n} : (R/M_i)^n \to (\Oo_i/\mm_i)^n$
be defined componentwise.
\begin{claim}\label{lone-claim}
  For any subspace $V \le K^n$, the direct image of
  \begin{equation*}
    (V \cap R^n + M_i^n)/M_i^n 
  \end{equation*}
  under $f^{\oplus n}$ is contained in
  \begin{equation*}
    (V \cap \Oo_i^n + \mm_i^n)/\mm_i^n.
  \end{equation*}
\end{claim}
\begin{proof}
  First, note that
  for any $M_i^n \le X \le R^n$, the direct image of $X/M_i^n$ under
  $f^{\oplus n}$ is exactly $(X + \mm_i^n)/\mm_i^n$, because the
  composition
  \begin{equation*}
    R^n \to R^n/M_i^n \stackrel{f^{\oplus n}}{\longrightarrow} \Oo_i^n/\mm_i^n
  \end{equation*}
  is the same as the composition $R^n \hookrightarrow \Oo_i^n
  \twoheadrightarrow \Oo_i^n/\mm_i^n$, by definition of $f$.  In
  particular, the direct image along $f^{\oplus n}$ sends
  \begin{equation*}
    (V \cap R^n +M_i^n)/M_i^n \mapsto ((V \cap R^n + M_i^n) + \mm_i^n)/\mm_i^n
  \end{equation*}
  Then it remains to show that for any $V \le K^n$,
  \begin{equation*}
    (V \cap R^n + M_i^n) + \mm_i^n \subseteq V \cap \Oo_i^n + \mm_i^n.
  \end{equation*}
  This is clear, since $M_i \subseteq \mm_i$ and $R \subseteq \Oo_i$.
\end{proof}
But then for any $V$,
\begin{align*}
  n \cdot \dim_K(V) & = \sum_{i = 1}^n \dim_{k_i}((V \cap R^n +
  \mm_i^n)/\mm_i^n) \\ & \le \sum_{i = 1}^n \dim_{k_i}((V \cap \Oo_i^n
  + \mm_i^n)/\mm_i^n) \stackrel{*}{=} n \cdot \dim_K(V),
\end{align*}
where the starred equality holds by Theorem~\ref{thm:calvin}.  So
equality holds in Claim~\ref{lone-claim}.  Thus the $n$-inflator
derived from $J$ is exactly the multi-valuation $n$-inflator of
Example~\ref{dorothy}.

\section{Fields of finite burden} \label{sec:model-theory-inflator}

Fix a saturated field $(\Kk,+,\cdot,0,1,\ldots)$ of finite burden, as
well as a small subfield $K_0$.  Recall from \S\ref{sec:lambda} the
lattice $\Lambda = \Lambda_1$ of type-definable $K_0$-linear subspaces
of $\Kk$, and the lattice $\Lambda^{00} = \Lambda^{00}_1$ obtained by
quotienting out by 00-commensurability.  By
Proposition~\ref{prop:genug}, the lattice $\Lambda^{00}$ is
cube-bounded.

\begin{definition}
  Let $r$ be the reduced rank of $\Lambda^{00}$.  A
  \emph{$K_0$-pedestal} is a group $A \in \Lambda$ whose image in
  $\Lambda^{00}$ is the base of a strict $r$-cube in $\Lambda^{00}$.
\end{definition}
By Proposition~\ref{prop-cs}, $r \le \dpr(\Kk)$.
\begin{remark}
  If $\Kk$ is NIP and $K_0$ is a magic subfield, then $\Lambda =
  \Lambda^{00}$, and so
  \begin{itemize}
  \item $\Lambda$ is a cube-bounded lattice of reduced rank $r$.
  \item A $K_0$-pedestal is a group $A \in \Lambda$ that is the base
    of a strict $r$-cube in $\Lambda$.
  \end{itemize}
  Thus this definition of ``$K_0$-pedestal'' generalizes
  Definition~8.4 in \cite{prdf2}.
\end{remark}

\subsection{The NIP case}

Recall the categories $\Hh$ and $\Hh^{00}$ of
\S\ref{sec:tdef-setting}-\ref{sec:tdef00-setting}.
\begin{theorem}\label{yeah-main}
  Suppose $\Kk$ is NIP, $K_0$ is a magic subfield, and $A$ is a
  $K_0$-pedestal.  Then there is a malleable $r$-inflator
  \begin{equation*}
    \varsigma : \Dir_\Kk(\Kk) \to \Dir_{\Hh^\flat}(\qsoc(\Kk/A))
  \end{equation*}
  given by
  \begin{equation*}
    \varsigma_n(V) = \qsoc((V + A^n)/A^n) \subseteq \qsoc(\Kk^n/A^n) \cong
    \qsoc(\Kk/A)^n,
  \end{equation*}
  where $\qsoc : \Hh \to \Hh^{\flat}$ is the quasi-socle functor of
  Theorem~\ref{thm:hey}.  The fundamental ring and ideal of $\varsigma$
  (see Proposition~\ref{o-i}) are
  \begin{align*}
    R &= \{ x \in \Kk : x \cdot A \subseteq A\} \\
    I &= \{ x \in R : \redrk(A / x \cdot A) = r\},
  \end{align*}
  where $r$ is $\redrk(\Kk)$ (in the category $\Hh$), or equivalently
  $\redrk(\Lambda)$.
\end{theorem}
\begin{proof}
  Note that
  \begin{equation*}
    \Kk\Vect^f \stackrel{F}{\to} \Hh \stackrel{G}{\to} K_0\Vect
  \end{equation*}
  is an instance of the Special Cube-bounded Configuration of
  Assumption~\ref{asm:2}, by Proposition~\ref{prop:genug}.  Then
  everything follows from Theorem~\ref{thm:actually-important},
  Proposition~\ref{prop:r-i}, and Proposition~\ref{prop:ped-mal}.
\end{proof}
\begin{remark}
  Theorem~\ref{yeah-main} verifies the structure predicted in
  \cite{prdf}, Speculative Remark~10.10.
\end{remark}
\begin{remark}
  The ring $R$ and ideal $I$ appearing in Theorem~\ref{yeah-main} are
  the same ring and ideal appearing in Proposition~10.15 in
  \cite{prdf}.
\end{remark}
\begin{remark}
  Let $\Kk$ be a saturated unstable dp-finite field.  In \cite{prdf2},
  we defined a notion of $\Kk$ having ``valuation type,'' meaning that
  the canonical topology is a V-topology.  We showed that if all
  dp-finite fields are either stable or valuation type, then the
  expected Shelah and henselianity conjectures hold, implying the
  expected classification of dp-finite fields.

  Inflators give a way to detect multi-valuation type: let $K_0$ be a
  magic subfield and $A$ be a $K_0$-pedestal.  Let $\varsigma$ be the
  induced inflator.  If $\varsigma$ is weakly of multi-valuation type
  (Definition~\ref{def-almost}), then there is a multi-valuation ring
  $R'$ on $\Kk$ and a non-trivial $R'$-submodule $M \le \Kk$ such that
  \begin{equation*}
    x \in M, y \in A \implies x \cdot y \in A.
  \end{equation*}
  Assuming $A \ne 0$, this implies that $A$ itself contains a
  non-trivial $R'$-submodule of $\Kk$.  By Theorem~8.11 in
  \cite{prdf2}, this implies that $\Kk$ has valuation type.  The
  degenerate case where $A = 0$ works as well; see Lemma~\ref{true-pc}
  below.

  The original hope for inflators was that every inflator would be
  weakly of multi-valuation type, completing the proof.  The examples
  of \S\ref{garbage-ex} show that this fails to hold, in general.  The
  remaining sections \S\ref{sec:mutation}-\ref{sec:mut4} show how we
  can partially fix the problem, by changing $A$ to a new pedestal
  $A'$ whose inflator is closer to being weakly multi-valuation type.
  This strategy successfully yields a valuation ring, but fails to
  prove the valuation conjecture (Conjecture~\ref{val-conj}).
\end{remark}

\subsection{The general case}

In the general finite burden case, we apparently lose malleability,
and need to consider everything up to 00-commensurability:
\begin{theorem}\label{thm:not-so-useful}
  Suppose $\Kk$ has finite burden, $K_0$ is a small subfield, and $A$
  is a $K_0$-pedestal.  Then there is an $r$-inflator
  \begin{equation*}
    \varsigma : \Dir_\Kk(\Kk) \to
    \Dir_{(\Hh^{00})^\flat}(\qsoc(\Kk/A))
  \end{equation*}
  given by
  \begin{equation*}
    \varsigma_n(V) = \qsoc((V + A^n)/A^n) \subseteq \qsoc(\Kk^n/A^n) \cong \qsoc(\Kk/A)^n
  \end{equation*}
  where $\qsoc : \Hh^{00} \to (\Hh^{00})^\flat$ is the quasi-socle
  functor of Theorem~\ref{thm:hey}.  The fundamental ring and ideal of
  $\varsigma$ are
  \begin{align*}
    R &= \{ x \in \Kk : x \cdot A \le A\} \\
    I &= \{ x \in R : \redrk(A / x \cdot A) = r\},
  \end{align*}
  where the $\le$ and $\redrk$ are in the lattice
  $\Lambda^{00}$ modulo commensurability.
\end{theorem}
For instance, $R$ should be understood as
\begin{equation*}
  \{x \in \Kk : (x \cdot A)/((x \cdot A) \cap A) \textrm{ is bounded}.\}
\end{equation*}
Combining Theorem~\ref{thm:not-so-useful} with earlier
Proposition~\ref{prop:1-fold} gives an extremely roundabout proof of
the following (easy) fact:
\begin{corollary}
  If $\Kk$ has burden 1 and $A$ is a type-definable $K_0$-linear
  subspace of $\Kk$, then
  \begin{equation*}
    \{ \alpha \in \Kk ~|~ \alpha A \le A\}
  \end{equation*}
  is a valuation ring on $\Kk$, where $A \le B$ means ``$A \cap B$ has
  bounded index in $A$.''
\end{corollary}
Of course this is true more generally without the $K_0$-linearity
assumption, for the simple reason that the lattice (modulo
commensurability) is totally ordered.  So for any $\alpha \in
\Kk^\times$, either
\begin{equation*}
  \alpha A \le A \textrm{ or } \alpha A \ge A,
\end{equation*}
and so
\begin{equation*}
  \alpha A \le A \textrm{ or } \alpha^{-1} A \le A.
\end{equation*}

\part{From inflators to multi-valuation rings}\label{part:3}
We rework \S 10.3 of \cite{prdf} in the language of inflators,
culminating in the construction of weakly definable valuation rings on
unstable dp-finite fields (Theorem~10.28 of \cite{prdf}).  This is
part of a more general construction of multi-valuation rings from
inflators.

Let $\varsigma$ be an $r$-inflator on $K$, with fundamental ring $R$.  In
\S\ref{sec:tame-locus} we defined a notion of \emph{tame} and
\emph{wild} elements of $K$, and showed that $R$ is a multi-valuation
ring when every element is tame.  It turns out that for any $a \in K$,
we can twist or ``mutate'' $\varsigma$ and obtain a new $r$-inflator
$\varsigma'$ whose fundamental ring and tame locus are larger than those
of $\varsigma$---and specifically $a$ is now tame.

The rough idea of mutation is as follows: given a line $L \le K^m$ and an $r$-inflator
\begin{equation*}
  \varsigma_n : \Sub_K(K^n) \to \Sub_S(M^n),
\end{equation*}
we define a new $r$-inflator $\varsigma'_\bullet$ by the formula
\begin{equation*}
  \varsigma'_n(V) = \varsigma'_{mn}(V \otimes L),
\end{equation*}
where we view $\otimes$ as a map
\begin{equation*}
  \otimes : \Sub_K(K^n) \times \Sub_K(K^m) \to \Sub_K(K^{mn}).
\end{equation*}
When trying to make a wild element $a \in K$ become tame, we mutate
along the line $K \cdot (1,a,a^2,\ldots,a^{r-1})$.  This approach
works because of a simple argument using Vandermonde matrices.

Now, if $\varsigma$ is an inflator, we define the \emph{limiting ring} of
$\varsigma$ to be the union
\begin{equation*}
  R^\infty_\varsigma := \bigcup \{R_{\varsigma'} : \varsigma' \textrm{ a mutation of } \varsigma\},
\end{equation*}
i.e., the union of the fundamental rings of the mutations of $\varsigma$.
The union turns out to be directed, so this is indeed a ring.  Because
we can make any element become tame via mutation, the ring
$R^\infty_\varsigma$ turns out to be a multi-valuation ring.  Because
$\varsigma$ is a trivial mutation of itself, $R^\infty_\varsigma \supseteq
R_\varsigma$.

One can also define the \emph{limiting ideal} of $\varsigma$ to be the
union
\begin{equation*}
  I^\infty_\varsigma := \bigcup \{I_{\varsigma'} : \varsigma' \textrm{ a mutation of } \varsigma\}.
\end{equation*}
Again, the union is directed, implying that $I^\infty_\varsigma$ is an
ideal in $R^\infty_\varsigma$, contained in the Jacobson radical of
$R^\infty_\varsigma$.  (Compare with Proposition~\ref{o-i}.)  Moreover,
$I^\infty_\varsigma \supseteq I_\varsigma$.

The upshot is that if $I_\varsigma$ is non-zero, then $R^\infty_\varsigma$ is
a multi-valuation ring with a non-zero Jacobson radical; therefore
$R^\infty_\varsigma$ is a non-trivial multi-valuation ring determining
finitely many non-trivial valuation rings.

If $\varsigma$ is one of the inflators on unstable dp-finite fields
constructed via Theorem~\ref{yeah-main}, then $I_\varsigma$ is
non-trivial, and this gives non-trivial weakly definable valuation
rings.  This is essentially the same construction of weakly definable
valuation rings as in Theorem~10.28 in \cite{prdf}.

Section \ref{sec:mutation} works through the construction of
multi-valuation rings from inflators.  Section \S\ref{sec:mut-def}
defines mutation, \S\ref{sec:mut-r-i} shows that mutation increases
the fundamental ring and ideal, \S\ref{sec:transitive} shows that
mutation is transitive (a mutation of a mutation is a mutation), and
\S\ref{sec:the-end} shows that mutation can make any element tame
(implying that the limiting ring is a multi-valuation ring).  The
additional sections \S\ref{sec:mut-mal}-\ref{sec:mut-ped} show that
mutation preserves malleability and the property ``comes from a
pedestal via the main construction.''

In \S\ref{sec:apps}, we apply these facts to fields of finite dp-rank
and finite burden.  For unstable dp-finite fields we recover the
construction of non-trivial valuation rings.  For fields of finite
burden, we recover non-trivial valuation rings only when the lattice
of type-definable subgroups is sufficiently rich.  We also discuss the
strategy of attacking Valuation Conjecture~\ref{val-conj} by trying to
show that malleable inflators can be mutated to have weakly
multi-valuation type.

Finally, \S\ref{sec:mut4} works through examples of mutation in some
of the inflators from Part~\ref{part:1}.  We see that mutation helps
undo the destruction wrought by the map
\begin{equation*}
  (V,W) \mapsto (V + W, V \cap W).
\end{equation*}

\section{Mutation}\label{sec:mutation}
\subsection{The definition of mutation}\label{sec:mut-def}
Let $K$ be a field.  For any line $L = K \cdot (a_1,\ldots,a_m) \le
K^m$, let $\xi^L_n : \Sub_K(K^n) \to \Sub_K(K^{mn})$ be the map
\begin{equation*}
  \xi^L_n(V) = \{(a_1\vec{x},\ldots,a_m \vec{x})^T : \vec{x} \in V\}
\end{equation*}
where $(-)^T$ denotes transpose:
\begin{equation*}
  (x_{1,1},x_{1,2},\ldots,x_{1,n},x_{2,1},\ldots,x_{m,n})^T =
  (x_{1,1},x_{2,1},\ldots,x_{m,1},x_{1,2},\ldots,x_{m,n}).
\end{equation*}
Note that $\xi^L_n(V)$ can be thought of as $V \otimes L \le K^n
\otimes K^m$.  In particular, $\xi^L_n(-)$ depends only on $L$, and
not on $\vec{a}$.

By inspection, the maps $\xi^L_n : \Sub_K(K^n) \to \Sub_K(K^{mn})$
satisfy the following properties:
\begin{align}
  V \subseteq W & \implies \xi^L_n(V) \subseteq \xi^L_n(W) \label{xi-orpres} \\
  \xi^L_{\ell + n}(V \oplus W) & = \xi^L_\ell(V) \oplus \xi^L_n(W) \label{xi-oplus} \\
  \xi^L_n(\mu \cdot V) &= (\mu \otimes I_m) \cdot \xi^L_n(V) \qquad \textrm{for } \mu \in GL_n(K) \label{xi-gl} \\
  \dim_K(\xi^L_n(V)) &= \dim_K(V). \label{xi-dim}
\end{align}
In fact, $\xi^L_\bullet$ is essentially the directory morphism
\begin{equation*}
  \Dir_K(K) \to \Dir_K(K^m)
\end{equation*}
obtained by pushforward along the morphism
\begin{align*}
  K & \to K^m \\
  x & \mapsto (a_1x, \ldots, a_mx).
\end{align*}
\begin{theorem}\label{mutations-exist}
  Let $\varsigma : \Dir_K(K) \to \Dir_R(M)$ be a ($K_0$-linear)
  $d$-inflator, where $R$ is a semisimple $K_0$-algebra and $M$ is an
  $R$-module of length $d$.  Let $L = K \cdot (a_1,\ldots,a_m)$ be a
  line in $K^m$.  Let $M' = \varsigma_m(L)$; so $M'$ is a submodule of
  $M^m$.  For $V \in \Sub_K(K^n)$, define
  \begin{equation*}
    \varsigma'_n(V) = \varsigma_{mn}(\xi^L_n(V)).
  \end{equation*}
  Then $\varsigma'_n(V)$ is an $R$-submodule of $(M')^n$ for each $n$, and
  the family $\varsigma'_\bullet$ forms a $d$-inflator
  \begin{equation*}
    \varsigma'_\bullet : \Dir_K(K) \to \Dir_R(M').
  \end{equation*}
\end{theorem}
\begin{proof}
  Let $V$ be a subspace of $K^n$.  By (\ref{xi-oplus}),
  \begin{equation*}
    \xi^L_n(K^n) = \xi^L_n(K^{\oplus n}) = (\xi^L_1(K))^{\oplus n}.
  \end{equation*}
  Now $\xi^L_1(K)$ is $L \le K^m$, so
  \begin{equation*}
    \xi^L_n(K^n) = L^n \le (K^m)^n.
  \end{equation*}
  By (\ref{xi-orpres}),
  \begin{equation*}
    V \le K^n \implies \xi^L_n(V) \le \xi^L_n(K^n).
  \end{equation*}
  Thus $\xi^L_n(V) \le \xi^L_n(K^n) = L^n$.
  Because $\varsigma$ is a directory morphism,
  \begin{equation*}
    \varsigma_{mn}(\xi^L_n(V)) \subseteq \varsigma_{mn}(L^n) =
    \varsigma_{mn}(L^{\oplus n}) = (\varsigma_m(L))^{\oplus n} = (M')^n.
  \end{equation*}
  This shows that $\varsigma'_n(V) \in \Sub_R((M')^n)$ for each $n$.
  Next, we verify that the $\varsigma'_\bullet$ maps constitute a morphism
  of directories:
  \begin{enumerate}
  \item The map $\varsigma'_n$ is order-preserving: it is the composition
    of $\varsigma_{mn}$, which is order preserving because $\varsigma_\bullet$
    is a directory morphism, and $\xi^L_n$, which is order-preserving
    by (\ref{xi-orpres}).
  \item The maps $\varsigma'_\bullet$ are compatible with $\oplus$: this
    follows from the analogous properties of $\varsigma_\bullet$ (it is a
    directory morphism) and $\xi^L_n$ (Equation (\ref{xi-oplus}) above):
    \begin{align*}
      \varsigma'_{\ell + n}(V \oplus W) & = \varsigma_{m \ell + m n}(\xi^L_{\ell
        + n}(V \oplus W)) = \varsigma_{m \ell + mn}(\xi^L_\ell(V) \oplus
      \xi^L_n(W)) \\ & = \varsigma_{m \ell}(\xi^L_\ell(V)) \oplus \varsigma_{m
        n}(\xi^L_n(V)) = \varsigma'_\ell(V) \oplus \varsigma'_n(W).
    \end{align*}
  \item The map $\varsigma'_n$ is compatible with the action of
    $GL_n(K_0)$.  This one is the most complicated.  Given $V \in
    \Sub_K(K^n)$ and $\mu \in GL_n(K_0)$, note
    \begin{equation*}
      \varsigma'_n(\mu \cdot V) = \varsigma_{mn}(\xi^L_n(\mu \cdot V))
    \end{equation*}
    By (\ref{xi-gl}) and the fact that $\varsigma_\bullet$ is a directory
    morphism,
    \begin{equation*}
      \varsigma_{mn}(\xi^L_n(\mu \cdot V)) = \varsigma_{mn}((\mu \otimes I_m)
      \cdot \xi^L_n(V)) = (\mu \otimes I_m) \cdot
      \varsigma_{mn}(\xi^L_n(V)) \in \Sub_R(M^{mn})
    \end{equation*}
    Now the subtle point is that the action of $(\mu \otimes I_m)$
    on $M^{mn}$ is the same as the action of $\mu$ on $(M^m)^n$, which
    restricts to the action of $\mu$ on $(M')^n$.  Thus
    \begin{equation*}
      \varsigma'_n(\mu \cdot V) = (\mu \otimes I_m) \cdot
      \varsigma_{mn}(\xi^L_n(V)) = \mu \cdot \varsigma'_n(V).
    \end{equation*}
  \end{enumerate}
  Thus $\varsigma'_\bullet$ is a valid directory morphism.

  Next, we verify that $\varsigma'_\bullet$ is a $d$-inflator.  First of
  all,
  \begin{equation*}
    \ell_R(M') = \ell_R(\varsigma_m(L)) = d \cdot \dim_K(L) = d.
  \end{equation*}
  Finally, for any $V \in \Sub_K(K^n)$,
  \begin{equation*}
    \ell_R(\varsigma'_n(V)) = \ell_R(\varsigma_{mn}(\xi^L_n(V))) = d \cdot
    \dim_K(\xi^L_n(V)) = d \cdot \dim_K(V),
  \end{equation*}
  using Equation~(\ref{xi-dim}) above.
\end{proof}
\begin{definition}
  If $\varsigma : \Dir_K(K) \to \Dir_R(M)$ is a $d$-inflator and $L \le
  K^m$ is a one-dimensional subspace, the \emph{mutation of $\varsigma$
    along $L$} is the $d$-inflator $\varsigma' : \Dir_K(K) \to \Dir_R(M')$
  constructed as in Theorem~\ref{mutations-exist}.
\end{definition}

\subsection{Mutation and the fundamental ring}\label{sec:mut-r-i}
Recall the notion of ``$b \in K$ specializes to $\varphi \in \End_R(M)$''
of \S\ref{sec:r-i}.
\begin{lemma}\label{gr-in-mutation}
  Let $\varsigma : \Dir_K(K) \to \Dir_R(M)$ be a $d$-inflator, let $L = K
  \cdot (a_1,\ldots,a_m)$ be a line in $K^m$, and let $\varsigma' :
  \Dir_K(K) \to \Dir_R(M')$ be the mutation, where $M' = \varsigma_m(L)$.

  If $b \in K$ specializes (with respect to $\varsigma$) to an
  endomorphism $\varphi \in \End_R(M)$, then the map
  \begin{align*}
    \varphi^{\oplus m} & \in \End_R(M^m) \\
    \varphi^{\oplus m}(x_1,\ldots,x_m) &= (\varphi(x_1),\ldots,\varphi(x_m))
  \end{align*}
  maps $M'$ into $M'$, and therefore induces an endomorphism $\varphi'
  \in \End_R(M')$.  The element $b$ specializes (with respect to
  $\varsigma'$) to $\varphi'$.
\end{lemma}
\begin{proof}
  Let $\Theta_b = K \cdot (1,b)$.  By definition of ``specializes to,''
  \begin{equation*}
    \varsigma_2(\Theta_b) = \{(x,\varphi(x)) : x \in M\} \le M^2.
  \end{equation*}
  Because $\varsigma_\bullet$ is compatible with $\oplus$,
  \begin{align*}
    \Theta_b^{\oplus m} &= \{(x_1, bx_1, x_2, bx_2, \ldots, x_m, bx_m) : \vec{x} \in K^m\} \\
    \varsigma_{2m}(\Theta_b^{\oplus m}) &= \{(x_1, \varphi(x_1), x_2, \varphi(x_2), \ldots, x_m, \varphi(x_m)) : \vec{x} \in M^m\}.
  \end{align*}
  Consider the subspace
  \begin{equation*}
    V = \{(x_1,x_2,\ldots,x_m,bx_1,bx_2,\ldots,bx_m) : \vec{x} \in K^m\} \le K^{2m}.
  \end{equation*}
  Because $\varsigma_{2m}$ preserves permutations,
  \begin{equation*}
    \varsigma_{2m}(V) =
    \{(x_1,x_2,\ldots,x_m,\varphi(x_1),\varphi(x_2),\ldots,\varphi(x_m)) :
    \vec{x} \in M^m\} = \{(\vec{x},\varphi^{\oplus m}(\vec{x})) : \vec{x} \in M^m\}.
  \end{equation*}
  Meanwhile, 
  \begin{align*}
    L &= \{(a_1x, a_2x, \ldots, a_mx) : x \in K\} \\
    \varsigma_m(L) & =: M'.
  \end{align*}
  By compatibility with $\oplus$,
  \begin{align*}
    L \oplus L &= \{(a_1x, a_2x, \ldots, a_mx, a_1y, a_2y, \ldots, a_my) : x, y \in K\} \\
    \varsigma_{2m}(L \oplus L) &= \{(x_1,\ldots,x_m, y_1,\ldots, y_m) : \vec{x}, \vec{y} \in M'\}.
  \end{align*}
  Now $\varsigma_{2m}(-)$ is order-preserving, so
  \begin{equation*}
    \varsigma_{2m}(V \cap (L \oplus L)) \subseteq \varsigma_{2m}(V) \cap \varsigma_{2m}(L).
  \end{equation*}
  By the above identifications of $V, L \oplus L, \varsigma_{2m}(V),$ and
  $\varsigma_{2m}(L \oplus L)$, we have
  \begin{align*}
    V \cap (L \oplus L) &= \{(a_1x, a_2x, \ldots, a_mx, a_1bx, a_2bx, \ldots, a_mbx) : x \in K\} \\
    \varsigma_{2m}(V) \cap \varsigma_{2m}(L \oplus L) &= \{(\vec{x}, \varphi^{\oplus m}(\vec{x})) : \vec{x} \in M',~ \varphi^{\oplus m}(\vec{x}) \in M'\}.
  \end{align*}
  The dimension of $V \cap (L \oplus L)$ as a $K$-vector space is 1.
  By the length-scaling law,
  \begin{equation*}
    \ell_R(\varsigma_{2m}(V \cap (L \oplus L))) = d \cdot \dim_K(V \cap (L \oplus L)) = d.
  \end{equation*}
  On the other hand
  \begin{equation*}
    \ell_R(\{(\vec{x}, \varphi^{\oplus m}(\vec{x})) : \vec{x} \in M', ~
    \varphi^{\oplus m}(\vec{x}) \in M'\}) \le \ell_R(M') = d,
  \end{equation*}
  with equality only if
  \begin{equation*}
    \vec{x} \in M' \implies \varphi^{\oplus m}(\vec{x}) \in M'.
  \end{equation*}
  Therefore, equality holds, $\varphi^{\oplus m}(\vec{x})$ maps $M'$
  into $M'$, and
  \begin{equation*}
    \varsigma_{2m}(V \cap (L \oplus L)) = \varsigma_{2m}(V) \cap \varsigma_{2m}(L
    \oplus L) = \{(\vec{x}, \varphi^{\oplus m}(\vec{x})) : \vec{x} \in
    M'\}.
  \end{equation*}
  Now
  \begin{align*}
    \xi^L_2(\Theta_b) &= \{(a_1\vec{x},\ldots,a_m\vec{x})^T : \vec{x} \in \Theta_b\} \\
    &= \{(a_1x, a_1bx, a_2x, a_2bx, \ldots, a_mx, a_mbx)^T : x \in K\} \\
    &= \{(a_1x, a_2x, \ldots, a_mx, a_1bx, a_2bx, \ldots, a_mbx) : x \in K\} \\ &= V \cap (L \oplus L).
  \end{align*}
  Therefore,
  \begin{align*}
    \varsigma'_2(\Theta_b) &= \varsigma_{2m}(\xi^L_2(\Theta_B)) = \varsigma_{2m}(V
    \cap (L \oplus L)) = \{(\vec{x},\varphi^{\oplus m}(\vec{x})) :
    \vec{x} \in M'\}.
  \end{align*}
  So $b$ specializes (with respect to $\varsigma'$) to the endomorphism of
  $M'$ induced by $\varphi^{\oplus m}$.
\end{proof}
\begin{proposition}[$\approx$ Lemma~10.20 in \cite{prdf}]\label{prop:growth}
  Let $\varsigma$ be a $d$-inflator on $K$, let $L = K \cdot (a_1, \ldots,
  a_m)$ be a line in $K^m$, and let $\varsigma'$ be the mutation of
  $\varsigma$ along $L$.  Let $R, R'$ be the fundamental rings of $\varsigma,
  \varsigma'$, and let $I, I'$ be the fundamental ideals of $\varsigma,
  \varsigma'$.  Then
  \begin{align*}
    R & \subseteq R' \\
    I & \subseteq I'.
  \end{align*}
\end{proposition}
\begin{proof}
  If $b \in R$, then by definition $b$ specializes (with respect to
  $\varsigma$) to some $\varphi \in \End_R(M)$.  By
  Lemma~\ref{gr-in-mutation}, $b$ specializes (with respect to
  $\varsigma'$) to some endomorphism $\varphi' \in \End_R(M')$.  Thus $b \in
  R \implies b \in R'$.  Moreover, $\varphi'$ is given by the restriction
  to $M'$ of $\varphi^{\oplus m} : M^m \to M^m$.  Thus, if $\varphi = 0$,
  then $\varphi' = 0$.  In other words, $b \in I \implies b \in I'$.
\end{proof}

\subsection{Iterated mutation}\label{sec:transitive}
\begin{proposition}\label{prop:iter}
  Suppose
  \begin{align*}
    \varsigma : \Dir_K(K) \to \Dir_R(M) \\
    \varsigma' : \Dir_K(K) \to \Dir_R(M') \\
    \varsigma'' : \Dir_K(K) \to \Dir_R(M'')
  \end{align*}
  are three $d$-inflators on $K$.  Suppose $\varsigma'$ is the mutation of
  $\varsigma$ along a line $L_1 \le K^{m_1}$.  Suppose $\varsigma''$ is the
  mutation of $\varsigma'$ along a line $L_2 \le K^{m_2}$.  Then $\varsigma''$
  is isomorphic to the mutation of $\varsigma$ along the line $L_2 \otimes
  L_1 \le K^{m_1m_2}$.
\end{proposition}
\begin{proof}
  Recall that we can think of $\xi^L_n(V)$ as $V \otimes L$.
  Therefore, for $V \in \Sub_K(K^n)$,
  \begin{equation*}
    \varsigma''_n(V) = \varsigma'_{m_2n}(V \otimes L_2) =
    \varsigma_{m_1m_2n}(V \otimes L_2 \otimes L_1).
  \end{equation*}
  Also,
  \begin{equation*}
    M'' = \varsigma''_1(K) = \varsigma_{m_1m_2}(K \otimes L_2 \otimes
    L_1) = \varsigma_{m_1m_2}(L_2 \otimes L_1),
  \end{equation*}
  so $M''$ is the expected submodule of $M^{m_1m_2}$.
\end{proof}

\begin{lemma}\label{lem:the-same}
  Let $L_1, L_2$ be two lines in $K^m$.  Suppose $L_2 = \mu \cdot L_1$
  for some $\mu \in GL_m(K_0)$.  Let $\varsigma$ be a $d$-inflator on $K$,
  and let $\varsigma', \varsigma''$ be the mutations of $\varsigma$ along $L_1$
  and $L_2$, respectively.  Then $\varsigma'$ is equivalent to $\varsigma''$.
\end{lemma}
\begin{proof}
  Let $\Dir_R(M)$ be the target of $\varsigma$.  Then
  \begin{align*}
    \varsigma'_n(V) &= \varsigma_{mn}(\xi^{L_1}_n(V)) \in \Sub_R((M')^n) \\
    \varsigma'_n(V) &= \varsigma_{mn}(\xi^{L_2}_n(V)) \in \Sub_R((M'')^n)
  \end{align*}
  where $M' = \varsigma_m(L_1)$ and $M'' = \varsigma_m(L_2)$.  Then
  \begin{equation*}
    M'' = \varsigma_m(L_2) = \varsigma_m(\mu \cdot L_2) = \mu \cdot \varsigma_m(L_2).
  \end{equation*}
  So $\mu : M^m \to M^m$ induces an isomorphism from $M'$ to $M''$.
  This in turn induces an isomorphism
  \begin{align*}
    \Dir_R(M') & \to \Dir_R(M'') \\
    \Sub_R((M')^n) & \to \Sub_R((M'')^n) \\
    V & \mapsto (I_n \otimes \mu) \cdot V.
  \end{align*}
  Now for any $V \in \Sub_K(K^n)$, we have
  \begin{align*}
    (I_n \otimes \mu) \cdot \varsigma'_n(V) &= (I_n \otimes \mu) \cdot
    \varsigma_{mn}(V \otimes L_1) \\ &= \varsigma_{mn}((I_n \otimes \mu) \cdot
    (V \otimes L_1)) \\ &= \varsigma_{mn}(V \otimes (\mu \cdot L_1)) \\ &=
    \varsigma_{mn}(V \otimes L_2) = \varsigma''_n(V).
  \end{align*}
  Therefore the diagram commutes
  \begin{equation*}
    \xymatrix{ \Dir_K(K) \ar[r]^{\varsigma'} \ar[dr]_{\varsigma''} &
      \Dir_R(M') \ar[d] \\ & \Dir_R(M'')}
  \end{equation*}
  and the vertical map on the right is an isomorphism.
\end{proof}
\begin{remark}\label{rem:commute}
  If $L_1$ is a line in $K^{m_1}$ and $L_2$ is a line in $K^{m_2}$, then
  the two lines $L_1 \otimes L_2$ and $L_2 \otimes L_1$ in $K^{m_1m_2}$
  are related by a permutation matrix, so they induce equivalent
  mutations.  Therefore, up to equivalence, mutating along $L_1$
  commutes with mutating along $L_2$.
\end{remark}

\subsection{The limiting ring} \label{sec:the-end}

\begin{proposition}\label{limitors}
  Let $\varsigma$ be a $d$-inflator on $K$.  For any $m$ and any line $L
  \le K^m$, let $\varsigma^L$ denote the mutation of $\varsigma$ along $L$.
  Let $R_L$ and $I_L$ denote the fundamental ring and ideal of
  $\varsigma^L$.  Then the following are directed unions:
  \begin{align*}
    R_\infty &= \bigcup_L R_L \\
    I_\infty &= \bigcup_L I_L.
  \end{align*}
  Therefore $R_\infty$ is a subring of $K$ and $I_\infty$ is an ideal
  in $R_\infty$.  Furthermore, $1 + I_\infty \subseteq
  R_\infty^\times$, so $I_\infty$ is in the Jacobson radical of
  $I_\infty$.
\end{proposition}
\begin{proof}
  The unions are directed by Proposition~\ref{prop:growth},
  Proposition~\ref{prop:iter}, and Remark~\ref{rem:commute}.  The
  remaining properties follow by Proposition~\ref{o-i}.
\end{proof}
\begin{definition}
  The \emph{limiting ring} and \emph{limiting ideal} of a $d$-inflator
  $\varsigma$ are the ring $R_\infty$ and ideal $I_\infty$ of
  Proposition~\ref{limitors}.
\end{definition}
Since $\varsigma$ is a trivial mutation of itself (along the line $K^1 \le
K^1$), one has
\begin{align*}
  R_\infty &\supseteq R \\
  I_\infty &\supseteq I
\end{align*}
where $R, I$ are the fundamental ring and ideal of $\varsigma$.

\begin{lemma}\label{hey}
  Let $\varsigma : \Dir_K(K) \to \Dir_R(M)$ be a $d$-inflator on $K$.  Let
  $a$ be an element of $K$ and $q$ be an element of $K_0$ such that $a
  \ne q$.  Let $\varsigma'$ be the mutation of $\varsigma$ along the line $K
  \cdot (1,a, \ldots, a^{d-1})$, and let $R'$ be the fundamental ring.
  If $1/(a-q) \notin R'$, then there is non-zero $\epsilon \in M$ such
  that
  \begin{equation*}
    (\epsilon, q \epsilon, \ldots, q^d \epsilon) \in
    \varsigma_{d+1}(\{(x,ax,a^2x,\ldots,a^dx) : x \in K\}).
  \end{equation*}
\end{lemma}
\begin{proof}
  Let $M' = \varsigma_d(K \cdot (1, a, \ldots, a^{d-1}))$, so that
  $\Dir_R(M')$ is the codomain of $\varsigma'$.

  Let $\Theta = K \cdot (1, 1/(a - q)) = K \cdot (a - q, 1)$.  By
  definition of $R'$ and by Lemma~\ref{o-alt-criterion},
  \begin{equation*}
    1/(a-q) \notin R' \implies \varsigma'_2(\Theta) \cap (0 \oplus M) > 0 \oplus 0.
  \end{equation*}
  Now
  \begin{align*}
    \varsigma'_2(\Theta) = \varsigma_{2d}(\{((a-q)x, (a-q)ax, (a-q)a^2x,
    \ldots, (a-q)a^{d-1}x; & \\  x, ax, \ldots, a^{d-1}x)& : x \in K\}).
  \end{align*}
  Let $N = \{(x, ax, a^2x, \ldots, a^d x) : x \in K\}$.  Note that
  \begin{align*}
    & \{((a-q)x, (a-q)ax, \ldots, (a-q)a^{d-1}x; x, ax, \ldots,
    a^{d-1}x) : x \in K\} \\ =& \{(x_1 - qx_0, x_2 - qx_1, \ldots, x_d -
    qx_{d-1}; x_0, x_1, \ldots, x_{d-1}) : (x_0,\ldots,x_d) \in N\}.
  \end{align*}
  By the techniques of \S\ref{sec:basic-tech}-\ref{sec:r-i},
  \begin{align*}
    \varsigma'_2(\Theta) = \{(x_1 - qx_0, x_2 - qx_2, \ldots, x_d -
    qx_{d-1};& \\ x_0, x_1, \ldots, x_{d-1})& : (x_0, \ldots, x_d) \in
    \varsigma_{d+1}(N)\}.
  \end{align*}
  So the assumption that $\varsigma'_2(\Theta) \cap (0 \oplus M)$ is
  non-trivial implies that there are $(x_0, \ldots, x_d) \in
  \varsigma_{d+1}(N)$ such that
  \begin{align*}
    \forall i &: x_{i+1} = qx_i \\
    \exists i &: x_i \ne 0.
  \end{align*}
  Evidently, then $x_i = q^i \epsilon$ for some non-zero $\epsilon \in
  M$.  Then
  \begin{equation*}
    (\epsilon, q \epsilon, q^2 \epsilon, \ldots, q^d \epsilon) \in \varsigma_{d+1}(N). \qedhere
  \end{equation*}
\end{proof}

\begin{lemma}[$\approx$ Lemma~10.21 in \cite{prdf}]
  Let $\varsigma : \Dir_K(K) \to \Dir_R(M)$ be a $d$-inflator on $K$.  Let
  $a$ be an element of $K$.  Let $\varsigma'$ be the mutation of $\varsigma$
  along the line $K \cdot (1, a, a^2, \ldots, a^{d-1})$.  Then $a$ is
  tame with respect to $\varsigma'$, i.e., $1/(a - q)$ is in the
  fundamental ring of $\varsigma'$ for almost all $q \in K_0$.
\end{lemma}
\begin{proof}
  Choose $q_0, q_1, \ldots, q_d$ distinct elements of $K_0$.  If $1/(a
  - q_i)$ is not in the fundamental ring of $\varsigma'$ for each $i$, then
  by Lemma~\ref{hey}, there exist $\epsilon_0, \ldots, \epsilon_d \in
  M$, all non-zero, such that for each $i$,
  \begin{equation*}
    (\epsilon_i, q_i \epsilon_i, q_i^2 \epsilon_i, \ldots, q_i^d
    \epsilon_i) \in \varsigma_{d+1}(K \cdot (1,a,\ldots,a^d)).
  \end{equation*}
  By invertibility of Vandermonde matrices, there is a matrix $\mu \in
  GL_{d+1}(K_0)$ such that
  \begin{equation*}
    \mu \cdot (1, q_i, \ldots, q_i^d) = \vec{u}_i,
  \end{equation*}
  where $\vec{u}_i$ is the $i$th standard basis vector.  Then
  \begin{equation*}
    \epsilon_i \vec{u}_i \in \mu \cdot \varsigma_{d+1}(K \cdot (1,a,\ldots,a^d))
  \end{equation*}
  for each $i$.  But the right side is an $R$-module of length $d
  \cdot \dim_K(K \cdot (1,a,\ldots,a^d)) = d$.  The elements on the
  left side generate an $R$-module
  \begin{equation*}
    (R\epsilon_0) \oplus (R\epsilon_1) \oplus \cdots \oplus (R\epsilon_d) \le M^{d+1}
  \end{equation*}
  of length at least $d + 1$, a contradiction.
\end{proof}

\begin{theorem}[$\approx$ Theorem~10.25 in \cite{prdf}]\label{thm:generate-mvals}
  If $\varsigma$ is a $d$-inflator of $K$, the limiting ring $R_\infty$ is a
  multi-valuation ring.  In fact, it is an intersection of at most $d$
  valuation rings.  If the fundamental ideal of $\varsigma$ is non-trivial,
  then $R_\infty$ is a non-trivial multi-valuation ring, i.e.,
  $R_\infty$ is a proper subring of $K$.
\end{theorem}
\begin{proof}
  Fix $q_1, q_2, \ldots, q_d$ distinct elements of $K_0$.  For any $a
  \in K$, there is a mutation $\varsigma'$, namely the mutation along $K
  \cdot (1, a, a^2, \ldots, a^{d-1})$, such that $a$ is tame with
  respect to $\varsigma'$.  By Lemma~\ref{tame-wild}, one of the elements
  \begin{equation*}
    x, \frac{1}{x - q_1}, \ldots, \frac{1}{x - q_d}
  \end{equation*}
  is in the fundamental ring $R_{\varsigma'}$ of $\varsigma'$, and therefore
  in the limiting ring $R^\infty_\varsigma$ of $\varsigma$.  By
  Lemma~\ref{algebra-case}, $R^\infty = R^\infty_\varsigma$ is a
  multi-valuation ring, an intersection of at most $d$ valuation rings
  on $K$.

  Now if the fundamental ideal $I_\varsigma$ is non-trivial, then the
  limiting ideal $I^\infty_\varsigma$ is non-trivial.  By
  Proposition~\ref{limitors}, $R^\infty$ has a non-trivial Jacobson
  radical.  The Jacobson radical of a field is trivial, so $R^\infty$
  must be a non-trivial multi-valuation ring.
\end{proof}

\subsection{Mutation and malleability}\label{sec:mut-mal}
The following fact isn't strictly necessary for the applications we
have in mind,\footnote{Because of Proposition~\ref{prop:ped-mal} above
  and Proposition~\ref{prop:still} below.} but it's nice to know conceptually.
\begin{proposition}\label{prop:presmal}
  Let $\varsigma$ be a malleable $d$-inflator and $\varsigma'$ be a mutation.
  Then $\varsigma'$ is malleable.
\end{proposition}
\begin{proof}
  Let $\varsigma$ be a morphism from $\Dir_K(K)$ to $\Dir_R(M)$ and let
  $\varsigma' : \Dir_K(K) \to \Dir_R(M')$ be the mutation along a line $L
  \le K^m$, where $M' = \varsigma_m(L) \le M^m$.

  Let $X, Z$ be subspaces of $K^n$ with $X \subseteq Z$.  Let $Y$ be a
  submodule of $(M')^n$ such that
  \begin{equation*}
    \varsigma_{mn}(\xi^L_n(X)) = \varsigma'_n(X) \subseteq Y \subseteq
    \varsigma'_n(Z) = \varsigma_{mn}(\xi^L_n(Z)).
  \end{equation*}
  \begin{equation*}
    \ell_R(Y) = \ell_R(\varsigma'_n(X)) + 1 = \ell_R(\varsigma_{mn}(\xi^L_n(X))) + 1.
  \end{equation*}
  By malleability of the map $\varsigma_{mn} : \Sub_K(K^{mn}) \to
  \Sub_R((M^m)^n)$, there is a subspace $Y' \le K^{mn}$ such that
  \[ \xi^L_n(X) \le Y' \le \xi^L_n(Z) \]
  \[ Y \le \varsigma_{mn}(Y') \]
  \[ \dim_K(Y') = \dim_K(\xi^L_n(X)) + 1.\]
  By Lemma~\ref{silly} below, the map $\xi^L_n(-)$ induces an
  isomorphism from the interval $[X,Z] \subseteq \Sub_K(K^n)$ to the
  interval $[\xi^L_n(X),\xi^L_n(Z)] \subseteq \Sub_K(K^{mn})$.
  Therefore $Y' = \xi^L_n(Y'')$ for a unique subspace $Y'' \le K^n$
  such that
  \begin{equation*}
    X \le Y'' \le Z.
  \end{equation*}
  Moreover, $\xi^L_n(-)$ preserves dimensions, by
  Equation~(\ref{xi-dim}) in \S\ref{sec:mut-def}.  Thus
  \begin{equation*}
    \dim_K(Y'') = \dim_K(\xi^L_n(Y')) = \dim_K(\xi^L_n(X)) + 1 = \dim_K(X) + 1.
  \end{equation*}
  So we have found $Y''$ between $X$ and $Z$ such that $\dim_K(Y''/X) = 1$ and
  \begin{equation*}
    Y \le \varsigma_{mn}(Y') = \varsigma_{mn}(\xi^L_n(Y'')) = \varsigma'_n(Y'').
  \end{equation*}
  This is the exact configuration required by malleability.
\end{proof}
\begin{lemma}\label{silly}
  For any $n$, for any line $L$, for any subspaces $X, Z \in
  \Sub_K(K^n)$ with $X \le Z$, the map $\xi^L_n(-)$ induces an
  isomorphism from the interval
  \begin{equation*}
    [X,Z] \subseteq \Sub_K(K^n)
  \end{equation*}
  to the interval
  \begin{equation*}
    [\xi^L_n(X),\xi^L_n(Z)] \subseteq \Sub_K(K^{nm}).
  \end{equation*}
\end{lemma}
\begin{proof}
  Let $L = K \cdot (a_1, \ldots, a_m)$.  Recall that $\xi^L_n(-)$ is
  given by
  \begin{equation*}
    \xi^L_n(V) = \{(a_1\vec{x},\ldots,a_m\vec{x})^T : \vec{x} \in V\}.
  \end{equation*}
  Choose some $\mu \in GL_m(K)$ such that $\mu \cdot (a_1,\ldots,a_m)
  = (1,0,0,\ldots,0)$.  Let $\tau \in GL_{mn}(K)$ be the permutation
  matrix inducing the transpose operation
  \begin{equation*}
    \tau \cdot (a_{1,1},a_{1,2},\ldots,a_{1,n},a_{2,1},\ldots,a_{m,n}) =
    (a_{1,1},a_{2,1},\ldots,a_{m,1},a_{1,2},\ldots,a_{m,n}).
  \end{equation*}
  Then $(\mu \otimes I_n) \cdot \tau$ is an invertible matrix in
  $GL_{mn}(K)$, inducing an automorphism of the $K$-vector space
  $K^{mn}$.  This automorphism sends
  \begin{equation*}
    (a_1\vec{x},a_2\vec{x},\ldots,a_m\vec{x})^T \mapsto
    (\vec{x},\underbrace{\vec{0},\ldots,\vec{0}}_{m-1 \textrm{ times}}).
  \end{equation*}
  The automorphism of $K^{mn}$ induces an automorphism of the lattice
  $\Sub_K(K^{mn})$.  This automorphism sends
  \begin{equation*}
    \xi^L_n(V) \mapsto V \oplus 0^{mn-n}.
  \end{equation*}
  Now the map
  \begin{align*}
    \Sub_K(K^n) \to \Sub_K(K^{mn}) \\
    V \mapsto V \oplus 0^{mn - n}
  \end{align*}
  clearly has the property of mapping intervals isomorphically onto
  intervals.
\end{proof}

\subsection{Mutation and pedestals}\label{sec:mut-ped}
Suppose we are in the Cube-bounded Configuration of
Assumption~\ref{asm:1}, so
\begin{itemize}
\item There is a faithful exact embedding of $K\Vect^f$ into an abelian category $\Cc$.
\item In the category $\Cc$, the object $K$ is cube-bounded, of
  reduced rank $d$.
\end{itemize}
Recall that a \emph{pedestal} is a subobject $A \in \Sub_\Cc(K)$ such that
\begin{equation*}
  \botrk(K/A) = \redrk(K) = d,
\end{equation*}
where the ranks are calculated in $\Cc$.  By the pedestal machine
(Theorem~\ref{thm:actually-important} and
  Proposition~\ref{bleal-form}), any pedestal $A$ determines a
  $d$-inflator $\varsigma$:
\begin{align*}
  \Dir_K(K) & \to \Dir_{\Pro \Cc}(B/A) \\
  V & \mapsto (V+A^n)/A^n \cap B^n/A^n = (V \cap B^n + A^n)/A^n,
\end{align*}
where $B/A$ is the socle of $K/A$ in the category $\Pro \Cc$.

\begin{proposition}[$\approx$ Lemma~10.20 in \cite{prdf}]\label{prop:still}
  Let $A \in \Sub_\Cc(K)$ be a pedestal, let $L = K \cdot
  (a_1,\ldots,a_m)$ be a line in $K^m$.  Then
  \begin{equation*}
    A' = a_1^{-1} A \cap a_2^{-1} A \cap \cdots \cap a_m^{-1} A
  \end{equation*}
  is a pedestal.  Moreover, if $\varsigma, \varsigma'$ are the $d$-inflators
  obtained from $A$ and $A'$, then $\varsigma'$ is the mutation of $\varsigma$
  along the line $L$.
\end{proposition}
\begin{proof}
  Let $\eta : K \to K^m$ be the map $x \mapsto
  (a_1x,a_2x,\ldots,a_mx)$.  Let $\xi^L : \Dir_K(K) \to \Dir_K(K^m)$
  be pushforward along $\eta$.  For each $n$, the map $\eta^{\oplus n}
  : K^n \to K^{mn}$ is
  \begin{align*}
    \vec{x} = (x_1,\ldots,x_n) &\mapsto (a_1x_1, a_2x_1, \ldots, a_mx_1, a_1x_2, a_2x_2, \ldots, a_mx_2, a_1x_3, \ldots, a_mx_n) \\
    & = (a_1x_1, a_1x_2, \ldots, a_1x_n, a_2x_1, \ldots, a_mx_n)^T \\
    & = (a_1\vec{x}, a_2\vec{x}, \ldots, a_m\vec{x})^T.
  \end{align*}
  Therefore $\xi^L_n : \Sub_K(K^n) \to \Sub_K(K^{mn})$ agrees with our
  earlier notation(!)

  We can also view $\eta$ as a morphism in $\Pro \Cc$, because of the
  faithful exact embeddings $K\Vect^f \hookrightarrow \Cc
  \hookrightarrow \Pro \Cc$.  Let
  \begin{equation*}
    \hat{\xi}^L : \Dir_{\Pro \Cc}(K) \to \Dir_{\Pro \Cc}(K^m)
  \end{equation*}
  be pushforward along $\eta$ in the category $\Pro \Cc$.  Let $B$ be
  the $\Pro \Cc$-subobject of $K$ such that $B/A$ is the socle of
  $K/A$.  Also let $B', A'$ be the pullbacks of $B^m$ and $A^m$ along
  $\eta : K \to K^m$.  Thus $B'$ is a $\Pro \Cc$-subobject of $K$, and
  $A'$ is a $\Cc$-subobject of $K$.  In fact,
  \begin{align*}
    A' &= a_1^{-1} A \cap \cdots \cap a_m^{-1} A \\
    B' &= a_1^{-1} B \cap \cdots \cap a_m^{-1} B,
  \end{align*}
  so $A'$ agrees with the $A'$ in the statement of the proposition.
  The map $\eta$ induces a monomorphism
  \begin{equation*}
    \iota : B'/A' \hookrightarrow B/A
  \end{equation*}
  in the category $\Pro \Cc$.  Let $\iota_*$ denote pushforward
  along this map.

  There is a commuting diagram
  \begin{equation} \label{xcom}
    \xymatrix{
      \Dir_K(K) \ar[r]^{\xi^L} \ar[d] & \Dir_K(K^m) \ar[d] \\
      \Dir_{\Pro \Cc}(K) \ar[r]^{\hat{\xi}^L} \ar@{->>}[d] & \Dir_{\Pro \Cc}(K^m) \ar@{->>}[d] \\
      \Dir_{\Pro \Cc}(B'/A') \ar[r]_{\iota_*} & \Dir_{\Pro \Cc}(B^m/A^m).
    }
  \end{equation}
  in which the $\twoheadrightarrow$ maps are interval retracts.  The
  bottom square commutes by Lemma~\ref{its-a-good-time} below, and the
  top square commutes because the embedding $K\Vect^f
  \hookrightarrow \Pro \Cc$ is exact (or because $\eta$ is a
  monomorphism).

  Let $\rho : \Dir_K(K) \to \Dir_{\Pro \Cc}(B^m/A^m)$ be the
  composition of the maps in the diagram.  Thus
  \begin{align*}
    \rho_n(V)& = \iota_{*,n}((V \cap (B')^n + (A')^n)/(A')^n) \\
    \rho_n(V)& = (\xi^L_n(V) \cap B^{mn} + A^{mn})/A^{mn} =
    \varsigma_{mn}(\xi^L_n(V)).
  \end{align*}
  Let $Q = \varsigma_m(L)$.  Then $Q$ is a semisimple subobject of
  $B^m/A^m$ of length $d \cdot \dim_K(L) = d$.  Let
  \begin{equation*}
    \varsigma'' : \Dir_K(K) \to \Dir_{\Pro \Cc}(Q)
  \end{equation*}
  be the mutation of $\varsigma$ along $L$.  Then
  \begin{equation*}
    \varsigma''_n(V) = \varsigma_{mn}(\xi^L_n(V)) = \rho_n(V).
  \end{equation*}
  So $\varsigma''$ and $\rho$ are identical except for their codomain.

  Let $\varsigma'$ be the composition of the left vertical maps in (\ref{xcom})
  \begin{equation*}
    \Dir_K(K) \to \Dir_{\Pro \Cc}(K) \twoheadrightarrow \Dir_{\Pro \Cc}(B'/A')
  \end{equation*}
  \begin{equation*}
    V \mapsto (V \cap (B')^n + (A')^n)/(A')^n.
  \end{equation*}
  By the commutative diagram, $\rho = \iota_* \circ \varsigma'$.  Then for
  any $V \in \Sub_K(K^n)$,
  \begin{equation}\label{cb7}
    \iota_{*,n}(\varsigma'_n(V)) = \rho_n(V) = \varsigma''_n(V).
  \end{equation}
  Taking $n = 1$ and $V = K$,
  \begin{equation*}
    \iota_*(B'/A') = \iota_*((K \cap B' + A')/A') =
    \iota_*(\varsigma'_1(K)) = \varsigma''_1(K).
  \end{equation*}
  Because $\varsigma'' : \Dir_K(K) \to \Dir_{\Pro \Cc}(Q)$ is an inflator,
  \begin{equation*}
    \iota_*(B'/A') = \varsigma''_1(K) = Q.
  \end{equation*}
  Therefore the image of the monomorphism
  \begin{equation*}
    \iota : B'/A' \to B/A
  \end{equation*}
  is exactly $Q$, and $\iota$ is an isomorphism from $B'/A'$ to $Q$.
  Then (\ref{cb7}) gives a commutative diagram in which the bottom map
  is an isomorphism
  \begin{equation*}
    \xymatrix{ \Dir_K(K) \ar[dr]^{\varsigma''} \ar[d]_{\varsigma'} & \\ \Dir_{\Pro \Cc}(B'/A') \ar[r]^{\sim} & \Dir_{\Pro \Cc}(Q)}.
  \end{equation*}
  Therefore $\varsigma'$ is a $d$-inflator equivalent to $\varsigma''$.  Also,
  $B'/A'$ is a semisimple object of length $d$ (because of the
  isomorphism to $Q$).  Since $A'$ is a $\Cc$-subobject of $K$, not
  merely a $\Pro \Cc$-subobject, the pro-object $K/A'$ has a socle
  $B''/A'$.  Semisimplicity of $B'/A'$ implies $B' \subseteq B''$.  Then
  \begin{equation*}
    d = \ell(B'/A') \le \ell(B''/A') = \ell(\qsoc(K/A')) =
    \botrk(K/A') \le \redrk(K/A') \le \redrk(K) = d.
  \end{equation*}
  Therefore equality holds, forcing $B' = B''$ and $\botrk(K/A') = d$.
  It follows that $A'$ is a pedestal.  Moreover, $B'/A' =
  \qsoc(K/A')$, and so $\varsigma'$ is the $d$-inflator associated to $A'$:
  \begin{align*}
    \varsigma'_n(V) = (V \cap (B')^n + (A')^n)/(A')^n & = (V +
    (A')^n)/(A')^n \cap (B')^n/(A')^n \\ &= \qsoc((V + (A')^n)/(A')^n).
  \end{align*}
  We have shown that $\zeta'$ (the inflator derived from $A'$) is
  equivalent to $\zeta''$ (the mutation along $L$).
\end{proof}
\begin{lemma}\label{its-a-good-time}
  Let $M$ be an object in an abelian category $\Cc$.  Let $A \le B \le
  M$ be subobjects.  Let $f : N \hookrightarrow M$ be a monomorphism.
  Let $A' = f^{-1}(A)$ and $B' = f^{-1}(B)$.  Let $f' : A'/B' \to A/B$
  be the monomorphism induced by $f$.  Then the diagram commutes
  \begin{equation*}
    \xymatrix{
    \Dir_\Cc(N) \ar[r]^{f_*} \ar@{->>}[d] & \Dir_\Cc(M) \ar@{->>}[d] \\
    \Dir_\Cc(B'/A') \ar[r]_{f'_*} & \Dir_\Cc(B/A)
    }
  \end{equation*}
  where the vertical maps are interval retracts and the horizontal
  maps are pushforwards.
\end{lemma}
\begin{proof}
  Without loss of generality $N \subseteq M$, so $A' = A \cap N$ and
  $B' = B \cap N$.  Consider the square
  \begin{equation*}
    \xymatrix{
    \Dir_\Cc(N) \ar[r]^{f_*} \ar@{->>}[d] & \Dir_\Cc(M) \ar@{->>}[d] \\
    \Dir_\Cc(B') \ar[r]_{f'_*} & \Dir_\Cc(B)
    }    
  \end{equation*}
  where the vertical maps are interval retracts onto $[0,B'] \subseteq
  \Sub_\Cc(N)$ and $[0,B] \subseteq \Sub_\Cc(M)$, and the bottom map
  is induced by the inclusion $B' = B \cap N \hookrightarrow B$.  Then
  the square commutes, because the upper right path sends $X \in
  \Sub_\Cc(N^n)$ like so:
  \begin{equation*}
    X \mapsto X \mapsto X \cap B^n;
  \end{equation*}
  the bottom left path sends $X$ like so:
  \begin{equation*}
    X \mapsto X \cap (B')^n \mapsto X \cap (B')^n;
  \end{equation*}
  and $X \cap (B')^n = X \cap B^n \cap N^n = X \cap B^n$ for $X \subseteq N^n$.

  Also, the square
  \begin{equation*}
    \xymatrix{
    \Dir_\Cc(B') \ar[r]^{f_*} \ar@{->>}[d] & \Dir_\Cc(B) \ar@{->>}[d] \\
    \Dir_\Cc(B'/A') \ar[r]_{f'_*} & \Dir_\Cc(B/A)
    }
  \end{equation*}
  commutes, because the maps are the pushforwards along the following
  commuting diagram in $\Cc$:
  \begin{equation*}
    \xymatrix{
      B \cap N \ar[r] \ar[d] & B \ar[d] \\
      (B \cap N)/(A \cap N) \ar[r] & B/A.}
  \end{equation*}
  Glueing the two squares together gives the desired commuting square
  of directories.
\end{proof}

\section{Application to fields of finite burden}\label{sec:apps}
\begin{theorem}[= Theorem~10.28 in \cite{prdf}] \label{thm:dp1}
  Let $\Kk$ be a saturated unstable dp-finite field.  Then $\Kk$
  admits a non-trivial $\Aut(\Kk/S)$-invariant valuation ring, for
  some small subset $S \subseteq \Kk$.
\end{theorem}
\begin{proof}
  Fix a magic subfield $K_0 \preceq \Kk$.  Let $\Lambda = \Lambda_1$
  be the lattice of type-definable $K_0$-linear subspaces of $\Kk$.
  By Proposition~10.4.1 in \cite{prdf}, we can find a non-zero
  pedestal $A \in \Lambda$.\footnote{The point is that (1) pedestals
    certainly exist, (2) $\botrk(\Lambda) = 1$ because the groups of
    infinitesimals $I_K$ are co-initial among nonzero elements of
    $\Lambda$, (3) therefore 0 cannot be a pedestal unless
    $\redrk(\Lambda) = 1$, (4) if $\redrk(\Lambda) = 1$, then
    $\Lambda$ is totally ordered and every group is a pedestal.}
  Let $K$ be a small model containing $K_0$ and type-defining $A$.
  Let $\varsigma$ be the $d$-inflator derived from $A$.  Let $R, I$ be the
  fundamental ring and ideal of $\varsigma$.  By Theorem~\ref{yeah-main},
  \begin{align*}
    R &= \{b \in \Kk : b \cdot A \subseteq A\} \\
    I &= \{b \in R : \redrk(A/b \cdot A) = d\}.
  \end{align*}
  By Proposition~10.15.5 in \cite{prdf}, $I \ne 0$; in fact $I$
  contains the $K$-infinitesimals $I_K$.  Let $R_\infty$ be the
  limiting ring of $\varsigma$.  By Theorem~\ref{thm:generate-mvals},
  $R_\infty$ is a non-trivial multi-valuation ring.  The ring
  $R_\infty$ is $\Aut(\Kk/K)$-invariant by construction.  The ring
  $R_\infty$ can be written as an intersection of incomparable
  valuation rings in a unique way, by Corollary~6.7 in \cite{prdf2}:
  \begin{equation*}
    R_\infty = \Oo_1 \cap \cdots \cap \Oo_n.
  \end{equation*}
  Non-triviality of $R_\infty$ implies non-triviality of the $\Oo_i$.
  The group $\Aut(\Kk/K)$ might permute the $\Oo_i$, but there is a
  finite set $S_0$ such that $\Aut(\Kk/S_0K)$ fixes each $\Oo_i$.
  Then each $\Oo_i$ is a non-trivial $\Aut(\Kk/S)$-invariant valuation
  ring, for $S = S_0 \cup K$.
\end{proof}

For general fields of finite burden, we can prove some weaker statements:
\begin{theorem}\label{thm:fb1}
  Let $\Kk$ be a saturated field of finite burden.  Let $K_0$ be a
  small infinite subfield.  Let $\Lambda$ be the lattice of
  type-definable $K_0$-linear subspaces of $\Kk$, and let
  $\Lambda^{00}$ be the quotient modulo 00-commensurability.  Suppose
  there is some $G \in \Lambda$ and non-zero $\epsilon \in \Kk$ such
  that in the lattice $\Lambda^{00}$,
  \begin{equation*}
    \epsilon \cdot G \le G
  \end{equation*}
  \begin{equation*}
    \redrk(G/\epsilon \cdot G) = \redrk(\Lambda^{00}).
  \end{equation*}
  Then $\Kk$ admits a non-trivial $\Aut(\Kk/S)$-invariant valuation
  ring for some small set $S \subseteq \Kk$.
\end{theorem}
\begin{proof}
  Let $d = \redrk(\Lambda^{00})$.  Take a strict $d$-cube in the
  interval $[\epsilon \cdot G, G] \subseteq \Lambda^{00}$.  Let $H \in
  \Lambda$ represent the base of the cube.  Then $H$ is a pedestal in
  $\Lambda^{00}$.  Moreover,
  \begin{equation*}
    \epsilon \cdot G \le H \le G,
  \end{equation*}
  and so
  \begin{equation*}
    \epsilon \cdot H \le \epsilon \cdot G \le H \le G \le \epsilon^{-1} \cdot H.
  \end{equation*}
  By choice of $H$, there is a strict $d$-cube in the interval $[H,G]
  \subseteq [H, \epsilon^{-1} \cdot H]$.  There is an isomorphism of
  intervals
  \begin{equation*}
    [H, \epsilon^{-1} \cdot H] \cong [\epsilon \cdot H, H].
  \end{equation*}
  Therefore $\redrk(H/\epsilon \cdot H) = d = \redrk(\Lambda^{00})$.

  Replacing $G$ with $H$, we may assume that $G$ is a pedestal.  Let
  $\varsigma$ be the associated $d$-inflator.  Let $R, I$ be the
  fundamental ring and ideal, and $R_\infty, I_\infty$ be the limiting
  ring and ideal.  By Theorem~\ref{thm:not-so-useful}
  \begin{equation*}
    \redrk(H/\epsilon \cdot H) = \redrk(\Lambda^{00}) \implies
    \epsilon \in I,
  \end{equation*}
  so $I$ is non-trivial.  By Theorem~\ref{thm:generate-mvals},
  $R_\infty$ is a non-trivial multi-valuation ring.  Proceed as in the
  proof of Theorem~\ref{thm:dp1}.
\end{proof}
\begin{theorem}\label{thm:fb2}
  Let $\Kk$ be a saturated field of finite burden.  Let $K_0$ be a
  small infinite subfield.  Let $\Delta$ be the lattice of definable
  $K_0$-linear subspaces of $\Kk$.  Let $G$ be an element of $\Delta$.
  Suppose that $\epsilon \cdot G \subseteq G$ for some non-zero
  $\epsilon \in \Kk$, and moreover
  \begin{equation*}
    \redrk(G/\epsilon \cdot G) = \redrk(\Delta)
  \end{equation*}
  in the lattice $\Delta$.  Then $\Kk$ admits a non-trivial
  $\Aut(\Kk/S)$-invariant valuation ring for some small set $S
  \subseteq \Kk$.
\end{theorem}
\begin{proof}
  Similar to Theorem~\ref{thm:fb1}.
\end{proof}
\begin{remark}
  In Theorem~\ref{thm:fb2}, it may be possible to show that the
  invariant valuation rings are, in fact \emph{definable}.
\end{remark}

We are also interested in the question of whether mutation terminates
in finitely many steps.  Recall from (\cite{prdf2}, Theorem~8.11) that
a dp-finite field $\Kk$ has \emph{valuation type} if it satisfies the
following equivalent conditions:
\begin{itemize}
\item The canonical topology on $\Kk$ is a V-topology.
\item For any small $K \preceq \Kk$, the $K$-infinitesimals $I_K$ are
  the maximal ideal of a valuation ring $\Oo_K$ on $\Kk$.
\item For some small $K \preceq \Kk$, the $K$-infinitesimals $I_K$
  contain a non-zero ideal of a multi-valuation ring on $\Kk$.
\item Some ``bounded'' group $J \subseteq (\Kk,+)$ contains a non-zero
  ideal of a multi-valuation ring on $\Kk$.
\end{itemize}
In \cite{prdf2}, we showed how to complete the classification of
dp-finite fields assuming the (unlikely) conjecture that unstable
dp-finite fields have valuation type.
\begin{lemma}[Pedestal criterion]
  \label{true-pc}
  Let $\Kk$ be a saturated dp-finite field.  Let $K_0$ be a magic
  subfield.  Let $A$ be a $K_0$-pedestal.
  \begin{enumerate}
  \item If $A$ contains a non-zero multi-valuation ideal, then $\Kk$
    has valuation type.
  \item If $\Stab(A) := \{x \in \Kk : x \cdot A \subseteq A\}$
    contains a non-zero multi-valuation ideal, then $\Kk$ has
    valuation type.
  \end{enumerate}
\end{lemma}
\begin{proof}
  The first point holds because $K_0$-pedestals are bounded
  (Remark~8.6 in \cite{prdf2}).  For the second point, suppose $R$ is
  a multi-valuation ring on $K$, $M$ is a non-zero $R$-submodule of
  $K$, and $\Stab(A) \supseteq M$.  There are two cases.  If $A \ne
  0$, take a non-zero $a_0 \in A$.  Then
  \begin{equation*}
    M \cdot a_0 \subseteq \Stab(A) \cdot A \subseteq A,
  \end{equation*}
  so $A$ contains the non-zero $R$-submodule $M \cdot a_0$.  Therefore
  $\Kk$ has valuation type by the first point.

  Otherwise, $A = 0$.  Let $\Lambda$ be the lattice of type-definable
  $K_0$-linear subspaces of $\Kk$, and let $d = \redrk(\Lambda)$.  By
  definition of pedestal, $d = \botrk([A,\Kk])$.  But $A = 0$, so the
  interval $[A,\Kk]$ is all of $\Lambda$.  Then $d = \botrk(\Lambda)$.
  By Proposition~10.1.6 in \cite{prdf}, $\botrk(\Lambda) = 1$.  Thus
  $d = 1$ and $\Lambda$ is totally ordered.  Therefore every element
  of $\Lambda$ is a pedestal other than $\Kk$.  In particular,
  $I_{K_0}$ is a pedestal.  If $\varsigma$ is the 1-inflator derived from
  $I_{K_0}$, then $\Stab(I_{K_0}) = R_\varsigma$ and $R_\varsigma$ is a
  valuation ring, by Theorem~\ref{yeah-main} and
  Proposition~\ref{prop:1-fold}.  Therefore $\Stab(I_{K_0})$ contains
  a non-zero multi-valuation ideal.  Replacing $A$ with $I_{K_0}$, we
  reduce to the $A \ne 0$ case.
\end{proof}

\begin{theorem}\label{thm:dp2}
  Let $\Kk$ be a saturated dp-finite field.  Let $K_0$ be a
  magic subfield.  At least one of the following holds:
  \begin{enumerate}
  \item $\Kk$ is stable.
  \item $\Kk$ has valuation type.
  \item \label{dp22} There is a $K_0$-pedestal $A$ such that the
    associated inflator $\varsigma$ has the following properties:
    \begin{itemize}
    \item $\varsigma$ is malleable
    \item $\varsigma$ is not weakly multi-valuation type
      (Definition~\ref{def-almost}).
    \item Every mutation of $\varsigma$ is malleable
    \item No mutation of $\varsigma$ is weakly multi-valuation type.
    \end{itemize}
    Moreover, all $K_0$-pedestals $A$ have this property.
  \end{enumerate}
\end{theorem}
\begin{proof}
  Suppose $\Kk$ is unstable and not of valuation type.  Let $d$ be the
  reduced rank of the lattice $\Lambda$ of type-definable $K_0$-linear
  subspaces of $\Kk$.  If $A$ is a pedestal and $\varsigma$ is the
  associated $d$-inflator, then
  \begin{equation*}
    \Stab(A) = R_\varsigma.
  \end{equation*}
  By the Pedestal Criterion (Lemma~\ref{true-pc}), it follows that
  $R_\varsigma$ is \emph{not} weakly multi-valuation type.  Meanwhile,
  $\varsigma$ is malleable by Theorem~\ref{yeah-main}.

  If $\varsigma'$ is any mutation of $\varsigma$, then $\varsigma'$ comes from a
  pedestal $A'$, by Proposition~\ref{prop:still}.  Therefore $\varsigma'$
  continues to be malleable and continues to fail to be weakly
  multi-valuation type.
\end{proof}
\begin{definition}
  An inflator $\varsigma$ on a field $K$ is \emph{wicked} if $\varsigma$ is
  malleable and no mutation of $\varsigma$ is weakly multi-valuation
  type.
\end{definition}
By Propositions~\ref{prop:iter} and \ref{prop:presmal}, any mutation
of a wicked inflator is wicked.  Theorem~\ref{thm:dp2} arguably
reduces the analysis of dp-finite fields to the algebraic analysis of
wicked inflators.  In a subsequent paper, we will demonstrate this
idea by analyzing wicked 2-inflators on fields of characteristic 0.
While we do not obtain a proof of the Shelah or henselianity
conjectures for fields of rank 2, we do prove the following useful
facts:
\begin{theorem}[to appear in \cite{prdf4}]
  Let $\Kk$ be a saturated unstable dp-finite field of characteristic
  0.
  \begin{itemize}
  \item The canonical topology on $\Kk$ is definable.
  \item There is a unique definable V-topology on $\Kk$.
  \end{itemize}
\end{theorem}
If the second point could be generalized to higher ranks, the Shelah
and henselianity conjectures for dp-finite fields would follow.

\section{Mutation examples} \label{sec:mut4}
\subsection{$\Rr$ and $\mathbb{C}$}
Consider the 2-inflator
\begin{align*}
  \varsigma : \Dir_{\mathbb{C}}(\mathbb{C}) &\to \Dir_{\mathbb{R}}(\mathbb{R}) \times \Dir_{\mathbb{R}}(\mathbb{R}) \\
  V & \mapsto (V + \overline{V}, V \cap \overline{V}).
\end{align*}
of Example~\ref{fiona}.  In Example~\ref{fiona-ring}, we saw that the
fundamental ring and ideal are $R = \Rr$ and $I = 0$.  Note that $\Rr$
is \emph{not} a multi-valuation ring on $\mathbb{C}$, by
Proposition~6.2.3 in \cite{prdf2}.

Let $\varsigma'$ be the mutation along $\mathbb{C} \cdot (1,i)$.
Then for any subspace $V \subseteq \mathbb{C}^n$, we have
\begin{equation*}
  \varsigma'(V) = \varsigma(\{(\vec{x},i\vec{x})^T : \vec{x} \in V\}).
\end{equation*}
The first component of $\varsigma'(V)$ is
\begin{equation*}
  W_1 := \{(\vec{x},i\vec{x})^T : \vec{x} \in V\} + \{(\vec{y},-i\vec{y})^T : \vec{y} \in \overline{V}\}
\end{equation*}
Thus $(\vec{a},\vec{b})^T \in W_1$ if and only if
\begin{align*}
  \vec{x} = \frac{1}{2}(\vec{a} - i\vec{b}) \in V \\
  \vec{y} = \frac{1}{2}(\vec{a} + i\vec{b}) \in \overline{V}.
\end{align*}
If $\vec{a}, \vec{b} \in \Rr^n$, these are equivalent, and so the
first component is essentially the image of $V$ under the forgetful
functor $\Dir_{\mathbb{C}}(\mathbb{C}) \to \Dir_\Rr(\mathbb{C})$.  On the
other hand, the second component of $\varsigma'(V)$ is
\begin{equation*}
  W_2 := \{(\vec{x},i\vec{x})^T : \vec{x} \in V\} \cap \{(\vec{y},-i\vec{x})^T : \vec{x} \in \overline{V}\}
\end{equation*}
This vanishes, since $\vec{x} = \vec{y}$ and $i\vec{x} = -i\vec{y}$
together imply $\vec{x} = \vec{y} = 0$.

So $\varsigma'$ is equivalent to the forgetful 2-inflator
\begin{equation*}
  \Dir_{\mathbb{C}}(\mathbb{C}) \to \Dir_\Rr(\mathbb{C}).
\end{equation*}
The new fundamental ring and ideal are $\mathbb{C}$ and 0.

\subsection{Another corrupted example}
Let $K$ be a field with two independent valuations, both of which have
residue field $k$, as in Examples~\ref{gerald}, \ref{gerald-ring}.  Let
$\varsigma$ be the 2-inflator
\begin{align*}
  \varsigma : \Dir_K(K) &\to \Dir_k(k) \times \Dir_k(k) \\
  \varsigma(V) &= (\varsigma^1(V) + \varsigma^2(V), \varsigma^1(V) \cap \varsigma^2(V))
\end{align*}
be the 2-inflator of Example~\ref{gerald}, where
$\varsigma^i$ is the 1-inflator from the $i$th valuation.

In Example~\ref{gerald-ring}, we saw that the fundamental ring and
ideal are exactly
\begin{align*}
  R &= \{x \in \Oo_1 \cap \Oo_2 : \res_1(a) = \res_2(a)\} \\
  I &= \{x \in \Oo_1 \cap \Oo_2 : \res_1(a) = \res_2(a) = 0\}.
\end{align*}
The ring $R$ is \emph{not} a multi-valuation ring on $K$.  Indeed, if
we choose $a \in K$ such that $\res_1(a) = 1$ and $\res_2(a) = 2$,
then one can check directly that none of the elements $a, 1/(a-q)$ are
in $R$, for any $q \in \Qq$.  So the criterion of
Lemma~\ref{algebra-case} fails to hold.

Fix such an $a$ and mutate along the line $\Theta_a = K \cdot (1,a)
\subseteq K^2$.  This yields a new mutation $\varsigma' : \Dir_K(K) \to
\Dir_k(B_1) \times \Dir_k(B_2)$, where
\begin{equation*}
  (B_1,B_2) = \varsigma(\Theta_a)) = (k^2,0).
\end{equation*}
Thus $B_2 = 0$, and what we really have is a map
\begin{equation*}
  \varsigma' : \Dir_K(K) \to \Dir_k(k^2).
\end{equation*}
The map is given by
\begin{equation*}
  V \mapsto \varsigma(V \otimes \Theta_a) = \varsigma^1(V \otimes
  \Theta_a) + \varsigma^2(V \otimes \Theta_a).
\end{equation*}
since the second component of $\varsigma(V \otimes \Theta_a)$
vanishes.  Now for $j = 1, 2$,
\begin{align*}
  \varsigma^j(V \otimes \Theta_a) &=
  \varsigma^j(\{(x_1,ax_1,x_2,ax_2,\ldots,x_n,ax_n) : (x_1,\ldots,x_n)
  \in V\}) \\
  &= \{(y_1,jy_1,y_2,jy_2,\ldots,y_n,jy_n) : (y_1,\ldots,y_n) \in \varsigma^j(V)\}
\end{align*}
since $a$ specializes to $j$ under $\varsigma^j$.  Thus,
\begin{equation*}
  \varsigma'(V) =
  \{(y_1+z_1,y_1+2z_1,y_2+z_2,y_2+2z_2,\ldots,y_n+z_n,y_n+2z_n) :
  \vec{y} \in \varsigma^1(V), \vec{z} \in \varsigma^2(V)\}.
\end{equation*}
After changing coordinates on $k^2$, this is just
\begin{equation*}
  V \mapsto \{(y_1,z_1,y_2,z_2,\ldots,y_n,z_n) : \vec{y} \in
  \varsigma^1(V), \vec{z} \in \varsigma^2(V)\}.
\end{equation*}
This is the composition
\begin{equation*}
  \Dir_K(K) \to \Dir_k(k) \times \Dir_k(k) \to \Dir_k(k^2)
\end{equation*}
where the first map is the product of $\varsigma^1$ and $\varsigma^2$,
and the second map is the map
\begin{equation*}
  (V,W) \mapsto V \oplus W
\end{equation*}
from Remark~\ref{rem:whatever}.

In particular, the fundamental ring and ideal are $\Oo_1 \cap \Oo_2$
and $\mm_1 \cap \mm_2$.  So again we see that mutation has more or
less undone the corrupting influence of the map
\begin{equation*}
  (V,W) \to (V + W, V \cap W).
\end{equation*}

\subsection{Endless mutation}\label{sec:ivan}
We give an example (promised in \cite{prdf}, Speculative
Remark~10.10.5) of an inflator which fails to have multi-valuation
type after any finite amount of mutation.

Let $\Gamma$ be the ordered abelian group $\Zz[1/3]$ (the ring
generated by $1/3$, with the order coming from the embedding into
$\Qq$ or $\Rr$).  Note that $(\Gamma,\le,+)$ is dp-minimal, and
$\Gamma/2\Gamma \cong \Zz/2\Zz$.

Let $K$ be the Hahn field $\mathbb{C}((t^\Gamma))$.  It is dp-minimal
as a pure field (by Theorem~1.1 in \cite{arxiv-myself}).  Moreover,
the valuation is definable from the pure field structure.\footnote{If
  $x \in K$ then $\sqrt{x}$ exists if and only if $\val(x) \in
  2\Gamma$.  (This can be seen using henselianity and the fact that
  the residue field $\mathbb{C}$ is algebraically closed.)  For any
  non-square $a \in K$, let $B_a$ be the definable set of $x \in K$
  such that $\sqrt{1 + x^2/a}$ exists.  Then $B_a = \{x \in K : \val(x)
  > \val(a)/2\}$.  The ``stabilizer'' $\{y \in K : y \cdot B_a
  \subseteq B_a\}$ is exactly the valuation ring, independent of the
  choice of $a$.}

Let $F$ be the subfield $\mathbb{C}((t^{2\Gamma}))$ inside $K$.  Then
$[K : F] = 2$, and $F$ is also dp-minimal.  The expansion of $K$ by a
predicate for $F$ has dp-rank 2, because of the bi-interpretation with
the pure field $F$.

The set $\{1,t\}$ is an $F$-linear basis of $K$.  So $K$ is an
internal direct sum of $F$ and $F \cdot t$.  Let $g, h : K \to K$ be
the projections onto $F$ and $F \cdot t$.  Then for any Hahn series $x
\in K$,
\begin{itemize}
\item $g(x)$ is the ``even component'' of $x$---the terms with even exponents
\item $h(x)$ is the ``odd component'' of $x$---the terms with odd
  exponents.
\end{itemize}
The functions $g$ and $h$ are definable in the rank-2 structure
$(K,F)$.  Note that
\begin{align*}
  g(xy) &= g(x)g(y) + h(x)h(y) \\
  h(xy) &= g(x)h(y) + h(x)g(y) \\
  \val(x) &= \min(\val(g(x)),\val(h(x))).
\end{align*}
Let $K_0 = \Qq$.  Let $R$ be the following definable $\Qq$-linear
subspace of $K$:
\begin{equation*}
  R = \{x \in K : \val(g(x)) \ge 0, ~ \val(h(x)) \ge 2\}.
\end{equation*}
Then $R$ is a pedestal in the definable picture $\Delta_\bullet$,
because it is the intersection of the following two incomparable
elements of $\Delta_1$:
\begin{equation*}
  \{x \in K : \val(g(x)) \ge 0\} \cap \{ x \in K : \val(h(x)) \ge 2\}.
\end{equation*}
\begin{claim}\label{r-is-r}
  $R$ is a unital ring, and
  \begin{equation*}
    R = \Stab(R) = \{x \in K : x \cdot R \subseteq R\}.
  \end{equation*}
\end{claim}
\begin{proof}
  If $x, y \in R$ then
  \begin{align*}
    \val(g(xy)) &= \val(g(x)g(y) + h(x)h(y)) \\ & \ge \min(\val(g(x))
    + \val(g(y)), \val(h(x)) + \val(h(y))) \\ & \ge \min(0 + 0, 2 + 2)
    = 0,
  \end{align*}
  and similarly
  \begin{align*}
    \val(h(xy)) &= \val(g(x)h(y) + h(x)g(y)) \\ & \ge \min(\val(g(x))
    + \val(h(y)), \val(h(x)) + \val(g(y))) \\ & \ge \min(0 + 2, 2 + 0)
    = 2.
  \end{align*}
  Thus $x, y \in R$, proving that $R$ is a ring.  Also $1 \in R$
  because $\val(g(1)) = 0 \ge 0$ and $\val(h(1)) = +\infty > 2$.  So
  $R$ is a unital ring.  Then $R \subseteq \Stab(R)$.  On the other
  hand, if $a \in \Stab(R)$ then $1 \in R \implies a \cdot 1 \in R$,
  so $\Stab(R) \subseteq R$.
\end{proof}
Let $\varsigma$ be the inflator associated to $R$ by the pedestal machine
(Theorem~\ref{thm:actually-important}), using the definable picture
$K\Vect^f \to \mathcal{D}$, rather than the usual type-definable
picture $K\Vect^f \to \mathcal{H}$.  By Proposition~\ref{prop:r-i}
and Claim~\ref{r-is-r}, the fundamental ring is $R$.
\begin{lemma}\label{lem:r-to-o}
  The integral closure of $R$ is the valuation ring $\Oo := \{x \in K
  : \val(x) \ge 0\}$.
\end{lemma}
\begin{proof}
  We first observe that $\Frac(R) = K$, so there is no ambiguity in
  the integral closure.  Indeed, given any element $x \in K^\times$,
  choose $y \in \Oo$ such that $\val(y) > 2$ and $\val(yx) > 2$.  Then
  $y$ and $yx$ are both in $R$, because of the general identity
  \begin{equation}
    \label{even-odd-split}
    \val(z) = \min(\val(g(z)),\val(h(z)).
  \end{equation}
  Thus $x \in \Frac(R)$.
  
  Let $R'$ be the integral closure of $R$.  As $\Oo$ is integrally
  closed and contains $R$, we must have $\Oo \supseteq R'$.
  Conversely, suppose $x \in \Oo$.  We claim $x \in R'$.  By the
  identity (\ref{even-odd-split}), both $g(x)$ and $h(x)$ are in $\Oo$.
  Now $g(x) \in F$ and $h$ vanishes on $F$, so $h(g(x)) = 0$.
  Then \[g(x) \in \Oo \implies g(x) \in R \implies g(x) \in R'.\] On
  the other hand, $h(x)^2$ is also in $F$, so $h(h(x)^2) = 0$ and
  similarly
  \begin{equation*}
    h(x) \in \Oo \implies h(x)^2 \in \Oo \implies h(x)^2 \in R
    \implies h(x) \in R'.
  \end{equation*}
  Then the fact that both $g(x)$ and $h(x)$ are in $R'$ implies that
  their sum $g(x) + h(x) = x$ is in $R'$.  This shows $\Oo \subseteq
  R'$, so $\Oo = R'$.
\end{proof}
\begin{proposition}\label{prop:endless}
  If $L$ is any line in $K^m$, then the mutation of $\varsigma$
  along $L$ fails to be of multi-valuation type.
\end{proposition}
\begin{proof}
  Without loss of generality $L = K \cdot (1,a_2,a_3,\ldots,a_m)$,
  where the $a_i$ are in $\Oo$.  Let $\varsigma'$ be the mutation of
  $\varsigma$ along $L$, and let $R'$ be the fundamental ring.  We will
  show that $R'$ is not a multi-valuation ring.  By
  Proposition~\ref{prop:growth}, $R' \supseteq R$.  Assume for the
  sake of contradiction that $R'$ is a multi-valuation ring.  Then
  $R'$ is integrally closed, hence contains the valuation ring $\Oo$,
  by Lemma~\ref{lem:r-to-o}.

  By Proposition~\ref{prop:still}, $\varsigma'$ is the 2-inflator obtained
  from the pedestal
  \begin{equation*}
    A' = R \cap a_2^{-1}R \cap \cdots \cap a_m^{-1}R.
  \end{equation*}
  By Proposition~\ref{prop:r-i},
  \begin{equation*}
    R' = \{x \in K : xA' \subseteq A'\}.
  \end{equation*}
  By the assumption that $R'$ is integrally closed, it follows that
  $R' \supseteq \Oo$ and so
  \begin{equation}
    \Oo \cdot A' \subseteq A' \label{yeah-right}
  \end{equation}
  Note that $x \in K$ belongs to $A'$ if and only if all of the
  following are true:
  \begin{enumerate}
  \item $\val(g(x)) \ge 0$
  \item \label{arr2} $\val(h(x)) \ge 2$
  \item \label{irr3} For every $2 \le i \le m$,
    \begin{equation*}
      \val(g(a_ix)) \equiv \val(g(a_i)g(x) + h(a_i)h(x)) \stackrel{?}{\ge} 0
    \end{equation*}
  \item \label{arr4} For every $2 \le i \le m$,
    \begin{equation*}
      \val(h(a_ix)) \equiv \val(g(a_i)h(x) + h(a_i)g(x)) \stackrel{?}{\ge} 2.
    \end{equation*}
  \end{enumerate}
  We may drop (\ref{irr3}) since it follows from the others---recall
  that the $a_i$ are in $\Oo$, so $\val(g(a_i)) \ge 0$ and
  $\val(h(a_i)) \ge 0$.  Assuming (\ref{arr2}), the term $g(a_i)h(x)$
  in (\ref{arr4}) always has valuation at least 2, hence is
  irrelevant.  So we see that $x \in A'$ if and only if the following
  conditions hold
  \begin{enumerate}
  \item $\val(g(x)) \ge 0$
  \item $\val(h(x)) \ge 2$.
  \item For every $2 \le i \le m$,
    \begin{equation*}
      \val(h(a_i)g(x)) \ge 2.
    \end{equation*}
  \end{enumerate}
  Thus
  \begin{equation*}
    A' = \{x \in K : \val(g(x)) \ge \gamma \textrm{ and } \val(h(x)) \ge 2\},
  \end{equation*}
  where $\gamma$ is the maximum of the set
  \begin{equation*}
    \{0\} \cup \{2 - \val(h(a_i)) : 2 \le i \le m\}.
  \end{equation*}
  Note that $\val(h(a_i)) > 0$, because $a_i \in \Oo$ and
  $\val(h(a_i))$ is odd.  Thus $\gamma < 2$.

  Because the value group $\Gamma$ is dense, we may find $x \in F$ and
  $y \in tF$ such that
  \begin{equation*}
    \gamma < \val(x) < \val(y) < 2.
  \end{equation*}
  Then $x \in A'$, $y \notin A'$ and $y/x \in \Oo$, contradicting
  (\ref{yeah-right}).
\end{proof}
Proposition~\ref{prop:endless} says that no amount of mutation
will convert $\varsigma$ into a 2-inflator of
multi-valuation type.

On the other hand, $\varsigma$ is \emph{already} of weak
multi-valuation type, because for $x \in K$,
\begin{equation*}
  \val(x) > 2 \iff \min(\val(g(x)),\val(h(x))) > 2 \implies x \in R,
\end{equation*}
and so $R$ contains a valuation ideal.

\begin{remark}
  I believe the 2-inflator of Example~\ref{harriet} also
  has ``endless mutation,'' in the sense that no mutation is of
  multi-valuation type.  The problem is that any line $L$ in $K^m$
  only involves finitely many $\sqrt{p}$, and therefore fails to undo
  the twisting by $\tau_i$ for $|i| \gg 0$.  Unlike the example of
  this section, Example~\ref{harriet} fails to be malleable.
\end{remark}

\appendix
\part{Appendices}\label{part:app}
\section{Review of abelian categories}\label{app:ab}
For a textbook account of abelian categories, see Chapter 8 of
\cite{cat-sheaves}.

A \emph{pre-additive category} is a category with the extra structure
of an abelian group on each hom-set $\Hom(A,B)$, such that for any objects $A, B, C$ the
composition operation
\begin{equation*}
  \Hom(A,B) \times \Hom(B,C) \to \Hom(A,C)
\end{equation*}
is bilinear.  If $K$ is a field, a $K$-linear pre-additive category is
defined the same way, using $K$-vector spaces rather than abelian
groups.  Note that a $K$-linear pre-additive category is a
pre-additive category with extra structure, and a pre-additive
category is a category with extra structure.

Let $\Cc$ be a pre-additive category.  If $A, B$ are two objects in
$\Cc$, there is a correspondence between the following three types of
data:
\begin{enumerate}
\item Product diagrams
  
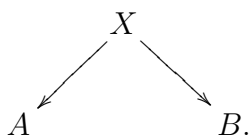
\begin{equation*}
    \xymatrix{ & X \ar[dl] \ar[dr] & \\ A & & B.}
  \end{equation*}
\item Co-product diagrams
  
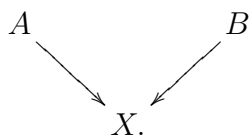
\begin{equation*}
    \xymatrix{ A \ar[dr] & & B \ar[dl] \\ & X. & }
  \end{equation*}
\item \label{bip} Diagrams of the form
  
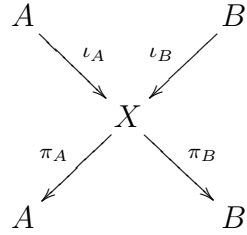
\begin{equation*}
    \xymatrix{A \ar[dr]^{\iota_A} & & B \ar[dl]_{\iota_B} \\ & X
      \ar[dl]_{\pi_A} \ar[dr]^{\pi_B} & \\ A & & B}
  \end{equation*}
  such that
  \begin{align*}
    \pi_A \circ \iota_A &= id_A \\
    \pi_A \circ \iota_B &= 0 \in \Hom(B,A) \\
    \pi_B \circ \iota_A &= 0 \in \Hom(A,B) \\
    \pi_B \circ \iota_B &= id_B.
  \end{align*}
\end{enumerate}
Thus binary products are equivalent to binary coproducts.  The
configurations of (\ref{bip}) are sometimes called ``biproducts.''

More generally, $n$-ary products are equivalent to $n$-ary coproducts
for any finite $n \ge 0$.  A pre-additive category $\Cc$ is
\emph{additive} if all finite products/coproducts exist.  The
existence of 0-ary and 2-ary products/coproducts is
sufficient.\footnote{For $n = 0$, a nullary product is a terminal
  object, and a nullary coproduct is an initial object.}

If $\Cc$ is an additive category, the pre-additive category structure
is determined by the underlying pure category structure
(\cite{cat-sheaves}, Theorem~8.2.14).  The analogue for $K$-linear
additive categories fails.

If $f : A \to B$ is a map in a pre-additive category, a \emph{kernel}
of $f$ is an equalizer of $f$ and the zero morphism $0 \in \Hom(A,B)$.
Cokernels are defined similarly.  An equalizer of two parallel maps
$f, g : A \to B$ is equivalent to a kernel of $f - g$.  Thus, a
pre-additive category has all equalizers if and only if it has all
kernels.  The same holds for coequalizers and cokernels.

A pre-additive category $\Cc$ is \emph{pre-abelian} if it has all
finite limits.  Equivalently, $\Cc$ is preabelian if $\Cc$ is additive
and has all kernels and cokernels.

If $\Cc$ is pre-abelian, one defines the \emph{image} of any morphism
$f : A \to B$ to be the kernel of the cokernel of $f$.  The
\emph{coimage} is defined dually as the cokernel of the kernel.  There
is a natural factorization of $f$ into
\begin{equation*}
  A \to \coim(f) \to \img(f) \to B.
\end{equation*}
One says that $\Cc$ is \emph{abelian} if the natural map $\coim(f) \to
\img(f)$ is an isomorphism for every $f$.

For $\Cc$ pre-abelian, it turns out that $\Cc$ is abelian if and only
if every monomorphism is a kernel and every epimorphism is a cokernel.

The standard example of an abelian category is the category $R\Mod$ of
(left) $R$-modules for any ring $R$, possibly noncommutative.  For $R
= \Zz$, this is the category $\Ab$ of abelian groups.  For $R = K$,
this is the category $K\Vect$ of $K$-vector spaces.

The category of free $\Zz$-modules of finite rank is an example of a
pre-abelian category that fails to be abelian.  Kernels exist because
submodules of free $\Zz$-modules are free.  Cokernels exist because
the category is self-dual.  On the other hand, the coimage/image
factorization of the inclusion $2\Zz \hookrightarrow \Zz$ turns out to
be
\begin{equation*}
  2\Zz \stackrel{=}{\longrightarrow} 2\Zz \stackrel{\not\cong}{\hookrightarrow} \Zz
  \stackrel{=}{\longrightarrow} \Zz.
\end{equation*}

\subsection{Additive and exact functors}
If $\Cc, \Dd$ are two pre-additive categories, a functor $F : \Cc \to
\Dd$ is \emph{additive} if the induced map on Hom-sets
\begin{equation*}
  \Hom_\Cc(A,B) \to \Hom_\Dd(F(A),F(B))
\end{equation*}
is a group homomorphism.  If $\Cc, \Dd$ are $K$-linear pre-additive
categories, one can analogously say that $F$ is a $K$-linear functor
if the map on Hom-sets is $K$-linear.

If $F : \Cc \to \Dd$ is an additive functor between \emph{additive}
categories, then $F$ preserves finite products/coproducts.  The
analogue for $K$-linear functors holds a fortiori.  Conversely, if $F
: \Cc \to \Dd$ is a product-preserving functor between two additive
categories, then $F$ is additive.  The analogue for $K$-linear
functors fails, since being $K$-linear is a strictly stronger
condition than being $\Zz$-linear.

Let $F : \Cc \to \Dd$ be an additive functor between abelian
categories.  One says that $F$ is \emph{left exact} if it satisfies
the following equivalent conditions:
\begin{itemize}
\item $F$ preserves kernels
\item $F$ preserves finite limits
\end{itemize}
Dually, one says that $F$ is \emph{right exact} if it satisfies the
following equivalent conditions:
\begin{itemize}
\item $F$ preserves cokernels
\item $F$ preserves finite colimits 
\end{itemize}
Finally, one says that $F$ is \emph{exact} if it satisfies the
following equivalent conditions:
\begin{itemize}
\item $F$ is left exact and right exact
\item $F$ preserves exact sequences
\item $F$ preserves short exact sequences.
\end{itemize}
If $\Cc, \Dd$ are general categories, a functor $F : \mathcal{C} \to
\mathcal{D}$ is said to be
\begin{itemize}
\item \emph{faithful} if for any $A, B \in \mathcal{C}$, the map
  \begin{equation}
    \Hom_\Cc(A,B) \to \Hom_{\mathcal{D}}(F(A),F(B)) \label{apply-f}
  \end{equation}
  is injective.
\item \emph{fully faithful} if for any $A, B \in \mathcal{C}$, the map
  (\ref{apply-f}) is bijective.
\item \emph{conservative} if for any $f : A \to B$ in $\mathcal{C}$,
  $f$ is an isomorphism if and only if $F(f)$ is an isomorphism.
\end{itemize}
If $\mathcal{C}, \mathcal{D}$ are abelian categories, and $F$ is an
exact functor, then the following are equivalent, by
(\cite{cat-sheaves}, Exercise~8.25).
\begin{itemize}
\item $F$ is faithful.
\item For every $A \in F$, $A = 0$ if and only if $F(A) = 0$.
\item $F$ is conservative.
\end{itemize}
\begin{remark} \label{out-of-finvec}
  Let $K\Vect^f$ be the category of finite-dimensional $K$-vector
  spaces.  If $F : K\Vect^f \to \mathcal{D}$ is an additive
  functor, then
  \begin{enumerate}
  \item $F$ is automatically exact, because every short exact sequence
    in $K\Vect^f$ splits.
  \item If $F(K) = 0$, then $F(X) = 0$ for all $X \in K\Vect^f$,
    because $X \cong K^{\dim(X)}$.
  \item If $F(K) \ne 0$, then $F$ is conservative and faithful.
  \end{enumerate}
\end{remark}
\begin{fact}[Mitchell embedding theorem]
  Every $K$-linear abelian category admits a fully faithful exact
  $K$-linear functor to a category $R\Mod$ of $R$-modules, for some
  $K$-algebra $R$.
\end{fact}
See Theorem~9.6.10 in \cite{cat-sheaves} for a proof.

\subsection{Subobjects}
If $A$ is an object in an arbitrary category $\Cc$, a ``subobject'' of
$A$ is an equivalence class of monomorphisms $X \hookrightarrow A$.
There is a partial order on subobjects, induced by the
``factors-through'' preorder on monomorphisms.  In the concrete
setting of $R\Mod$, subobjects correspond exactly to $R$-submodules.

In an abelian category, the subobject poset $\Sub(A)$ is always a
bounded modular lattice, with the bounded lattice operations given as
follows:
\begin{itemize}
\item The intersection of $X \hookrightarrow A$ and $Y \hookrightarrow
  A$ is the pullback $X \times_A Y \hookrightarrow A$.
\item The join of $X \hookrightarrow A$ and $Y \hookrightarrow A$ is
  the image of $X \oplus Y \to A$.
\item The bottom element is the zero monomorphism $0 \hookrightarrow A$.
\item The top element is the identity $id_A : A \to A$.
\end{itemize}
The fact that these operations satisfy the axioms of bounded modular
lattices can be checked using the Mitchell embedding theorem to reduce
to the case of $R\Mod$.

Because the subobject lattice $\Sub_\Cc(A)$ is modular, it makes sense
to talk about the length of $A$, and to talk about $A$ having
``finite-length.''

Also, if $F : \Cc \to R\Mod$ is a Mitchell embedding, or even a
faithful exact functor, then there is an embedding of bounded lattices
\begin{equation}\label{too-bad}
  \Sub_\Cc(A) \hookrightarrow \Sub_R(F(A))
\end{equation}
for any $A \in \Cc$.  Thus, one can view a subobject of $A$ as an
$R$-submodule of $F(A)$.  In general, the map (\ref{too-bad}) is not
surjective, so not all $R$-submodules of $F(A)$ correspond to
subobjects of $A$.\footnote{If the map were always surjective, it would
  always be an isomorphism of lattices.  This would force
  $\Sub_\Cc(A)$ to be a complete lattice.  However, this cannot hold
  in general, essentially because the class of abelian categories is
  an elementary class, and the class of complete lattices is not.  For
  example, in the category of finite abelian groups, for
  each $n$ we can find an object $A_n$ (namely $\Zz/2^{n-1}\Zz$) whose
  subobject lattice is the totally ordered set of size $n$.  In an
  elementary extension, we can find an object $A$ for which $\Sub(A)$
  is a pseudo-finite total order.  Pseudo-finite total orders are
  \emph{never} complete lattices.}

\subsection{Serre quotients}\label{serre-quot}
Let $\Cc$ be an abelian category.  A \emph{Serre subcategory} is a
full subcategory $\mathcal{S} \subseteq \Cc$ containing 0, such that
for any short exact sequence,
\begin{equation*}
  0 \to X \to Y \to Z \to 0
\end{equation*}
in $\Cc$, the following holds:
\begin{equation*}
  Y \in \mathcal{S} \iff (X \in \mathcal{S} \textrm{ and } Y \in \mathcal{S}).
\end{equation*}
This implies:
\begin{itemize}
\item If $X, Y \in \mathcal{S}$, then $X \oplus Y \in \mathcal{S}$
\item If $X$ is a subobject or quotient object of $Y$, then $Y \in
  \mathcal{S} \implies X \in \mathcal{S}$.
\item If $X \cong Y$, then $X \in \mathcal{S} \iff Y \in \mathcal{S}$.
\end{itemize}
If $\mathcal{S}$ is a Serre subcategory, let $W_{\mathcal{S}}$ be the
set of morphisms $f : X \to Y$ in $\Cc$ such that the kernel and
cokernel are both in $\mathcal{S}$.  The class $W_{\mathcal{S}}$ has a
calculus of left fractions and right fractions (\cite{cat-sheaves},
Exercise~8.12).   The \emph{Serre quotient} is the localization
\begin{equation*}
  \Cc/\mathcal{S} := \Cc[W_{\mathcal{S}}^{-1}].
\end{equation*}
See \S7.1 of \cite{cat-sheaves} for information on localization.
(There, a calculus of fractions is called a multiplicative system.)

By Exercises~8.11-8.12 in \cite{cat-sheaves}, the Serre quotient
$\Cc/\mathcal{S}$ is itself an abelian category, and the localization
functor
\begin{equation*}
  \Cc \to \Cc/\mathcal{S}
\end{equation*}
is exact.  If $X, Y \in \mathcal{C}$, then the morphisms $X \to Y$ in
$\Cc/\mathcal{S}$ can be represented by diagrams
\begin{equation*}
  X \longleftarrow X' \longrightarrow Y
\end{equation*}
where $X' \to X$ is in $W_{\mathcal{S}}$, or alternatively by diagrams
\begin{equation*}
  X \longrightarrow Y' \longleftarrow Y,
\end{equation*}
where $Y \to Y'$ is in $W_{\mathcal{S}}$.
\begin{lemma}\label{serro}
  If $A$ is an object in $\Cc$, then $A \cong 0$ holds in
  $\Cc/\mathcal{S}$ if and only if $A \in \mathcal{S}$.
\end{lemma}
\begin{proof}
  The following statements are equivalent:
  \begin{enumerate}
  \item \label{ser1} $A \cong 0$ in $\Cc/\mathcal{S}$
  \item \label{ser2} There is an object $B \in \Cc$ such that the zero
    morphism $0_{A,B} : A \to B$ is in $W_{\mathcal{S}}$.
  \item \label{ser3} There is an object $B \in \Cc$ such that $A, B
    \in \mathcal{S}$.
  \item \label{ser4} $A \in \mathcal{S}$.
  \end{enumerate}
  The equivalence (\ref{ser1})$\iff$(\ref{ser2}) holds by the
  calculus of fractions; see Exercise~7.3 in \cite{cat-sheaves}.
  The equivalence (\ref{ser2})$\iff$(\ref{ser3}) is by definition
  of $W_{\mathcal{S}}$.  The equivalence
  (\ref{ser3})$\iff$(\ref{ser4}) is trivial.
\end{proof}
\begin{lemma}\label{serre-lift}
  Let $A, B$ be objects of $\Cc$ and $f : A \to B$ be a morphism in
  $\Cc/\mathcal{S}$.
  \begin{enumerate}
  \item There is an isomorphism $B \cong B'$ in $\Cc/\mathcal{S}$ such
    that $A \to B'$ lifts to a morphism in $\Cc$.
  \item If $f$ is a monomorphism in $\Cc/\mathcal{S}$, there is an
    isomorphism $A' \cong A$ in $\Cc/\mathcal{S}$ such that $A' \to B$
    lifts to a monomorphism in $\Cc$.
  \end{enumerate}
  The dual statements hold as well.
\end{lemma}
\begin{proof}
  The first point holds by the calculus of fractions.  Now suppose $f$
  is a monomorphism.  Then $f$ is the kernel of some morphism $g : B
  \to C$ in $\Cc/\mathcal{S}$, because the Serre quotient
  $\Cc/\mathcal{S}$ is an abelian category.  By the first point, we
  can change the object $C$ by an isomorphism and arrange $g$ to lift
  to a morphism $\tilde{g}$ in the category $\Cc$.  Take an exact sequence
  \begin{equation*}
    0 \to A' \to B \stackrel{\tilde{g}}{\to} C
  \end{equation*}
  in the category $\Cc$.  By exactness of the localization functor
  $\Cc \to \Cc/\mathcal{S}$, this induces an exact sequence in
  $\Cc/\mathcal{S}$.  In the category $\Cc/\mathcal{S}$, the two
  monomorphisms $A \to B$ and $A' \to B$ are both kernels of $g : B
  \to C$, so they must be isomorphic to each other.
\end{proof}
\begin{proposition} \label{oh-that-2}
  Let $\Cc$ be an abelian category, $\mathcal{S}$ be a Serre
  subcategory, and $A$ be an object in $\Cc$.  The exact localization
  functor $\Cc \to \Cc/\mathcal{S}$ induces a lattice homomorphism
  \begin{equation*}
    \Sub_\Cc(A) \to \Sub_{\Cc/\mathcal{S}}(A).
  \end{equation*}
  This lattice homomorphism is onto.  Two subobjects $X_1, X_2 \in
  \Sub_\Cc(A)$ map to the same subobject in $\Cc/\mathcal{S}$ if and
  only if
  \begin{align*}
    X_1/(X_1 \cap X_2) &\in \mathcal{S} \\
    X_2/(X_1 \cap X_2) &\in \mathcal{S}.
  \end{align*}
\end{proposition}
\begin{proof}
  The map is onto by Lemma~\ref{serre-lift}.  Since the localization
  functor is exact, it preserves constructs like $X_1/(X_1 \cap X_2)$.
  Therefore, $X_1$ and $X_2$ map to the same subobject in
  $\Cc/\mathcal{S}$ if and only if the identities
  \begin{align*}
    X_1/(X_1 \cap X_2) &\cong 0 \\
    X_2/(X_1 \cap X_2) &\cong 0
  \end{align*}
  hold in the Serre quotient $\Cc/\mathcal{S}$.  By Lemma~\ref{serro},
  this is equivalent to the stated conditions on $X_1, X_2$.
\end{proof}
This gives another way to think about the Serre quotient.  The Serre
quotient $\Cc/\mathcal{S}$ has the same underlying objects as $\Cc$,
but the morphisms from $A$ to $B$ correspond to subobjects $\Gamma
\subseteq A \oplus B$ such that
\begin{align*}
  \Gamma + (0 \oplus B) &\approx A \oplus B \\
  \Gamma \cap (0 \oplus B) &\approx 0 \oplus 0,
\end{align*}
where $\approx$ denotes commensurability modulo $\mathcal{S}$.

As an example, let $M$ be some model-theoretic structure and let $\Cc$
be the category of interpretable abelian groups in the structure $M$.
Then $\Cc$ is an abelian category.  Let $\mathcal{S}$ be the
subcategory of finite interpretable abelian groups.  Then
$\mathcal{S}$ is a Serre subcategory, because in a short exact
sequence
\begin{equation*}
  0 \to X \to Y \to Z \to 0,
\end{equation*}
$Z$ is finite if and only if $X$ and $Y$ are both finite.  Therefore
we can form the Serre quotient.  If $X \in \Cc$, then the subobject
lattice of $X$ in $\Cc/\mathcal{S}$ is the lattice of definable
subgroups, modulo commensurability.

If $M$ is a structure in which $G^0$ always exists, then we can give a
more explicit description of the Serre quotient $\Cc/\mathcal{S}$:
\begin{itemize}
\item Objects are interpretable abelian groups $(G,+)$ such that $G =
  G^0$.
\item The subobject lattice of $G$ is the lattice of connected
  subgroups.
\item A morphism from $G$ to $H$ is a connected subgroup $\Gamma
  \subseteq G \oplus H$ such that $\Gamma + (0 \oplus H) = G \oplus H$
  (i.e., $\Gamma$ projects onto $G$) and such that $(\Gamma \cap (0
  \oplus H))^0 = 0 \oplus 0$, (i.e., $\Gamma \cap (0 \oplus H)$ is
  finite).
\end{itemize}
So a morphism from $G$ to $H$ is an ``endogeny'' or
``quasi-homomorphism'' from $G$ to $H$, in the sense of
(\cite{stab-groups}, \S1.5) or (\cite{frank-jan}, \S3).  The category
$\Cc/\mathcal{S}$ is similar to the isogeny category of abelian
varieties.

\subsection{A criterion for recognizing abelian categories}

\begin{lemma}\label{lem:reflector}
  Let $\Cc, \Dd$ be additive categories, with $\Dd$
  abelian.  Let $F : \Cc \to \Dd$ be a faithful additive functor.
  Suppose the following conditions hold:
  \begin{enumerate}
  \item \label{ker-basically} For every $f : X \to Y$ in $\Cc$, there
    is $e : W \to X$ in $\Cc$ such that the diagram
    \begin{equation*}
      0 \to F(W) \to F(X) \to F(Y)
    \end{equation*}
    is exact in $\Dd$.
  \item For every $f : X \to Y$ in $\Cc$, there is $g : X \to Z$ in
    $\Cc$ such that the diagram
    \begin{equation*}
      F(X) \to F(Y) \to F(Z) \to 0
    \end{equation*}
    is exact in $\Dd$.
  \item \label{ker-other} Let $f : X \to Y$ and $g : X' \to Y$ be
    morphisms in $\Cc$.  Suppose $F(g) : F(X') \to F(Y)$ is a
    monomorphism in $\Dd$ and $F(f) : F(X) \to F(Y)$ factors through
    $F(g)$.  Then $f$ factors through $g$.
  \item \label{coker-other} Let $f : X \to Y$ and $g : X \to Y'$ be
    morphisms in $\Cc$.  Suppose $F(g) : F(X) \to F(Y')$ is an
    epimorphism in $\Dd$ and $F(f) : F(X) \to F(Y)$ factors through
    $F(g)$.  Then $f$ factors through $g$.
  \end{enumerate}
  Then $\Cc$ is abelian and $F$ is an exact functor.
\end{lemma}
\begin{proof}
  Because $F$ is faithful, $F$ reflects monomorphisms.  Indeed,
  suppose $f : X \to Y$ is a morphism in $\Cc$, $F(f)$ is a
  monomorphism, and $g_1, g_2 : W \to X$ are parallel morphisms such
  that $f \circ g_1 = f \circ g_2$.  Then $F(f) \circ F(g_1) = F(f)
  \circ F(g_2)$.  As $F(f)$ is a monomorphism, $F(g_1) = F(g_2)$.
  Then $g_1 = g_2$ because $F$ is faithful.  A similar argument shows
  that $F$ reflects epimorphisms.

  We claim that $\Cc$ has kernels and $F$ preserves them.  Let $f : X
  \to Y$ be a morphism, and let $e : W \to X$ be as in
  (\ref{ker-basically}).  Then $F(f) \circ F(e) = 0$, so $f \circ e =
  0$ by faithfulness of $F$.  We claim that $e$ is the kernel of $f$.
  Let $e' : W' \to X$ be some morphism in $\Cc$ such that $f \circ e'
  = 0$.  Then $F(f) \circ F(e') = 0$.  As $F(e)$ is the kernel of
  $F(f)$, we see that $F(e') : F(W') \to F(X)$ factors through the
  monomorphism $F(e) : F(W) \hookrightarrow F(X)$.  By
  (\ref{ker-other}), we see that $e'$ factors through $e$.  Because
  $F$ reflects monomorphisms, $e$ is a monomorphism, and so the
  factorization of $e'$ through $e$ is unique.  This proves that $e$
  is the kernel of $f$.  Therefore $\Cc$ has kernels and $F$ preserves
  them.  By duality, $\Cc$ has cokernels and $F$ preserves them.
  
  Therefore, $\Cc$ is pre-abelian, and if $\Cc$ is abelian, then $F$
  is exact.

  Next we claim that $F$ is conservative (i.e., $F$ reflects
  isomorphisms).  Indeed, suppose $f : X \to Y$ is such that $F(f) :
  F(X) \to F(Y)$ is an isomorphism.  Note that $id_{F(Y)} = F(id_Y)$
  factors through the monomorphism $F(f)$.  By (\ref{ker-other}),
  $id_{Y}$ factors through $f$, i.e.,
  \begin{equation*}
    id_Y = f \circ g
  \end{equation*}
  for some morphism $g : Y \to X$.  Similarly,
  \begin{equation*}
    id_X = h \circ f
  \end{equation*}
  for some morphism $h : Y \to X$.  Then $f$ has a two-sided inverse
  given by
  \begin{equation*}
    h = h \circ id_Y = h \circ f \circ g = id_X \circ g = g.
  \end{equation*}
  So $f$ is an isomorphism.  This shows that $F$ reflects
  isomorphisms.

  Finally, let $f : X \to Y$ be any morphism.  As $\Cc$ is
  pre-abelian, there is the usual coimage/image factorization
  \begin{equation*}
    \xymatrix{ \ker(f) \ar[r] & X \ar[dr] \ar[rrr] & & & Y \ar[r] & \coker(f). \\
       & & \coim(f) \ar[r] & \img(f) \ar[ur] & &}
  \end{equation*}
  Because $F$ preserves kernels and cokernels, the diagram obtained by
  applying $F$ is also a coimage/image factorization.  In particular,
  $F(\coim(f))$ is the coimage and $F(\img(f))$ is the image of $F(f)
  : F(X) \to F(Y)$.  As $\Dd$ is abelian, the morphism $F(\coim(f))
  \to F(\img(f))$ is an isomorphism.  As $F$ is convervative, it
  follows that $\coim(f) \to \img(f)$ is an isomorphism, which is
  exactly what it means for $\Cc$ to be an abelian category.
\end{proof}

\section{Review of ind- and pro-objects}\label{app:ind}
For a textbook account of ind-objects and pro-objects, see Chapter 6
and Section 8.6 of \cite{cat-sheaves}.
\subsection{The category of ind-objects}
Let $\Cc$ be a small category.  The Yoneda lemma gives a fully
faithful embedding of $\Cc$ into the functor category
$\Fun(\Cc^{op},\Set)$.  For $A \in \Cc$, let $h_A$ be the
corresponding representable functor
\begin{equation*}
  h_A(B) := \Hom_\Cc(B,A).
\end{equation*}
The category $\Fun(\Cc^{op},\Set)$ has all small limits and colimits;
they are constructed pointwise.
If $\{A_i\}_{i \in I}$ is a filtered diagram of objects in $A$, one
defines
\begin{equation*}
  \textrm{``}\varinjlim_{i \in I}\textrm{''} A_i := \varinjlim_{i \in I} h_{A_i}.
\end{equation*}
One defines the \emph{indization} $\Ind \Cc$ to be the full subcategory of
$\Fun(\Cc^{op},\Set)$ consisting of the objects of the form
\begin{equation*}
  \textrm{``}\varinjlim_{i \in I}\textrm{''} A_i.
\end{equation*}
Objects of $\Ind \Cc$ are called \emph{ind-objects}.  There is a fully
faithful embedding $\Cc \to \Ind \Cc$ sending $A$ to its representable
functor $h_A$.

From the Yoneda lemma, one gets the usual formula
\begin{equation}
  \Hom_{\Ind\Cc}\left(\textrm{``}\varinjlim_{j \in
    J}\textrm{''} B_j, \textrm{``}\varinjlim_{i \in I}\textrm{''} A_i\right)
  = \varprojlim_{j \in J}\varinjlim_{i \in
    I}\Hom(B_j,A_i). \label{usual-formula}
\end{equation}
So one can alternatively define $\Ind \Cc$ as the category of formal
colimits $\textrm{``}\varinjlim_{i \in I} \textrm{''} A_i$ with
morphisms given by (\ref{usual-formula}).  The downside of this
alternative approach is that it now takes some work to define
composition and verify associativity.

\subsection{The case of posets}

If $P$ is a poset, say that a subset $I \subseteq P$ is an
\emph{ideal} if the following conditions hold:
\begin{itemize}
\item $I$ is downwards-closed, i.e.,
  \begin{equation*}
    \left(x \in I \textrm{ and } y \le x\right) \implies y \in I.
  \end{equation*}
\item $I$ is upwards-directed, i.e.,
  \begin{equation*}
    \forall x, y \in I ~\exists z \in I : x, y \le z.
  \end{equation*}
\item $I$ is non-empty.
\end{itemize}
Any object $x \in P$ determines a principal ideal
\begin{equation*}
  (x) = \{y \in P : y \le x\}.
\end{equation*}
\begin{proposition}\label{poset-case}
  If $P$ is a poset, then $\Ind P$ is (isomorphic to) the poset of
  ideals in $P$, ordered by inclusion.  The map $P \to \Ind P$ is the
  map $x \to (x)$.
\end{proposition}
\begin{proof}
  By formula (\ref{usual-formula}), every Hom-set in $\Ind P$ has at
  most one element, and so $\Ind P$ is equivalent to a poset.  Let
  $\{x_i\}_{i \in I}$ and $\{y_j\}_{j \in J}$ be two filtered diagrams
  in $P$.  These determine ind-objects
  \begin{align*}
    x &= \textrm{``}\varinjlim_{i \in I}\textrm{''} x_i \in \Ind P \\
    y &= \textrm{``}\varinjlim_{j \in J}\textrm{''} y_j \in \Ind P
  \end{align*}
  as well as ideals
  \begin{align*}
    E_x &= \{z \in P ~|~ \exists i \in I : z \le x_i\} \\
    E_y &= \{z \in P ~|~ \exists j \in J : z \le x_j\}.
  \end{align*}
  Now one has equivalences
  \begin{equation*}
    x \le y \iff \left(\forall i \in I ~ \exists j \in J : x_i \le
    y_j\right) \iff E_x \subseteq E_y,
  \end{equation*}
  where the first equivalence is a disguised form of
  (\ref{usual-formula}), and the second equivalence is easy.  Thus
  $\Ind P$ embeds into the poset of ideals via the map $x \mapsto
  E_x$.  But the map is onto, because any ideal $I \subseteq P$
  determines a filtered diagram $\{i\}_{i \in I}$, mapping to $I$
  under the embedding.

  Finally, if $x \in P$, then the ind-object $x$ is represented by the
  filtered diagram $\{x\}_{i \in 1}$, where $1$ is the terminal
  category.  The corresopnding ideal is the principal ideal generated
  by $x$.
\end{proof}

\subsection{The case of abelian categories}
See \S 8.6 of \cite{cat-sheaves} for a textbook account of indization
of abelian categories.

Let $\Cc$ be a small abelian category.  Then
\begin{itemize}
\item $\Ind \Cc$ is abelian.
\item The embedding $\Cc \to \Ind \Cc$ is additive, fully faithful,
  and exact.
\item The category $\Ind \Cc$ has directed limits (i.e., filtered
  colimits), and they are exact.
\end{itemize}
See Theorem~8.6.5 in \cite{cat-sheaves}.  The analogous facts hold for
$k$-linear abelian categories.  In particular, if $\Cc$ is a
$k$-linear abelian category, then $\Ind \Cc$ naturally has the
structure of a $k$-linear category.

We also need some facts about subobject lattices in $\Ind \Cc$.
\begin{lemma}\label{mor}
  Let $Y$ be an object in a small abelian category $\Cc$.  If $\{f_i :
  X_i \hookrightarrow Y\}$ is a filtered diagram of monomorphisms in
  $\Cc$, then
  \begin{equation*}
    \textrm{``}\varinjlim_{i \in I}\textrm{''}X_i \to Y
  \end{equation*}
  is a monomorphism in $\Ind \Cc$.  Up to isomorphism over $Y$, all
  $\Ind\Cc$-monomorphisms into $Y$ arise this way.
\end{lemma}
\begin{proof}
  Because the embedding $\Cc \hookrightarrow \Ind\Cc$ is exact, each
  morphism $X_i \to Y$ is a monomorphism in $\Ind \Cc$.  Filtered
  colimits are exact in $\Ind \Cc$, so the filtered colimit of all
  these monomorphisms is a monomorphism.\footnote{Note that in the
    category $\Ind \Cc$, the formal filtered colimit
    $\textrm{``}\varinjlim_{i \in I}\textrm{''}X_i$ is the actual
    filtered colimit.  This is because $\Ind \Cc$ gets its filtered
    colimits from the larger category $\Fun(\Cc^{op},\Set)$.}

  Now let
  \begin{equation*}
    f : \textrm{``}\varinjlim_{i \in I}\textrm{''}X_i \hookrightarrow Y
  \end{equation*}
  be some arbitrary monomorphism in $\Ind\Cc$.  This map arises as the
  filtered colimit of a filtered diagram $\{f_i : X_i \to Y\}_{i \in
    I}$ of morphisms in $\Cc$, but the $f_i$ are not guaranteed to be
  monic.  Nevertheless, we can form a filtered family of diagrams of
  the form
  \begin{equation*}
    X_i \to \img(f_i) \to Y \to \coker(f_i).
  \end{equation*}
  This induces a diagram
  \begin{equation*}
    \textrm{``}\varinjlim_{i \in I}\textrm{''}X_i \to
    \textrm{``}\varinjlim_{i \in I}\textrm{''}\img(f_i) \to Y \to
    \textrm{``}\varinjlim_{i \in I}\textrm{''}\coker(f_i).
  \end{equation*}
  Because filtered colimits in $\Ind \Cc$ are exact, we see that
  \begin{equation*}
    \textrm{``}\varinjlim_{i \in I}\textrm{''}\coker(f_i) \cong \coker(f),
  \end{equation*}
  and then
  \begin{equation*}
    \textrm{``}\varinjlim_{i \in I}\textrm{''}\img(f_i) \cong \ker(\coker(f)).
  \end{equation*}
  Because $f$ is monic, it follows that $f$ is isomorphic to the monomorphism
  \begin{equation*}
    \textrm{``}\varinjlim_{i \in I}\textrm{''}\img(f_i) \hookrightarrow Y,
  \end{equation*}
  which has the desired form.
\end{proof}
\begin{proposition}
  If $\Cc$ is an abelian category and $A \in \Cc$, then the natural map
  \begin{align*}
    \Ind\Sub_\Cc(A) & \to \Sub_{\Ind \Cc}(A)
    \\ \left(\textrm{``}\varinjlim_{i \in I}\textrm{''}(X_i
    \hookrightarrow A)\right) & \mapsto
    \left(\left(\textrm{``}\varinjlim_{i \in I}\textrm{''} X_i\right)
    \hookrightarrow A\right)
  \end{align*}
  is an isomorphism of lattices.
\end{proposition}
\begin{proof}
  By Lemma~\ref{mor}, the map is onto.  It remains to check that it is
  strictly order-preserving.  Let $\{X_i\}_{i \in I}$ and $\{Y_j\}_{j
    \in J}$ be two directed diagrams of $\Cc$-subobjects of $A$.  Let
  \begin{align*}
    X &= \textrm{``}\varinjlim_{i \in I}\textrm{''} X_i \\
    Y &= \textrm{``}\varinjlim_{j \in J}\textrm{''} Y_j
  \end{align*}
  be the corresponding $\Ind\Cc$-subobjects of $A$.
  \begin{claim}
    For each $i$, $X_i \subseteq Y$ if and only if there is some $j$
    such that $X_i \subseteq Y_j$.
  \end{claim}
  \begin{claimproof}
    There is a natural morphism $Y_j \to Y$.  As $Y_j, Y$ are both
    subobjects of $A$, it follows that $Y_j \subseteq Y$.  If $X_i
    \subseteq Y_j$, then $X_i \subseteq Y$.

    Conversely, suppose $X_i \subseteq Y$.  Then the inclusion $X_i
    \to Y$ factors through some $Y_j \to Y$, by
    Formula~(\ref{usual-formula}).  Then $X_i \subseteq Y_j$.
  \end{claimproof}
  \begin{claim}
    $X \subseteq Y$ if and only if $X_i \subseteq Y$ for all $i$.
  \end{claim}
  \begin{claimproof}
    Note $X_i \subseteq X$ because of the morphism $X_i \to X$ over
    $A$.  Thus $X \subseteq Y$ clearly implies $X_i \subseteq Y$ for
    all $i$.  Conversely, suppose $X_i \subseteq Y$ for all $i$.  For
    any arrow $i \to i'$ in the category $I$, the diagram
    \begin{equation*}
      \xymatrix{ X_i \ar[d] \ar[dr] & \\ X_{i'} \ar[r] & Y}
    \end{equation*}
    commutes, because the same diagram commutes when $Y$ is replaced
    with $A$, and the map $Y \to A$ is a monomorphism.  Therefore the
    maps $X_i \to Y$ assemble into a morphism
    \begin{equation*}
      X = \textrm{``}\varinjlim_{i \in I}\textrm{''} X_i \to Y
    \end{equation*}
    whose composition with $Y \to A$ is the given morphism
    $X \to A$.  Therefore $X \subseteq Y$.
  \end{claimproof}
  By the two claims, we see that
  \begin{equation*}
    X \subseteq Y \iff \left(\forall i ~ \exists j : X_i \subseteq Y_j\right).
  \end{equation*}
  As in the proof of Proposition~\ref{poset-case}, the right hand side
  corresponds to
  \begin{equation*}
    \textrm{``}\varinjlim_{i \in I}\textrm{''}(X_i \hookrightarrow A) \le
        \textrm{``}\varinjlim_{j \in J}\textrm{''}(Y_j \hookrightarrow A)
  \end{equation*}
  in the poset $\Ind\Sub_\Cc(A)$.
\end{proof}

\subsection{Pro-objects}\label{sec:pro}
If $\Cc$ is any category, the category $\Pro\Cc$ of pro-objects is
defined as $(\Ind(\Cc^{op}))^{op}$.  Objects of $\Pro\Cc$ can be
thought of as formal inverse limits
\begin{equation*}
  \textrm{``}\varprojlim_{i \in I}\textrm{''} A_i.
\end{equation*}
If $M$ is a bounded lattice, then $\Pro M$ is dual to the lattice of
filters on $M$, where a \emph{filter} is a subset $F \subseteq M$ such
that
\begin{itemize}
\item $\top \in F$
\item If $x, y \in F$, then $x \wedge y \in F$
\item If $x \in F$ and $y \in M$, then $x \vee y \in F$.
\end{itemize}
If $\Cc$ is an abelian category, then $\Pro \Cc$ is an abelian
category.  The inclusion $\Cc \hookrightarrow \Pro \Cc$ is fully
faithful and exact.  If $A \in \Cc$, then there is a canonical isomorphism
\begin{equation*}
  \Pro \Sub_\Cc(A) \cong \Sub_{\Pro \Cc}(A).
\end{equation*}

\section{Boring proofs}\label{app:dirs}
We give the proofs of several statements from
\S\ref{sec:directories}-\ref{sec:lambda} that sound obvious, but
perhaps aren't.

\subsection{Semisimple directories}\label{app:ssd}
For completeness, we should give a proof of
Proposition~\ref{prop:to-prove}.  Recall Schur's lemma:
\begin{lemma}[Schur]
  Let $A, B$ be simple objects in an abelian category $\Cc$.  Every
  morphism from $A$ to $B$ is either zero or an isomorphism.
  Therefore,
  \begin{itemize}
  \item If $A, B$ are non-isomorphic, $\Hom_\Cc(A,B) = 0$.
  \item $\End_\Cc(A) := \Hom_\Cc(A,A)$ is a division ring.
  \end{itemize}
\end{lemma}
We will also use Morita equivalence---the fact that the category of
$M_n(R)$-modules is equivalent to the category of $R$-modules via an
equivalence sending $M_n(R)$ to $R^n$.

Recall the notion of ``neighborhood'' from \S\ref{sec:neighborhoods}.
\begin{proposition}\label{prop:sureish}
  If $A$ is a semisimple object in an abelian category $\Cc$, then
  there is a ring $R$ such that
  \begin{enumerate}
  \item $R$ is a finite product of division rings $D_1 \times \cdots
    \times D_k$
  \item The neighborhood of $A$ in $\Cc$ is equivalent to the category of
    finitely generated $R$-modules, or equivalently, the neighborhood of
    $R$ in $R\Mod$.
  \end{enumerate}
\end{proposition}
\begin{proof}
  For any noncommutative ring $R$, let $Mat(R)$ denote the category in
  which
  \begin{itemize}
  \item objects are nonnegative integers
  \item morphisms from $n$ to $m$ are $m \times n$ matrices
  \item composition is matrix multiplication.
  \end{itemize}
  Equivalently, $Mat(R)$ is the category of free $R^{op}$-modules of
  finite rank.  More generally, if $A$ is any object in an abelian category,
  then $Mat(\End_\Cc(A)^{op})$ is equivalent to the full category
  $\{0,A,A^2,A^3,\ldots\} \subseteq \Cc$.

  Write $A \cong A_1^{n_1} \oplus \cdots \oplus A_k^{n_k}$ for some
  pairwise non-isomorphic simple $A_i \in \Cc$ and $n_k > 0$.  Let
  $D_i$ be the division ring $\End_\Cc(A_i)^{op}$.  There is a natural
  functor
  \begin{align*}
    Mat(D_1) \times \cdots \times Mat(D_k) & \to \Cc
    \\ (j_1,\ldots,j_k) & \mapsto A_1^{j_1} \oplus \cdots \oplus A_k^{j_k}
  \end{align*}
  This functor is fully faithful by Schur's lemma, and the essential
  image is the neighborhood of $A$.  So the neighborhood of $A$ in $\Cc$ is
  equivalent to the category $Mat(D_1) \times \cdots \times Mat(D_k)$.

  Now let $R = D_1 \times \cdots \times D_k$.  In the category
  $R\Mod$, the object $R$ decomposes as an internal direct sum $D_1
  \oplus \cdots \oplus D_k$, where $\End_R(D_k) \cong D_k^{op}$.
  Replacing $\Cc$ with $R\Mod$ in the above argument, we see that the
  neighborhood of $R$ in $R\Mod$ is equivalent to $Mat(D_1) \times \cdots
  \times Mat(D_k)$.  Thus the neighborhood of $R$ in $R\Mod$ is also
  equivalent to the neighborhood of $A$ in $\Cc$.
\end{proof}
The following variant holds as well:
\begin{proposition}\label{prop:sureish2}
  If $A$ is a semisimple object in an abelian category $\Cc$, and $R
  = \End_\Cc(A)^{op}$, then
  \begin{enumerate}
  \item $R$ is a finite product of matrix algebras over division
    rings.
  \item The neighborhood of $A$ in $\Cc$ is equivalent to the neighborhood of
    $R$ in $R\Mod$, and the equivalence sends $A$ to $R$.
  \end{enumerate}
\end{proposition}
\begin{proof}
  By Proposition~\ref{prop:sureish} there is $R_0 = D_1 \times \cdots
  \times D_k$ such that the neighborhood of $A$ in $\Cc$ is equivalent to
  the neighborhood of $R_0$ in $R_0\Mod$.  Let $M \in R_0\Mod$ be the
  object corresponding to $A$ under this equivalence.  Write $M$ as
  $D_1^{n_1} \times \cdots \times D_k^{n_k}$.  For each $k$, there is
  a Morita equivalence from the category of $M_{n_k}(D_k)$-modules to
  the category of $D_k$-modules.  This equivalence sends
  $M_{n_k}(D_k)$ to $D_k^{n_k}$.  Let $R = M_{n_1}(D_1) \times \cdots
  \times M_{n_k}(D_k)$.  The Morita equivalences assemble to an
  equivalence
  \begin{equation*}
    R\Mod = (M_{n_1}(D_1) \times \cdots \times M_{n_k}(D_k))\Mod
    \stackrel{\sim}{\to} (D_1 \times \cdots \times D_k)\Mod = R_0\Mod
  \end{equation*}
  under which $R$ maps to $D_1^{n_1} \times \cdots \times D_k^{n_k} =
  M$.  Thus the neighborhood of $M$ in $R_0\Mod$ is equivalent to the
  neighborhood of $R$ in $R\Mod$, via an equivalence sending $M$ to $R_0$.

  Chaining things together, there is an equivalence from the neighborhood
  of $A$ in $\Cc$ to the neighborhood of $R$ in $R\Mod$, and this
  equivalence sends $A$ to $R$.  It follows that $\End_\Cc(A)
  = \End_R(R) = R^{op}$.
\end{proof}
Now Proposition~\ref{prop:to-prove} follows from
Propositions~\ref{prop:sureish} and \ref{prop:sureish2}, because
$\Dir_\Cc(A)$ is determined by the neighborhood of $A$ in $\Cc$.  Also,
Proposition~\ref{prop:sureish2} implies the harder
(\ref{aw1})$\implies$(\ref{aw2}) direction of the Artin-Wedderburn
theorem (Theorem~\ref{thm:aw}).\footnote{The easier
  (\ref{aw2})$\implies$(\ref{aw1}) direction follows by Morita
  equivalence, I suppose.}  This is essentially the standard proof of
the Artin-Wedderburn theorem.

\subsection{Examples of directory morphisms}
The next two propositions verify the (intuitively obvious) fact that
Examples~\ref{from-morphism}-\ref{from-functor} are valid.
\begin{proposition}\label{prop:ffs1}
  If $f : A \to B$ is a morphism in an abelian category $\Cc$, then the
  pushforward and pullback maps
  \begin{align*}
    f_* : \Dir_\Cc(A) &\to \Dir_\Cc(B) \\
    f^* : \Dir_\Cc(B) &\to \Dir_\Cc(A)
  \end{align*}
  of Example~\ref{from-morphism} are directory morphisms.
\end{proposition}
\begin{proof}
  It is trivial that $f^*$ and $f_*$ are levelwise order-preserving.
  For $GL_n(K_0)$-equivariance, let $\mu$ be an invertible $n \times
  n$ matrix over $K_0$.  Then there is a commuting diagram
  \begin{equation*}
    \xymatrix{ A^n \ar[r]^{f^{\oplus n}} \ar[d] & B^n \ar[d] \\
      A^n \ar[r]_{f^{\oplus n}} & B^n}
  \end{equation*}
  in which the vertical maps are isomorphisms induced by $\mu$.  Given
  a subobject $X \le A^n$, the two ways of pushing forward to the
  bottom left corner are equal, which shows that
  \begin{equation*}
    (f^{\oplus n})_*(\mu \cdot X) = \mu \cdot (f^{\oplus n})_*(X).
  \end{equation*}
  A similar argument using $\mu^{-1}$ instead of $\mu$, and pullbacks
  instead of pushforwards, shows that
  \begin{equation*}
    (f^{\oplus n})^*(\mu \cdot X) = \mu \cdot (f^{\oplus n})^*(X).
  \end{equation*}
  Finally, for $\oplus$-compatibility, note that for $X \le A^n$ and $Y \le A^m$,
  \begin{equation*}
    (f^{\oplus (n + m)})_*(X \oplus Y) = (f^{\oplus n})_*(X) \oplus (f^{\oplus m})_*(Y)    
  \end{equation*}
  because, choosing a Mitchell embedding, both sides are
  \begin{equation*}
    \{(f(x_1),\ldots,f(x_n),f(y_1),\ldots,f(y_m)) : \vec{x} \in X, ~ \vec{y} \in Y\}.
  \end{equation*}
  Similarly, for $X \le B^n$ and $Y \le B^m$,
  \begin{equation*}
    (f^{\oplus (n + m)})^*(X \oplus Y) = (f^{\oplus n})^*(X) \oplus (f^{\oplus m})^*(Y) 
  \end{equation*}
  because, choosing a Mitchell embedding, both sides are
  \begin{equation*}
    \{(\vec{x},\vec{y}) \in A^n \times A^m : (f(x_1),\ldots,f(x_n))
    \in X \textrm{ and } (f(y_1),\ldots,f(y_m)) \in Y\} \qedhere
  \end{equation*}
\end{proof}

\begin{proposition}\label{prop:ffs2}
  If $F : \Cc \to \Cc'$ is a left-exact functor and $A \in \Cc$, then the map
  \begin{equation*}
    F_* : \Dir_\Cc(A) \to \Dir_{\Cc'}(F(A))
  \end{equation*}
  of Example~\ref{from-functor} is a directory morphism.
\end{proposition}
\begin{proof}
  Fix Mitchell embeddings of $\Cc$ and $\Cc'$, so that we can identify
  subobjects of $A$ with literal subgroups of $A$.  Then the map from
  $\Sub_\Cc(A^n)$ to $\Sub_{\Cc'}(F(A)^n)$ is the map sending a
  subobject $X \subseteq A^n$ to the image of $F(X) \to F(A)^n$.

  The map is order-preserving: suppose $X \subseteq Y \subseteq A^n$.
  Then $F(X) \to F(A)^n$ factors through $F(Y) \to F(A)^n$, so $F_*(X)
  \le F_*(Y)$.

  The map is is $GL_n(K_0)$-equivariant: suppose $X \subseteq A^n$ and
  $\mu$ is an invertible $n \times n$ matrix over $K_0$.  Since $F$ is
  $K_0$-linear, the image of
  \begin{equation*}
    A^n \stackrel{\mu}{\to} A^n
  \end{equation*}
  under $F$ is
  \begin{equation*}
    F(A)^n \stackrel{\mu}{\to} F(A)^n.
  \end{equation*}
  Now $\mu \cdot X$ is the image of the composition
  \begin{equation*}
    X \stackrel{\subseteq}{\to} A^n \stackrel{\mu}{\to} A^n,
  \end{equation*}
  and so $F_*(\mu \cdot X)$ is the image of
  \begin{equation*}
    F(X) \to F(A)^n \stackrel{\mu}{\to} F(A)^n
  \end{equation*}
  The first map is essentially the inclusion of $F_*(X)$ into
  $F(A)^n$, showing that $F_*(\mu \cdot X) = \mu \cdot F_*(X)$.

  The map is compatible with $\oplus$: suppose $X \subseteq A^n$ and
  $Y \subseteq A^m$.  Apply $F$ to the commuting diagram
  \begin{equation*}
    \xymatrix{ X \ar[r] \ar[d] & X \oplus Y \ar[d] & Y \ar[l] \ar[d] \\
      A^n \ar[r] & A^{n+m} & A^m. \ar[l]}
  \end{equation*}
  This yields the following diagram,
  \begin{equation*}
    \xymatrix{ F(X) \ar[r] \ar[d] & F(X \oplus Y) \ar[d] & F(Y) \ar[l] \ar[d] \\
      F(A)^n \ar[r] & F(A)^{n+m} & F(A)^m. \ar[l]}
  \end{equation*}
  The images of the vertical maps are $F_*(X), F_*(X \oplus Y),
  F_*(Y)$, respectively.  Moreover, the vertical maps are inclusions,
  by left-exactness.  Changing the top row by an isomorphism, we get a diagram in which all the vertical maps are inclusions:
  \begin{equation*}
    \xymatrix{ F_*(X) \ar[r] \ar[d] & F_*(X \oplus Y) \ar[d] & F_*(Y) \ar[l] \ar[d] \\
      F(A)^n \ar[r] & F(A)^{n+m} & F(A)^m. \ar[l]}    
  \end{equation*}
  The top row is a coproduct diagram, because $F$ preserves
  coproducts.  Therefore $F_*(X \oplus Y) = F_*(X) \oplus F_*(Y)$.
\end{proof}

\subsection{The theorem on interval directories}\label{app:interval-proof}
\begin{proposition}[= Proposition~\ref{prop:sq}]\label{prop:sq-copy}
  Let $D_\bullet$ be a directory and $a \le b$ be elements of $D_1$.
  Let $D^{[a,b]}_n$ be the interval $[a^{\oplus n},b^{\oplus n}]$
  inside $D_n$.
  \begin{enumerate}
  \item \label{sq1} $D^{[a,b]}_\bullet$ forms a substructure of
    $D_\bullet$, i.e., $D^{[a,b]}_\bullet$ is closed under $\oplus,
    \vee, \wedge,$ and the $GL_\bullet(K_0)$-action.
  \item \label{sq2} The substructure $D^{[a,b]}_\bullet$ is itself a
    directory.
  \item \label{sq3} The inclusion maps $i_n : D^{[a,b]}_n
    \stackrel{\subseteq}{\to} D_n$ form a directory morphism $i :
    D^{[a,b]}_\bullet \to D_\bullet$.
  \item \label{sq4} There is a directory morphism $r : D_\bullet \to
    D^{[a,b]}_\bullet$ given by
    \begin{equation*}
      r_n(x) = (x \vee a^{\oplus n}) \wedge b^{\oplus n} = (x \wedge
      b^{\oplus n}) \vee a^{\oplus n},
    \end{equation*}
    and $r$ is a retract of $i$: $r \circ i$ is the identity map on
    $D^{[a,b]}_\bullet$.
  \item \label{sq5} If $f : D' \to D$ is some directory morphism, then
    $f$ factors through $D^{[a,b]} \to D$ if and only if $f_n(x) \in
    [a^{\oplus n},b^{\oplus n}]$ for all $n$ and all $x \in D'_n$.
  \item \label{sq6} If $D_\bullet = \Dir(C)$ for some object $C$ in an
    abelian category, and if $a, b$ correspond to subobjects $A
    \subseteq B \subseteq C$, then $D^{[a,b]}_\bullet$ is isomorphic
    to $\Dir(B/A)$ via the maps from the isomorphism theorems.
  \end{enumerate}
\end{proposition}
\begin{proof}
  Changing everything by an isomorphism, we may assume we are in the
  setting of (\ref{sq6}).  Let $f : B \hookrightarrow C$ be the
  inclusion and $g : B \twoheadrightarrow B/A$ be the quotient map.
  By Example~\ref{from-morphism}, we get a commutative diagram
  \begin{equation} \label{perseus}
    \xymatrix{ \Dir(B/A) \ar[r]^{g^*} \ar[d]_{id} & \Dir(B)
      \ar[r]^{f_*} \ar[d]_{id} & \Dir(C) \ar[dl]^{f^*} \\ \Dir(B/A) &
      \Dir(B) \ar[l]^{g_*} & }
  \end{equation}
  Thus $r^0 := g_* \circ f^*$ is a retract of $i^0 := f_* \circ g^*$.

  At level $n$, the resulting diagram of functions is
  \begin{equation*}
    \xymatrix{ \Sub(B^n/A^n) \ar[r]^{g^*} \ar[d]_{id} & \Sub(B^n)
      \ar[r]^{f_*} \ar[d]_{id} & \Sub(C^n) \ar[dl]^{f^*} \\ \Sub(B^n/A^n) &
      \Sub(B^n) \ar[l]^{g_*} & }
  \end{equation*}
  By the isomorphism theorems, this diagram is isomorphic to
  \begin{equation}\label{theseus}
    \xymatrix{ [a^{\oplus n},b^{\oplus n}] \ar[r] \ar[d]_{id} &
      [\bot,b^{\oplus n}] \ar[r] \ar[d]_{id} & \Sub(C^n) \ar[dl]
      \\ [a^{\oplus n},b^{\oplus n}] & [\bot,b^{\oplus n}] \ar[l]}
  \end{equation}
  where all the intervals are in $\Sub(C^n)$, the rightward maps are
  inclusions, and the leftward maps are $- \wedge b^{\oplus n}$ and $-
  \vee a^{\oplus n}$.

  Therefore, there is \emph{some} choice of directory structures on
  $D^{[a,b]}_\bullet$ and $D^{[\bot,b]}_\bullet$ such that the maps of
  (\ref{theseus}) assemble into directory morphisms forming a diagram
  isomorphic to (\ref{perseus}):
  \begin{equation*}
    \xymatrix{ D^{[a,b]}_\bullet \ar[r] \ar[d]_{id} &
      D^{[\bot,b]}_\bullet \ar[r] \ar[d]_{id} & \Dir(C) \ar[dl]
      \\ D^{[a,b]}_\bullet & D^{[\bot,b]}_\bullet \ar[l] & }
  \end{equation*}
  Dropping the middle column, we get
  \begin{equation} \label{minos}
    \xymatrix{ D^{[a,b]}_\bullet \ar[r]^i \ar[d]_{id} & \Dir(C) \ar[dl]^r
      \\ D^{[a,b]}_\bullet &  }
  \end{equation}
  where $i$ is the inclusion and $r$ is the map
  \begin{equation*}
    x \mapsto (x \wedge b^{\oplus n}) \vee a^{\oplus n}
  \end{equation*}
  The diagram (\ref{minos}) ensures that $D^{[a,b]}_\bullet$ is an
  induced substructure of $\Dir(C) = D_\bullet$, at least for the
  following parts of the structure:
  \begin{itemize}
  \item The $\oplus$ operator
  \item The $GL_n(K_0)$-action
  \item The partial order $\le$.
  \end{itemize}
  As for $\vee$ and $\wedge$, it is a general fact that any interval
  $[c,d]$ in a lattice $L$ is naturally a sublattice of $L$.

  This proves all the points but (\ref{sq5}).  Let $f : D'_\bullet \to
  D_\bullet$ be a directory morphism.  If $f$ factors through the
  inclusion $i : D^{[a,b]}_\bullet \hookrightarrow D_\bullet$, then
  clearly $f_n(x) \in D^{[a,b]}_n = [a^{\oplus n},b^{\oplus n}]$ for
  all $x \in D'_n$.  Conversely, if $f_n(x) \in D^{[a,b]}_n$ for all
  $x$, then $f = i \circ r \circ f$, so $f$ factors through $i$.
\end{proof}

\subsection{The category $\Hh$}\label{sec:annoy}
As in \S\ref{sec:tdef-setting}, let $\Kk$ be a monster model of a
field, possibly with extra structure.  Let $K_0$ be a small infinite
subfield.  Let $\Lambda_n$ be the lattice of type-definable
$K_0$-linear subspaces of $\Kk$.
\begin{proposition}[= Proposition~\ref{prop:here-is-hh-copy}]\label{prop:here-is-hh}
  There is a $K_0$-linear pre-additive category $\Hh$ in which
  \begin{itemize}
  \item an object is a quotient $A/B$ where $B \le A \in \Lambda_n$
    for some $n$.
  \item a morphism from $A/B$ to $A'/B'$ is a $K_0$-linear function $f
    : A/B \to A'/B'$ such that the set
    \begin{equation*}
      \{(x,y) \in A \times A' : y + B' = f(x + B)\}
    \end{equation*}
    is type-definable.
  \item the composition of $f : A/B \to A'/B'$ and $g : A'/B' \to
    A''/B''$ is the usual composition $g \circ f$.
  \item the $K_0$-vector space structure on $\Hom_\Hh(A/B,A'/B')$ is
    induced by the usual operations, i.e., induced as a subspace of
    $\Hom_{K_0\Vect}(A/B,A'/B')$.
  \end{itemize}
\end{proposition}
\begin{proof}
  If $f : A/B \to A'/B'$ is a function, we call the set
  \begin{equation*}
    T_f = \{(x,y) \in A \times A' : y + B' = f(x + B)\}
  \end{equation*}
  the \emph{trace} of $f$.  We call $f$ an \emph{$\Hh$-morphism} if
  $f$ is $K_0$-linear and the trace of $f$ is type-definable.

  We must check that all the operations on morphisms are well-defined:
  \begin{enumerate}
  \item The identity map on $A/B$ is an $\Hh$-morphism. Its trace is the type-definable set
    \begin{equation*}
      \{(x,y) \in A \times A : x - y \in B\}.
    \end{equation*}
  \item The zero morphism on $A/B$ is an $\Hh$-morphism.  Its trace is
    the type-definable set $A \times B$.
  \item If $f : A/B \to A'/B'$ and $g : A'/B' \to A''/B''$ are
    $\Hh$-morphisms, then so is the composition $g \circ f$.  Indeed,
    the trace of $g \circ f$ is the type-definable set
    \begin{equation*}
      \{(x,z) \in A \times A'' ~|~ \exists y \in A' : ((x,y) \in T_f
      \textrm{ and } (y,z) \in T_g)\}.
    \end{equation*}
  \item If $f : A/B \to A'/B'$ and $g : A/B \to A'/B'$ are parallel
    $\Hh$-morphisms, the sum $f + g$ is an $\Hh$-morphism.  Indeed,
    its trace is the type-definable set
    \begin{equation*}
      \{(x,y) \in A \times A' ~|~ \exists w \in A' : ((x,w) \in T_f
      \textrm{ and } (x,y-w) \in T_g)\}.
    \end{equation*}
  \item If $f : A/B \to A'/B'$ is an $\Hh$-morphism and $\alpha \in
    K_0$, the product $\alpha \circ f$ is an $\Hh$-morphism.  Indeed,
    the trace is the type-definable set
    \begin{equation*}
      \{(x,y) \in A \times A' : (x,\alpha^{-1} \cdot y) \in T_f\}.
    \end{equation*}
  \end{enumerate}
  The associative, distributive, etc. laws are trivial, because they
  hold in $K_0\Vect$.  This shows that $\Hh$ is a $K_0$-linear
  pre-additive category.
\end{proof}
\begin{lemma}
  The category $\Hh$ is additive, i.e., finite products and coproducts
  exist.
\end{lemma}
\begin{proof}
  If $A$ is any element of $\Lambda_1$, such as $A = 0$, then the
  object $A/A$ has the property that its identity and zero
  endomorphisms are equal.  Thus $\Hh$ has a zero object.  It remains
  to verify that binary products/coproducts exist.  (Since we are in a
  pre-additive category, they are equivalent.)  Let $A/B$ and $A'/B'$
  be two objects in $\Hh$, with $A, B \in \Lambda_n$ and $A', B' \in
  \Lambda_m$.  Then $A \times A'$ and $B \times B'$ are two objects in
  $\Lambda_{n +m}$, and $A \times A' \ge B \times B'$.  Therefore $(A
  \times A')/(B \times B')$ is an object of $\Hh$.  It remains to
  produce a diagram
  \begin{equation*}
    \xymatrix{ A/B \ar[dr]^{\iota_1} & & A/B \\ & (A \times A')/(B
      \times B') \ar[ur]^{\pi_1} \ar[dr]_{\pi_2} & \\ A'/B' \ar[ur]_{\iota_2} & &
      A'/B'}
  \end{equation*}
  in $\Hh$ such that the composition $\pi_j \circ \iota_i$ is the
  identity for $i = j$ and the zero morphism for $i \ne j$.  Note that
  $(A \times A')/(B \times B')$ is isomorphic as a $K_0$-vector space
  to $(A/B) \times (A'/B')$.  We take $\iota_1, \iota_2$ to be the
  obvious inclusions and $\pi_1, \pi_2$ to be the obvious projections.

  It remains to show that $\iota_1, \iota_2, \pi_1, \pi_2$ are
  $\Hh$-morphisms.  Note that for $x \in A$ and $(y,z) \in A \times A'$,
  \begin{equation*}
    \iota_1(x + B) = (y,z) + (B \times B') \iff (x,0) + (B \times B')
    = (y,z) + (B \times B')
  \end{equation*}
  Thus the trace of $\iota_1$ is exactly
  \begin{equation*}
    \{(x,y,z) \in A \times A \times A' : x-y \in B \textrm{ and } z \in B'\},
  \end{equation*}
  which is type-definable.  So $\iota_1$ is an $\Hh$-morphism, and
  $\iota_2$ is too, by symmetry.  Similarly, for $(x,y) \in A \times
  A'$ and $z \in A$,
  \begin{equation*}
    \pi_1((x,y) + (B \times B')) = z + B \iff x + B = z + B,
  \end{equation*}
  and so the trace of $\pi_1$ is exactly
  \begin{equation*}
    \{(x,y,z) \in A \times A' \times A : x - z \in B\}.
  \end{equation*}
  Thus $\pi_1$ is an $\Hh$-morphism, and so is $\pi_2$.  This shows
  that the diagram exists and $(A \times A')/(B \times B')$ is truly
  the biproduct of $A/B$ and $A'/B'$.
\end{proof}

\begin{lemma}\label{cool-story-bro}
  Let $f : A/B \to A'/B'$ be an $\Hh$-morphism.  Let $A''$ be the
  kernel of the composition
  \begin{equation*}
    A \twoheadrightarrow A/B \stackrel{f}{\to} A'/B'
  \end{equation*}
  Then $A''$ is type-definable, $A''/B$ is an object in $\Hh$, the
  natural inclusion $A''/B \hookrightarrow A/B$ is an $\Hh$-morphism,
  and the diagram
  \begin{equation*}
    0 \to A''/B \to A/B \stackrel{f}{\to} A'/B'
  \end{equation*}
  is exact in the category of $K_0$-vector spaces.
\end{lemma}
\begin{proof}
  The group $A''$ is type-definable because it is exactly
  \begin{equation*}
    \{x \in A ~|~ \exists y \in B' : (x,y) \in T_f\}.
  \end{equation*}
  The group $A''$ contains $B$ because $B$ is in the kernel of $A
  \twoheadrightarrow A/B$.  Therefore $A''/B$ is an object in $\Hh$.
  The inclusion $A''/B \hookrightarrow A/B$ is an $\Hh$-morphism
  because its trace is the type-definable set
  \begin{equation*}
    \{(x,y) \in A'' \times A : x - y \in B\}.
  \end{equation*}
  Finally, the sequence
  \begin{equation*}
    0 \to A''/B \to A/B \stackrel{f}{\to} A'/B'
  \end{equation*}
  is exact in the category of $K_0$-vector spaces by a trivial diagram
  chase.
\end{proof}
\begin{lemma}
  Let $f : A/B \to A''/B''$ and $g : A'/B' \to A''/B''$ be
  $\Hh$-morphisms, such that $g$ is injective and $f$ factors through
  $g$ set-theoretically.  Let $h : A/B \to A'/B'$ be the unique
  $K_0$-linear map such that $f = g \circ h$.  Then $h$ is an
  $\Hh$-morphism.
\end{lemma}
\begin{proof}
  The trace of $h$ is the type-definable set
  \begin{equation*}
    \{(x,y) \in A \times A' ~|~ \exists z \in A'' : ((x,z) \in T_f
    \textrm{ and } (y,z) \in T_g)\}. \qedhere
  \end{equation*}
\end{proof}
\begin{lemma}
  Let $f : A/B \to A'/B'$ be an $\Hh$-morphism.  Then there is a
  type-definable $B''$ such that $B'' \subseteq A'$, $B'' \supseteq
  B'$, the image of $f$ is $B''/B'$, the quotient map $A'/B'
  \twoheadrightarrow A'/B''$ is an $\Hh$-morphism, and the diagram
  \begin{equation*}
    A/B \stackrel{f}{\to} A'/B' \twoheadrightarrow A'/B'' \to 0
  \end{equation*}
  is exact in the category of $K_0$-vector spaces.
\end{lemma}
\begin{proof}
  Let $B''$ be the type-definable set
  \begin{equation*}
    B'' = \{y \in A' ~|~ \exists x \in A : (x,y) \in T_f\}.
  \end{equation*}
  Then $B''$ is certainly a type-definable subgroup of $A'$,
  containing $B'$ (take $x = 0 \in A$).  It is clear that $B''/B'$ is
  the image of $f : A/B \to A'/B'$.  A trivial diagram chase shows
  that
  \begin{equation*}
    A/B \to A'/B' \to A'/B'' \to 0
  \end{equation*}
  is exact.  Lastly, the quotient map $A'/B' \to A'/B''$ is an
  $\Hh$-morphism because its trace is exactly the type-definable set.
  \begin{equation*}
    \{(x,y) \in A' \times A' : x - y \in B''\} \qedhere
  \end{equation*}
\end{proof}
\begin{lemma}
  Let $f : A/B \to A'/B'$ be a surjective $\Hh$-morphism.  Let $g :
  A/B \to A''/B''$ be another $\Hh$-morphism such that $g = h \circ f$
  for some $K_0$-linear $h : A'/B' \to A''/B''$.  Then $h$ is an
  $\Hh$-morphism.
\end{lemma}
\begin{proof}
  By surjectivity of $f$, the trace of $h$ is exactly the
  type-definable set
  \begin{equation*}
    \{(y,z) \in A' \times A'' ~|~ \exists x \in A : ((x,y) \in T_f
    \textrm{ and } (x,z) \in T_g)\}. \qedhere
  \end{equation*}
\end{proof}

\begin{proposition}
  The category $\Hh$ is abelian and the forgetful functor $\Hh \to
  K_0\Vect$ is exact.
\end{proposition}
\begin{proof}
  We use the criterion of Lemma~\ref{lem:reflector}.  The previous
  four lemmas verify all the necessary conditions.
\end{proof}

\begin{lemma}
  Let $A/B$ be an object in $\Hh$, with $A, B \in \Lambda_n$.  Suppose
  $A' \in [B,A] \subseteq \Lambda_n$.  Then $A'/B$ is an object in
  $\Hh$, and the inclusion $A'/B \hookrightarrow A/B$ is a
  monomorphism in $\Hh$.  The induced map
  \begin{align*}
    [B,A] & \to \Sub_\Hh(A/B) \\
    A' & \mapsto A'/B
  \end{align*}
  is an isomorphism of lattices.
\end{lemma}
\begin{proof}
  It is trivial that $A'/B$ is an object of $\Hh$.  The inclusion
  $A'/B \hookrightarrow A/B$ is an $\Hh$-morphism because its trace is
  the type-definable set
  \begin{equation*}
    \{(x,y) \in A' \times A : x - y \in B\}.
  \end{equation*}
  The inclusion is an $\Hh$-monomorphism because the forgetful functor
  $\Hh \to K_0\Vect$ reflects monomorphisms, by virtue of being
  faithful.  Thus $A'/B$ yields a subobject of $A/B$.

  We claim that all subobjects of $A/B$ arise this way.  Because $\Hh$
  is abelian, every subobject of $A/B$ arises as the kernel of some
  morphism $A/B \to A''/B''$.  In Lemma~\ref{cool-story-bro} we found
  $A' \in [A,B]$ such that the diagram
  \begin{equation*}
    0 \to A'/B \to A/B \to A''/B''
  \end{equation*}
  becomes exact after applying the forgetful functor.  But the
  forgetful functor is faithful and exact, so the diagram is
  \emph{already} exact in $\Hh$.  Therefore the subobject in question
  is represented by $A'/B$.  Thus every $\Hh$-subobject of $A/B$ has
  the form $A'/B$, and the map $[B,A] \to \Sub_\Hh(A/B)$ is onto.

  It remains to show that $[B,A] \to \Sub_\Hh(A/B)$ is strictly
  order-preserving.  Because the forgetful functor is faithful and
  exact, the induced map $\Sub_\Hh(A/B) \to \Sub_{K_0\Vect}(A/B)$ is
  strictly order-preserving.  Therefore it suffices to show that for
  $A', A'' \in [B,A]$,
  \begin{equation*}
    A' \subseteq A'' \iff A'/B \subseteq A''/B.
  \end{equation*}
  This is trivial.
\end{proof}

We put everything together into the following theorem
\begin{theorem}[= Theorem~\ref{thm:h-lambda}]\label{thm:h-lambda-copy}
  Let $\Hh$ be the category of Proposition~\ref{prop:here-is-hh}.
  \begin{itemize}
  \item $\Hh$ is a $K_0$-linear abelian category.
  \item The forgetful functor $\Hh \to K_0\Vect$ is a $K_0$-linear
    exact functor.
  \item $\Lambda_\bullet$ is isomorphic to $\Dir_\Hh(\Kk)$, and is
    therefore a directory.
  \end{itemize}
\end{theorem}

\begin{remark}
  The technique used to prove Theorem~\ref{thm:h-lambda-copy}
  probably generalizes to prove the following:
\end{remark}
\begin{proposition-eh}\label{winner}
  Let $R$ be a $K_0$-algebra and $M$ be an $R$-module.  For every $n$,
  let $D_n$ be some bounded sublattice of $\Sub_M(R)$.  Suppose that
  for any $m \times n$ matrix $\mu$ with coefficients from $K_0$, the
  structure $\Lambda_\bullet$ is closed under pushforward and pullback
  along the map $M^n \to M^m$ induced by $\mu$:
  \begin{align*}
    A \in D_n & \implies \{\mu \cdot \vec{x} : \vec{x} \in A\} \in D_m \\
    A \in D_m & \implies \{ \vec{x} \in M^n : \mu \cdot \vec{x} \in A\} \in D_n.
  \end{align*}
  Then $D_\bullet$ is a substructure of $\Dir_R(M)$, and $D_\bullet$
  is a directory.
\end{proposition-eh}

\section{Further remarks on directories}\label{app:dirs2}
We discuss
\begin{itemize}
\item How $\End(A)$ is uniformly interpretable in $\Dir(A)$.
\item How $\Dir(A)$ is uniformly interpretable in $\End(A)$ when $A$
  is semisimple.
\item Why the class of directories is probably an elementary class.
\end{itemize}
The first two points yield the intuition that $\Dir(A)$ is somehow a
``generalized endomorphism ring.''

\subsection{Interpreting the endomorphism ring}
Let $A$ be an object in a $K_0$-linear abelian category.
\begin{proposition}\label{prop:interpret-end}
  The structure $\Dir(A)$ determines the $K_0$-algebra $\End(A)$.  In
  fact, $\End(A)$ is interpretable in $\Dir(A)$.
\end{proposition}
\begin{proof}
  Let $\Dir(A) = (D_1,D_2,\ldots)$ be given as an abstract structure.
  Let $\bot_n$ and $\top_n$ denote the bottom and top elements of
  $D_n$.  (So $\bot_n = 0$ and $\top_n = A^n$.)  Set
  \begin{align*}
    X &= \top_1 \oplus \bot_1 \in D_2 \\
    Y &= \bot_1 \oplus \top_1 \in D_2.
  \end{align*}
  (So $X = A \oplus 0$ and $Y = 0 \oplus A$.)  Set
  \begin{equation*}
    R = \{\Gamma \in D_2 ~|~ \Gamma \wedge Y = \bot_2,~ \Gamma \vee Y = \top_2\}.
  \end{equation*}
  (So $R$ is the set of graphs of endomorphisms on $A$.)
  Given $\Gamma, \Gamma' \in R$, define $f(\Gamma,\Gamma') \in D_3$ by
  \begin{equation*}
    f(\Gamma,\Gamma') = (\top_1 \oplus \Gamma') \wedge (\Gamma \oplus \top_1)
  \end{equation*}
  (So
  \begin{equation*}
    f(\Gamma,\Gamma') = \{(x, \varphi(x), \varphi'(\varphi(x))) ~|~ x \in A\}
  \end{equation*}
  where $\Gamma$ and $\Gamma'$ are the graphs of $\varphi, \varphi'
  \in \End(A)$, respectively.)  Let $\nu \in GL_3(K_0)$ be the matrix
  swapping the second and third coordinates:
  \begin{equation*}
    \nu = 
    \begin{pmatrix}
      1 & 0 & 0 \\
      0 & 0 & 1 \\
      0 & 1 & 0
    \end{pmatrix}
  \end{equation*}
  Define $\Gamma' \circ \Gamma$ to be
  the unique $\Gamma'' \in R$ such that
  \begin{equation*}
    \nu \cdot (\Gamma'' \oplus \top_1) = f(\Gamma, \Gamma') \vee (\bot_1 \oplus \top_1 \oplus \bot_1),
  \end{equation*}
  (So $\Gamma''$ is the graph of $\varphi' \circ \varphi$.)
  If $\Gamma, \Gamma' \in R$, let
  \begin{equation*}
    g(\Gamma,\Gamma') := (\Gamma \oplus \top_1) \wedge (\nu \cdot
    (\Gamma' \oplus \top_1)),
  \end{equation*}
  where $\nu$ is the matrix swapping second and third coordinates, as
  above. (So
  \begin{equation*}
    g(\Gamma,\Gamma') = \{(x, \varphi(x), \varphi'(x)) ~|~ x \in A\}.
  \end{equation*}
  where $\Gamma, \Gamma'$ are the graphs of $\varphi, \varphi'$,
  respectively.)  Let $\rho \in GL_3(K_0)$ be the matrix sending
  $(x,y,z) \mapsto (x,y+z,z)$.  Let
  \begin{equation*}
    h(\Gamma,\Gamma') = (\rho \cdot g(\Gamma, \Gamma')) \vee (\bot_2
    \oplus \top_1).
  \end{equation*}
  (So
  \begin{equation*}
    h(\Gamma,\Gamma') = \{(x, \varphi(x) + \varphi'(x), y) ~|~ x,y \in A\}.
  \end{equation*}
  holds.)  Let $\Gamma + \Gamma'$ be the unique $\Gamma'' \in R$ such
  that
  \begin{equation*}
    \Gamma'' \oplus \top_1 = h(\Gamma,\Gamma').
  \end{equation*}
  (So $\Gamma''$ is the graph of $\varphi + \varphi' \in \End(A)$.)
  Finally, if $\alpha \in K_0$, define $\iota_\alpha \in R$ to be
  \begin{equation*}
    \lambda_\alpha \cdot (\top_1 \oplus \bot_1),
  \end{equation*}
  where $\lambda_\alpha \in GL_2(K_0)$ sends $(x,y) \mapsto (x,y+
  \alpha x)$.  (So $\iota_\alpha$ is
  \begin{equation*}
    \iota_\alpha = \{(x, \alpha x) ~|~ x \in A\},
  \end{equation*}
  the graph of $\alpha$ times the identity endomorphism $id_A
  \in \End(A)$.)

  Thus, we have recovered the set $\End(A)$, the ring structure on
  $\End(A)$, and the $K_0$-algebra structure map $K_0 \to \End(A)$,
  using only the pure directory structure.
\end{proof}

For semisimple directories, this is in fact a bi-interpretation:
\begin{proposition-eh}\label{prop:semisimple-bi-interp}
  If $A$ is semisimple, then $\Dir_\Cc(A)$ and $\End_\Cc(A)$ are
  bi-interpretable.
\end{proposition-eh}
\begin{proof}[Proof sketch]
  We interpret $\End_\Cc(A^n)$ as the matrix ring over $\End_\Cc(A)$.
  The following claim shows that $\End_\Cc(A^n)$ interprets
  $\Sub_\Cc(A^n)$.
  \begin{claim}
    If $B \in \Cc$ is semisimple, then $\Sub_\Cc(B)$ is interpretable
    in $\End_\Cc(B)$.
  \end{claim}
  \begin{claimproof}
    Let $I$ be the set of idempotent elements in $\End_\Cc(B)$.  Note
    that for $u_1, u_2 \in I$, we have
    \begin{equation*}
      u_2 \circ u_1 = u_1 \iff \img(u_1) \subseteq \img(u_2).
    \end{equation*}
    The $\implies$ direction is straightforward.  Conversely, if
    $\img(u_1) \subseteq \img(u_2)$, then for any $x \in B$,
    \begin{equation*}
      u_1(x) \in \img(u_1) \subseteq \img(u_2),
    \end{equation*}
    so that $u_1(x) = u_2(u_1(x))$.

    Furthermore, every subobject of $B$ is of the form $\img(u)$ for
    some $u \in I$, by semisimplicity.  Therefore the relation
    \begin{equation*}
      u_1 \le u_2 \iff u_2 \circ u_1 = u_1
    \end{equation*}
    defines a pre-order on $I$ and the induced poset is $\Sub_\Cc(B)$.
  \end{claimproof}
  Thus we can interpret the lattice $\Sub_\Cc(A^n)$ in $\End_\Cc(A)$.
  We leave the remaining details as an exercise to the reader.
\end{proof}

\subsection{Elementarity}\label{sec:directory-speculation}

Let $A$ be an object in a $K_0$-linear abelian category $\Cc$, and let
$\Cc'$ be the neighborhood of $A$, i.e., the category of subquotients
of finite powers of $A$.  It appears that $\Cc'$ is determined up to
equivalence by $\Dir_\Cc(A)$, using the methods of
Propopsition~\ref{prop:interpret-end}.

More precisely, for every $n$ let $\Cc'_n$ be the category whose
\begin{itemize}
\item \emph{objects} are triples $(X,Y,m)$ where $m \le n$, where $X,
  Y \in \Sub(A^m)$, and $X \ge Y$.
\item \emph{morphisms} from $(X,Y,m)$ to $(X',Y',m')$ are
  $\Cc$-morphisms from $X/Y$ to $X'/Y'$.
\end{itemize}
Then $\Cc'_n$ is equivalent to the full subcategory of $\Cc$
consisting of the subquotients of $A^n$.  Moreover, $\Cc'_m$ is a full
subcategory of $\Cc'_n$ for $m \le n$ and $\Cc'$ is equivalent to the
direct limit
\begin{equation*}
  \varinjlim_n \Cc'_n.
\end{equation*}
Now, using the techniques of Proposition~\ref{prop:interpret-end},
each $\Cc'_n$ should be interpretable in $\Dir_\Cc(A)$---uniformly
across all $\Cc$ and $A$.  Up to equivalence, $\Cc'$ is thus
``ind-interpretable.''

Assuming things work out, we provisionally make the following
definition:
\begin{definition}\label{def:neighborhood-directory}
  If $D_\bullet$ is an abstract directory, the \emph{neighborhood of
    $D_\bullet$} is the neighborhood of $A \in \Cc$, for any abelian
  category $\Cc$ and object $A \in \Cc$ such that $D_\bullet \cong
  \Dir_\Cc(A)$.
\end{definition}
This should be well-defined up to equivalence, and comes with a
canonical object $A$ whose directory is $D_\bullet$.

Another likely consequence is
\begin{proposition-eh}\label{prop:elementary-directories}
  The class of directories is elementary.
\end{proposition-eh}
\begin{proof}[Proof sketch]
  Let $(D_1,D_2,\ldots)$ be a multi-sorted structure of the
  appropriate signature.  Roughly speaking, we can assert that
  $D_\bullet$ is a directory by interpreting the neighborhood and
  asserting that the neighborhood is a $K_0$-linear abelian category.

  Unfortunately, one needs a more intricate argument, because the
  neighborhood is merely ind-interpretable, rather than fully
  interpretable.  One approach is to replace the category $\Cc'$ with
  the multi-sorted structure $(\Cc'_1,\Cc'_2,\ldots)$, and write down
  axioms ensuring that
  \begin{equation*}
    \varinjlim_n \Cc'_n
  \end{equation*}
  is a $K_0$-linear abelian category.  One needs axioms like the
  following:
  \begin{itemize}
  \item Each $\Cc'_n$ is a $K_0$-linear pre-additive category.
  \item Each inclusion functor $\Cc'_n \to \Cc'_{n+1}$ is fully
    faithful, $K_0$-linear, and injective on objects.
  \item If $X, Y \in \Cc'_n$, then $X \oplus Y$ exists in $\Cc'_{2n}$.
  \item If $f : X \to Y$ is a morphism in $\Cc'_n$, then $\ker(f)$ and
    $\coker(f)$ exist in $\Cc'_n$.
  \item The inclusion functors preserve kernels and cokernels.
  \item If $f : X \to Y$ is a morphism in $\Cc'_n$, then the canonical
    map from $\coker(\ker(f))$ to $\ker(\coker(f))$ is an isomorphism.
  \end{itemize}
  One should also add a constant for the element $A = (A,0,1) \in
  \Cc'_1$.  Call such a structure a ``gadget.''  Then 
  \begin{itemize}
  \item The class of gadgets should be elementary.
  \item There should be a uniform way to interpret gadgets in
    directories, using the techniques of
    Proposition~\ref{prop:interpret-end}.
  \item Given a gadget $(\Cc'_1,\Cc'_2,\ldots)$ with selected object
    $A \in \Cc'_1$, the directory
    \begin{equation*}
      \Dir_{\varinjlim_n \Cc'_n}(A)
    \end{equation*}
    should be interpretable in a straightforward, uniform fashion.
  \item These should combine to show that every directory is
    bi-interpretable with a gadget\footnote{But not vice versa.} in a
    uniform fashion.
  \end{itemize}
  Finally, one can assert that a general structure $D_\bullet$ is a
  directory by attempting to interpret the resulting gadget and the
  resulting directory.  As long as the resulting gadget satisfies the
  gadget axioms, and as long as the resulting directory $D'_\bullet$
  is isomorphic to $D_\bullet$ in the expected way, we can conclude
  that $D_\bullet$ is a directory.
\end{proof}
\subsection{An incomplete axiomatization}
It would be nice to have a more explicit and concise set of axioms for
directories, however.  One candidate list is the following:
\begin{itemize}
\item Each $D_i$ should be a bounded modular lattice.
\item The $\oplus$ operation should be associative.
\item The map $D_n \times D_m \to D_{n + m}$ should be an injective
  morphism of bounded lattices.
\item For any $n, m$, the map
  \begin{align*}
    D_n &\to D_m \\
    V & \mapsto V \oplus 0^m
  \end{align*}
  should be an isomorphism from $D_n$ onto an interval in $D_m$.
\item The $GL_n(K_0)$-action should preserve the lattice structure.
\item If $\mu_1 \in GL_n(K_0)$ and $\mu_2 \in GL_m(K_0)$ and $\mu_3$
  is the block matrix
  \begin{equation*}
    \begin{pmatrix}
      \mu_1 & 0 \\ 0 & \mu_2
    \end{pmatrix}
  \end{equation*}
  then the identity should hold:
  \begin{equation*}
    (\mu_1 \cdot X) \oplus (\mu_2 \cdot Y) = \mu_3 \cdot (X \oplus Y)
  \end{equation*}
\item If $X \in D_n, Y \in D_m$ and $\tau$ is the block matrix
  \begin{equation*}
    \begin{pmatrix}
      0 & I_m \\ I_n & 0
    \end{pmatrix}
  \end{equation*}
  then the identity should hold:
  \begin{equation*}
    \tau \cdot (X \oplus Y) = Y \oplus X.
  \end{equation*}
\end{itemize}
Unfortunately, this list is incomplete.  One can show (again using the
methods of Proposition~\ref{prop:interpret-end}) that the $GL_n(\Zz)$
action is determined by the action of the symmetric group
$\mathcal{S}_n$.  Given a genuine directory $D_\bullet =
(D_1,D_2,\ldots)$, we can create another structure $D'_\bullet$ in
which the $GL_n(K_0)$-action is twisted by the automorphism
\begin{equation*}
  \mu \mapsto (\mu^{-1})^T.
\end{equation*}
This automorphism fixes permutation matrices, but is non-trivial for
$n > 1$.  We can find a directory $D_\bullet$ for which the
$GL_2(\Zz)$ action on $D_2$ is faithful.  The resulting twisted
structure $D'_\bullet$ seems to satisfy the axioms listed above.  But
$D'_\bullet$ cannot be a directory, because
\begin{itemize}
\item $D'_\bullet$ and $D_\bullet$ have the same $\oplus$-structure,
  lattice structure, and $\mathcal{S}_n$-action.
\item $D'_\bullet$ and $D_\bullet$ have distinct $GL_n(\Zz)$-actions.
\item For genuine directories, the $GL_n(\Zz)$-action is determined by
  the $\oplus$-structure, lattice structure, and
  $\mathcal{S}_n$-action.
\end{itemize}

\subsection{Another approach}
If $A$ is an object in a $K_0$-linear abelian category, define
$\Dir^+(A)$ to be
\begin{equation*}
  (\Sub(A),\Sub(A^2),\Sub(A^3),\ldots)
\end{equation*}
with the following structure:
\begin{itemize}
\item For each $n$, the bounded lattice structure on $\Sub(A^n)$.
\item For any $m \times n$ matrix $\mu$ over $K_0$, the pullback and
  pushforward maps
  \begin{align*}
    \mu^* : \Sub(A^m) &\to \Sub(A^n) \\
    \mu_* : \Sub(A^n) &\to \Sub(A^m)
  \end{align*}
  along the morphism $A^n \to A^m$ induced by $\mu$.
\end{itemize}
Restricting $\mu$ to invertible matrices, one recovers the
$GL_n(K_0)$-action.  The $\oplus$ operation
\begin{equation*}
  \oplus : \Sub(A^n) \times \Sub(A^m) \to \Sub(A^{n+m})
\end{equation*}
is determined as well:
\begin{equation*}
  V \oplus W = (\pi_1^* V) \wedge (\pi_2^* W),
\end{equation*}
where $\pi_1, \pi_2$ are the matrices corresponding to the coordinate projections
\begin{align*}
  K_0^{n+m} &\twoheadrightarrow K_0^n \\
  K_0^{n+m} &\twoheadrightarrow K_0^m.  
\end{align*}
So $\Dir(A)$ is interpretable in $\Dir^+(A)$.  Conversely, $\Dir^+(A)$
is interpretable in $\Dir(A)$ as follows:
\begin{enumerate}
\item Let $\mu$ be a given $m \times n$ matrix over $K_0$.
\item Let $\mu'$ be the invertible block triangular matrix
  \begin{equation*}
    \mu' = 
    \begin{pmatrix}
      I_m & \mu \\ 0 & I_n.
    \end{pmatrix}
  \end{equation*}
\item Let $f$ and $f'$ be the morphisms
  \begin{align*}
    f &: A^n \to A^m \\
    f' &: A^{m+n} \to A^{m+n}
  \end{align*}
  induced by $\mu, \mu'$.  Then
  \begin{equation*}
    f'(\vec{x},\vec{y}) = (\vec{x} + f(\vec{y}), \vec{y})
  \end{equation*}
  for $\vec{x} \in A^m$ and $\vec{y} \in A^n$.
\item Let $X$ be a given subobject of $A^n$.  Then
  \begin{equation*}
    \mu' \cdot (0^m \oplus X) = f'_*(0^m \oplus X) =
    \{(f(\vec{y}),\vec{y}) : \vec{y} \in X\}
  \end{equation*}
  \begin{equation*}
    (\mu' \cdot (0^m \oplus X)) \vee (0^m \oplus X) = \{(f(\vec{y}),
    \vec{z}) : \vec{y}, \vec{z} \in X\} = f_*(X) \oplus A^n.
  \end{equation*}
  Then $f_*(X)$ is the unique $Y \in \Sub(A^m)$ such that
  \begin{equation*}
    (\mu' \cdot (0^m \oplus X)) \vee (0^m \oplus X) = Y \oplus A^n.
  \end{equation*}
  This shows that $f_* : \Sub(A^n) \to \Sub(A^m)$ is interpretable in
  $\Dir(A)$.  Then $f^* : \Sub(A^n) \to \Sub(A^m)$ is also
  interpretable, because it is characterized by the Galois connection
  with $f_*$:
  \begin{equation*}
    \forall X \in \Sub(A^n),~ Y \in \Sub(A^m) : \left( f_*(X) \le Y
    \iff X \le f^*(Y). \right)
  \end{equation*}
\end{enumerate}
Thus $\Dir(A)$ and $\Dir^+(A)$ are bi-interpretable.  Say that a
structure $D_\bullet$ is an \emph{extended directory} if it is
isomorphic to one of the form $\Dir^+_\Cc(A)$.  Extended directories
are equivalent to directories.  We opted to use directories rather
than extended directories because most of the directory morphisms we
are interested in fail to preserve the extended directory structure.
\begin{example}
  Let $f : A \to B$ be a morphism.  Consider the commutative diagram,
  in which the horizontal maps are the projections onto the first
  coordinate, induced by the $1 \times 2$ matrix $\begin{pmatrix} 1 &
    0 \end{pmatrix}$:
  \begin{equation*}
    \xymatrix{ A^2 \ar[r]^{\pi_A} \ar[d]_{f^{\oplus 2}} & A \ar[d]^f \\ B^2 \ar[r]^{\pi_B} & B.}
  \end{equation*}
  Then for general $X \in \Sub(A)$,
  \begin{equation*}
    \pi_B^* f_* X \ne f^{\oplus 2}_* \pi_A^* X.
  \end{equation*}
  Therefore the pushforward map $f_* : \Dir^+(A) \to \Dir^+(B)$ does
  not preserve the added structure of $\Dir^+(-)$.
\end{example}
On the other hand, the class of extended directories appears to be
closed under the following:
\begin{enumerate}
\item \label{z1} Ultraproducts, because extended directories are the image of an
  elementary class (pointed abelian categories) under an
  interpretation.
\item \label{z2} Substructures, by Proposition~\ref{winner} and Mitchell embedding
\item \label{z3} Quotients, by a Serre quotient construction similar
  to \S\ref{sec:tdef00-setting}.
\item \label{z4} Products, by \S\ref{sec:products}, which should
  generalize to infinite products.  An infinite product of abelian
  categories is an abelian category, by Remark~8.3.6(i) in
  \cite{cat-sheaves}.
\end{enumerate}
If points (\ref{z2}-\ref{z4}) hold, then the class of extended
directories is cut out by a universal equational theory.  If points
(\ref{z1}-\ref{z2}) hold, the class is cut out by a universal theory.
Either way, the classes of extended directories and plain directories
would therefore be elementary.

Here is a candidate axiomatization for extended directories:
\begin{conjecture}\label{oho}
  Let $D_\bullet$ be a structure $(D_0,D_1,D_2,\ldots)$ with a poset
  structure on each $D_i$, and connecting maps
  \begin{align*}
    \mu^* : \Sub(A^m) &\to \Sub(A^n) \\
    \mu_* : \Sub(A^n) &\to \Sub(A^m)
  \end{align*}
  for every $m \times n$ matrix over $K_0$.  Suppose the following
  axioms hold:
  \begin{enumerate}
  \item Each $D_n$ is a bounded modular lattice.
  \item \label{a2} $D_0$ is the trivial one-element modular lattice.
  \item The maps $\mu^*$ and $\mu_*$ are order-preserving:
    \begin{align*}
      x \le y & \implies \mu^*(x) \le \mu^*(y) \\
      x \le y & \implies \mu_*(x) \le \mu_*(y).
    \end{align*}
  \item The maps $\mu^*$ and $\mu_*$ depend functorially on $\mu$:
    \begin{align*}
      (\mu \cdot \nu)^* &= \mu^* \circ \nu^* \\
      (\mu \cdot \nu)_* &= \nu_* \circ \mu_* \\
      (I_n)^* &= id_{D_n} \\
      (I_n)_* &= id_{D_n}.
    \end{align*}
  \item The maps $\mu^*$ and $\mu_*$ have a Galois connection:
    \begin{equation*}
      \mu_*(x) \le y \iff x \le \mu^*(y).
    \end{equation*}
  \end{enumerate}
  Then $D$ is an extended directory, i.e., $D \cong \Dir^+(A)$ for
  some object $A$ in a $K_0$-linear abelian category.
\end{conjecture}
If this conjecture is true, one could view a directory as a functor
from $K_0\Vect^f$ to a category of bounded modular lattices with
Galois connections as morphisms.

On the other hand, Axiom~\ref{a2} feels out of place, and suggests
that we need some additional axioms for compatibility with $\oplus$.

\section{A note on ranks in abelian categories}\label{app:ranks}
Let $\Cc$ be an abelian category.  Let $\rk : \Cc \to \Zz_{\ge 0}$ be
a function assigning each object in $\Cc$ a non-negative integer rank.
Say that $\rk$ is a \emph{weak rank} if the following conditions hold:
\begin{enumerate}
\item If $f : A \to B$ is an epimorphism, then $\rk(A) \ge \rk(B)$.
\item If $f : A \to B$ is a monomorphism, then $\rk(A) \le \rk(B)$.
\item $\rk(A \oplus B) \ge \rk(A) + \rk(B)$.
\item $\rk(A) = 0 \iff A \cong 0$.
\end{enumerate}
Say that $\rk$ is a \emph{strong rank} if the following additional
condition holds: for any short exact sequence
\begin{equation*}
  0 \to A \to B \to C \to 0
\end{equation*}
we have
\begin{equation*}
  \rk(B) \le \rk(A) + \rk(C).
\end{equation*}
A rank function $\rk(-)$ is a strong rank if and only if it satisfies
the properties of Proposition~\ref{redrk}.  In particular, reduced
rank is a strong rank in a cube-bounded category.
\begin{proposition}\label{prop:motiv1}
  Let $\Cc$ be an abelian category and $\rk : \Cc \to \Zz_{\ge 0}$ be
  a weak rank.  Then $\Cc$ is cube-bounded and
  \begin{equation*}
    \redrk(A) \le \rk(A)
  \end{equation*}
  for any $A$.
\end{proposition}
\begin{proof}
  Same as Claim~\ref{genug-claim} in the proof of
  Proposition~\ref{prop:genug}.
\end{proof}
With Proposition~\ref{redrk}, this implies the following result, which
explains the name ``reduced rank:''
\begin{proposition}
  Let $\Cc$ be an abelian category.  The following are equivalent:
  \begin{enumerate}
  \item $\Cc$ is cube-bounded (Definition~\ref{def:redrk}).
  \item\label{e2} $\Cc$ admits a weak rank.
  \item\label{e3} $\Cc$ admits a strong rank.
  \end{enumerate}
  If the equivalent conditions hold, then $\redrk(-)$ is the smallest
  weak rank and the smallest strong rank.
\end{proposition}
In particular, there is some way to upgrade weak ranks (like burden)
into strong ranks (like dp-rank, reduced rank).

Here is a more general statement in this direction:
\begin{proposition}
  Let $\overline{\Nn}$ be the extended natural numbers $\Zz_{\ge 0}
  \cup \{+\infty\}$.  Let $\Cc$ be an abelian category.  Let $\star : \overline{\Nn}
  \times \overline{\Nn} \to \overline{\Nn}$ be a function such that
  \begin{equation*}
    0 \star 0 = 0
  \end{equation*}
  \begin{equation*}
    x, y < \infty \implies x \star y < \infty.
  \end{equation*}
  Let $\rk : \Cc \to \overline{\Nn}$ be a function assigning each
  object $A \in \Cc$ a rank $\rk(A) \in \overline{\Nn}$, satisfying
  the following properties:
  \begin{enumerate}
  \item Given a short exact sequence
    \begin{equation*}
      0 \to A \to B \to C \to 0,
    \end{equation*}
    we have
    \begin{equation*}
      \max(\rk(A),\rk(C)) \le \rk(B) \le \rk(A) \star \rk(B).
    \end{equation*}
  \item For any $A, B$,
    \begin{equation*}
      \rk(A \oplus B) \ge \rk(A) + \rk(B).
    \end{equation*}
  \end{enumerate}
  Then there is another rank $\rk' : \Cc \to \overline{\Nn}$
  satisfying the same properties with $+$ instead of $\star$;
  moreover
  \begin{align*}
    \rk'(A) = 0 &\iff \rk(A) = 0 \\
    \rk'(A) = \infty &\iff \rk(A) = \infty.
  \end{align*}
\end{proposition}
\begin{proof}
  Let $\Cc'$ be the (Serre) subcategory of objects of finite rank.
  Let $\Cc''$ be the Serre subcategory of objects of rank 0.  Then
  $\rk$ induces a weak rank on the Serre quotient $\Cc'/\Cc''$.
  Therefore $\Cc'/\Cc''$ has a strong rank, which pulls back to an
  $\Nn$-valued rank on $\Cc'$ having the desired properties.  This
  extends to an $\overline{\Nn}$-valued rank on $\Cc$ by setting rank
  to $\infty$ outside of $\Cc'$.
\end{proof}
For example, if $\Dd$ is the category of interpretable abelian groups
in some structure, then $\bdn : \Dd \to \overline{\Nn}$ satisfies the
assumptions, with
\begin{equation*}
  x \star y := x + y + xy,
\end{equation*}
by the sub-multiplicativity of burden proven in (\cite{Ch}, Corollary
2.6).

Therefore there is a rank $\rk : \Dd \to \overline{\Nn}$ with all the
good sub-additivity properties of dp-rank, such that
\begin{equation*}
  \rk(A) < \infty \iff \left(\bdn(A) \textrm{ is finite.}\right)
\end{equation*}

\bibliographystyle{plain} \bibliography{mybib}{}

\end{document}